\UseRawInputEncoding
\documentclass[11pt,reqno]{amsart}
\RequirePackage[OT1]{fontenc}
\usepackage{palatino}
    \usepackage{graphicx}   
    \usepackage{multirow}   
    \usepackage{amsthm,amsmath,amsfonts,amssymb,color,stmaryrd,dsfont} 
    \usepackage{bm}
\usepackage{ulem}
\usepackage{tikz}
\usepackage{hyperref,url}
\usepackage{graphicx} 
\usepackage{stmaryrd}
\usepackage{float} 
\usepackage{subfigure,caption} 
\usepackage[a4paper, total={6.7in, 8.5in}]{geometry}

\begin{document} 

\newcommand{\bbA}{{\bf A}}
	\newcommand{\cbbA}{{\check \bbA}}
	\newcommand{\mbbA}{{\mathcal A}}
	\newcommand{\bbbA}{{\bar \mbbA}}
	\newcommand{\hbbA}{{\widehat \bbA}}
	\newcommand{\tbbA}{{\tilde \bbA}}
	\newcommand{\bba}{{\bf a}}
	\newcommand{\bbB}{{\bf B}}
	\newcommand{\bbC}{{\bf C}}
	\newcommand{\bbc}{{\bf c}}
	\newcommand{\bbD}{{\bf D}}
	\newcommand{\bbd}{{\bf d}}
	\newcommand{\bbe}{{\bf e}}
	\newcommand{\bbE}{{\bf E}}
	\newcommand{\rE}{{\rm E}}
	\newcommand{\bbf}{{\bf f}}
	\newcommand{\bbF}{{\bf F}}
	\newcommand{\bbP}{{P}}
	\newcommand{\bbg}{{\bf g}}
	\newcommand{\bbG}{{G}}
	\newcommand{\bbK}{{\bf K}}
	\newcommand{\bbH}{{H}}
	\newcommand{\bbh}{{\bf h}}
	\newcommand{\bbw}{{\bf w}}
	\newcommand{\bbI}{{\bf I}}
	\newcommand{\bbi}{{\bf i}}
	\newcommand{\bbj}{{\bf j}}
	\newcommand{\bbJ}{{\bf J}}
	\newcommand{\bbk}{{\bf k}}
	\newcommand{\bbl}{{\bf 1}}
	\newcommand{\bbM}{{\bf M}}
	\newcommand{\bbm}{{\bf m}}
	\newcommand{\bbN}{{\bf N}}
	\newcommand{\bbn}{{\bf n}}
	\newcommand{\bbQ}{{Q}}
	\newcommand{\bbq}{{\bf q}}
	\newcommand{\bbO}{{\bf O}}
	\newcommand{\bbR}{{\bf R}}
	\newcommand{\bbr}{{\bf r}}
	\newcommand{\bbs}{{\bf s}}
	\newcommand{\cbbs}{{\check \bbs}}
	\newcommand{\hbbs}{{\hat \bbs}}
	\newcommand{\tbbs}{{\tilde \bbs}}
	\newcommand{\bbS}{{\bf S}}
	\newcommand{\cbbS}{{\check \bbS}}
	\newcommand{\wB}{{\widehat B}}
	\newcommand{\tbbS}{\widetilde{\bf S}}
	\newcommand{\hbbS}{\widehat{\bf S}}
	\newcommand{\obbS}{\overline{\bf S}}
	\newcommand{\bbt}{{\bf t}}
	\newcommand{\bbT}{{\bf T}}
	\newcommand{\hbbT}{\widetilde{\bf T}}
	\newcommand{\tbbT}{\widetilde{\bf T}}
	\newcommand{\obT}{{\overline{\bf T}}}
	\newcommand{\bbU}{{\bf U}}
	\newcommand{\bbu}{{\bf u}}
	\newcommand{\bbV}{{\bf V}}
	\newcommand{\bbv}{{\bf v}}
	\newcommand{\tbbv}{\widetilde{\bf v}}
	\newcommand{\bbW}{{W}}
	\newcommand{\tbbW}{\widetilde{\bf W}}
	\newcommand{\hbbW}{\widehat{\bf W}}
	\newcommand{\bbX}{{X}}
	\newcommand{\tbby}{\tilde {\bf y}}
	\newcommand{\tbbX}{\widetilde {\bf X}}
	\newcommand{\hbbX}{\widehat {\bf X}}
	\newcommand{\bbx}{{\bf x}}
	\newcommand{\obbx}{{\overline{\bf x}}}
	\newcommand{\tbbx}{{\widetilde{\bf x}}}
	\newcommand{\hbbx}{{\widehat{\bf x}}}
	\newcommand{\obby}{{\overline{\bf y}}}
	\newcommand{\hbbY}{{\widehat{\bf Y}}}
	\newcommand{\tbbY}{{\widetilde{\bf Y}}}
	\newcommand{\bbY}{{\bf Y}}
	\newcommand{\bby}{{\bf y}}
	\newcommand{\bbZ}{{\bf Z}}
	\newcommand{\bbz}{{\bf z}}
	\newcommand{\bbb}{{\bf b}}
	\newcommand{\cD}{{\cal D}}
	\newcommand{\bbL}{{\bf L}}
	\newcommand{\bxi} {\boldsymbol  \xi}
	\newcommand{\bbeta} {\boldsymbol  \eta}
	\newcommand{\utm}{\underline{ \tilde m}}
	\newcommand{\um}{\underline{m}}
	\newcommand{\la}{\langle}
	\newcommand{\ra}{\rangle}
	\newcommand{\bla}{\big{\langle}}
	\newcommand{\bra}{\big{\rangle}}
	\newcommand{\Bla}{\Big{\langle}}
	\newcommand{\Bra}{\Big{\rangle}}
	
	\newcommand{\rdd}{\textcolor{red}}
	\newcommand{\bll}{\textcolor{blue}}
	
	\newcommand{\md}{\mbox{d}}
	\newcommand{\non}{\nonumber\\}
	\newcommand{\tr}{{\rm tr}}
	\newcommand{\Tr}{{\rm Tr}}
	\newcommand{\E}{{\mathbb{E}}}
	\newcommand{\rP}{{\mathbb{P}}}
	\newcommand{\bqa}{\begin{eqnarray}}
		\newcommand{\eqa}{\end{eqnarray}}
	
	\newcommand{\bqn}{\begin{eqnarray*}}
		\newcommand{\eqn}{\end{eqnarray*}}
		
	\theoremstyle{plain}
	\newtheorem{thm}{Theorem}[section]
	\newtheorem{corollary}[thm]{Corollary}
	\newtheorem{defin}[thm]{Definition}
	\newtheorem{prop}[thm]{Proposition}
	\newtheorem{remark}[thm]{Remark}
	\newtheorem{lemma}[thm]{Lemma}
	\newtheorem{assumption}[thm]{Assumption}
          \allowdisplaybreaks[4]


\newpage

\begin{center}
\large\bf
Spectral Statistics of Sample Block Correlation Matrices
\end{center}

\vspace{0.5cm}
\renewcommand{\thefootnote}{\fnsymbol{footnote}}
\hspace{5ex}	
\begin{center}
 \begin{minipage}[t]{0.35\textwidth}
\begin{center}
Zhigang Bao\footnotemark[1]  \\
\footnotesize {Hong Kong University of Science and Technology}\\
{\it mazgbao@ust.hk}
\end{center}
\end{minipage}
\hspace{8ex}
\begin{minipage}[t]{0.35\textwidth}
\begin{center}
Jiang Hu\footnotemark[2]  \\ 
\footnotesize {Northeast Normal University}\\
{\it huj156@nenu.edu.cn}
\end{center}
\end{minipage}
\end{center}
\vspace{0.2cm}
\begin{center}
 \begin{minipage}[t]{0.35\textwidth}
\begin{center}
Xiaocong Xu\footnotemark[3]  \\
\footnotesize {Hong Kong University of Science and Technology}\\
{\it xxuay@connect.ust.hk}
\end{center}
\end{minipage}
\hspace{8ex}
\begin{minipage}[t]{0.35\textwidth}
\begin{center}
Xiaozhuo Zhang\footnotemark[4]  \\ 
\footnotesize {Northeast Normal University}\\
{\it zhangxz722@nenu.edu.cn}
\end{center}
\end{minipage}
\end{center}

\footnotetext[1]{Supported by  Hong Kong RGC grant GRF 16301520 and 16305421}
\footnotetext[2]{Supported by NSFC No. 12171078 and 11971097}
\footnotetext[3]{Supported by  Hong Kong RGC grant GRF 16300618 and 16301519 }
\footnotetext[4]{Supported by NSFC No. 12171078}
\vspace{0.8cm}
\begin{center}
 \begin{minipage}{0.8\textwidth}\footnotesize{
Abstract: A fundamental concept in multivariate statistics, sample correlation matrix,  is often used to infer the correlation/dependence structure among random variables, when the population mean and covariance are unknown. A natural block extension of it, {\it sample block correlation matrix}, is proposed to take on the same role, when random variables are generalized to random sub-vectors.  In this paper, we establish a spectral theory of the sample block correlation matrices and apply it to group independent test and related problem, under the high-dimensional setting. More specifically, we consider a  random vector of dimension $p$, consisting of $k$ sub-vectors of dimension $p_t$'s,  where $p_t$'s  can vary from $1$ to order $p$. Our primary goal is to investigate the dependence of the $k$ sub-vectors. We construct a random matrix model called sample block correlation matrix based on $n$ samples for this purpose. The spectral statistics of the sample block correlation matrix include the classical Wilks' statistic and Schott's statistic as special cases.  It turns out that the spectral statistics do not depend on the unknown population mean and covariance, under the null hypothesis that the sub-vectors are independent. Further,  the limiting behavior  of the spectral statistics can be described with the aid of the Free Probability Theory. Specifically, under three different settings of possibly $n$-dependent $k$ and $p_t$'s, we show that the empirical spectral distribution of the sample block correlation matrix converges to the free Poisson binomial distribution, free Poisson distribution (Marchenko-Pastur law) and free Gaussian distribution (semicircle law), respectively. We then further derive the CLTs for the linear spectral statistics of the block correlation matrix  under general setting.  Our results are established under general distribution assumption on the random vector. It turns out that the CLTs are universal and do not depend on the $4$-th cumulants of the vector components, due to a self-normalizing effect of the correlation type matrices. Based on our theory, real data analysis on stock return data and gene data are also conducted.}
\end{minipage}
\end{center}


\thispagestyle{headings}
\section{Introduction and main results}\label{Sec. Intro and main}
\subsection{Main problem and matrix model}
Assume that $\mathbf{y} = (\mathbf{y}'_1, \mathbf{y}'_2, \dots , \mathbf{y}'_k)'$ is a $p$-dimensional random (column) vector, in which the sub-vector $\mathbf{y}_t$ possesses dimension $p_t$ for $t \in [\![ k ]\!]$, such that $\sum_{t=1}^k p_t = p$, where $ [\![ k ]\!] := \{1,2,\dots,k\}$. Denote by $\mathbf{\mu}_t=\mathbb{E}(\mathbf{y}_t)$ the mean vector, by $\Sigma_{ij}=\text{Cov}(\mathbf{y}_i, \mathbf{y}_j)$ the cross covariance matrix, by $\mu = \mathbb{E}(\mathbf{y})$ and $\Sigma = \text{Cov}(\mathbf{y}, \mathbf{y})$ the mean and covariance of the full vector. A fundamental hypothesis testing problem is
\begin{align}
	\mathbf{H}_0: \text{$\mathbf{y}_t$'s are independent}, \hspace{5ex} \text{v.s.} \hspace{5ex} \mathbf{H}_1: \text{ not } \mathbf{H}_0. \label{the test problem}
\end{align}
In case $p_t=1$ for all $t$, i.e., $\mathbf{y}_t$'s are scalars, the problem boils down to the simplest complete independence test.  To this end, we draw $N$ observations of $\mathbf{y}$, namely $\mathbf{y}(1), \dots, \mathbf{y}(N)$. In addition, the $i$th sub-vector of $\mathbf{y}(j)$, i.e., the $j$th sample of $\mathbf{y}_t$,  will be denoted by $\mathbf{y}_t(j)$,  for all $t \in [\![ k ]\!]$ and $j \in [\![ N ]\!]$. Hence, collecting all the observations, we can construct the data matrices,
 $$Y := (\mathbf{y}(1), \dots, \mathbf{y}(N)), \quad Y_t := (\mathbf{y}_t(1), \dots, \mathbf{y}_t(N)), \quad t \in [\![ k ]\!].$$ 
 In this work, we will consider the problem (\ref{the test problem}) under high dimensional setting, namely,  $p\equiv p(N)$ is comparably large as $N$ or even much larger than $N$. Our $p_t$'s can vary from $1$ to order $p$, and $p_t$'s may be of different orders. Our default setting is that the mean vector $\mathbf{\mu}=\mathbb{E}(\mathbf{y})$ and the covariance $\Sigma=\text{Cov}(\mathbf{y}, \mathbf{y})$ are unknown.

In the classical low dimensional case,  when $p$ is fixed and $N$ is large (and thus $p_t$'s and $k$ are all fixed), assuming the Gaussian population, the hypothesis testing problem (\ref{the test problem}) dates back to  \cite{wilks1935independence}, \cite{Hotelling}. Especially, in the classical case, the Wilks' statistic is asymptotically $\chi^2$ distributed.  In the high dimensional case, when all $p_t$ and $N$ are comparably large and $k$ is fixed (and thus $p$ is also comparably large as $N$), a CLT is established for the Wilks' statistics in \cite{jiang2013central}. This result can be regarded as a high-dimensional refinement of the result in the classical low-dimensional case. In  \cite{jiang2013central}, it is assumed that $p+2<N$ and $p_t/N\to \hat{y}_t \in (0,1)$ when $N\to \infty$. Again, in the high-dimensional case, a tracial statistic based on F matrix was constructed for the test (\ref{the test problem}) in \cite{jiang2013testing}, but the assumption on $p_t$'s is stronger. A Schott type statistic is then used for the same test in \cite{bao2017test} for the high-dimensional case, and the assumption on $p_t$'s is weaker. All these three works on the high-dimensional case were done for the Gaussian population. 
For generally distributed population, a very special case, $p_t=1$ for all $t$, has been widely studied in the high dimensional setting when $p$ and $N$ are comparably large. In this case, various testing statistics have been constructed based on the sample correlation matrices and their limiting laws have been derived in the literature. For instance, a testing statistic based on the largest  off-diagonal entry of the sample correlation matrix was considered  in \cite{jiang2004asymptotic}; and testing statistics based on linear spectral statistics of the sample correlation matrix were adopted in  \cite{gao2017high} and \cite{yinspectral}. Very recently,  the work \cite{dornemann2022likelihood} went beyond the scalar case, again under the high-dimensional setting and generally distribution assumption. More specifically, \cite{dornemann2022likelihood} imposed some technical assumptions on $p_t$'s and the total $p$ is comparably  large as $N$ with the restriction $p/N\to \hat{y}\in (0,1)$ when $N\to \infty$. However, $k$ is not necessarily fixed. Actually, in  \cite{dornemann2022likelihood}, all $p_t$'s and $k$ can be $N$ dependent, but $\inf_t p_tk$ is comparable with $N$ (or $p$) for all $t$.  We also refer to   earlier works \cite{dette2020likelihood, bodnar2019testing} for related study, but on Gaussian population. An essential reason for the work  {\cite{jiang2013central}, \cite{dornemann2022likelihood}, \cite{dette2020likelihood} to restrict on the case $p<N$ is that all of them chose the Wilks' statistic which involves the log determinant of the sample covariance matrices, and thus the positive definiteness was required in all these papers.  
 
 In this work, we aim at establishing a general theory for (\ref{the test problem}) which can include nearly all the previous high-dimensional results as special cases and further go far beyond the previous settings.  To this end, we first introduce the  following sample  block correlation matrix model, which is a natural block extension of the sample correlation matrix. We separate the cases when the population mean is  unknown or known. For brevity, we denote by 
 \begin{align*}
\widehat{Y} := (\mathbf{y}(1)-\bar{\mathbf{y}}, \dots, \mathbf{y}(N)-\bar{\mathbf{y}}), \quad \widehat{Y}_i := (\mathbf{y}_t(1)-\bar{\mathbf{y}}_t, \dots, \mathbf{y}_t(N)-\bar{\mathbf{y}}_t), \quad t \in [\![ k ]\!],
 \end{align*}
 where $\bar{\mathbf{y}}=\frac{1}{N} \sum_{i=1}^N \mathbf{y}(i)$ and $\bar{\mathbf{y}}_t=\frac{1}{N} \sum_{i=1}^N \mathbf{y}_t(i)$ are the sample means.  

\begin{defin}[Sample block correlation matrix (with unknown mean)]  \label{defi-unknown mean}
	For any $k \in \mathbb{N}$,  when the population mean $\mu$ is unknown, the sample block correlation matrix $\widehat{\mathcal{B}} :=\widehat{\mathcal{B}}(Y_1, \dots, Y_k)$ is defined as follows,
	\begin{align*}
		\widehat{\mathcal{B}} := {\rm diag}\big((\widehat{Y}_t\widehat{Y}_t')^{-\frac{1}{2}}\big)_{t=1}^k \cdot \widehat{Y}\widehat{Y}' \cdot {\rm diag}\big((\widehat{Y}_t\widehat{Y}_t')^{-\frac{1}{2}}\big)_{t=1}^k. 
	\end{align*}
\end{defin}
Alternatively, we also have the following definition when the population mean $\mu$ is known. In this case, since one can subtract the known mean from the samples, without loss of generality, we assume $\mu=0$. 
\begin{defin}[Sample block correlation matrix (with  mean $0$)]
	For any $k \in \mathbb{N}$,  with the population mean $\mu=0$, the sample block correlation matrix $\mathcal{B} := \mathcal{B}(Y_1, \dots, Y_k)$ is defined as follows,
	\begin{align}
		\mathcal{B} := {\rm diag}\big((Y_tY_t')^{-\frac{1}{2}}\big)_{t=1}^k \cdot YY' \cdot {\rm diag}\big((Y_tY_t')^{-\frac{1}{2}}\big)_{t=1}^k. \label{def of B}
	\end{align}
\end{defin}

\begin{remark} \label{rmk060601}Our final goal is to establish all the main results for $\widehat{\mathcal{B}}$. But for brevity, most of the derivations in this paper will be displayed for $\mathcal{B}$ at first. At the end, we extend all results to $\widehat{\mathcal{B}}$. Further, notice that in case $p_t=1$ for all $t$, $\widehat{\mathcal{B}}$ boils down to the classical sample correlation matrix. Also notice that the classical Wilks' statistic for (\ref{the test problem}) (c.f., \cite{jiang2013central}) is simply $\log \det \widehat{\mathcal{B}}$, and the Schotts' statistic for (\ref{the test problem}) (c.f., \cite{bao2017test}) is simply $\text{\Tr} \widehat{\mathcal{B}}^2$,  and both of them are the so-called linear spectral statistics of $\widehat{\mathcal{B}}$ (c.f., (\ref{def of LSS})). This motivates us to study the random matrix model $\widehat{\mathcal{B}}$ and construct general testing statistics of
(\ref{the test problem}) based on the spectrum of   $\widehat{\mathcal{B}}$. 
\end{remark}
In random matrix theory and high-dimensional multivariate statistics, a commonly adopted structural assumption is 
\begin{align*}
	\mathbf{y}_t(i) = T_t\mathbf{x}_t(i) + \mathbf{\mu}_t, \quad T_tT_t' = \Sigma_{tt} \succ 0, \quad t \in [\![ k ]\!] \quad \text{and} \quad i \in [\![ N ]\!],
\end{align*}
where $\mathbf{x}_i$ consists of independent mean $0$ variance $1$ components and $T_t\in \mathbb{R}^{p_t\times p_t}$ is an invertible matrix.  For simplicity, we present the notations and results for the case that $\mu = 0$ and consider the matrix $\mathcal{B}$ at first. With the normalized vector $\hat{\mathbf{x}}_t(i) : = \mathbf{x}_t(i) / \sqrt{N}$ for all $t \in [\![ k ]\!]$ and $i \in [\![ N ]\!]$, we can define the following normalized sample matrices, 
\begin{align}
X := (\hat{\mathbf{x}}(1), \dots, \hat{\mathbf{x}}(N)), \quad X_t := (\hat{\mathbf{x}}_t(1), \dots, \hat{\mathbf{x}}_t(N)), \quad t \in [\![ k ]\!]. \label{062901}
\end{align}
Instead of studying the spectral statistics of the sample block correlation matrix $\mathcal{B}$, we can turn to study the following matrix which has the same non-zero eigenvalues,
\begin{align}
\bbH :=Y' \cdot {\rm diag}((Y_tY_t')^{-1})_{t=1}^k \cdot Y = \sum_{t=1}^k Y_t'(Y_tY_t')^{-1}Y_t = \sum_{t=1}^k X_t'(X_tX_t')^{-1}X_t=:\sum_{t=1}^k {P}_t.  \label{def of H}
\end{align}
Notice that $\bbH$ is a sum of $k$ random projections, and it does not depend on the unknown $\Sigma_{tt}$'s (or $T_t$'s). This further implies that the spectral statistics of $\mathcal{B}$ does not depend on the unknown population covariance matrix $\Sigma$, under the null hypothesis $\mathbf{H}_0$. Hence $\mathcal{B}$ (or $\widehat{\mathcal{B}}$) has the same advantage as its scalar counterpart, sample correlation matrix, when one considers the spectral statistics. We denote  
the  ordered eigenvalues of $H$ by
$$
\lambda_1(H) \ge \lambda_2(H) \ge \dots \ge \lambda_N(H),
$$
and consider their statistics in the sequel. The empirical spectral distribution (ESD) of $H$ is defined as
\begin{align}
\mu_N := \frac{1}{N}\sum_{i=1}^N\delta_{\lambda_i(H)}. \label{def of mu H}
\end{align}
The limiting behaviour of the ESD can often be studied via the Green function of $H$,  $G(z) := (H - z)^{-1}$, and its normalized trace, also known as  the  Stieltjes transform of $H$ 
\begin{align}
m_N(z) = \frac{1}{N} \Tr G(z)=\int \frac{1}{\lambda-z} \mu_N({\rm d} z), \qquad z\in \mathbb{C}^+:=\{w\in \mathbb{C}: \Im w>0\}.  \label{def of m_N}
\end{align}
The linear spectral statistics (LSS) of $H$ is then defined as follows.
\begin{defin}[Linear spectral statistics] 
	For a test function $f:\mathbb{R}\to \mathbb{R}$, the linear spectral statistics (LSS) of $H$ is defined as
	\begin{align}
		\sum_{i=1}^N f(\lambda_i(H))  = \Tr f(H)=N\int f(\lambda) \mu_N({\rm d}\lambda). \label{def of LSS}
	\end{align}
\end{defin}
Similarly, we can define LSS for any square matrix. Hereafter we use $\Tr A$ to denote the trace of a matrix $A$ and use $\tr A := N^{-1}\Tr A $ to denote the normalized (by $N$) trace for any square matrix $A$, no matter the dimension of $A$ is $N\times N$ or not. 

From the LSS of $H$, one can easily recover the LSS of $\mathcal{B}$ since they share the same non-zero eigenvalues. It turns out that the classical Wilks' statistic and Schott's statistic are both LSS of $\mathcal{B}$, as mentioned in Remark \ref{rmk060601}. Our aim is to establish a general CLT for $\Tr f(H)$ for a general class of test function $f$ with rather general assumption on $k$, $p_t$'s and $N$,  and thus provide a class of testing statistics for the problem (\ref{the test problem}). To this end, according to (\ref{def of LSS}), it is clear that one needs to first study the limiting law for the random measure $\mu_N$.

\subsection{Free additive convolution} In (\ref{def of H}), 
observe that  ${P}_t$ is a projection matrix, whose  ESD is trivially 
$\mu_t\equiv \mu_{t}^N=y_t\delta_1+(1-y_t)\delta_0$ (almost surely), with $y_t:= p_t/N.$  If the $P_t$'s are replaced by independent Bernoulli random variables in classical probability theory, then the sum of these Bernoulli is distributed as the classical convolution of Bernoulli random variables. 
Heuristically, if we view $P_t$'s as certain random variables in a non-commutative probability space, and regard the ESD $\mu_t$ as the distribution of $P_t$, the random matrix $H$ can be regarded as a sum of $k$ random variables $P_t$ with Bernoulli distributions. This motivates us to consider convolution in certain non-commutative probability space. The right candidate is the free additive convolution from Free Probability Theory. Free probability was initiated by Voiculescu in  \cite{voiculescu1985symmetries}, and later a connection with random matrices was discovered in \cite{voiculescu1991limit}. Free additive convolution, the analogue of classical convolution, was first introduced in \cite{voiculescu1986addition} via R-transform. In the sequel, we will rely on an analytical definition of free additive convolution based on the subordination functions which dates back to \cite{voiculescu1993analogues}. 
Given a generic probability measure $\mu$ on $\mathbb{R}$, its Stieltjes transform, $m_{\mu}$,  is defined by
$
	m_{\mu}(z) := \int_{\mathbb{R}} (x-z)^{-1}{\rm d} \mu(x)$, $z\in \mathbb{C}^{+}.
$
We denote by $F_{\mu}$ the negative reciprocal Stieltjes transform of $\mu$, i.e.
$
	F_{\mu} (z) := -(m_{\mu}(z))^{-1}$, $z\in \mathbb{C}^{+}.
$
Observe that 
\begin{align}\label{F1}
	\lim_{\eta \to \infty} \frac{F_\mu(\mathrm{i}\eta)}{\mathrm{i}\eta} = 1,
\end{align}
and note that $F_\mu$ is analytic on $\mathbb{C}^{+}$ with nonnegative imaginary part. We refer to \cite{belinschi2007new} and \cite{chistyakov2011arithmetic} for instance for the following definition of  free additive convolution. 
\begin{prop}\label{freeaddprop}
	Given k probability measures $\mu_1,\dots,\mu_k$ on $\mathbb{R}$, there exist unique analytic functions, $\omega_1,\dots, \omega_k: \mathbb{C}^{+} \to \mathbb{C}^{+}$, such that,
	\begin{itemize}
		\item [i)] for all $z \in \mathbb{C}^{+}$, $\Im \omega_1, \Im \omega_2 ,\dots, \Im \omega_k \ge \Im z$, and
		\begin{align}\label{sub1}
			\lim_{\eta \to \infty} \frac{\omega_t(\mathrm{i}\eta)}{\mathrm{i}\eta} = 1, \quad t \in [\![ k ]\!],
		\end{align}
		\item [ii)] for all $z \in \mathbb{C}^{+}$,
		\begin{align}\label{sub2}
			F_{\mu_1}(\omega_1(z)) = F_{\mu_2}(\omega_2(z)) = \dots = F_{\mu_k}(\omega_k(z)),
		\end{align}
		and for all $t \in [\![ k ]\!]$, 
		\begin{align}\label{sub3}
			\omega_1(z) + \omega_2(z) + \dots + \omega_k(z) - z = (k-1)F_{\mu_t}(\omega_t(z)).
		\end{align}
	\end{itemize}
\end{prop}
It follows from \eqref{sub1} that the analytic function $F_{\boxplus}:\mathbb{C}^{+} \to \mathbb{C}^{+}$ defined by
\begin{align}\label{F2}
	F_{\boxplus}(z) := F_{\mu_t}(\omega_t(z)), \quad t \in [\![ k]\!],
\end{align} 
satisfies the analogues of \eqref{F1}. Hence, from \cite{akhiezer2020classical}, we know that  $F_{\boxplus}$ is the negative reciprocal Stieltjes transform of a probability measure $\mu_{\boxplus}$, called the free additive convolution of $\mu_t$'s, usually denoted by $\mu_{\boxplus} = \mu_1\boxplus\dots\boxplus\mu_k$. The functions $\omega_t$'s in Proposition \ref{freeaddprop} are called subordination functions and $m_{\boxplus}:=-1/{F_{\boxplus}}$ is said to be subordinated to $m_{\mu_t}$. Apparently, we can also rewrite (\ref{F2}) as 
\begin{align}
m_\boxplus(z)=m_{\mu_t}(\omega_t(z)),  \quad t \in [\![ k]\!]. \label{m2}
\end{align}

\subsection{Main results}

With the aid of the notations introduced above, in this part, we can present our results on the limiting behavior of $\mu_N$ and the CLTs for LSS. To this end, we start with our main assumptions. As we mentioned above, we will first present the result for the matrix $\mathcal{B}$ in (\ref{def of B}) and at the end we extend the result to $\widehat{\mathcal{B}}$.
\begin{assumption}[Assumption on matrix entrices]\label{assum1}
Assume that $X=(X_{ab})$ in (\ref{062901}) has i.i.d. columns. For its  entries, we further impose the following assumptions,
\begin{itemize}
	\item 	Under $\mathbf{H}_0$, $X_{ab}$'s $(a \in [\![ p ]\!], b \in [\![ N ]\!])$ are independent.
	\item 	$\mathbb{E}[ X_{ab} ]=0, \mathbb{E}[ |X_{ab}|^2 ]= 1 / N$ for all $a \in [\![ p ]\!]$ and $b \in [\![ N ]\!]$.
	\item  For each $\ell \in \mathbb{N}$, there exists a constant $C_\ell$ such that $\mathbb{E} [|\sqrt{N}X_{ab}|^{\ell} ] < C_\ell$ for all $N, a, b$.
\end{itemize}
\end{assumption}
Here we emphasize that we do not need $X_{ab}$'s to be i.i.d.  But we assume that the columns of the data matrix $X$ are i.i.d, which is natural since we have i.i.d. samples. When an entry of $X$ is the $(c,d)$ entry of a sub-matrix $X_t$ (c.f. (\ref{062901})), we also write this entry as $X_{t,cd}$.  Further, the moment condition $\mathbb{E} [|\sqrt{N}X_{ab}|^{\ell} ] < C_\ell$ will be eventually replaced by a weaker $4+\delta$-moment condition in Assumption \ref{assum3}, assuming the continuity of the matrix entry distributions. 

\begin{assumption}[Assumption on dimensional parameters]\label{assum2}
For the dimensional parameters, we impose the following assumptions,
\begin{itemize}
	\item $\sum_{t=1}^ky_t =: y\to \hat{y} \in (0,\infty)$ as $N\to \infty$.
	\item $y_t\to \hat{y}_t \in [0,1), t = 1,\dots,k$ as $N\to \infty$.
	\item there exists some small constant $c>0$, such that $y - \max_ty_t \geq c$, for sufficiently large $N$.
\end{itemize}
\end{assumption}

\begin{remark} \label{rem071301}
The second assumption on dimensional parameters is to ensure that the sample covariance matrix $X_tX_t'$ is invertible with high probability. The last assumption guarantees that the variance for LSS is typically of order $1$. See, for instance, Corollaries \ref{coroschott} and \ref{corowilk}. We also refer to Remark \ref{rem.080901} below for more discussion. 
\end{remark}

In the sequel,  for convenience, we also write 
$
p_{\max}:=\max_tp_t$, $y_{\max}:=\max_t y_t. 
$
Recall $\mu_N$ defined in (\ref{def of mu H}). 
\begin{thm}\label{1stLimit}
Under Assumptions \ref{assum1} and \ref{assum2} , the following convergence holds in probability: For any fixed integer $\ell>0$, we have the convergence in moment of $\mu_N$ to $\mu_1\boxplus\dots\boxplus\mu_k$,
\begin{align}
\int x^\ell{\rm d} \mu_N - \int x^\ell {\rm d}\mu_1\boxplus\dots\boxplus\mu_k \stackrel{\mathbb{P}}\longrightarrow 0.  \label{062703}
\end{align}
Further, we have the following convergence in probability in two special cases

	Case 1: $k \sim 1$, $p_t \sim N$ for all $t \in [\![ k ]\!]$
	 \begin{align}
		 \mu_N\stackrel{\mathbb{P}}\Longrightarrow \mu^{\infty}_{\boxplus} :=\mu^{\infty}_1\boxplus\dots\boxplus \mu^{\infty}_k, \qquad \mu^{\infty}_t\sim \text{Ber} (\hat{y}_t).  \label{062701}
 	\end{align}
 	
	Case 2: $k \gg 1$, $ p_t \le  N^{1-\epsilon}$ for some small constant $\epsilon > 0$ for all $t \in [\![ k ]\!]$ 
	 \begin{align}
 		\mu_N\stackrel{\mathbb{P}}\Longrightarrow \mu_{mp,\hat{y}} := \frac{\sqrt{([(1+\sqrt{\hat{y}})^2-x][x-(1-\sqrt{\hat{y}})^2])_+}}{2x} {\rm d} x+(1-\hat{y})_+\delta_0.  \label{062702}
 \end{align}
Here $A_N \stackrel{\mathbb{P}}\Longrightarrow A $ means $A_N$ converge weakly  to $A$ in probability, and the notations $\sim$ and $\gg$ are introduced in Section \ref{s.NC}.
\end{thm}
\begin{remark} Here we remark that from Free Probability point of view, the above convergence of measure $\mu_N$ can be regarded as the non-commutative analogues of the limiting laws for sum of independent Bernoulli in classical probability. Specifically, case 1 is analogous to the sum of (fixed) $k$  independent Bernoulli whose distribution is known as Poisson Binomial distribution. Hence, we can call the measure in the RHS of (\ref{062701}) as Free Poisson Binomial distribution. Cases 2 corresponds to the classical Poisson convergence, and indeed Marchenko-Pastur law is called Free Poisson law in Free Probability Theory.  In Section \ref{s.case 3 semicircle} of Appendix, we also state the discussion for an alternative regime which is not covered by Theorem \ref{1stLimit}. Specifically, we consider the case when $\sum_{t=1}^k y_t(1-y_t)\to \infty$ as $N$ goes to infinity. Notice that this is a necessary condition for the sum of $k$ independent $ \text{Ber} (y_t)$'s to be asymptotically Gaussian in classical CLT.  We show under certain additional assumption that in this regime the ESD of a rescaled version of $H$ converges weakly in probability to the Free Gaussian law, i.e, semicircle law. We also remark here in the general statement (\ref{062703}), $\mu_1\boxplus\mu_2\dots\boxplus\mu_k$ could be $N$-dependent.
\end{remark}

The above theorem depicts the first order behaviour of $\mu_N$. For the second order, we shall derive  CLT for the LSS in (\ref{def of LSS}). 

We will start with Cauchy integral formula, writing the LSS as a contour integral involving the Stieltjes transform for test functions analytic in certain domain. Therefore, to avoid the singularity at the origin for some specific but important test functions, i.e., the logarithm function (the one used in Wilks' statistics), our discussion will be separated into two cases. For the case of $\hat{y}\in (0,1)$ and thus $H$ has trivial $0$ eigenvalues, we will use a contour which does not enclose $0$. While for other cases, we will simply use sufficiently large contour that encloses $\text{supp}(\mu_\boxplus)$ (c.f., Lemma \ref{supportofmu}). 
For the former case, we will work on the contour $\gamma^0$ shown in Fig.~\ref{contourFig}. We only depict the part in the upper half plane, and the lower half can be completed by complex conjugation. The upper half contour of $\gamma^0$ can be parameterized as the following. For sufficiently small $\epsilon_1 > \epsilon_2 > 0$, and sufficiently large $M_2 > M_1 > 0$, let
\begin{align}
	&\mathcal{C}_1\equiv \mathcal{C}_1(\epsilon_1,\epsilon_2) := \{z: |z| = \epsilon_1, \Im z \ge 0, \Re z \ge -\epsilon_2 \},\notag\\
	&\mathcal{C}_2\equiv  \mathcal{C}_2(\epsilon_1,\epsilon_2,M_1) := \{z: \Im z = \sqrt{\epsilon_1^2-\epsilon_2^2}, -M_1\leq \Re z \leq -\epsilon_2 \}, \notag\\
	&\mathcal{C}_3\equiv \mathcal{C}_3(\epsilon_1,\epsilon_2,M_1,M_2) := \{z: \sqrt{\epsilon_1^2-\epsilon_2^2} \leq \Im z \leq M_2 , \Re z = -M_1 \},\notag\\
	&\mathcal{C}_4\equiv \mathcal{C}_4(M_1,M_2) := \{z: \Im z = M_2,  -M_1 \le \Re z \le M_1 \},\notag\\
	&\mathcal{C}_5\equiv \mathcal{C}_5(\epsilon_1,\epsilon_2,M_1,M_2) := \{z: \sqrt{\epsilon_1^2-\epsilon_2^2} \leq  \Im z \leq M_2, \Re z = M_1 \},\notag\\
	&\mathcal{C}_6\equiv \mathcal{C}_6(\epsilon_1,\epsilon_2,M_1) := \{z: 0 \le \Im z \le \sqrt{\epsilon_1^2-\epsilon_2^2} , \Re z = M_1 \}.  \label{062033}
\end{align}
In summary, the contour for the case of $\hat{y}\in (0,1)$ is 
$
\gamma^0:=\mathcal{C}^0\cup\overline{\mathcal{C}^0}$, where $\mathcal{C}^0\equiv \mathcal{C}^0(\epsilon_1,\epsilon_2,M_1,M_2):= \bigcup_{a=1}^6 \mathcal{C}_a .   
$

\begin{figure}
        \centering
        \begin{minipage}{.5\textwidth}
            \centering
            \begin{tikzpicture}
		\draw[->](-3.2,0)--(3.2,0) node[left,below,font=\tiny]{$x$};
		\draw[->](0,-0.2)--(0,2.2) node[right,font=\tiny]{$y$};
		\draw[color=red, thick, smooth, domain=-2:-0.07]plot(\x,0.05);
		\draw[color=green, thick, smooth, domain=0.035:2]plot(-2,\x);
		\draw[color=black, thick, smooth, domain=-2:2]plot(\x,2);
		\draw[color=green, thick, smooth, domain=0.1:2.015]plot(2,\x);
		\draw[color=blue, thick, smooth, domain=0:0.1]plot(2,\x);
		\draw[color= blue, thick] (0.1,0) arc (0:150:0.1);
	\end{tikzpicture}
            \caption{Contour of $z$ in half plane (for test \\functions which are singular at origin)}
            \label{contourFig}
        \end{minipage}%
        \begin{minipage}{.5\textwidth}
            \centering
            \begin{tikzpicture}
			\draw[->](-3.2,0)--(3.2,0) node[left,below,font=\tiny]{$x$};
			\draw[->](0,-0.2)--(0,2.2) node[right,font=\tiny]{$y$};
			\draw[color=green, thick, smooth, domain=0.035:2]plot(-2,\x);
			\draw[color=black, thick, smooth, domain=-2:2]plot(\x,2);
			\draw[color=green, thick, smooth, domain=0.1:2.015]plot(2,\x);
			\draw[color=blue, thick, smooth, domain=0:0.1]plot(2,\x);
			\draw[color=blue, thick, smooth, domain=0:0.1]plot(-2,\x);
			\end{tikzpicture}
	        \caption{Contour of $z$ in half plane (for analytic test functions)}
            \label{zcontour2}
        \end{minipage}%
  \end{figure}
    
  For the other case, we choose the contour $\gamma$ with  upper half  shown in Fig. \ref{zcontour2} and its complex conjugate. Let
$
	\mathcal{C}_7 = \{z: 0 \le \Im z \le \epsilon_2 , \Re z = -M_1 \}.
$
The contour now becomes
$
\gamma:=\mathcal{C}\cup\overline{\mathcal{C}}$, where $\mathcal{C}\equiv \mathcal{C}(\epsilon_1,\epsilon_2,M_1,M_2) := \bigcup_{a=3}^7 \mathcal{C}_a.
$  
With the above configuration, we first define the contour used for the CLT. For sufficiently small $\epsilon_{1i} > \epsilon_{2i} > 0, i = 1,2$ and sufficiently large $M_{2i} > M_{1i} > 0, i=1,2$, let
\begin{align}
	\gamma^0_1:=\gamma(\epsilon_{11}, \epsilon_{21},M_{11},M_{21}),\quad \gamma^0_2:=\gamma(\epsilon_{12}, \epsilon_{22},M_{12},M_{22})\notag \\
	\gamma_1:=\gamma(\epsilon_{11}, \epsilon_{21},M_{11},M_{21}),\quad \gamma_2:=\gamma(\epsilon_{12}, \epsilon_{22},M_{12},M_{22})\label{contour12}
\end{align}
be counterclockwise contours.
Notice that by choosing sufficiently well separated parameters $\epsilon_{1i}$, $\epsilon_{2i}$, $M_{1i}$ and $M_{2i}$, $i=1,2$, the contours $\gamma_1^0$ ($\gamma_1$) and $\gamma_2^0$ ($\gamma_2$) are nonintersecting. In addition, by choosing $\epsilon_{1i}$, $\epsilon_{2i}$, $M_{1i}$ and $M_{2i}$, $i=1,2$ appropriately, we can always have that $\{m_{\boxplus}(z): z\in \gamma_1^0\}$ and $ \{m_{\boxplus}(z): z\in \gamma_2^0\}$ are well separated and  the same holds if $(\gamma_1^0,\gamma_2^0)$ is replaced by $(\gamma_1,\gamma_2)$ (c.f.,  Section \ref{Discussion on the contours} in Appendix). Notice that all the contours enclose the set $\text{supp}(\mu_{\boxplus}) \setminus 0$ (c.f.,  Lemma \ref{supportofmu}). The CLT for LSS can be summarized as the following.

\begin{thm}\label{FreeCLT}
Recall $ \mu_\boxplus=\mu_1\boxplus\dots\boxplus\mu_k$ and $m_\boxplus$ its Stieltjes transform. 
If $\hat{y} \in (0,1)$ and f is analytic inside $\gamma_1^0$ and $\gamma_2^0$, under Assumptions \ref{assum1} and \ref{assum2} , we have 
\begin{align}
\frac{\Tr f(H)-N\int f(x) {\rm d} \mu_\boxplus(x)-a_f}{\sigma_f}\Rightarrow \mathcal{N}(0, 1) \label{062230}
\end{align}
 if  $\sigma_f^2 > c$ with some small constant $c > 0$. Here
\begin{align}
a_f=-\frac{1}{4\pi \mathrm{i}} \oint_{\gamma_1^0} f(z) \bigg[ \sum_{t=1}^k \frac{\omega_t''(z)}{\omega_t'(z)}+(k-1)\left( \frac{2m'_{\boxplus}(z)}{m_\boxplus(z)}-\frac{m_\boxplus''(z)}{m_{\boxplus}'(z)}\right)\bigg] {\rm d} z \label{062201}
\end{align}
 and 
\begin{align}
\sigma_f^2=& \frac{-1}{2\pi^2} \oint_{\gamma_1^0}\oint_{\gamma_2^0} f(z_1)f(z_2)\bigg[ \sum_{t=1}^k \frac{\omega_t'(z_1)\omega_t'(z_2)}{(\omega_t(z_1)-\omega_t(z_2))^2} \notag\\
&\qquad-\frac{1}{(z_1-z_2)^2}
-\frac{(k-1)m'_\boxplus(z_1)m'_\boxplus(z_2)}{(m_\boxplus(z_1)-m_\boxplus(z_2))^2}\bigg] {\rm d}z_2 {\rm d}z_1. \label{072201}
\end{align}
The same result holds if  $\hat{y} \in (0,\infty)$ with $\gamma_{1}^0$ and $\gamma_{2}^0$ replaced by $\gamma_1$ and $\gamma_2$, respectively. 
\end{thm}

\begin{remark}\label{rem.080901} Here we further illustrate the generality of our result, in contrast to previous works in the special cases such as \cite{jiang2013central}, \cite{dornemann2022likelihood}, \cite{dette2020likelihood}, \cite{bao2017test}. From the perspective of the generality of test function, the results shown in \cite{jiang2013central}, \cite{dornemann2022likelihood}, \cite{dette2020likelihood} only considered the log likelihood ratio test statistics, i.e., Wilks' statistics,  while our results can be applied to a larger class of test functions, such us the polynomial test functions. The non--analytic function is not of great interest in statistics. Thus we do not pursue this direction here. For the assumption on the dimensional parameters, since the previous work \cite{jiang2013central}, \cite{dornemann2022likelihood}, \cite{dette2020likelihood} considered the Wilks' statistics, they need to impose the condition $p < N $ to ensure the existence of the log determinant of sample covariance matrices. But if we choose test functions which are analytic at 0 as well, we can simply remove this restriction.  Further, we do not require our $p_t$'s to be comparably large. For instance, $k=p/2$ and $p_1=\dots=p_{k}=2$ apparently satisfies Assumption \ref{assum2}; also, the case $k=4$, $p_1=p_2=2$, $p_3=p_4=p/2-2$ satisfies Assumption \ref{assum2} as well. But none of the previous works can cover both cases, say. In addition,  as we mentioned in Remark \ref{rem071301}, the variance $\sigma_f^2$ will typically degenerate if $y - \max_ty_t \geq c$ does not hold. However, we claim here that our CLT will still hold with a degenerate $\sigma_f^2$ even when $y - \max_ty_t$ degenerates at certain moderate rate. Actually, from our analysis, one can easily generalize the assumption to $y - \max_ty_t \gg N^{-\frac16}$, so that the CLT is still valid. This will further generalize our setting on $p_t$'s.  But for brevity, we will not pursue the direction on optimizing the degenerate rate of $y - \max_ty_t$ for the validity of CLT in the current work. 
\end{remark}

\begin{remark}
For many classical random matrix models such as Wigner matrices and sample covariance matrices, the CLT of LSS often depends on the $4$-th moment/cumulant of the matrix entries. Then the $4$-th moment/cumulant of the matrix entries shall be estimated  from real data in applications, which may not be feasible. However, for sample block correlation matrices, we notice that the above CLT does not depend on the $4$-th moment/cumulant of the matrix entries. Hence, under the null hypothesis, this asymptotic result indeed does not involve any additional unknown parameter. The independence of $4$-th moment/cumulant is essentially due to a self-normalizing effect of the correlation type of matrices. 
\end{remark}
Based on the foregoing Theorem, we can obtain the following asymptotic normality results for Schott's statistics ($f(x) = x^2$) and Wilks' statistics ($f(x) = \log(x)$).
\begin{corollary}\label{coroschott}
(Schott's statistics). If $\hat{y} \in (0,\infty)$, under Assumptions \ref{assum1} and \ref{assum2} , we have
	\begin{align*}
		\frac{\Tr \mathcal{B}^2 - a_1}{b_1} \Rightarrow \mathcal{N}(0, 1),
	\end{align*}
	where
	\begin{align*}
		&a_1 = N(\sum_{r \neq s}^k y_ry_s + \sum_{t=1}^k y_t),\quad b_1 = 4\sum_{r \neq s}^ky_ry_s(1-y_r)(1-y_s).
	\end{align*}
	\end{corollary}
\begin{corollary}\label{corowilk}
	(Wilks' statistics). If $\hat{y} \in (0,1)$, under Assumptions \ref{assum1} and \ref{assum2} , we have
	\begin{align*}
		\frac{\Tr \log(\mathcal{B}) - a_2}{b_2} \Rightarrow \mathcal{N}(0, 1),
	\end{align*}
	where 
	\begin{align*}
		&a_2 = \sum_{t=1}^{k}\left(N-p_t-\frac{1}{2}\right)\log(1-y_t) - \left(N-Ny- \frac{1}{2}\right)\log(1-y),\\
		&b_2 = -2\log(1-y) + 2\sum_{t=1}^k \log(1-y_t).
	\end{align*}
	
\end{corollary}

\begin{remark} The derivations of the above two corollaries from Theorem \ref{FreeCLT} require us to compute $a_f$ and $\sigma_f^2$ which are seemingly involved integrals. The calculation schemes (c.f., Section \ref{Sec Proof of 1.14} in Appendix) for the integrals in these two special cases, i.e., $f(x)=x^2$ and $f(x)=\log (x)$ can actually be applied to more general test functions satisfying the assumptions in Theorem \ref{FreeCLT} . For brevity, we will not display the results for other test functions here. 
\end{remark}

If we impose stronger condition on the relation between $p_t$ and $N$, the following theorem provides another approximation of the asymptotic distribution of the LSS, where the free additive convolution is approximated by the simpler  \textit{Machenko-Pastur law}.

\begin{thm}\label{MPCLT}
Denote by  $m_{y}$ the Stieltjes transform of $\mu_{mp,y}$. Let ${p_{\max}} \le N^{1/2-\epsilon}$ for any given (small) constant $\epsilon > 0$, and f is analytic inside $\gamma_1^0$ and $\gamma_2^0$. Under Assumptions \ref{assum1} and \ref{assum2} , we have 
\begin{align*}
\frac{\Tr f(H)-N\int f(x) {\rm d} \mu_{mp,y}-a_f }{\sigma_f}\Rightarrow \mathcal{N}(0, 1)
\end{align*}
if $\sigma_f^2 > c$ with some small constant $c > 0$, where 
\begin{align*}
a_f=\frac{1}{2\pi \mathrm{i}} \oint_{\gamma_1^0} f(z) \Big[ -\sum_{t=1}^k \frac{Ny_t^2}{1-y_t}\frac{m_y(z)m_y'(z)}{(1+m_y(z))^3} +y\frac{(m'_y(z))^2-m_y^2(z)m_y'(z)}{m_y(z)(1+m_y(z))^3}\Big]{\rm d} z
\end{align*}
 and 
\begin{align*}
\sigma_f^2=&\frac{-1}{2\pi^2} \oint_{\gamma_1^0}\oint_{\gamma_2^0} f(z_1)f(z_2)\bigg[\frac{m_y'(z_1)m'_y(z_2)}{(m_y(z_1)-m_y(z_2))^2}-\frac{1}{(z_1-z_2)^2} \notag\\
&\qquad\qquad-\frac{ym_y'(z_1)m_y'(z_2)}{(1+m_y(z_1))^2(1+m_y(z_2))^2}\bigg]{\rm d}z_2{\rm d}z_1.
\end{align*}
The same result holds if  $\hat{y} \in (0,\infty)$ with $\gamma_{1}^0$ and $\gamma_{2}^0$ replaced by $\gamma_1$ and $\gamma_2$, respectively.
\end{thm}

Recall the setting with unknown population mean in Definition \ref{defi-unknown mean}. We have the following theorem. 

\begin{thm}\label{RemoveSampleMean}
Let $\hat{H}: = \hat{Y}' \cdot {\rm diag}((\hat{Y}_i\hat{Y}_i')^{-1})_{i=1}^k \cdot \hat{Y}$ be the matrix which has the same non-zero eigenvalues as $\hat{\mathcal{B}}$. 
Further let $ \tilde{\mu}_\boxplus=\tilde{\mu}_1\boxplus\dots\boxplus\tilde{\mu}_k$ with $\tilde{\mu}_t \sim Ber(p_t/(N-1))$ and $\tilde{m}_\boxplus$ its Stieltjes transform. 
If $\hat{y} \in (0,1)$ and f is analytic inside $\gamma_1^0$ and $\gamma_2^0$, under Assumptions \ref{assum1} and \ref{assum2} , we have 
\begin{align*}
\frac{\Tr f(\hat{H})-(N-1)\int f(x) {\rm d} \tilde{\mu}_\boxplus(x)-\tilde{a}_f}{\tilde{\sigma}_f}\Rightarrow \mathcal{N}(0, 1) 
\end{align*}
if  $\tilde{\sigma}_f^2 > c$ with some small constant $c > 0$, where
\begin{align*}
\tilde{a}_f=-\frac{1}{4\pi \mathrm{i}} \oint_{\gamma_1^0} f(z) \bigg[ \sum_{t=1}^k \frac{\tilde{\omega}_t''(z)}{\tilde{\omega}_t'(z)}+(k-1)\left( \frac{2\tilde{m}'_{\boxplus}(z)}{\tilde{m}_\boxplus(z)}-\frac{\tilde{m}_\boxplus''(z)}{\tilde{m}_{\boxplus}'(z)}\right) - \frac{1}{z}\bigg] {\rm d} z 
\end{align*}
 and 
\begin{align*}
\tilde{\sigma}_f^2=& \frac{-1}{2\pi^2} \oint_{\gamma_1^0}\oint_{\gamma_2^0} f(z_1)f(z_2)\bigg[ \sum_{t=1}^k \frac{\tilde{\omega}_t'(z_1)\tilde{\omega}_t'(z_2)}{(\tilde{\omega}_t(z_1)-\tilde{\omega}_t(z_2))^2} \notag\\
&\qquad-\frac{1}{(z_1-z_2)^2}
-\frac{(k-1)\tilde{m}'_\boxplus(z_1)\tilde{m}'_\boxplus(z_2)}{(\tilde{m}_\boxplus(z_1)-\tilde{m}_\boxplus(z_2))^2}\bigg] {\rm d}z_2 {\rm d}z_1,
\end{align*}
Here $\tilde{\omega}_t(z), t \in [\![k ]\!]$ are the subordination functions determined by Proposition \ref{freeaddprop} with $\mu_t \equiv \tilde{\mu}_t$.

The same result holds if  $\hat{y} \in (0,\infty)$ with $\gamma_{1}^0$ and $\gamma_{2}^0$ replaced by $\gamma_1$ and $\gamma_2$, respectively. 
In addition, Corollaries \ref{coroschott} and \ref{corowilk} still hold if we replace $\mathcal{B}$ by $\hat{\mathcal{B}}$ and $N$ by $N-1$ simultaneously. 
\end{thm}

Finally, we present the following theorem with relaxed moment condition on matrix entries.
\begin{assumption}[Assumption on matrix entries]\label{assum3}
Keeping the first two assumptions in Assumption \ref{assum1}, and replacing the third assumption with
\begin{itemize}
	\item $X_{ab}$'s follow continuous distributions, and there exists a constant $\delta > 0$ such that $\mathbb{E} [|\sqrt{N}X_{ab}|^{4+\delta} ] < C$ for all $N, a, b$.
\end{itemize}
\end{assumption}

\begin{thm}\label{mainth22}
	Theorems \ref{FreeCLT}, \ref{MPCLT}, \ref{RemoveSampleMean} and Corollaries \ref{coroschott}, \ref{corowilk} still hold under Assumptions \ref{assum3} and \ref{assum2}.
\end{thm}

\begin{remark} Here we remark that in the very special case $p_t=1$ for all $t$, our CLT matches the result in \cite{yinspectral}, where the independence of the $4$-th moment/cumulant was also observed. 
\end{remark}

\subsection{Proof Strategy}

CLT for linear spectral statistics is a classical and central topic in Random Matrix Theory. There is a vast body of  literature. Most of the works have been done for the classical random matrix models such as Wigner matrices and sample covariance matrices. We refer to \cite{lytova2009central, bai2005convergence, bai2008clt, bao2021quantitative, cipolloni2020functional, landon2022almost, shcherbina2011central} and the reference therein. We also refer to \cite{bao2015spectral,zheng2017clt,li2018structure,wang2021eigenvalues,li2021central,li2020asymptotic,zheng2019hypothesis} and reference therein for statistical applications of CLT for LSS of various random matrix models. For more general polynomials in classical random matrices, the notion ``second order freeness" was raised in \cite{mingo2006second}, \cite{mingo2007second}, and \cite{collins2006second}, which can be used to describe the fluctuation of the LSS. However, second order freeness have been  established  only for Gaussian matrices and Haar orthogonal/unitary matrices which bear a Gaussian nature as well. Further, the second order freeness does not deal with the polynomial with growing number of terms.  For our matrix model $H$ in (\ref{def of H}), a sum of random projections,  on one hand, the projections are generally distributed; on the other hand, our matrix polynomial has $k$ terms and $k\equiv k(N)$ might diverge. Hence, the previous works and methods therein do not apply to our model directly. 
In the sequel, we provide a brief description on our proof strategy.

For the first order behaviour of $\mu_N$ in Theorem \ref{1stLimit}, in order to prove its closeness with $\mu_1\boxplus\dots\boxplus\mu_k$, we shall turn to study the closeness between $m_N$ and $m_\boxplus$, which are the Stieltjes transforms of $\mu_N$ and $\mu_1\boxplus\dots\boxplus\mu_k$, 
respectively. Since $m_\boxplus$ is defined in terms of the subordination system in Proposition \ref{freeaddprop}, 
it is then natural to establish a perturbed subordination system of $m_N$ as well. Then by a stability analysis of the subordination  system we can conclude the closeness between $m_N$ and $m_\boxplus$.  To this end, we first define the approximate subordination functions $\omega_t^c$'s in (\ref{Appsub1}). There are two key steps for the above strategy: 1, establishing a perturbed subordination system for $\omega_t^c$'s; 2, analyzing the stability of the subordination system. For the first step, we rely on the cumulant expansion approach \cite{khorunzhy1996asymptotic, lytova2009central, he2020mesoscopic, lee2018local}. The key idea in this step is to carefully choose the right quantities to start with for the cumulant expansion. 
From the definition of $\omega_t^c$ in (\ref{Appsub1}), it is natural to start the cumulant expansion for $\text{tr} P_tG$'s, where $G=(H-z)^{-1}$ is the Green function of $H$. For instance, if we want to compute the expectation of $\text{tr}P_tG$, we can write $\mathbb{E}\text{tr} P_tG$ as $N^{-1}\sum_{ij} \mathbb{E}(X_{t,ji}[\dots]_{ij})$ and apply cumulant expansion formula in Lemma \ref{cumulant expansion} in Appendix w.r.t. $X_{t,ji}$'s.  However, after applying the expansion to $\text{tr} P_t G$, one will get a new quantity $\text{tr} Q_t G$, with $Q_t=X_t'(X_tX'_t)^{-2}X_t$. There is no a priori estimate for   $\text{tr} Q_t G$, and thus $\text{tr} P_t G$ cannot be estimated if we stop here. Nevertheless,  if we further apply cumulant expansion to $\text{tr} Q_t G$, we will further create new terms which do not have known estimates. Then the system will never be closed. A key observation to solve this issue  is that, if we start with $\text{tr}P_tGP_t$ rather than $\text{tr} P_t G$ and perform the cumulant expansion, we can actually get another relation between $\text{tr} P_tGP_t$ and $\text{tr} Q_tG$, which together with the first relation we can close the system by the trivial fact $\text{tr} P_t GP_t=\text{tr} P_tG$. This discussion shows that even one starts from essentially the same quantity, such as $\text{tr} P_t GP_t$ and $\text{tr} P_tG$, just by writing them in different form, the cumulant expansion may lead to different algebraic relations. Such kind of tricky choices of quantities in cumulant expansions are needed in various steps throughout the work. For the second step, after we get the perturbed system for the approximate subordination functions, we shall compare it with the original system in Proposition \ref{freeaddprop}. A stability analysis for the subordination system  is then necessary to exploit the closeness between $\omega_t^c$'s and $\omega_t$'s, which further implies the closeness between $m_N$ and $m_\boxplus$, in light of (\ref{Appsub2}) and (\ref{sub3})-(\ref{m2}).   When $k=2$, a stability analysis of the subordination system has been done in \cite{bao2016local, bao2020spectral}, for instance. The argument in \cite{bao2016local, bao2020spectral} can be naturally extended to the case of any fixed $k$. However, here in our setting, $k\equiv k(N)$ can be diverging. This will further requires us to exploit a fluctuation averaging for the $k$ error terms  in the perturbed subordination system of $\omega_t^c$'s, in order to counter balance the growth of $k $ (c.f., (\ref{Second error}) in Lemma \ref{Lemma error estimates}). 

For the second order behaviour, i.e, CLT for linear spectral statistics, we turn to estimate the characteristic functions of the centered statistics $\text{Tr}f(H)-\mathbb{E}\text{Tr} f(H)$. Such a strategy dates back to \cite{lytova2009central}. More specifically, we aim at establishing an approximate ODE $(\phi_f^N(x))'=-x\sigma_f^2 \phi_f^N(x)+\text{error}$.  The quantity $(\phi_f^N(x))'$ is suitable to start a cumulant expansion. Thanks to the estimates in the first order part, one can estimate various terms produced via cumulant expansion by $\omega_t$'s and $m_\boxplus$. This eventually gives the expression of $\sigma_f$. Then what remains is the estimate of $\mathbb{E}\text{Tr} f(H)$.  Notice that $\sigma_f$ is of order $1$, but the leading order of $\mathbb{E}\text{Tr} f(H)$ is $N$. Hence, we need to take a step further to identify the second order term of $\mathbb{E}\text{Tr} f(H)$, which is order $1$. Finally, for specific test functions such as those in Corollaries \ref{coroschott} and \ref{corowilk}, we derive simpler expressions for the mean and variance of LSS from Theorem  \ref{FreeCLT}, via involved residue calculation, with the aid of the subordination system in Proposition  \ref{freeaddprop}.

\subsection{Notations and Conventions} \label{s.NC}

Throughout this paper, we regard $N$ as our fundamental large parameter. Any quantities that are not explicit constant or fixed may depend on $N$; we almost always omit the argument $N$ from our notation. We use $\|u\|_\alpha$ to denote the $\ell^\alpha$-norm of a vector $u$. We further use $\|A\|_{(\alpha,\beta)}$ to denote the induced norm $\sup_{x\in \mathbb{C}^n, \|x\|_\alpha=1} \|Ax\|_\beta$ for an $A\in \mathbb{C}^{m\times n}$. We write  $\|A\|\equiv \|A\|_{(2,2)}$ for the usual operator norm of a matrix $A$.  We use $C$ to denote some generic (large) positive constant. The notation $a\sim b$ means $C^{-1}b \leq |a| \leq Cb$ for some positive constant $C$. Similarly, we use $a\lesssim b$ to denote the relation $|a|\leq Cb$ for some positive constant $C$. When we write $a \ll b$ and $a \gg b$ for possibly $N$-dependent quantities $a\equiv a(N)$ and $b\equiv b(N)$, we mean $|a|/b \to 0$ and $|a|/b \to \infty$ when $N\to \infty$, respectively. 

For any Hermitian matrix $A\in \mathbb{C}^{n\times n}$ we use $\lambda_1(A)\geq \dots\geq \lambda_n(A)$ to denote the ordered eigenvalues of $A$ and sometime we also use the notation $\lambda_{\max}(A)\equiv \lambda_{1}(A)$ and $\lambda_{\min}(A)\equiv \lambda_{n}(A)$ instead. For any rectangle matrix $B\in \mathbb{C}^{n\times m}$, we use $\sigma_1(B)\geq \dots \geq \sigma_{m\wedge n}(B)$ to denote the ordered singular values of $B$.

Throughout the paper, for $A\in\mathbb{C}^{n\times  n}$, we use $\text{Tr} A$ to denote the trace of $A$, and use $\text{tr} A=N^{-1}\text{Tr}A$ to denote the  trace normalized by $N$. 

\subsection{Organization}
The paper is organized as follows. In Section \ref{Sec. Intro and main}, we introduce the matrix model and state our main results and proof strategy.  Numerical performances of our methods are investigated using both simulated and real data in Section \ref{Sec. Simulation}.  Section \ref{Sec. Preliminaries} is devoted to the preliminaries, which is crucial for later discussion.  Then in Section \ref{Proofof1stLimit}, we prove (\ref{062703}) in Theorem \ref{1stLimit}. The proof of our main result, Theorem \ref{FreeCLT}, is stated in Section \ref{ProofofCLT}. The proofs of other lemmas, propositions and theorems are provided in Appendix.

\section{Simulation studies and real data analysis}\label{Sec. Simulation}
In this section, we perform simulation studies to demonstrate the finite sample behaviors of Schott's and Wilks' statistics, with comparison to the two existing methods proposed by \cite{jiang2013testing} and \cite{YamadaH17T}. 
Furthermore, we illustrate how these methods apply to real datasets' block-diagonal covariance selection problem. Due to the space limitation, we only present some selected figures with significant properties in the paper. The full results can be collected from \url{https://github.com/huj156/Block-Correlation-Matrix.git} including the source codes.

\subsection{Simulation}
 Our objectives in the simulation are as the following: (i) Examine the asymptotic properties as delineated in the theorems for finite sample sizes; (ii) Provide some empirical observation on their relative rates of convergence in the high dimensional cases; (iii) Explore the robustness of our proposed statistics against two existing methods proposed by \cite{jiang2013testing} and \cite{YamadaH17T}, which are called  JBZ statistic and YHN statistic in the sequel, respectively.

  Let $\mathbf{y} = (\mathbf{y}'_1,  \dots , \mathbf{y}'_k)'=T\mathbf{x}+\mathbf{\mu}$, where $\mathbf{x}=(x_1,\dots,x_p)$ with i.i.d. entries and  $\Sigma=TT'=(\Sigma_{ij})_{k\times k}$ is the block  covariance matrix corresponds to $(\mathbf{y}'_1, \dots , \mathbf{y}'_k)'$. In the following numerical studies, for simplicity, we set  $\mathbf{\mu}={0}_{p\times1}$ and $T=\Sigma^{1/2}$, and we always use the result in Theorem \ref{RemoveSampleMean}, i.e., the sample mean is subtracted although the population mean is taken to be $0$ here. 
  We examine the following three  different distributions of $\mathbf{x}$: 
  $$\mbox{D1:}~x_1\sim N(0,1);~~\mbox{D2:}~x_1\sim (\chi^2(1)-1)/\sqrt{2};~~\mbox{D3:}~x_1\sim t_{5}/\sqrt{5/3},$$
   where $t_5$ and $\chi^2(1)$ stand for Student's $t$-distribution with five degrees of freedom and  $\chi^2$-distribution with one degree of freedom, respectively. 
Notice that the above three population distributions have different kurtosis and the fifth moment of $t_{5}$ does not exist.
   For the covariance matrix, we also set three structures: 
  \begin{itemize}
  	\item[M1:] $\Sigma_{tt}=I_{p_t}$, $t \in [\![ k]\!]$;
  	\item[M2:] $\Sigma_{tt}=B\cdot \left(0.3^{|j_1-j_2|^{1 / 3}}\right)_{j_1,j_2=1}^{p_t} \cdot B$, $t \in [\![ k]\!]$,  where  $$B={\rm diag}\big((0.5+1 /(p_t+1))^{1 / 2}, \ldots,(0.5+p_t /(p_t+1))^{1 / 2}\big),$$ and $j_1$, $j_2$ are the coordinates of the  matrix entries.
  	\item[M3:] $\Sigma_{tt}=\frac{1}{p_t}U_tU_t'$, $t \in [\![ k]\!]$, where $U_t$ is a $p_t\times 2p_t$ random matrix whose entries follow the continuous uniform distribution $U(1,5)$.

  \end{itemize}
   We remark here the setting M2 is adopted from \cite{YamadaH17T}.  Since YHN depends on the estimation of the covariance matrices, we choose a more general block structure, M3, to examine the performance of YHN.  The settings of the sample sizes are constructed as follows: 
   $$\mbox{G1}:N=2p;~~\mbox{G2}:N=p+3;~~\mbox{G3}:N=3\max\{p_t,t \in [\![ k]\!]\}.$$
The above three different choices of sample sizes are set for the examination of the cases that $y$ is smaller than $1$, close to $1$, and bigger than $1$, respectively. Notice that in the case $p>N-1$,   Wilks and JBZ are not well defined. For the choice of the different groups, we consider the following three scenarios:
    $$\mbox{S1: $k=4$, $p_1=p_2=p_3=p_4=p/4$;~~S2: $k=p/2$, $p_1=\dots=p_k=2$}$$
    $$\mbox{S3: $k=4$, $p_1=p_2=2$, $p_3=p_4=p/2-2$.}$$
The settings S1 and S2 are the cases of equally big and small blocks, respectively. For the sake of comparison, we set S3 as the case for two small blocks and two large blocks. In the current numerical studies, the null hypothesis is defined as
$$\mbox{H0:	$\Sigma_{ts}={0}_{p_t\times p_s}$, for any $s, t \in [\![ k]\!]$, $t\neq s$}.$$
For  the alternative hypothesis, we adopt the  following  three  settings 
$$\mbox{H1}:\Sigma_{ts}=\rho\mathds{1}_{p_t}\mathds{1}_{p_s}';~~\mbox{H2}:\Sigma_{ts}=\rho\mathds{I}_{ts};~~\mbox{H3}:\Sigma_{ts}=\rho\mathds{1}_{ts}, \text{ for any } s, t \in [\![ k]\!], t\neq s. $$
Here $\mathds{1}_{p_t}$ is the all-one vector of dimension $p_t$,  $\mathds{I}_{ts}$ is the $p_t\times p_s$ rectangular matrix whose main diagonal entries are $1$ and the others are $0$, and $\mathds{1}_{ts}$ is the $p_t\times p_s$ rectangular matrix whose first entry is equal to $1$ and the others are equal to $0$.
The empirical results are obtained based on 10,000 replications with the dimension $p = 32$ and $p=160$, respectively.

In the captions of these figures, ``D$\ast$M$\ast$G$\ast$S$\ast$" stands for the setting D$\ast$, M$\ast$, G$\ast$, and S$\ast$. All the presented figures are based on the dimension $p=32$. Under settings G3S1 and G3S2,   JBZ and Wilks are not applicable. Thus the corresponding figures  do not contain these simulated curves. The simulated distributions under the null hypothesis are based on the kernel density estimation method with a normal kernel function. Since   Schott, Wilks and JBZ are scale-invariant, and  YHN does not perform very well under the null hypothesis with non-identity covariance matrices (e.g., Figure \ref{D3M3G3S3}), hence we do not consider the settings M2 and M3 in the empirical power studies.

\begin{figure}[htbp]
\centering
\includegraphics[height=2.2cm]{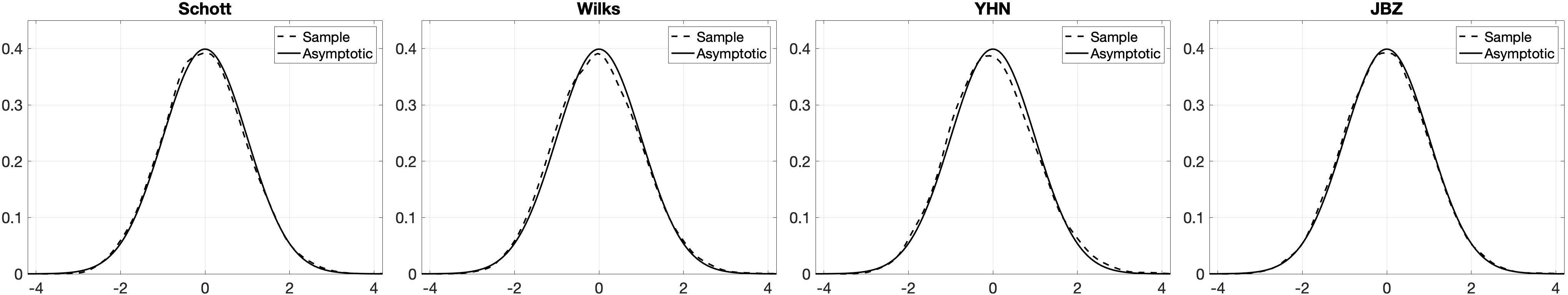}
\caption{Simulated distribution under the null hypothesis for D1M1G1S1}
\label{D1M1G1S1}\
\end{figure}

\begin{figure}[htbp]
\centering
\includegraphics[height=2.2cm]{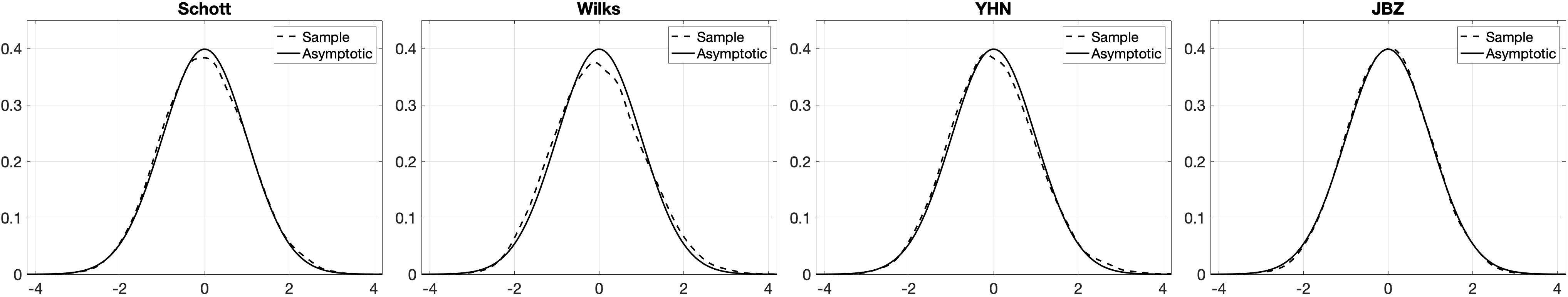}
\caption{Simulated distribution under the null hypothesis for D3M1G3S3}
\label{D3M1G3S3}
\end{figure}

\begin{figure}[htbp]
\centering
\includegraphics[height=2.2cm]{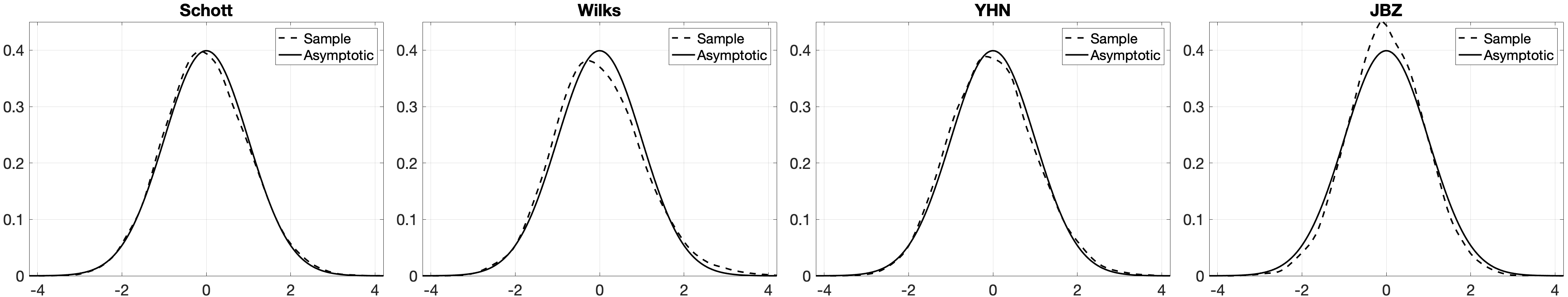}
\caption{Simulated distribution under the null hypothesis for D3M1G2S3}
\label{D3M1G2S3}
\end{figure}

\begin{figure}[htbp]
\centering
\includegraphics[height=2.2cm]{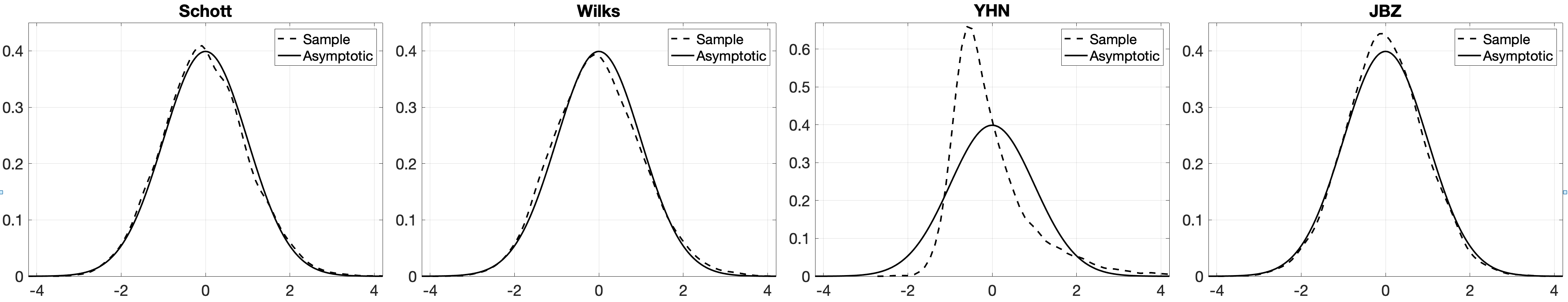}
\caption{Simulated distribution under the null hypothesis for D3M3G3S3}
\label{D3M3G3S3}
\end{figure}

\begin{figure}[htbp]
\centering
\begin{subfigure}
         \centering
\includegraphics[height=2.2cm]{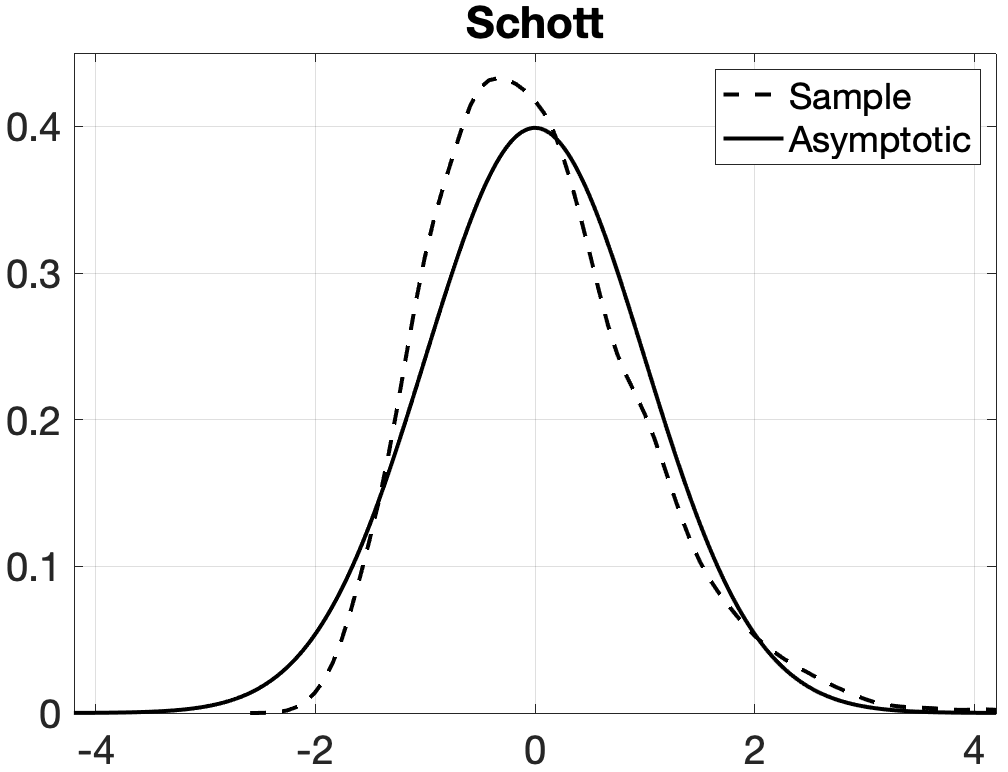}
     \end{subfigure}
\begin{subfigure}
         \centering
\includegraphics[height=2.2cm]{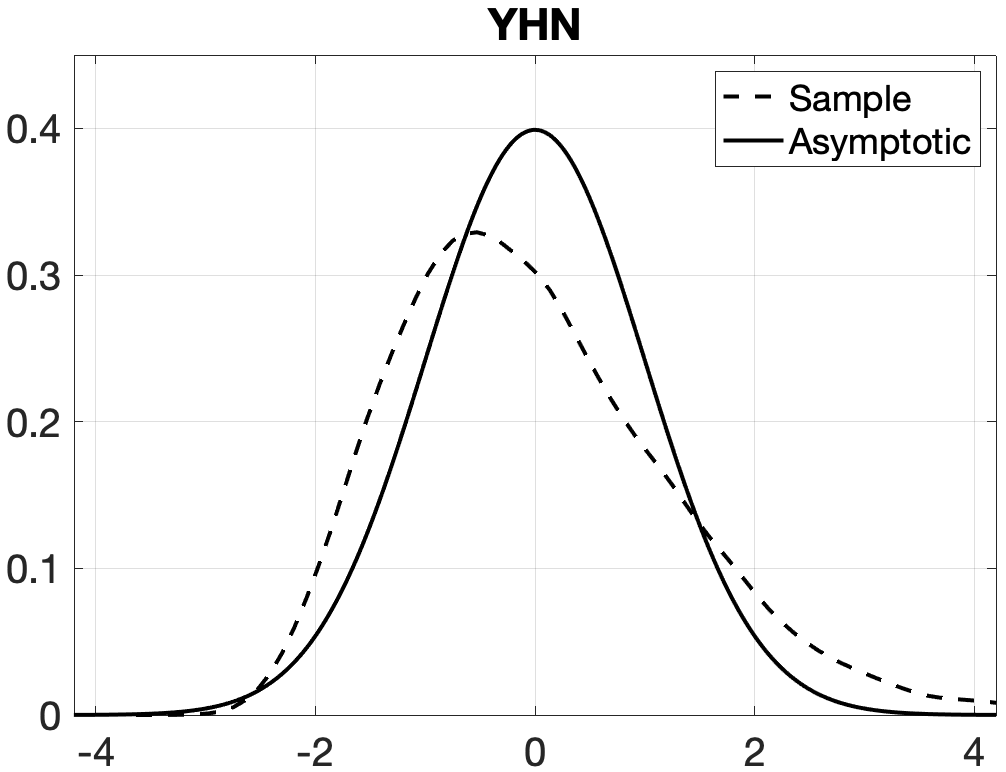}
     \end{subfigure}
\caption{Simulated distribution under the null hypothesis for D3M1G3S2}
\label{D3M1G3S2}
\end{figure}

\begin{figure}[htbp]
\centering
\includegraphics[height=3cm]{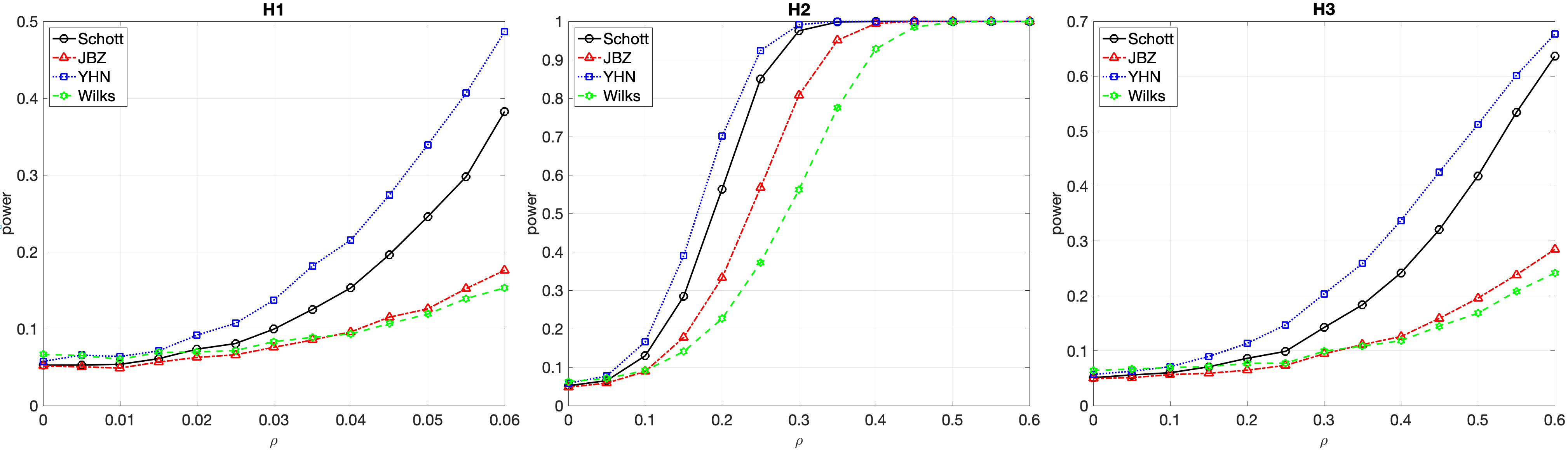}
\caption{Empirical power for D3M1G2S1}
\label{D3M1G2S1_power}
\end{figure}
\begin{figure}[htbp]
\centering
\includegraphics[height=3cm]{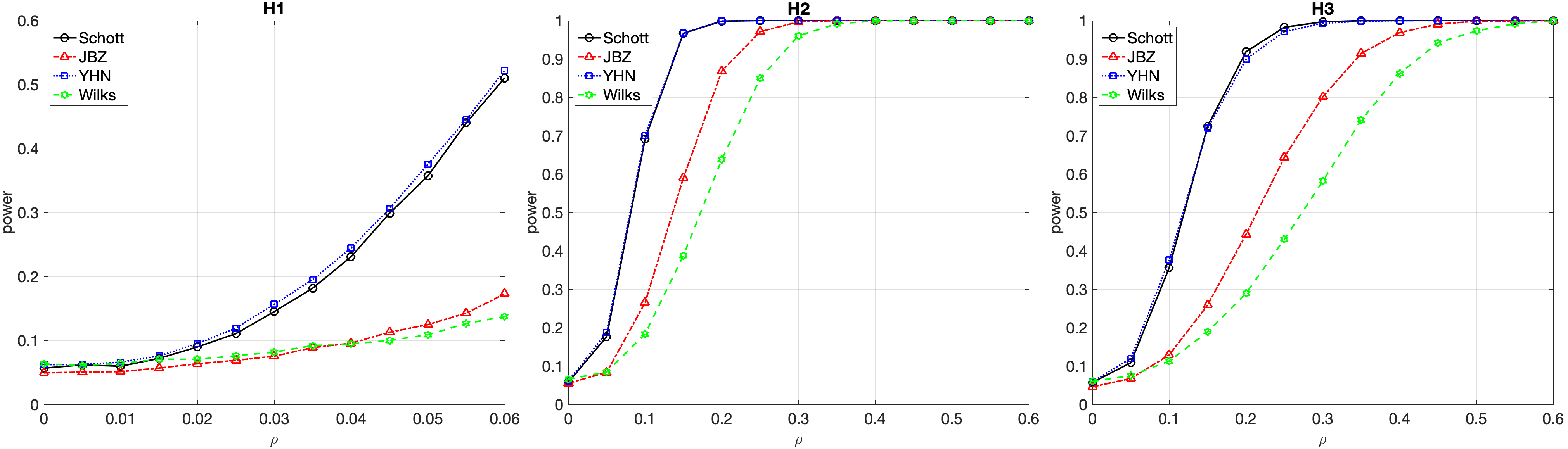}
\caption{Empirical power for D3M1G2S2}
\label{D3M1G2S2_power}
\end{figure}
\begin{figure}[htbp]
\centering
\includegraphics[height=3cm]{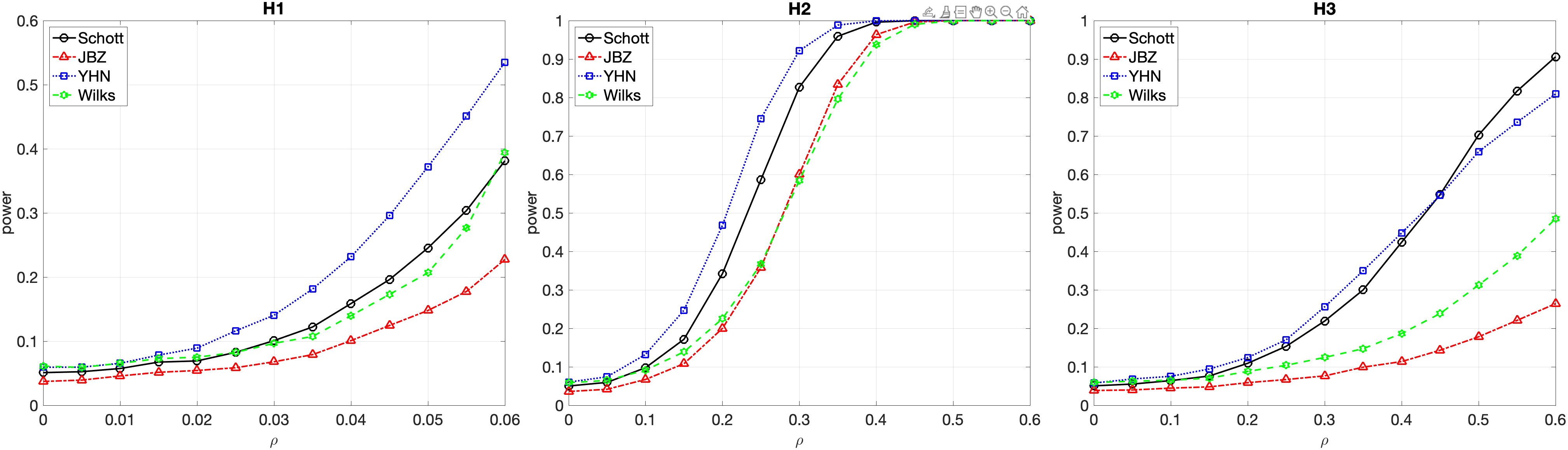}
\caption{Empirical power for D3M1G3S3}
\label{D3M1G3S3_power}
\end{figure}

Below are our conclusions based on our simulation studies:
\begin{enumerate}
\item For the null hypothesis, under suitable conditions,  the approximation accuracy of all the four statistics improves as the values of $p$ and $N$ increase. But in fact, the empirical distributions fit well enough when the values of $p$ and $N$ are moderate. Different underlying distributions with finite fourth moments do not significantly affect the empirical distributions of the four statistics when the sample size is large (even the CLT of JBZ is proved under Gaussian distribution (e.g., Figures \ref{D1M1G1S1} and \ref{D3M1G3S3}). By comparing these numerical results of the four statistics under different settings, we find that  Schott is the most stable one. Schott performs well in all settings except the case that the sample size $N$ and the dimensions of each block $p_t$ are small, but the number of the blocks $k$ is big (see Figure \ref{D3M1G3S2}). It is worth noting that this case violates Assumption \ref{assum2}, which means our CLTs could not hold for $y\to\infty$. When the total dimension $p$ is close to the sample size $N$ and one relatively large block exists, Wilks and JBZ have negligible bias,  but the bias of JBZ is slightly more severe than Wilks  (e.g., Figure \ref{D3M1G2S3}). For YHN, since it relies on the estimation of the population covariance matrices, non-identity covariance matrices could give rise to its unreliable approximation (e.g., Figure \ref{D3M1G3S3}). More importantly, the population covariance matrices are unknown in practice. Hence it is hard to be confident that the Type one error is controlled when applying YHN.

\item For the alternative hypothesis, we remark that the settings H1, H2 and H3 stand for the dense covariance matrix, sparse covariance matrix and extremely sparse covariance matrix, respectively. From the numerical results, we find that the empirical powers for all the four statistics increase reasonably as $\rho$ increases. In comparison, in most cases,   Schott and YHN are more powerful than  Wilks and JBZ. Especially under setting S2,  Schott and YHN perform very close to each other (e.g., Figure \ref{D3M1G2S2_power}).  Moreover, under settings H1 and H2,   YHN seems to be more powerful than    Schott. On the contrary,  under setting H3,    Schott would be more powerful.  For the comparison of  Wilks and JBZ, 
except for the settings G2S1 and G2S2,   Wilks shows better performance than   JBZ in general (e.g., Figures \ref{D3M1G2S1_power}-\ref{D3M1G3S3_power}). It is worth noting that we do not simulate the empirical powers under non-identity covariance matrice settings, because under which the approximation of   YHN could be unreliable anymore. Therefore, in summary, we recommend  Schott for the independence test problem for several groups, since it is simple, robust and powerful.
\end{enumerate}

In this sequel, we analyze two datasets for illustrating the efficacy of our methods. 
\subsection{Real data illustration with stock return data}
The dataset is collected from the Center for Research in Security Prices and contains the arithmetic daily stock returns of 75 companies for the trading days of the first half of 2014, i.e., from 1 January 2014 to 30 June 2014, with a total of 124 trading days.
 According to Fama and French's 48-industry classification \cite{FamaF97I}, the 75 companies belong to five industries, i.e., Non-Metallic and Industrial Metal Mining (Mines), Apparel (Clths), Healthcare (Hlth), Medical equipment (MedEq) and Food products (Food), and each industry contains 15 stocks. 
  That is the sample size $N=124$, the total dimension $p=75$, the number of the blocks $k=5$, and the dimensions of each block $p_1=\dots=p_5=15$.
Our interest is whether the five blocks are independent.

First, we use the four statistics Schott, Wilks, 
YHN and JBZ to test the five industries with the original dataset, respectively. 
 It is not surprising that all four statistics reject the null hypothesis with significantly small $p$-values.  Next, we apply principal component analysis to the stock data and remove one common factor. Then the smallest $p$-value among the four statistics is 0.255. The $p$-values are summarized in Table \ref{realdata1}. This means we have no evidence to reject the independence of the five blocks that removes one common factor. However, in this case, the $p$-values for each block obtained by the four statistics are all smaller than 0.001 (see Table \ref{realdata2}). Therefore, we can consider that the common factor is the market effect and the rest of the components as industry-specific effects. Detailed discussions for the structure of stock returns can be found in \cite{FanL13L,BiH22S}. 

\begin{table}[H]
\centering
\begin{tabular}{c|cccc}
\hline
   &Schott & Wilks & YHN &JBZ\\
      \hline
    Before&  $4.3866\times 10^{-18}$&$5.3873\times 10^{-7}$&$2.2292\times10^{-231}$& $1.4427\times10^{-4}$\\
       \hline
      After& 0.2550&0.4491&0.3998&0.3258\\
       \hline
\end{tabular}
\caption{The $p$-values obtained by statistics Schott, Wilks,  YHN  and JBZ for total of the five industries dataset that before and after removing one common factor, respectively. }
\label{realdata1}
\end{table}

\begin{table}[H]
\centering
\begin{tabular}{ccccc}
\hline
   &Schott & Wilks & YHN &JBZ\\
  
      \hline
                Mines & $9.1347\times10^{-12}$&$1.6640\times10^{-11}$&$0.0044$&$1.9010\times10^{-8}$\\    
          \hline
   Clths  & $3.3190\times10^{-180}$&$2.3775\times10^{-41}$&$2.7299\times10^{-122}$&$4.8450\times10^{-34}$\\
       \hline
          Hlth  & 0&$7.5521\times10^{-319}$&0&$4.8914\times10^{-168}$\\
       \hline
          MedEq  & $6.2410\times10^{-51}$&$9.2337\times10^{-18}$&$1.6016\times10^{-8}$&$2.4543\times10^{-15}$\\
       \hline
          Food  & $2.2444\times10^{-291}$&$1.6253\times10^{-73}$&$7.5335\times10^{-209}$&$2.0750\times10^{-57}$\\
       \hline

\end{tabular}
\caption{The $p$-values obtained by statistics Schott, Wilks,  YHN  and JBZ for each of the five industries dataset that after removing one common factor, respectively. }
\label{realdata2}
\end{table}

\subsection{Real data illustration in block-diagonal covariance selection}
It is well known that the covariance matrix plays a central role in multivariate statistical analysis. Also, the sample covariance 
matrix is no longer a consistent estimator of the high-dimensional population covariance matrix. Therefore, there exists an amount of work for high-dimensional covariance matrix estimation under different structure assumptions, such as sparsity (e.g., \cite{ElKaroui08O,LamF09S,CaiL11C}), handedness (e.g., \cite{WuP03N,BickelL08R,QiuC12T}), block-diagonal (e.g., \cite{CaiZ10O,DonohoG18O,DevijverG18B,Perrot-DockesL22E}). Especially,  the block-diagonal covariance matrix has the inherent advantage of reducing computational complexity, because under such an assumption, one can easily separate the variables into several uncorrelated parts. For instance,  in the analysis of gene expression data, the number of observations could be much smaller than the number of variables. Thus,  it is more feasible to analyze some unrelated subsets one by one. Therefore, in practice, the proposed statistics in this paper can be used to test whether the selected subset variables are uncorrelated. 

For illustration, we apply the proposed methods to a gene dataset.
The original data are the expression of 52,580 genes across 69 observations and can be collected from the Recount database \cite{FrazeeL11R}.  \cite{DevijverG18B}  identified 200 most variable genes among the 52,580 genes, and then investigated the block covariance structure among the 200 genes. Finally, the 200 genes are partitioned into four blocks of size 18, 13, 8 and 5, four blocks of size 3, two blocks of size 2,  and the remaining blocks are of size 1.
Detailed analysis can be found in the supplementary materials of \cite{DevijverG18B}. 

We use the four statistics Schott, Wilks
YHN and JBZ to test whether these blocks are uncorrelated. For the sake of being well defined for Wilks and JBZ, we only consider the blocks whose sizes are bigger than one, i.e,  ten blocks of size 18, 13, 8, 5, 3, 3, 3, 3, 2, 2, respectively.  Therefore, the sample size $N=69$, the total dimension $p=60$, and the number of the blocks $k=10$. The $p$-values of the four statistics are stated in Table \ref{table_p}. From these results, we have strong evidence to believe that these blocks are dependent. In fact,  we also apply the statistic Schott to test the independence of each pair of the ten blocks.  It is found that all the $p$-values for the paired blocks are smaller than 0.05 except for the pair of the block of size 8 and one of the blocks of size 2, whose $p$-value is 0.6814. Therefore, we think the selected blocks in \cite{DevijverG18B} should be re-examined. 
\begin{table}[H]
\centering
\begin{tabular}{cccc}
\hline
   Schott & Wilks & YHN &JBZ\\
  
      \hline
       0 &$ 5.2694\times 10^{-215}$ &$2.1667\times 10^{-9}$&$5.6930\times 10^{-162}$\\
       \hline

\end{tabular}
\caption{The $p$-values obtained by statistics Schott, Wilks,  YHN  and JBZ for the gene dataset. }
\label{table_p}
\end{table}

\section{Preliminaries}\label{Sec. Preliminaries}
In this section we collect some necessary  notations and technical tools that are used throughout the paper.

We first introduce the notions of stochastic domination which was introduced in \cite{erdHos2013averaging}.
 
 \begin{defin}[Stochastic domination] \label{def.sd}
 	Let 
 	$$
 		\mathsf{X}=\left(\mathsf{X}^{(N)}(u): N \in \mathbb{N}, u \in \mathrm{U}^{(N)}\right), \mathrm{Y}=\left(\mathsf{Y}^{(N)}(u): N \in \mathbb{N}, u \in \mathsf{U}^{(N)}\right)
 	$$
 	be two families of random variables, where $\mathsf{Y}$ is nonnegative, and $\mathsf{U}^{
(N)}$ is a possibly $N$-dependent parameter set.

We say that $\mathrm{X}$ is stochastically dominated by $\mathsf{Y}$, uniformly in $u$, if for all small $\epsilon > 0$  and large $D > 0$ ,
$$
\sup _{u \in \mathsf{U}^{(N)}} \mathbb{P}\left(\left|\mathsf{X}^{(N)}(u)\right|>N^{\varepsilon} \mathsf{Y}^{(N)}(u)\right) \leqslant N^{-D}
$$
for large enough $N > N_0(\epsilon, D)$. If $\mathsf{X}$ is stochastically dominated by $\mathsf{Y}$, uniformly in $u$, we use the notation
$\mathsf{X} \prec \mathsf{Y}$ , or equivalently $\mathsf{X} = O_{\prec}(\mathsf{Y})$. Note that in the special case when $\mathsf{X}$ and $\mathsf{Y}$ are deterministic, $\mathsf{X} \prec \mathsf{Y}$ means that for any given $\epsilon > 0$, $|\mathsf{X}^{(N)}(u)| \le N^{\epsilon}\mathsf{Y}^{(N)}(u)$ uniformly in $u$, for all sufficiently large $N \ge N_0(\epsilon)$.
\end{defin}

\begin{defin}[High probability event] \label{def.high-probab}
We say an event $\mathcal{E}\equiv \mathcal{E}(N)$ holds with {\it high probability } (in $N$) if for any fixed $D>0$, $\mathbb{P}(\mathcal{E}^c)\leq N^{-D}$ when $N$ is sufficiently large. 
\end{defin}
 
We have the following elementary result about stochastic domination.
\begin{lemma} \label{prop_prec} Let
	\begin{equation*}
	\mathsf{X}_i=(\mathsf{X}^{(N)}_i(u):  N \in \mathbb{N}, \ u \in \mathsf{U}^{(N)}), \   \mathsf{Y}_i=(\mathsf{Y}_i^{(N)}(u):  N \in \mathbb{N}, \ u \in \mathsf{U}^{(N)}),\quad i=1,2
	\end{equation*}
	be families of  random variables, where $\mathsf{Y}_i, i=1,2,$ are nonnegative, and $\mathsf{U}^{(N)}$ is a possibly $N$-dependent parameter set.	Let 
	\begin{align*}
	\Psi=(\Psi^{(N)}(u): N \in \mathbb{N}, \ u \in \mathsf{U}^{(N)})
	\end{align*}
	be a family of deterministic nonnegative quantities. We have the following results:
	
(i)	If $\mathsf{X}_1 \prec \mathsf{Y}_1$ and $\mathsf{X}_2 \prec \mathsf{Y}_2$ then $\mathsf{X}_1+\mathsf{X}_2 \prec \mathsf{Y}_1+\mathsf{Y}_2$ and  $\mathsf{X}_1 \mathsf{X}_2 \prec \mathsf{Y}_1 \mathsf{Y}_2$.

 (ii) Suppose $\mathsf{X}_1 \prec \Psi$, and there exists a constant $C>0$ such that  $|\mathsf{X}_1^{(N)}(u)| \leq N^{C}\Psi^{(N)}(u)$ a.s. uniformly in $u$ for all sufficiently large $N$. Then $\E \mathsf{X}_1 \prec \Psi$. 
\end{lemma}

The foregoing lemma indicates that to get the expectation bound of a random quantity, we should have both a typical  bound with high probability and a deterministic crude bound for it. Therefore, to facilitate the estimations in our paper, we define the following ``truncated" expectation operator. Let $\chi(x)$ be a smooth cutoff which  equals $0$ when $x>2N^K$ and $1$ when $x<N^K$ for some sufficiently large constant $K>0$ and $|\chi^{(n)}(x)|=O(1)$ for all $n \ge 1$. We define for any random variable $\xi$ in the sequel
\begin{align}
	\E^{\chi}(\xi) := \E (\xi\cdot \Xi) \label{EXit}
\end{align}
where 
\begin{align}
	\Xi:=\prod_{t=1}^k\chi(\tr (X_tX_t')^{-1})\chi(\tr (X_tX_t')) \label{Xit}
\end{align}
is used to control $\|(X_tX_t')^{-1}\|$ and $\|X_tX_t'\|$ crudely but deterministically.

Next, we define the approximate subordination functions, which will play a key role throughout this paper.
\begin{defin}[Approximate subordination functions]
	\begin{align}\label{Appsub1}
		\omega_t^c(z) := z - \frac{\sum_{s\neq t}\tr\bbP_s\bbG(z)}{m_N(z)}=\frac{\tr P_tG(z) - 1}{m_N(z)}, \quad z \in \mathbb{C}^{+},\quad t = [\![ k]\!].
	\end{align}
\end{defin}
The functions $\omega_t^c(z)$'s turn out to be good approximations to the subordination functions \eqref{sub2}. A direct consequence of the definition in \eqref{Appsub1} is that
\begin{align}\label{Appsub2}
	z - \omega_1^c(z) - \omega_2^c(z) - \cdots - \omega_k^c(z) = \frac{k-1}{m_{N}(z)}	, \quad z \in \mathbb{C}^{+},
\end{align}
which matches (\ref{sub3}). The following property of $\omega_t^c(z)$ implies that $\Im \omega_t^c(z)$ is large when $\Im z$ is large with high probability. 
\begin{lemma}\label{Imomegatc}
	Under Assumptions \ref{assum1} and \ref{assum2} , there exists sufficiently large constant $\eta_0 > 0$, such that for any fixed $z \in \mathbb{C}^{+}$ with $\Im z \ge \eta_0$, we have 
	\begin{align}
	\Im \omega_t^c(z) - \Im z \ge c,	 \label{lower bound for im omega}
	\end{align}
for some small constant $c > 0$ with high probability. 
\end{lemma}
The proof of Lemma \ref{Imomegatc} is given in Section \ref{Sec A} of Appendix. For notational simplicity, we further define
\begin{align*}
	Q_t : = X_t'(X_tX'_t)^{-2}X_t, \quad W_t := (X_tX'_t)^{-1}X_t, \quad t \in [\![ k]\!].
\end{align*}
Finally, we introduce the shorthand notations 
\begin{align*}
\sum_{ij}^{(t)}:=\sum_{i=1}^N\sum_{j=1}^{p_t},\qquad \partial_{t,ji} := \frac{\partial}{\partial X_{t,ji}}.
\end{align*}
With these definitions and notations, we prove (\ref{062703}) in Theorem \ref{1stLimit} in Section \ref{Proofof1stLimit}, and further prove Theorem \ref{FreeCLT} in Section \ref{ProofofCLT}. The proofs of other main results are stated in Appendix.

\section{First order limit: proof of (\ref{062703}) in Theorem \ref{1stLimit}}\label{Proofof1stLimit}
In this section, we investigate the first order behavior of the ESD of $H$. Specially, we will prove the general case in Theorem \ref{1stLimit}, i.e, (\ref{062703}). The proof of (\ref{062701}) and (\ref{062702}) will be stated in Section \ref{Sec Proof of 1.9} of Appendix. The proof of (\ref{062703}) will involve a stability analysis for the subordination system in Proposition \ref{freeaddprop}. We note that when $k$ is fixed (case 1), the stability analysis is an extension of the counterpart in \cite{bao2016local} for the case of $k=2$; see \cite{kargin2015subordination} and \cite{bao2020spectral} also.   We emphasize that the stability analysis in \cite{bao2016local}, \cite{kargin2015subordination} and \cite{bao2020spectral} are down towards the local scale, but here we only need a discussion on global scale. However, for our general result (\ref{062703}), since $k$ can be $N$-dependent, the stability analysis become more delicate. Especially, we need to further exploit a fluctuation averaging for linear combinations of error terms in the approximate subordination system. 

In the very first step, we provide the following lemma which gives a rough description of the support of $\mu_\boxplus$ under different setting of $\hat{y}$, which will be helpful for our later analysis.
\begin{lemma}\label{supportofmu}
	Under Assumptions \ref{assum1} and \ref{assum2}, if $\hat{y} \in (0,1)$, there exists positive constants $b > a > 0$, such that
	\begin{align}
		\mathrm{supp}(\mu_{\boxplus}) \subset \{ 0\} \cup [a,b]. \label{supp of mu}
	\end{align}
	Generally, if $\hat{y} \in (0,\infty)$, there exists positive constants $C > 0$, such that
	\begin{align}
		\mathrm{supp}(\mu_{\boxplus}) \subset [0,C]. \label{supp of mu2}
	\end{align}
\end{lemma}
Then we set up our working domains. These domains are not only for the proof of (\ref{062703}) in Theorem \ref{1stLimit}, but also for the proof of Theorem \ref{FreeCLT} in the next section. Recall the definition in (\ref{contour12}). Let $\bar{\gamma}^0_1$ and $\bar{\gamma}^0_2$ be the parts of $\gamma^0_1$ and $\gamma^0_2$ with $|\Im z| \geq N^{-K}$ for some large (but fixed) $K$, and $\bar{\gamma}_1$ and $\bar{\gamma}_2$ are defined analogously. The truncation here is to ensure that $\| G(z)\|$ has deterministic upper bound, so that we can do high order moment estimates for $G(z)$ and its functionals when $z$ lies on the truncated contours, in light of Lemma \ref{prop_prec} (ii).

Further denote $(\bar{\gamma}^0_1)^{+}$ be the part of $\bar{\gamma}^0_1$ with $\Im z \ge N^{-k}$, and $(\bar{\gamma}^0_2)^{+}$, $(\bar{\gamma}_1)^{+}$ and $(\bar{\gamma}_2)^{+}$ are defined similarly. 
For simplicity, we state the estimates on $(\bar{\gamma}^0_1)^{+}$, $(\bar{\gamma}^0_2)^{+}$, $(\bar{\gamma}_1)^{+}$ and $(\bar{\gamma}_2)^{+}$ only. The  estimates on their complex conjugate are analogous. The following lemma gives the high probability bound of the Green function $G(z)$ and $\tr G(z)$.

\begin{lemma}\label{Gbound}
	Under Assumptions \ref{assum1} and \ref{assum2}, we have that $\| H\| $ is bounded with high probability. Further, if $\hat{y} \in (0,1)$, for any fixed $z \in {\gamma}^0_1\cup{\gamma}^0_2$,  we have $|\tr G(z)|,\; \| G(z)\| \sim 1$ with high probability. The same bounds hold for $z \in {\gamma}_1\cup{\gamma}_2$ with $\hat{y} \in (0,\infty)$. 
\end{lemma}
Now we present one of the key estimates of our paper. The following proposition gives an approximation for $\tr G(z)$, when $z$ lies in our working domains.  
\begin{prop}\label{keyProp}
	Under Assumptions \ref{assum1} and \ref{assum2}, if $\hat{y} \in (0,1)$, for any fixed $z \in (\bar{\gamma}^0_1)^{+}\cup(\bar{\gamma}^0_2)^{+}$,  we have
	\begin{align}
		\left|{\omega}^c_{t}(z)-{\omega}_{t}(z)\right| \prec \frac{1}{N},\qquad t \in [\![k]\!], \label{omegatAppro}
	\end{align}
	and
	\begin{align}
		\left| \tr \bbG(z) - m_{\boxplus}(z) \right| \prec  \frac{1}{N} . \label{trGGap}
	\end{align}
	The same bounds hold for $z \in (\bar{\gamma}_1)^{+}\cup(\bar{\gamma}_2)^{+}$ with $\hat{y} \in (0,\infty)$. 
\end{prop} 

With the help of the above lemmas and proposition, we can prove (\ref{062703}).
\begin{proof}[Proof of (\ref{062703})]
	From Lemmas \ref{supportofmu} and \ref{Gbound}, we know that both $\mu_\boxplus$ and $\mu_N$ are compactly supported (with high probability). Therefore,  by Cauchy's integral formula, we have for any fixed integer $\ell > 0$, 
	\begin{align*}
		\int x^\ell{\rm d} \mu_N - \int x^\ell {\rm d}\mu_\boxplus =& \frac{-1}{2\pi{\rm i}}\oint_{\gamma_1} z^{\ell} (\tr G(z) - m_{\boxplus}(z)) {\rm d}z \\
		=& \frac{-1}{2\pi{\rm i}}\oint_{\bar{\gamma}_1} z^{\ell} (\tr G(z) - m_{\boxplus}(z)) {\rm d}z + O_{\prec}(N^{-K})
	\end{align*}
	with high probability. Then using (\ref{trGGap}), we get
	$
		| \int x^\ell{\rm d} \mu_N - \int x^\ell {\rm d}\mu_\boxplus| \prec N^{-1},
	$
	which completes the proof of (\ref{062703}).
\end{proof}

The rest of this section is devoted to the proof of Proposition \ref{keyProp}. The proofs of Lemmas \ref{supportofmu} and \ref{Gbound} are given in Section \ref{Sec B} of Appendix.

\begin{proof}[Proof of Proposition \ref{keyProp}]
	Recall the subordination system in (\ref{sub2})-(\ref{F2}). Our aim is to establish an approximate system for $\omega_i^c(z)$'s (c.f., (\ref{Appsub1})) and $m_N(z)$ (c.f., (\ref{def of m_N})), which can be regarded as a perturbation of the system in (\ref{sub2})-(\ref{F2}). Then, by inverting the system, we can get the closeness between $\omega_i^c(z)$ and $\omega_i(z)$, and also that between $m_N(z)$ and $m_{\boxplus}(z)$.   
	
	To ease the presentation, we organize the subordination system in a more compact form. We define a function $\Phi_{\mu_1,\cdots,\mu_k}: \left( \mathbb{C}^{+} \right)^{k+1} \to \mathbb{C}^{k}$ as
\begin{align}\label{Phi1}
	\Phi_{\mu_1,\cdots,\mu_k} (\omega_1,\omega_2,\cdots,\omega_k,z):=\left(
	\begin{array}{c}
		(k-1)F_{\mu_{1}}\left(\omega_{1}\right)-\omega_{1}-\omega_2-\cdots-\omega_{k}+z \\ 
		(k-1)F_{\mu_{2}}\left(\omega_{2}\right)-\omega_{1}-\omega_2-\cdots-\omega_{k}+z \\
		\vdots\\
		(k-1)F_{\mu_{k}}\left(\omega_{k}\right)-\omega_{1}-\omega_2-\cdots-\omega_{k}+z
	\end{array}
	\right).
\end{align} 
Considering $\mu_t$ as fixed, the equation
\begin{align}\label{Phi2}
	\Phi_{\mu_1,\cdots,\mu_k} (\omega_1,\cdots,\omega_k,z) = 0
\end{align}
is equivalent to \eqref{sub2}, and by Proposition \ref{freeaddprop}, there are unique analytic functions $\omega_i(z)$'s satisfying \eqref{sub1} that solve \eqref{Phi2} in terms of $z$. Within this section, we use the conventions that $\omega_i$'s are generic variables in $\mathbb{C}^+$ and  (with a slight abuse of notation) $\omega_i(z)$'s are the subordination functions solving (\ref{Phi2}) in terms of $z$.

We set 
\begin{align*}
	\Gamma_{\mu_1,\cdots,\mu_k}(\omega_1, \ldots, \omega_k) := \| \left(\mathrm{D}\Phi(\omega_1, \cdots,\omega_k)\right)^{-1}  \|_{\infty,\infty}.
\end{align*}
 Here $\mathrm{D}\Phi(\omega_1,\omega_2, \cdots,\omega_k)$ is the partial Jacobian matrix of \eqref{Phi1} w.r.t. $\omega_i$'s, i.e.,
\begin{align}\label{DPhi}
	\mathrm{D}\Phi & (\omega_1, \cdots,\omega_k):=\left(\begin{array}{cccc}(k-1)F_{\mu_{1}}^{\prime}\left(\omega_{1}\right)-1 & -1 &\cdots &-1 \\ 
	-1 & (k-1)F_{\mu_{2}}^{\prime}\left(\omega_{2}\right)-1 & \cdots & -1\\
	\vdots & \vdots & \ddots & \vdots\\
	-1 & -1 & \cdots &(k-1)F_{\mu_{k}}^{\prime}\left(\omega_{k}\right)-1
	\end{array}\right).
\end{align}

We mainly consider the case when $z \in (\bar{\gamma}^0_1)^{+}\cup(\bar{\gamma}^0_2)^{+}$ with $\hat{y} \in (0,1)$, while the case for $z \in (\bar{\gamma}_1)^{+}\cup(\bar{\gamma}_2)^{+}$ with $\hat{y} \in (0,\infty)$ is similar. To this end,  we will first establish a perturbed system of \eqref{Phi2} for $\omega_t^c(z)$'s, and then show that \eqref{omegatAppro} holds for $z$ with sufficiently large $\Im z$. Finally, by a continuity argument, we show that \eqref{omegatAppro} holds for each fixed $z \in (\bar{\gamma}^0_1)^{+}\cup(\bar{\gamma}^0_2)^{+}$. 

Define for any $t \in [\![ k]\!]$ and for any fixed $z \in (\bar{\gamma}^0_1)^{+}\cup(\bar{\gamma}^0_2)^{+}$,
\begin{align}
	&\gamma_{t1}(z) :=  \tr P_tG(z) - \left( \tr (X_tX_t')^{-1} -\tr Q_tG(z)  \right)\left( \tr G(z) - \tr P_tG(z) \right),\label{gamma1}\\
	&\gamma_{t2}(z) := \tr (X_tX_t')^{-1} -\frac{y_t}{1-y_t},\label{gamma2}\\
	&\gamma_{t3}(z) := \frac{1}{1-y_t}\tr P_tG(z) - \tr Q_tG(z).\label{gamma3}
\end{align}
Writing $\tr P_tG(z)$ and $\tr G(z)$ as $(1 + \omega_t^c(z)m_N(z))$ and $m_N(z)$ respectively in (\ref{gamma1}), and then combining with \eqref{gamma2} and \eqref{gamma3}, we have
\begin{align}
	1 + \omega_t^c(z)m_N(z) -\left( \frac{y_t}{1-y_t} - \frac{1+\omega_t^c(z)m_N(z)}{1-y_t} \right)\left(m_N(z) - 1 - \omega_t^c(z)m_N(z)\right) \notag \\= \gamma_{t1} + \left(m_N(z) - 1 - \omega_t^c(z)m_N(z) \right) \left(\gamma_{t2} + \gamma_{t3} \right). \label{070403}
\end{align}
Multiplying both sides by $(1-y_t)$, and then dividing both sides by $\theta_t = m_N(z)\omega_t^c(z)(1-\omega_t^c(z))$, we have
\begin{align*}
	m_N(z) =& \frac{y_t}{1-\omega_t^c(z)} - \frac{1-y_t}{\omega_t^c(z)} + \frac{1-y_t}{\theta_t}\gamma_{t1} + \frac{(1-y_t)(m_N(z) - 1 - \omega_t^c(z)m_N(z))}{\theta_t} \left(\gamma_{t2} + \gamma_{t3} \right) \\
	=&-\frac{1}{F_{\mu_t}(\omega_t^c(z))} +\frac{1-y_t}{\theta_t}\gamma_{t1} + \frac{(1-y_t)(m_N(z) - 1 - \omega_t^c(z)m_N(z))}{\theta_t} \left(\gamma_{t2} + \gamma_{t3} \right).
\end{align*}
Using \eqref{Appsub2}, and then performing expansion of $(-m_N(z))^{-1}$ around $F_{\mu_t}(\omega_t^c(z))$, we  obtain
\begin{align}
	(k-1)F_{\mu_t}(\omega_t^c(z))- \omega_1^c(z)  - \cdots - \omega_k^c(z) + z =& \tilde{\gamma}_{t1}(z) + \tilde{\gamma}_{t2}(z) + \tilde{\gamma}_{t3}(z) \notag\\
	&+O_{\prec}\left(\frac{\tilde{\gamma}^2_{t1}(z) + \tilde{\gamma}^2_{t2}(z) +\tilde{\gamma}^2_{t3}(z)}{(k-1)F_{\mu_t}(\omega_t^c(z))} \right), \quad z \in \mathbb{C}^{+}. \label{062001}
\end{align}
where
\begin{align}
	&\tilde{\gamma}_{t1}(z) =(k-1)F^2_{\mu_t}(\omega_t^c(z))\cdot\frac{1-y_t}{\theta_t}\gamma_{t1}(z),\notag\\
	&\tilde{\gamma}_{t2}(z) =(k-1)F^2_{\mu_t}(\omega_t^c(z))\cdot\frac{(1-y_t)(m_N(z) - 1 - \omega_t^c(z)m_N(z))}{\theta_t} \gamma_{t2}(z),\notag\\
	&\tilde{\gamma}_{t3}(z) =(k-1)F^2_{\mu_t}(\omega_t^c(z))\cdot\frac{(1-y_t)(m_N(z) - 1 - \omega_t^c(z)m_N(z))}{\theta_t} \gamma_{t3}(z). \label{0617100}
\end{align}
It turns out that the above equations for all $t \in [\![k ]\!]$ will form a perturbed system of \eqref{Phi2}. The estimates of the error terms $\gamma_{ti}$, $i\in [\![3]\!]$, can be summarised as the following lemma.
\begin{lemma}\label{Lemma error estimates}
	Under Assumptions \ref{assum1} and \ref{assum2}, if $\hat{y} \in (0,1)$, for any $t \in [\![ k]\!]$ and for any fixed $z \in (\bar{\gamma}^0_1)^{+}\cup(\bar{\gamma}^0_2)^{+}$,  we have
	\begin{align}\label{First error}
		\gamma_{t1}(z),\; \gamma_{t2}(z),\; \gamma_{t3}(z) = O_{\prec}\left( {\sqrt{\frac{p_t}{N^3}}} \right).
	\end{align}
	Furthermore, let $\mathcal{F}: \mathbb{C}^2\mapsto \mathbb{C}$. 
If  $\mathcal{F}(x,y)$ is analytic at $(x,y)=(\tr P_tG(z), \tr G(z))$ and the partial derivatives of $\mathcal{F}$ at these points are $O_\prec(1)$, we have
\begin{align}\label{Second error}
	\sum_{t=1}^k\mathcal{W}_t\gamma_{t1}(z),\; \sum_{t=1}^k\mathcal{W}_t\gamma_{t2}(z),\; \sum_{t=1}^k\mathcal{W}_t\gamma_{t3}(z) = O_{\prec}\left( \frac{1}{N} \right),
\end{align}
	where
$
		\mathcal{W}_t := \mathcal{F}(\tr P_tG(z), \tr G(z)).
$ 

The same estimates hold when $z \in (\bar{\gamma}_1)^{+}\cup(\bar{\gamma}_2)^{+}$ with $\hat{y} \in (0,\infty)$.
\end{lemma}
With the perturbed system and the error estimates, we then first start from $z$ with sufficiently large $\Im z$, i.e., $z \in \mathcal{C}_4$. 
To show the  bounds in Proposition \ref{keyProp} initially for such $z$, we further give some preliminary estimations for the negative reciprocal Stieltjes transform. 
\begin{lemma}\label{Flbound}
	Suppose $\hat{y} \in (0,1)$. For any $z \in (\bar{\gamma}^0_1)^{+}\cup(\bar{\gamma}^0_2)^{+}$, let $\omega_t(z), t \in [\![ k ]\!]$ be the subordination functions, we have
	\begin{align*} 
		&|\omega_t(z)| ,\; |F_{\mu_t}(\omega_t(z))|,\; |F'_{\mu_t}(\omega_t(z))| \sim 1,\qquad |F^{(n)}_{\mu_t}(\omega_t(z))| \sim y_t, \quad n \ge 2,\\
		& \Big|1 - \frac{1}{k-1}\sum_{t=1}^{k}\frac{1}{F'_{\mu_t}(\omega_t(z))}\Big| \gtrsim \frac{1}{k}.
	\end{align*}
	The same bounds hold for generic variable $\omega_t\in \mathbb{C}$ instead of the subordination function $\omega_t(z)$ if $|\omega_t|$ is sufficiently large .
\end{lemma}
The proofs of Lemmas \ref{Lemma error estimates} and \ref{Flbound} are given in Section \ref{Sec B} of Appendix. 

With the above two lemmas at hand, we can start the proof for $z$ with large imaginary part. For notational simplicity, we write $G := G(z)$,  $m_N := m_N(z) $, $\omega_t^c := \omega_t^c(z)$, $\gamma_{ti} := \gamma_{ti}(z)$, and $\tilde{\gamma}_{ti} := \tilde{\gamma}_{ti}(z)$, $i \in [\![ 3]\!]$ for short, whenever there is no confusion.

\noindent $\bullet${\it Sufficiently large $\Im z$.} Notice that by Lemma \ref{Imomegatc}, $\Im \omega_t^c \to \infty$ as $\Im z \to \infty$. Therefore, when $\Im z$ is sufficiently large,   we have by Lemma \ref{Flbound}, for any $t \in [\![ k]\!]$, with high probability,
\begin{align}
	&|F'_{\mu_t}(\omega_t^c)| \sim 1,\qquad |F^{(n)}_{\mu_t}(\omega_t^c)| \sim y_t, \quad n \ge 2, \notag\\
	& \Big|1 - \frac{1}{(k-1)}\sum_{t=1}^{k}\frac{1}{F'_{\mu_t}(\omega_t^c)}\Big| \gtrsim \frac{1}{k}.\label{Fomegatc1}
\end{align}
As $\Im z \to \infty$, by Lemma \ref{Imomegatc}, we have with high probability
	\begin{align}
		\frac{|\theta_t|}{|F^2_{\mu_t}(\omega_t^c)|} = \left| \frac{m_N(z)(y_t-1+\omega_t^c(z))^2}{\omega_t^c(z)(1-\omega_t^c(z))}\right|  \ge \frac{1}{2}|m_N(z)| \gtrsim \frac{1}{\Im z}. \label{6010}
	\end{align}

Therefore, by \eqref{Fomegatc1} and (\ref{6010}), we have with high probability
\begin{align}
	\left|\frac{F^2_{\mu_t}(\omega_t^c)}{F'_{\mu_t}(\omega_t^c)}\cdot\frac{1-y_t}{\theta_t}\right| \lesssim \left|\frac{F^2_{\mu_t}(\omega_t^c)}{\theta_t}\right| \lesssim \Im z, \label{6020}
\end{align}
when $\Im z$ is sufficiently large.
Notice that by resolvent identity,
	\begin{align}
		 (\mathrm{D}\Phi)^{-1}(\omega_1^c,\cdots,\omega_k^c) = (\mathcal{D}(\omega^c)-\mathds{1}\mathds{1}')^{-1} = \mathcal{D}^{-1}(\omega^c) + \frac{\mathcal{D}^{-1}(\omega^c)\mathds{1}\mathds{1}'\mathcal{D}^{-1}(\omega^c)}{1-\mathds{1}\mathcal{D}^{-1}(\omega^c)\mathds{1}'}.  \label{6001}
	\end{align}
	Here $\mathds{1}$ is the all-one vector, and  $\mathcal{D}(\omega^c) := {\rm diag}(\mathcal{D}_t(\omega_t^c))_{t=1}^k$ with $\mathcal{D}_t(\omega_t^c) := (k-1)F_{\mu_t}'(\omega^c_t(z))$. We consider the $\ell^\infty$-norm of $(\mathrm{D}\Phi)^{-1}\cdot\Phi:= (\mathrm{D}\Phi)^{-1}\cdot\Phi(\omega_1^c,\cdots,\omega_k^c)$. Then 
\begin{align}
	\left\|(\mathrm{D}\Phi)^{-1}\cdot\Phi\right\|_{\infty} \le& \|\mathcal{D}^{-1}(\omega^c)\cdot\Phi \|_{\infty}+ \left\|\frac{\mathcal{D}^{-1}(\omega^c)\mathds{1}\mathds{1}'\mathcal{D}^{-1}(\omega^c)}{1-\mathds{1}'\mathcal{D}^{-1}(\omega^c)\mathds{1}}\cdot \Phi\right\|_{\infty} \notag\\
	\leq & \max_t\left|\mathcal{D}_t^{-1}(\omega_t^c)\tilde{\gamma}_{t1} \right| + \max_t\left|\mathcal{D}_t^{-1}(\omega_t^c)\tilde{\gamma}_{t2} \right|+\max_t\left|\mathcal{D}_t^{-1}(\omega_t^c)\tilde{\gamma}_{t3}\right| \notag\\
	&+ \max_s \left| \frac{\mathcal{D}_s^{-1}(\omega_s^c)}{1-\mathds{1}'\mathcal{D}^{-1}(\omega^c)\mathds{1}}\sum_{t=1}^k \mathcal{D}_t^{-1}(\omega_t^c)\tilde{\gamma}_{t1}\right|+ \max_s \left| \frac{\mathcal{D}_s^{-1}(\omega_s^c)}{1-\mathds{1}'\mathcal{D}^{-1}(\omega^c)\mathds{1}}\sum_{t=1}^k \mathcal{D}_t^{-1}(\omega_t^c)\tilde{\gamma}_{t2}\right| \notag\\
	&+ \max_s \left| \frac{\mathcal{D}_s^{-1}(\omega_s^c)}{1-\mathds{1}'\mathcal{D}^{-1}(\omega^c)\mathds{1}}\sum_{t=1}^k \mathcal{D}_t^{-1}(\omega_t^c)\tilde{\gamma}_{t3}\right| + O_{\prec}\left(\max_s\left|\frac{\tilde{\gamma}^2_{s1} + \tilde{\gamma}^2_{s2} +\tilde{\gamma}^2_{s3}}{(k-1)\mathcal{D}_s(\omega_s^c)F_{\mu_s}(\omega_s^c)} \right|\right) \notag\\
	&+O_{\prec}\left(\max_s\left|\frac{\mathcal{D}_s^{-1}(\omega_s^c)}{1-\mathds{1}'\mathcal{D}^{-1}(\omega^c)\mathds{1}}\sum_{t=1}^k\frac{\tilde{\gamma}^2_{t1} + \tilde{\gamma}^2_{t2} +\tilde{\gamma}^2_{t3}}{(k-1)\mathcal{D}_t(\omega_t^c)F_{\mu_t}(\omega_t^c)} \right|\right), \label{0617101}
\end{align}
where the last step follows from (\ref{062001}).
By (\ref{gamma1})-(\ref{0617100}), (\ref{Fomegatc1}) and Lemma \ref{Lemma error estimates}, one can crudely bound the following terms
\begin{align}
	\max_s\left|\frac{\tilde{\gamma}^2_{s1} + \tilde{\gamma}^2_{s2} +\tilde{\gamma}^2_{s3}}{(k-1)\mathcal{D}_s(\omega_s^c)F_{\mu_s}(\omega_s^c)} \right|  \prec \frac{p_{\max}}{N^3}, \qquad \max_s\left|\frac{\mathcal{D}_s^{-1}(\omega_s^c)}{1-\mathds{1}'\mathcal{D}^{-1}(\omega^c)\mathds{1}}\sum_{t=1}^k\frac{\tilde{\gamma}^2_{t1} + \tilde{\gamma}^2_{t2} +\tilde{\gamma}^2_{t3}}{(k-1)\mathcal{D}_t(\omega_t^c)F_{\mu_t}(\omega_t^c)} \right| \prec \frac{p_{\max}}{N^2}. \label{6050}
\end{align}
Here we also used the fact $\mathcal{D}_t(\omega_t^c)\sim  k$ with high probability, in light of (\ref{Fomegatc1}) and the definition of $\mathcal{D}_t(\omega_t^c)$ in (\ref{6001}). 
Next, we consider the other terms in (\ref{0617101}). By Lemma \ref{Lemma error estimates} together with (\ref{6020}), we have
\begin{align}
	\mathcal{D}_t^{-1}(\omega_t^c)\tilde{\gamma}_{t1} = \frac{F^2_{\mu_t}(\omega_t^c)}{F'_{\mu_t}(\omega_t^c)}\cdot\frac{1-y_t}{\theta_t}\gamma_{1t} = O_{\prec}\left(\Im z\sqrt{\frac{p_t}{N^3}} \right).
\label{6051}\end{align}
Similarly, we can obtain
\begin{align}
		\mathcal{D}_t^{-1}(\omega_t^c)\tilde{\gamma}_{t2} = O_{\prec}\left(\Im z\sqrt{\frac{p_t}{N^3}} \right), \quad \mathcal{D}_t^{-1}(\omega_t^c)\tilde{\gamma}_{t3} = O_{\prec}\left(\Im z\sqrt{\frac{p_t}{N^3}}\right).
\label{6052}
\end{align}
Next, we turn to estimate the fourth to the sixth terms in (\ref{0617101}). Using \eqref{Fomegatc1}, we know that
\begin{align*}
	\left| \frac{\mathcal{D}_s^{-1}(\omega_s^c)}{1-\mathds{1}'\mathcal{D}^{-1}(\omega^c)\mathds{1}}\sum_{t=1}^k \mathcal{D}_t^{-1}(\omega_t^c)\tilde{\gamma}_{t1} \right| \lesssim\left| \sum_{t=1}^{k} \frac{F^2_{\mu_t}(\omega_t^c)}{F'_{\mu_t}(\omega_t^c)}\cdot\frac{1-y_t}{\theta_t}\gamma_{t1}\right|
\end{align*}
Then, we set in (\ref{Second error})
\begin{align*}
	\mathcal{W}_t := \frac{F^2_{\mu_t}(\omega_t^c)}{F'_{\mu_t}(\omega_t^c)}\cdot\frac{1-y_t}{\theta_t}.
\end{align*}
 Hence, using (\ref{Second error}) in Lemma \ref{Lemma error estimates}, we have
\begin{align}
	\Big| \frac{\mathcal{D}_s^{-1}(\omega_s^c)}{1-\mathds{1}'\mathcal{D}^{-1}(\omega^c)\mathds{1}}\sum_{t=1}^k \mathcal{D}_t^{-1}(\omega_t^c)\tilde{\gamma}_{t1} \Big| \prec \frac{1}{N}.
\label{6053}\end{align}
Similarly, we can obtain
\begin{align}
	\Big| \frac{\mathcal{D}_s^{-1}(\omega_s^c)}{1-\mathds{1}'\mathcal{D}^{-1}(\omega^c)\mathds{1}}\sum_{t=1}^k \mathcal{D}_t^{-1}(\omega_t^c)\tilde{\gamma}_{t2} \Big| \prec \frac{1}{N},\qquad \Big| \frac{\mathcal{D}_s^{-1}(\omega_s^c)}{1-\mathds{1}'\mathcal{D}^{-1}(\omega^c)\mathds{1}}\sum_{t=1}^k \mathcal{D}_t^{-1}(\omega_t^c)\tilde{\gamma}_{t3} \Big| \prec \frac{1}{N}.
\label{6054}\end{align}
As a result, by combining (\ref{6050})-(\ref{6054}), we get
\begin{align}
	\left\|(\mathrm{D}\Phi)^{-1}\cdot\Phi\right\|_\infty \prec \frac{1}{N}. \label{071201}
\end{align}
Further, by (\ref{6020}), we have 
\begin{align}
	\| (\mathrm{D}\Phi)^{-1}\cdot \mathrm{D}^{2}\Phi \|_{(\infty,\infty)} \le& \|\mathcal{D}^{-1}(\omega^c) \cdot \mathrm{D}^{2}\Phi \|_{(\infty,\infty)} + \left\|\frac{\mathcal{D}^{-1}(\omega^c)\mathds{1}\mathds{1}'\mathcal{D}^{-1}(\omega^c)}{1-\mathds{1}'\mathcal{D}^{-1}(\omega^c)\mathds{1}} \cdot \mathrm{D}^{2}\Phi \right\|_{(\infty,\infty)}\notag\\
	=& \max_s \left| \frac{F^{(2)}_{\mu_s}(\omega_s^c)}{F'_{\mu_s}(\omega_s^c)} \right|+ \max_s\sum_{t=1}^{k}\left|\frac{(k-1)F^{(2)}_{\mu_t}(\omega_t^c)}{(k-1)^2F'_{\mu_s}(\omega_s^c)F'_{\mu_t}(\omega_t^c)(1-\mathds{1}'\mathcal{D}^{-1}(\omega^c)\mathds{1})}\right| \notag\\
	\lesssim & \max_s y_s + \sum_{t=1}^k y_t \lesssim 1. \label{9010}
\end{align}
 Therefore, we have
	\begin{align}
		s_0 :=&\| (\mathrm{D}\Phi)^{-1}\cdot \mathrm{D}^{2}\Phi \|_{(\infty,\infty)}  \left\|(\mathrm{D}\Phi)^{-1}\cdot\Phi\right\|_{\infty}= O_{\prec}\left(  {\frac{1}{N}}	\right) < \frac{1}{2} \label{071210}
	\end{align}
	with high probability when $\Im z$ is sufficiently large . 
Recall the Newton-Kantorvich theorem (Theorem \ref{NewtonKantorvich} in Appendix) and set $b \equiv \left\|(\mathrm{D}\Phi)^{-1}\cdot\Phi\right\|_\infty $ and $L \equiv \| (\mathrm{D}\Phi)^{-1}\cdot \mathrm{D}^{2}\Phi \|_{(\infty,\infty)}$. Together with (\ref{071201}), (\ref{9010}) and (\ref{071210}),  we have that there is for every such $z$ a  unique collection of $\hat{\omega}_t(z)$'s s.t. 
	\begin{align}
	\Phi_{\mu_1,\cdots,\mu_k} (\hat{\omega}_1(z),\cdots,\hat{\omega}_k(z),z) =0 \label{9081}
	\end{align} 
	with
	\begin{align}\label{Gap4}
		\left|{\omega}^c_{t}(z)-\hat{\omega}_{t}(z)\right| \leq &\frac{1 - \sqrt{1 - 2s_0}}{s_0} \left\|(\mathrm{D}\Phi)^{-1}\cdot\Phi\right\|_{\infty} = O_{\prec}\left(\frac{1}{N} \right).
	\end{align}
	where we used the conclusion in the Newton-Kantorvich theorem that $x_{*} \in B[x_0, t_{*}]$. 
	Finally, using Lemma \ref{Imomegatc},  we note that $\Im \hat{\omega}_t(z) = \Im \hat{\omega}_t(z) - \Im{\omega}^c_t(z) + \Im {\omega}^c_t(z) \ge \Im z $  with high probability, when $\Im z$ is sufficiently large. It further follows that 
	$\Gamma_{\mu_1,\cdots,\mu_k} (\hat{\omega}_1(z),\cdots,\hat{\omega}_k(z)) \neq 0$ for all $z \in \mathbb{C}^{+}$ when $\Im z$ is sufficiently large. Thus $\hat{\omega}_t(z)$ is analytic when $\Im z$ is sufficiently large,  since $F_{\mu_t}(\hat{\omega}(z))$ is. By the definition of $F_{\mu_t}$, we have
	\begin{align}
		F_{\mu_t}(\hat{\omega}_t(z)) = \frac{-1}{m_{\mu_t}(\hat{\omega}_t(z))}= \hat{\omega}_t(z) -y_t+ \frac{y_t(1-y_t)}{1-y_t-\hat{\omega}_t(z)}, \quad t \in [\![ k]\!]. \label{9080}
	\end{align}
	Plugging (\ref{9080}) into (\ref{9081}) we have
$
		((k-1)\mathrm{I} - \mathds{1}\mathds{1}')\cdot \hat{x} = \hat{b},
$
	where 
	\begin{align*}
	&\hat{x} := \bigg(\frac{\hat{\omega}_1(z) }{ z}, \cdots,  \frac{\hat{\omega}_k (z)}{  z}\bigg)', \notag\\
	&\hat{b} := \bigg(-1 +\frac{(k-1)y_1}{z}- \frac{(k-1)y_1(1-y_1)}{(1-y_1-\hat{\omega}_1(z)) z}, \cdots, -1 +\frac{(k-1)y_k}{z}-\frac{(k-1)y_k(1-y_k)}{(1-y_k-\hat{\omega}_k(z))z}\bigg)'.
	\end{align*}
	 Solving the above linear system by inverting $((k-1)\mathrm{I} - \mathds{1}\mathds{1}')$, we have
$
		\hat{x} = (k-1)^{-1}(\mathrm{I} - \mathds{1}\mathds{1}')\cdot \hat{b},
$
	which gives
	\begin{align*}
		\frac{\hat{\omega}_t(z)}{z} =& 1 +\frac{y_t}{z} +\sum_{i=1}^{k}\frac{y_i}{z} -  \frac{y_t(1-y_t)}{(1-y_t-\hat{\omega}_t(z))z} - \sum_{t=1}^k\frac{y_t(1-y_t)}{(1-y_t-\hat{\omega}_t(z))z} = 1 + O(|\Im z|^{-1}),
	\end{align*}
	as $\Im z \to \infty$. Therefore, we can obtain
	\begin{align} 
		\lim_{ \eta \to \infty}\frac{\hat{\omega}_t(\mathrm{i}\eta)}{\mathrm{i}\eta} = 1.
	\end{align}
	Thus by the uniqueness claim in Proposition \ref{freeaddprop}, $\hat{\omega}_t(z)$ agrees with $\omega_t(z)$ when $\Im z$ is sufficiently large. Therefore, \eqref{Gap4} implies that
	\begin{align}\label{Gap5}
		\left|{\omega}^c_{t}(z)-{\omega}_{t}(z)\right| \prec \frac{1}{N},
	\end{align}
	for  all $z \in \mathbb{C}^{+}$ with $\Im z$ being sufficiently large. Subtracting \eqref{Appsub2} from \eqref{sub3}, we have
	\begin{align}\label{DPer2}
		\sum_{t=1}^{k}(\omega_t^c(z) - \omega_t(z)) = -\frac{k-1}{m_N(z)} + \frac{k-1}{m_\boxplus(z)} = (k-1)\frac{m_N(z)-m_\boxplus(z)}{m_N(z)m_\boxplus(z)}.
	\end{align}
	As a result,
	\begin{align}\label{initial2}
		\left|m_N(z) - m_\boxplus(z) \right| = O_{\prec}\left(\frac{1}{k-1} \sum_{t=1}^{k}|\omega_t^c(z) - \omega_t(z)| \right) = O_{\prec}\left( \frac{1}{N} \right),
	\end{align}
	for  all $z \in \mathbb{C}^{+}$ with  $\Im z$ being sufficiently large.

Next, taking \eqref{initial2} as an input, we can use the continuity argument to obtain the bound for each fixed $z \in (\bar{\gamma}^0_1)^{+}\cup(\bar{\gamma}^0_2)^{+}$.

\noindent
$\bullet$ {\it Any fixed $z\in (\bar{\gamma}^0_1)^{+}\cup(\bar{\gamma}^0_2)^{+}$.} 
Let $z_0\in (\bar{\gamma}^0_1)^{+}\cup(\bar{\gamma}^0_2)^{+}$ has sufficiently large imaginary part so that (\ref{Gap5}) and (\ref{initial2}) hold at $z_0$. Then, let $z\in (\bar{\gamma}^0_1)^{+}\cup(\bar{\gamma}^0_2)^{+}$ be any fixed point. We can find a sequence of number $z_i\in (\bar{\gamma}^0_1)^{+}\cup(\bar{\gamma}^0_2)^{+}, i=1,\ldots, N^2$, so that $z=z_{N^2}$ and $|z_i-z_{i-1}|=O(N^{-2})$ for all $i=1,\ldots, N^2$. 
Our strategy is to extend the estimates in (\ref{Gap5}) and (\ref{initial2}) from $z_0$ to $z$ via the intermediate points $z_i$'s step by step, using a continuity argument. In the sequence, we will only show the detailed argument for the first step from $z_0$ to $z_1$. The remaining steps are the same. 

We first consider the bound of $\omega'_t(z)$. Similar as we did in (\ref{9010}), with $\omega_t^c$ replaced by $\omega_t$, together with Lemma \ref{Flbound}, we can get
\begin{align}
		\| (\mathrm{D}\Phi)^{-1}\|_{(\infty,\infty)} \lesssim 1. \label{90101}
\end{align}
Differentiating the equation (\ref{Phi2}) w.r.t. $z$, we get
\begin{align}
	\mathrm{D}\Phi \cdot \omega'(z) = \mathds{1}, \label{8007}
\end{align}
where $\omega'(z) := (\omega_1'(z),\cdots,\omega_k'(z))$.  Together with (\ref{90101}), we get
\begin{align}
	|\omega_t'(z)| \lesssim 1, t \in [\![k ]\!] \label{080711}
\end{align}
by inverting (\ref{8007}). 
Recall the definition in (\ref{Appsub1}), we have
\begin{align*}
	(\omega_t^c(z))'  = \frac{\tr G(z)\tr P_tG^2(z) - (\tr P_tG(z) -1) \tr G^2(z)}{(\tr G(z))^2}.
\end{align*}
By Lemma \ref{Gbound}, we can obtain
$
	|(\omega_t^c(z))'| \lesssim 1
$
with high probability.

Therefore, from  (\ref{Gap5}) for $z_0$ and  the continuity of $\omega_t^c$ and $\omega_t$, we have
\begin{align}
	\left|{\omega}^c_{t}(z_1)-{\omega}_{t}(z_1)\right| \prec \frac{1}{N} + \frac{1}{N^2}. \label{7001}
\end{align}
More precisely, by the definition of the stochastic domination, (\ref{7001}) implies that  for any large $D > 0$ and small $\epsilon > 0$, there exists an high probability event $\mathcal{E}_1 \equiv \mathcal{E}(z_1,\epsilon,D)$, satisfying $\mathbb{P}(\mathcal{E}_1^c) < N^{-D}$, such that on the event $\mathcal{E}_1$, we have 
\begin{align}
		|\omega_t^c(z_1) - \omega_t(z_1)| \le N^{\epsilon}(N^{-1} + N^{-2}) \le 2N^{-1+\epsilon}, \label{7001E1}
	\end{align}
for any small $\epsilon > 0$.

From (\ref{062001}) we recall the perturbed system at $z_1$ 
	\begin{align*}
	&	(k-1)F_{\mu_t}(\omega_t^c(z_1)) - \omega_1^c(z_1)  - \cdots -\omega_k^c(z_1) + z_1 \notag\\
		&= \tilde{\gamma}_{t1}(z_1) + \tilde{\gamma}_{t2}(z_1) + \tilde{\gamma}_{t3}(z_1)+O_{\prec}\left(\frac{\tilde{\gamma}^2_{t1}(z_1) + \tilde{\gamma}^2_{t2}(z_1) +\tilde{\gamma}^2_{t3}(z_1)}{(k-1)F_{\mu_t}(\omega_t^c(z_1))} \right),
	\end{align*}
	with the definitions in (\ref{0617100}). 
Using (\ref{7001}), Lemma \ref{Lemma error estimates} and Lemma \ref{Flbound} , one can easily check that 
\begin{align*}
	\frac{\tilde{\gamma}^2_{t1}(z_1) + \tilde{\gamma}^2_{t2}(z_1) +\tilde{\gamma}^2_{t3}(z_1)}{(k-1)F_{\mu_t}(\omega_t^c(z_1))} \prec \frac{p_t}{N^2}.
\end{align*}
Therefore, let $\Omega_t(z_1):= \omega_t^c(z_1) - \omega_t(z_1)$, and apply expansion $F_{\mu_t}(\omega_t^c(z_1))$ around $\omega_t(z_1)$ with (\ref{Appsub2}).  We get
	\begin{align}
		&(k-1)F'_{\mu_t}(\omega_t(z_1))\Omega_1(z_1) - \Omega_1(z_1) - \cdots -\Omega_k(z_1) \notag \\
		&= \tilde{\gamma}_{t1}(z_1) + \tilde{\gamma}_{t2}(z_1) + \tilde{\gamma}_{t3} (z_1)- (k-1)\sum_{n\ge 2}\frac{1}{n!}F_{\mu_t}^{(n)}(\omega_t(z_1))\Omega_t^n(z_1) +O_{\prec}\left(\frac{p_t}{N^2} \right). \label{062010}
	\end{align}
	Denoted by
$
		\tilde{\gamma}_a(z_1) := (\tilde{\gamma}_{1a}(z_1),\cdots,\tilde{\gamma}_{ka}(z_1))',
$ for $a=1,2,3$. 
Writing the system (\ref{062010}) (in $t\in [\![ k ]\!]$) in terms of $\mathrm{D}\Phi$ in (\ref{DPhi}) and inverting $\mathrm{D}\Phi$,  we have
	\begin{align}
		\|\Omega(z_1)\|_{\infty} \leq& \|(\mathrm{D}\Phi)^{-1} \cdot \left(\tilde{\gamma}_1(z_1) + \tilde{\gamma}_2(z_1) + \tilde{\gamma}_3(z_1)\right)\|_{\infty} \notag\\
		&+\sum_{n\ge 2}\left\| \frac{k-1}{n!} (\mathrm{D}\Phi)^{-1}\cdot \Omega_F^n(z_1)\right\|_{\infty} + O_{\prec}\left(\| (\mathrm{D}\Phi)^{-1}\|_{(\infty,\infty)}\cdot \frac{p_{\max}}{N^2} \right), \label{9020}
	\end{align}
	where 
	$$\Omega_F^n(z_1) := \left( F_{\mu_1}^{(n)}(\omega_1(z_1))\Omega_1^n(z_1), \cdots,  F_{\mu_k}^{(n)}(\omega_k(z_1))\Omega_k^n(z_1)\right)'.$$
	Using (\ref{90101}), (\ref{9020}) can be written as
	\begin{align}
		\|\Omega(z_1)\|_{\infty} \leq& \|(\mathrm{D}\Phi)^{-1} \cdot \left(\tilde{\gamma}_1(z_1) + \tilde{\gamma}_2(z_1) + \tilde{\gamma}_3(z_1)\right)\|_{\infty} \notag\\
		&+\sum_{n\ge 2}\left\| \frac{k-1}{n!} (\mathrm{D}\Phi)^{-1}\cdot \Omega_F^n(z_1)\right\|_{\infty}+ O_{\prec}\left(  \frac{p_{\max}}{N^2} \right) \label{9030}
	\end{align}
	
First, we consider the high order terms in (\ref{9030}). Recall (\ref{6001}) and let $\mathcal{D}(\omega) := \mathcal{D}(\omega_1,\cdots,\omega_k)$, We have
	\begin{align*}
		\frac{k-1}{n!} (\mathrm{D}\Phi)^{-1}\cdot \Omega_F^n(z_1) = \frac{k-1}{n!}\mathcal{D}^{-1}(\omega)\cdot \Omega_F^n(z_1)  + \frac{k-1}{n!}\frac{\mathcal{D}^{-1}(\omega)\mathds{1}\mathds{1}'\mathcal{D}^{-1}(\omega)}{1-\mathds{1}'\mathcal{D}^{-1}(\omega)\mathds{1}}\cdot \Omega_F^n(z_1).
	\end{align*}
	Taking $\ell^\infty$-norm on both sides, by triangular inequality, we have
	\begin{align*}
		\left\| \frac{k-1}{n!} (\mathrm{D}\Phi)^{-1}\cdot \Omega_F^n(z_1)\right\|_{\infty}  \le \left\| \frac{k-1}{n!}\mathcal{D}^{-1}(\omega)\cdot \Omega_F^n(z_1) \right\|_{\infty}  + \left\|\frac{k-1}{n!}\frac{\mathcal{D}^{-1}(\omega)\mathds{1}\mathds{1}'\mathcal{D}^{-1}(\omega)}{1-\mathds{1}'\mathcal{D}^{-1}(\omega)\mathds{1}}\cdot \Omega_F^n(z_1)\right\|_{\infty} .
	\end{align*}
Notice that  for $n \ge 2$
	\begin{align}
		\left\| \frac{k-1}{n!}\mathcal{D}^{-1}(\omega)\cdot \Omega_F^n(z_1) \right\|_{\infty} =& \left\|\frac{1}{n!}\left( \frac{F_{\mu_1}^{(n)}(\omega_1(z_1))}{F_{\mu_1}'(\omega_1(z_1))}\Omega_1^n(z_1), \cdots,  \frac{F_{\mu_k}^{(n)}(\omega_k(z_1))}{F_{\mu_k}'(\omega_k(z_1))}\Omega_k^n(z_1) \right)'\right\|_{\infty}  \notag\\
		\lesssim& \max_{1\le t \le k}\frac{y_t}{|F_{\mu_t}'(\omega_t(z_1))|} |\Omega_t(z_1)|^n \lesssim \max_{1 \le t \le k}y_t|\Omega_t(z_1)|^n, \label{9001}
	\end{align}
 	and
	\begin{align}
		&\left\|\frac{k-1}{n!}\frac{\mathcal{D}^{-1}(\omega)\mathds{1}\mathds{1}'\mathcal{D}^{-1}(\omega)}{1-\mathds{1}'\mathcal{D}^{-1}(\omega)\mathds{1}}\cdot \Omega_F^n(z_1)\right\|_{\infty} \notag\\
		 =& \left\|\frac{k-1}{n!(1-\mathds{1}'\mathcal{D}^{-1}(\omega)\mathds{1})}\left( \sum_{t=1}^{k}\frac{ F_{\mu_t}^{(n)}(\omega_t(z_1))\Omega_t^n(z_1)}{\mathcal{D}_1(\omega_1)\mathcal{D}_{t}(\omega_i)},\cdots, \sum_{t=1}^{k} \frac{F_{\mu_t}^{(n)}(\omega_t(z_1))\Omega_t^n(z_1)}{\mathcal{D}_k(\omega_k)\mathcal{D}_{t}(\omega_t)}\right)  \right\|_{\infty}  \notag\\
		 \lesssim &\frac{1}{(k-1)(1-\mathds{1}'\mathcal{D}^{-1}(\omega)\mathds{1})}\sum_{t=1}^{k}\frac{y_t|\Omega_t(z_1)|^n}{|F_{\mu_t}'(\omega_t(z_1))|}\cdot\max_{1\le s\le k} \frac{1}{|F_{\mu_s}'(\omega_s(z_1))|}\lesssim \max_{1 \le t \le k}|\Omega_t(z_1)|^n,\label{9002}
	\end{align}
	where we used Lemma \ref{Flbound}.
	Plugging (\ref{9001}) and (\ref{9002}) into (\ref{9030}), we have
	\begin{align*}
		\|\Omega(z_1) \|_{\infty} \lesssim \| (\mathrm{D}\Phi)^{-1} \cdot  \left(\tilde{\gamma}_1(z_1) + \tilde{\gamma}_2(z_1) + \tilde{\gamma}_3(z_1)\right)  \|_{\infty} + \sum_{n\ge 2}\|\Omega(z_1)\|^n_{\infty} + O_{\prec}\left(  \frac{p_{\max}}{N^2} \right).
	\end{align*}
	Which means that there exists an high probability event $\mathcal{E}_2 = \mathcal{E}_2(z_1,\epsilon,D)$, satisfying $\mathbb{P}(\mathcal{E}_2^c) < N^{-D}$, such that on the event $\mathcal{E}_1\cap\mathcal{E}_2$,
	\begin{align*}
		\|\Omega(z_1) \|_{\infty} \lesssim \| (\mathrm{D}\Phi)^{-1} \cdot  \left(\tilde{\gamma}_1(z_1) + \tilde{\gamma}_2(z_1) + \tilde{\gamma}_3(z_1)\right)  \|_{\infty} + \sum_{n\ge 2}\|\Omega(z_1)\|^n_{\infty} + O\left(  \frac{p_{\max}}{N^{2-\epsilon}} \right).
	\end{align*}
	for any small $\epsilon > 0$. Since $\|\Omega(z_1) \|_{\infty}  \le 2N^{-1+\epsilon}$ on $\mathcal{E}_1$ in light of (\ref{7001E1}), we can absorb the quadratic term into the LHS, which gives 
	\begin{align}
		\|\Omega(z_1) \|_{\infty}  \lesssim &\| (\mathrm{D}\Phi)^{-1} \cdot  \left(\tilde{\gamma}_1(z_1) + \tilde{\gamma}_2(z_1) + \tilde{\gamma}_3(z_1)\right) \|_{\infty} + O\left(  \frac{p_{\max}}{N^{2-\epsilon}} \right) \notag\\
		\le& \| (\mathrm{D}\Phi)^{-1} \cdot  \tilde{\gamma}_1(z_1) \|_{\infty}   + \| (\mathrm{D}\Phi)^{-1} \cdot  \tilde{\gamma}_2(z_1)  \|_{\infty}  +\| (\mathrm{D}\Phi)^{-1} \cdot \tilde{\gamma}_3(z_1) \|_{\infty} + O\left(  \frac{p_{\max}}{N^{2-\epsilon}} \right) ,  \label{162019}
	\end{align}
	on the event $\mathcal{E}_1\cap\mathcal{E}_2$. 
	For the first term in the RHS, we have
	\begin{align}
		 \| (\mathrm{D}\Phi)^{-1} \cdot  \tilde{\gamma}_1(z_1) \|_{\infty}   \le& \|\mathcal{D}^{-1}(\omega)\cdot  \tilde{\gamma}_1(z_1) \| _{\infty} + \left\|\frac{\mathcal{D}^{-1}(\omega)\mathds{1}\mathds{1}'\mathcal{D}^{-1}(\omega)}{1-\mathds{1}'\mathcal{D}^{-1}(\omega)\mathds{1}}\cdot \tilde{\gamma}_1(z_1)\right\|_{\infty} \notag\\
		 =&\max_s \Big|\mathcal{D}_i^{-1}(\omega_s)\tilde{\gamma}_{s1} \Big| + \max_s \Big| \frac{\mathcal{D}_s^{-1}(\omega_s)}{1-\mathds{1}'\mathcal{D}^{-1}(\omega)\mathds{1}}\sum_{t=1}^k \mathcal{D}_t^{-1}(\omega_t)\tilde{\gamma}_{t1}(z_1)\Big|. \label{062018}
	\end{align}
For each $t \in [\![ k]\!]$, 
\begin{align}
	\mathcal{D}_t^{-1}(\omega_t)\tilde{\gamma}_{t1}(z_1) = \frac{F^2_{\mu_t}(\omega_t^c(z_1))}{F'_{\mu_t}(\omega_t(z_1))}\cdot\frac{1-y_t}{\theta_t}\gamma_{t1}(z_1). \label{062017}
\end{align}
By the continuity of function $F_{\mu_t}$, we have on the event $\mathcal{E}_1\cap\mathcal{E}_2$
\begin{align}
	\Big|\frac{F^2_{\mu_t}(\omega_t^c(z_1))}{F'_{\mu_t}(\omega_t(z_1))}\cdot\frac{1-y_t}{\theta_t} \Big| \lesssim \Big|\frac{F^2_{\mu_t}(\omega_t^c(z_1))}{\theta_t} \Big| = \Big| \frac{F^2_{\mu_t}(\omega_t(z_1))}{m_{\boxplus}(z_1)\omega_t(z_1)(1-\omega_t(z_1))}  \Big| + O\left( \frac{1}{N^{1-\epsilon}}\right),  \label{062015}
\end{align}
where the last equality comes from (\ref{7001}). 
Further  note
\begin{align}
\Big| \frac{F^2_{\mu_t}(\omega_t(z_1))}{m_{\boxplus}(z_1)\omega_t(z_1)(1-\omega_t(z_1))}  \Big| \lesssim 1. \label{062016}
\end{align}

Combining (\ref{062017}), (\ref{062015}), (\ref{062016}), (\ref{gamma1}) 
with  Lemma \ref{Lemma error estimates},  we can obtain
$
	\left| \mathcal{D}_t^{-1}\tilde{\gamma}_{t1} \right| \prec  \sqrt{\frac{p_t}{N^3} } .
$
This completes the estimate for the first term in the RHS of (\ref{062018}). For the second term in (\ref{062018}), we have
\begin{align*}
	\bigg| \frac{\mathcal{D}_s^{-1}(\omega_s)}{1-\mathds{1}'\mathcal{D}^{-1}(\omega)\mathds{1}}\sum_{t=1}^k \mathcal{D}_t^{-1}(\omega_t)\tilde{\gamma}_{t1}(z_1) \bigg| \lesssim\bigg| \sum_{t=1}^{k} \frac{F^2_{\mu_t}(\omega_t^c(z_1))}{F'_{\mu_t}(\omega_t^c(z_1))}\cdot\frac{1-y_t}{\theta_t}\gamma_{t1}(z_1)\bigg| \prec \frac{1}{N} ,
\end{align*}
where in the first step we used Lemma \ref{Flbound}, and in the last step we used Lemma \ref{Lemma error estimates}.

As a result,
$
 \| (\mathrm{D}\Phi)^{-1} \cdot  \tilde{\gamma}_1(z_1) \|_{\infty} \prec N^{-1}. 
$
Similarly, we can obtain,
$
	 \| (\mathrm{D}\Phi)^{-1} \cdot  \tilde{\gamma}_2(z_1) \|_{\infty}  \prec N^{-1}$, $\| (\mathrm{D}\Phi)^{-1} \cdot  \tilde{\gamma}_3(z_1) \|_{\infty}  \prec N^{-1}.
$
Again by the definition of the stochastic domination, we have there exists an high probability event $\mathcal{E}_3 \equiv \mathcal{E}(z_1,\epsilon,D)$, satisfying $\mathbb{P}(\mathcal{E}_3^c) < N^{-D}$, such that on the event $\mathcal{E}_1 \cap \mathcal{E}_2 \cap \mathcal{E}_3$,
\begin{align}
	\| (\mathrm{D}\Phi)^{-1} \cdot  \tilde{\gamma}_a(z_1) \|_{\infty} \le N^{1-\epsilon}, \quad a = 1,2,3, \label{080401}
\end{align}
for any small $\epsilon > 0$.
Therefore, plugging (\ref{080401}) into (\ref{162019}), we have event $\mathcal{E}_1 \cap \mathcal{E}_2 \cap \mathcal{E}_3$
\begin{align}
	\max_t|\omega_t(z_1) - \omega_t^c(z_1) | =  \|\Omega(z_1) \|_{\infty} \le CN^{1-\epsilon} , \label{080403}
\end{align}
 Repeating the above procedure for $O(N^2)$ steps, we will generate a series of events, i.e., $\mathcal{E}_1,\cdots,\mathcal{E}_{CN^2}$. Therefore, we have on the event $\bigcap_{i=1}^{O(N^2)}\mathcal{E}_i$, (\ref{080403}) still holds. Since all these events are high probability events, we obtain
	\begin{align}
		\max_t|\omega_t(z_1) - \omega_t^c(z_1) | \prec  {\frac{1}{N}}. \label{fine bound of omega}
	\end{align}
	Using \eqref{initial2}, we can again have
	\begin{align}
		|m_N(z_1) - m_{\boxplus}(z_1)| \prec \frac{1}{N}.  \label{071310}
	\end{align}
Repeating the above procedure we can prove (\ref{omegatAppro}) and (\ref{trGGap}) for all $z\in (\bar{\gamma}^0_1)^{+}\cup(\bar{\gamma}^0_2)^{+}$. 

\end{proof}

\section{Proof of Theorem \ref{FreeCLT}}\label{ProofofCLT}
In this section, we prove Theorem \ref{FreeCLT}. To  present the results, we  define the following two point function. For any fixed $z_1, z_2 \in \mathbb{C}\setminus \text{supp}(\mu_\boxplus)$, let
\begin{align*}
	\mathcal{L}_t(z_1,z_2) := \left(1 + \frac{(\omega_t(z_1)m_{\boxplus}(z_1)-\omega_t(z_2)m_{\boxplus}(z_2))^2}{(z_1-z_2)(m_{\boxplus}(z_1) - m_{\boxplus}(z_2))} + \frac{\omega_t(z_1) - \omega_t(z_2)}{\frac{1}{m_\boxplus(z_1)}-\frac{1}{m_{\boxplus}(z_2)}}  \right) .
\end{align*}
Further, if $z_1 = z_2$, we denote
\begin{align*}
	\mathcal{L}_t(z_1,z_1) := \frac{(1+\omega_t(z_1)m_{\boxplus}(z_1))(m_{\boxplus}(z_1)-1-\omega_t(z_1)m_{\boxplus}(z_1) - \omega_t(z_1)(\omega_t(z_1)m_{\boxplus}(z_1))') }{m_{\boxplus}(z_1)-1-2m_{\boxplus}(z_1)\omega_t(z_1)}.
\end{align*}
Based on Proposition \ref{keyProp}, we have the following approximations for tracial quantities and the entries of matrices involving Green functions. The proof is given in Section \ref{Sec C} of Appendix. 
\begin{prop}\label{keyprop2}
	Under Assumptions \ref{assum1} and \ref{assum2}, if $\hat{y} \in (0,1)$, for any $t \in [\![ k]\!]$ and for any fixed $z \in (\bar{\gamma}^0_1)^{+}\cup(\bar{\gamma}^0_2)^{+}$,  we have
	\begin{align}
		&\left| \tr \bbP_t\bbG(z) - \left(1+ \omega_t(z)m_{\boxplus}(z)\right) \right|,\; \left|\tr \bbQ_t\bbG(z) - \frac{1+ \omega_t(z)m_{\boxplus}(z)}{1-y_t} \right| \prec  \sqrt{\frac{p_t}{N^3}}, \label{080411}\\
		&\left| \tr \bbP_t\bbG(z)\bbP_t\bbG(z) - \mathcal{L}_t(z,z)\right|, \; \left|\tr \bbQ_t\bbG(z)\bbP_t\bbG(z) - \frac{\mathcal{L}_t(z,z)}{1-y_t}\right| \prec  \sqrt{\frac{p_t}{N^3}}.\label{080412}
	\end{align}
	The same bounds hold for $z \in (\bar{\gamma}_1)^{+}\cup(\bar{\gamma}_2)^{+}$ with $\hat{y} \in (0,\infty)$.
	 
	In addition, for any $t \in [\![ k]\!]$ and for any fixed $z_1 \in (\bar{\gamma}^0_1)^{+}$ and $z_2 \in (\bar{\gamma}^0_2)^+$,  we have
	\begin{align}
		\left| \tr \bbP_t\bbG(z_1)\bbP_t\bbG(z_2) - \mathcal{L}_t(z_1,z_2)\right|, \; \left|\tr \bbQ_t\bbG(z_1)\bbP_t\bbG(z_2) - \frac{\mathcal{L}_t(z_1,z_2)}{1-y_t}\right| \prec  \sqrt{\frac{p_t}{N^3}}. \label{0807120}
	\end{align}
	The same bounds hold for $z_1 \in (\bar{\gamma}_1)^{+}$ and $z_2 \in (\bar{\gamma}_2)^{+}$ with $\hat{y} \in (0,\infty)$. 
\end{prop}

In addition to the concentration results for the tracial quantities, we also have the following concentration results for the diagonal entries of the involved random matrices. 
\begin{prop}\label{DiagApproByFC}
	Under Assumptions \ref{assum1} and \ref{assum2}, if $\hat{y} \in (0,1)$, for any fixed $z \in (\bar{\gamma}^0_1)^{+}\cup(\bar{\gamma}^0_2)^{+}$,  we have	
	\begin{align}
		\left| [G(z)]_{ii} - m_{\boxplus}(z) \right| \prec  \frac{1}{\sqrt{N} }, \label{GiiGap}
	\end{align}
	and for any $t \in [\![ k]\!]$,
	\begin{align}
		&\left| [\bbP_t\bbG(z)]_{ii} - \left(1+ \omega_t(z)m_{\boxplus}(z)\right) \right|, \;\left| [\bbP_t\bbG(z)\bbP_t]_{ii} -\left(1+ \omega_t(z)m_{\boxplus}(z)\right)  \right|\prec  \frac{1}{\sqrt{N}},\label{PGPiiGap}\\
		&\left| [(X_tX_t')^{-1}]_{jj} - \frac{1}{1-y_t} \right|, \; \left| [\bbW_t\bbG(z)\bbW'_t]_{jj} - \frac{1+ \omega_t(z)m_{\boxplus}(z)}{y_t(1-y_t)} \right|\prec  \frac{1}{\sqrt{N}}.\label{QGiiGap}
	\end{align}
	The same bounds hold for $z \in (\bar{\gamma}_1)^{+}\cup(\bar{\gamma}_2)^{+}$ with $\hat{y} \in (0,\infty)$.
\end{prop}

  The proof of Proposition \ref{DiagApproByFC} is also given in Section \ref{Sec C} of Appendix.
  
The above bounds can be easily extended to the bound for the derivatives w.r.t. $z$ using Cauchy integral. For instance,  choosing sufficiently small (but of constant order) contour $\gamma$ centred at $z$, we have 
\begin{align}
	|\partial_z \tr G(z) - \partial_z m_{\boxplus}(z)| = \left|\frac{1}{2\pi {\rm i}}\oint_{\gamma}\frac{\tr G(\tilde{z}) - m_{\boxplus}(\tilde{z})}{(\tilde{z}-z)^2}{\rm d}\tilde{z} \right| \prec  \frac{1}{N}.  \label{06281152}
\end{align} Here for the last step in (\ref{06281152}) we need a bound of $\tr G(\tilde{z}) - m_{\boxplus}(\tilde{z})$ for all $\tilde{z}\in \gamma$. A desired bound for those $\tilde{z}$ with $|\Im \tilde{z}|\geq N^{-K}$ can be easily obtained from Propsition \ref{keyProp} by slightly modifying the contours $\gamma_a^0$'s to its neighborhood, and for those $\tilde{z}$ with $|\Im \tilde z|<N^{-K}$ one can use Lemma \ref{Gbound} and the fact $m_\boxplus(\tilde{z})\sim 1$ directly.
Similar arguments can be applied to the derivatives of all quantities in (\ref{080411})-(\ref{QGiiGap}), and the same bounds hold for any fixed order derivatives. Hence, in the sequel, we will often apply the error bounds obtained in the estimates of  the Green function functionals to their derivatives directly without further explanation. 

Based on Proposition \ref{keyProp}, \ref{keyprop2}, and \ref{DiagApproByFC}, we can now prove Theorem \ref{FreeCLT}. Our proof basically consists of two parts. In the first part, based on the method of characteristic function, we prove the asymptotic normality of the LSS, and the corresponding mean and variance can be expressed in terms of the expectation of the Green function and its functionals. In the second part, we show that, the expectation of the Green function and its functionals can be well approximated by functionals of $m_{{\boxplus}}$ and $\omega_t, t \in [\![ k]\!]$.  
The conclusion for the first part can be summarized as the following theorem. Recall $\E^{\chi}$ defined in (\ref{EXit}).
\begin{thm}\label{GeneralCLT}
	Let $\hat{y} \in (0,1)$ and f be analytic inside $\gamma_1^0$ and $\gamma_2^0$. 
	Denote by
	\begin{align}
		\alpha_t(z) =\E^{\chi} [ \tr( X_t X_t')^{-1} -\tr Q_tG(z) ], \qquad \beta_t(z) =\frac{1}{1-y_t} \E^{\chi} [\tr G(z) - \tr P_tG(z) ],  \label{def of alphabeta}
	\end{align}
	for all $t \in [\![ k ]\!]$. Under Assumptions \ref{assum1} and \ref{assum2} , we have
	\begin{align*}
		\frac{\Tr f(H) - \oint_{\bar{\gamma}^0_1}\E^{\chi}[ \Tr G(z)]f(z){\rm d} z}{\sigma_f} \Rightarrow \mathcal{N}(0, 1)
	\end{align*} 
	with
	\begin{align}
		\sigma_f^2 = -\frac{1}{2\pi^2} \oint_{\bar{\gamma}_1^0}\oint_{\bar{\gamma}_2^0} \mathcal{K}(z_1,z_2)f(z_1)f(z_2){\rm d}z_2{\rm d}z_1, \label{061360}
	\end{align}
	and
	\begin{align}
	\mathcal{K}(z_1,z_2) =&\left(z_1 - \sum_{s=1}^k\frac{\alpha_s(z_1)}{1+\alpha_s(z_1) + \beta_s(z_1)} \right)^{-1}  \sum_{t=1}^k\E^{\chi}\Bigg[ \partial_{z_2}\frac{  \tr\bbQ_t\bbG(z_1)\bbP_t\bbG(z_2)- \tr\bbQ_t\bbG(z_1)\bbG(z_2)}{1 + \alpha_t(z_1) + \beta_t(z_1)}\Bigg] . \label{def of K(zz)}
	\end{align}
	The same result holds if  $\hat{y} \in (0,\infty)$ with $\gamma_{1}^0$ and $\gamma_{2}^0$ replaced by $\gamma_1$ and $\gamma_2$, respectively.  
\end{thm} 

\begin{proof}[Proof of Theorem \ref{GeneralCLT}]
We only prove the case when $\hat{y} \in (0,1)$ with contours $\gamma^0_1$ and $\gamma^0_2$, and the proof of the case $\hat{y} \in (0,\infty)$ with contours $\gamma_1$ and $\gamma_2$ is similar.
First, from the proof of  Lemma \ref{supportofmu} in Appendix, we can actually conclude that the nonzero eigenvalues of $H$ are bounded below and above by positive constants with high probability. Then by Cauchy's formula, by choosing our contour $\gamma_a^0$ properly to enclose all eigenvalues (with high probability), we have
\begin{align*}
	\Tr f(H)=\frac{-1}{2\pi{\rm i}}\oint_{\gamma^0_a} \Tr G(z)f(z){\rm d} z=\frac{-1}{2\pi{\rm i}}\oint_{\bar{\gamma}^0_a} \Tr G(z)f(z){\rm d} z+O_\prec(N^{-K+1}), \quad a=1,2.
\end{align*}
 Recall the definition in (\ref{Xit}), we have
\begin{align}
		\Tr f(H)=& \frac{-1}{2\pi{\rm i}}\oint_{\bar{\gamma}^0_a} \Tr G(z)f(z){\rm d} z \cdot \Xi +O_\prec(N^{-K+1}) =:L_N^a(f)+O_\prec(N^{-K+1}), \quad a=1,2 \label{062810}
\end{align}
for some large (but fixed) $K$. Therefore, to prove the asymptotic normality of $\Tr f(H)$, it suffices to prove the asymptotic normality of $L_N^1(f)$. An advantage of $L_N^1(f)$ over $\Tr f(H)$ is that the Green function $G(z)$ in the integrand of $L_N^1(f)$ and the derivatives of various functionals of $G(z)$ w.r.t. matrix entries always possess deterministic crude bound. It is crucial for us to transform high probability estimates into the expectation estimates in the following calculations, in the spirit of Lemma \ref{prop_prec} (ii).  Let
	$
		\phi_f^N(x) = \E \big[ {\rm e}^{\mathrm{i}x \la L_N^1(f)\ra  } \big]
	$
	be the characteristic function of $\la L_N^1(f)\ra $. Taking derivative w.r.t. $x$, we have
	\begin{align}\label{Dcharac}
		(\phi_f^N(x))' =\mathrm{i} \E \Big[\la L_N^1(f)\ra {\rm e}^{\mathrm{i}x\la L_N^1(f)\ra } \Big] = \frac{-1}{2\pi}\oint_{\bar{\gamma}^0_1} \E^{\chi} \Big[ \bla\Tr G(z_1)\bra {\rm e}^{\mathrm{i}x\la L_N^1(f)\ra } \Big] f(z_1){\rm d}z_1,
	\end{align}
	Our aim is to derive an approximate ODE for $\phi_f^N(x)$, from which we can solve $\phi_f^N(x)$. Then our task boils down to calculating the term,
	$
		\E^{\chi} \big[ \bla\Tr G(z_1)\bra {\rm e}^{\mathrm{i}x\la L_N^1(f)\ra } \big]$, $z_1\in \bar{\gamma}_1^0. 
	$
	By definition, we have the trivial identity
$
		\sum_{t=1}^k \Tr P_tG(z_1) = \Tr HG(z_1)  = N + z_1\Tr G(z_1).
$
	Then we can write
	\begin{align*}
		\E^{\chi} \Big[ \bla\Tr G(z_1)\bra {\rm e}^{\mathrm{i}x\la L_N^1(f)\ra } \Big]= \frac{1}{z_1}\sum_{t=1}^{k}\E^{\chi} \Big[ \bla\Tr P_tG(z_1)\bra {\rm e}^{\mathrm{i}x\la L_N^1(f)\ra } \Big]. 
	\end{align*}
	By the cumulant expansion (c.f., Lemma \ref{cumulant expansion} in Section \ref{Sec A} of Appendix),
	\begin{align}
&\E^{\chi}\Big[ \bla\Tr\bbP_t\bbG(z_1)\bra {\rm e}^{\mathrm{i}x\la L_N^1(f)\ra }\Big]\notag\\
&=\E^{\chi}\Big[ \sum_{ij}^{(t)}\bbX_{t,ij}[\bbW_t\bbG(z_1)]_{ji} \bla {\rm e}^{\mathrm{i}x\la L_N^1(f)\ra }\bra\Big] =\frac{1}{N}\E^{\chi}\Big[ \sum_{ij}^{(t)}(\partial_{t,ji}[\bbW_t\bbG(z_1)]_{ji})\bla {\rm e}^{\mathrm{i} xL_N^1(f)}\bra\Big] \notag\\
&\quad+\frac{1}{N}\E^{\chi}\Big[\sum_{ij}^{(t)}[\bbW_t\bbG(z_1)]_{ji}\partial_{t,ji}{\rm e}^{\mathrm{i}x\la L_N^1(f)\ra }\Big] 
+\sum_{ij}^{(t)}\sum_{a=0}^2 {2\choose a}\frac{\kappa_3^{t,j}}{2{N}^{3/2}}\E^{\chi} \Big[  (\partial^a_{t,ji}[\bbW_t\bbG(z_1)]_{ji}) \partial^{2-a}_{t,ji}\bla {\rm e}^{\mathrm{i}x\la L_N^1(f)\ra } \bra\Big] \notag\\
&\quad+\sum_{ij}^{(t)}\sum_{a=0}^3{3\choose a}\frac{\kappa^{t,j}_{4}}{6{N}^{2}}\E^{\chi} \Big[ (\partial^a_{t,ji}[\bbW_t\bbG(z_1)]_{ji})\partial^{3-a}_{t,ji}\bla {\rm e}^{\mathrm{i}x\la L_N^1(f)\ra }\bra\Big] \notag\\
&\quad+ \sum_{s\ge1}^lO\left(\frac{1}{N^{\frac{s+1}{2}}}	\right)\sum_{ij}^{(t)}\sum_{\substack{s_0+s_1=s\\s_1 \ge 1}}\E\Big[\partial_{t,ji}^{s_0}\left\{[W_tG(z_1)]_{ji} \bla {\rm e}^{\mathrm{i}x\la L_N^1(f)\ra }\bra \right\} \partial_{t,ji}^{s_1}\Xi \Big] \notag\\
&\quad+ \sum_{s\ge 4 }^lO\left(\frac{1}{N^{\frac{s+1}{2}}}	\right)\sum_{ij}^{(t)}\E^{\chi}\Big[\partial^s_{t,ji}\left\{ [W_tG(z_1)]_{ji} \bla {\rm e}^{\mathrm{i}x\la L_N^1(f)\ra }\bra\right\}  \Big] + \sum_{ij}^{(t)}  \E\Big[\mathcal{R}_{l+1}^{t,ji} \Big] \notag \\
&\quad=:\mathsf{I}_{t1}+\mathsf{I}_{t2}+\mathsf{I}_{t3}+\mathsf{I}_{t4}+ \mathsf{E}_1 + \mathsf{E}_2+\mathsf{E}_3, \label{061303}
\end{align}
where
\begin{align}
	\mathcal{R}_{l+1}^{t,ji} =& O(1) \cdot \sup _{X_{t,ji}\in \mathbb{R}}\left| \partial_{t,ji}^{l+1}\left\{ [ W_t G(z_1)]_{ji}\bla {\rm e}^{\mathrm{i}x\la L_N^1(f)\ra }\bra \Xi \right\} \right| \cdot \mathbb{E}\Big[|X_{t,ji}|^{l+2} \mathbf{1}_{|X_{t,ji}|>N^{-\frac12+\epsilon}}\Big]\notag\\
	&+O(1) \cdot \mathbb{E}\big[|X_{t,ji}|^{l+2} \big]\cdot \sup _{|X_{t,ji}| \leqslant N^{-\frac12+\epsilon}}\left| \partial_{t,ji}^{l+1}\left\{ [ W_t G]_{ji}\bla {\rm e}^{\mathrm{i}x\la L_N^1(f)\ra }\bra \Xi\right\} \right| . \label{062501}
\end{align}
Here we also used $\kappa_a^{t,j}, a=3,4$  to denote the $a$-th cumulant of $X_{t,ji}$'s. Note that since $i$ is the index of i.i.d. samples, according to our assumption the cumulant depends on $j$ only but not $i$. The following lemma gives the estimates of each term in (\ref{061303}).

\begin{lemma}\label{Error in CLT}
	\begin{align*}
	&\mathsf{I}_{t1} 
	= \alpha_t \E^{\chi}\Big[ \bla  \Tr\bbG(z_1) \bra {\rm e}^{\mathrm{i}x\la L_N^1(f)\ra }\Big] - (\alpha_t + \beta_t)\E^{\chi} \Big[ \bla  \Tr \bbP_t\bbG(z_1)  \bra {\rm e}^{\mathrm{i}x\la L_N^1(f)\ra }\Big]+O_{\prec}\left( \sqrt{\frac{p_t}{N^3}} \right),\\
	&\mathsf{I}_{t2}= \frac{x}{\pi}\oint_{\bar{\gamma}^0_2}\E^{\chi}\partial_{z_2}\Big[   \tr\bbQ_t\bbG(z_1)\bbP_t\bbG(z_2)- \tr\bbQ_t\bbG(z_1)\bbG(z_2)\Big]f(z_2){\rm d}z_2 \E \big[{\rm e}^{\mathrm{i}x\la L_N^1(f)\ra }\big]
	+O_{\prec}\left( \sqrt{\frac{p_t}{N^3}} \right),  \\
	&\mathsf{I}_{t3} = O_{\prec}\left(\frac{p_t}{N^{\frac76}}\right),\; \mathsf{I}_{t4}= O_{\prec}\left(\frac{p_t}{N^\frac32} \right), \; \mathsf{E}_1 + \mathsf{E}_2 + \mathsf{E}_3 = O_{\prec}\left( \frac{p_t}{N^{\frac32}} \right).
\end{align*}
\end{lemma}
The proof of Lemma \ref{Error in CLT} is given in Section \ref{Sec C} of Appendix.

First plugging the estimates of  $\mathsf{I}_{t1}$, $\mathsf{I}_{t3}$ and $\mathsf{I}_{t4}$ in (\ref{061303}), we arrive at
\begin{align*}
	&\E^{\chi} \bigg[  \Tr P_tG(z_1)\bla {\rm e}^{\mathrm{i}x\la L_N^1(f)\ra }\bra \bigg] = \alpha_t(z_1) \E^{\chi}\Big[  \Tr\bbG(z_1) \bla{\rm e}^{\mathrm{i}x\la L_N^1(f)\ra }\bra\Big] \notag\\
	& \qquad - (\alpha_t(z_1) + \beta_t(z_1))\E^{\chi} \Big[\Tr \bbP_t\bbG(z_1)\bla {\rm e}^{\mathrm{i}x\la L_N^1(f)\ra }\bra\Big]+ \mathsf{I}_{t2} + O_{\prec}\left( \frac{p_t}{N^{\frac76}} \right).
\end{align*} 
Moving the second term in the RHS to the LHS, and then dividing both sides by $(1 + \alpha_t(z_1) + \beta_t(z_1))$, we have
\begin{align*}
	\E^{\chi} \Big[  \Tr \bbP_t\bbG(z_1) \bla {\rm e}^{\mathrm{i}x\la L_N^1(f)\ra }\bra\Big]=& \frac{\alpha_t(z_1)\E^{\chi}\Big[ \Tr\bbG(z_1) \bla {\rm e}^{\mathrm{i}x\la L_N^1(f)\ra }\bra\Big]}{1 + \alpha_t(z_1) + \beta_t(z_1)} + \frac{\mathsf{I}_{t2}}{1 + \alpha_t(z_1) + \beta_t(z_1)}+ O_{\prec}\left( \frac{p_t}{N^{\frac76}} \right).
\end{align*}
Here we actually used the fact that $|1 + \alpha_t(z_1) + \beta_t(z_1)|$ is bounded below (implied by Lemma \ref{lem.062830} below). 
Summing over $t$ and using the trivial identity $\sum_t P_t G=zG+I$, we arrive at an  equation  for $\E^{\chi}\big[  \Tr\bbG(z_1) \bla{\rm e}^{\mathrm{i}x\la L_N^1(f)\ra }\bra\big]$. Solving the equation gives 
\begin{align}
	&\E^{\chi}\Big[  \Tr\bbG(z_1)\bla {\rm e}^{\mathrm{i}x\la L_N^1(f)\ra }\bra \Big] = \left(z_1 - \sum_{s=1}^k\frac{\alpha_s(z_1)}{1+\alpha_s(z_1) + \beta_s(z_1)} \right)^{-1} \sum_{t=1}^k\frac{\mathsf{I}_{t2}}{1 + \alpha_t(z_1) + \beta_t(z_1)} + O_{\prec}\left( \frac{1}{N^{\frac16}} \right). \label{061401}
\end{align}
Substituting this result together with the estimate of $\mathsf{I}_{t2}$ back to \eqref{Dcharac}, we have
\begin{align*}
	(\phi_f^N(x))' = -x\sigma^2_f\E \Big[ {\rm e}^{\mathrm{i}x\la L_N^1(f)\ra } \Big] + O_{\prec}\left( \frac{1}{N^{\frac16}} \right)=-x\sigma^2_f \phi_f^N(x) + O_{\prec}\left( \frac{1}{N^{\frac16}} \right),
\end{align*}
where $\sigma_f$ is given in (\ref{061360}). 
This proves the asymptotic normality of $\la L_N^1(f)\ra $, which further implies the normality of $\text{Tr} f(H)$. 
\end{proof}

From Theorem \ref{GeneralCLT} to Theorem \ref{FreeCLT}, one needs to approximate the tracial quantities  by functions of $m_{\boxplus}$ and $\omega_t$'s.  To this end,  we first provide the following lemma, whose proof is postponed to Section \ref{Sec C} in Appendix.
\begin{lemma} \label{lem.062830}
	Under the assumption of Theorem \ref{FreeCLT}, if $\hat{y} \in (0,1)$, for any fixed $z \in (\bar{\gamma}^0_1)^{+}\cup(\bar{\gamma}^0_2)^{+}$,  we have
	\begin{align}
		1 + \alpha_t(z) + \beta_t(z) =\frac{ m_{\boxplus}(z)-1-2\omega_t(z)m_{\boxplus}(z)}{1-y_t} + O_{\prec}\left(\frac{1}{N} \right),
	\label{1plusalphatbetat}
	\end{align}
	and
	\begin{align}
		\left|m_{\boxplus}(z)-1-2\omega_t(z)m_{\boxplus}(z)\right| > 0. \label{lowerboundof1plusalphatbetat}
	\end{align}
The same bounds hold for $z \in (\bar{\gamma}_1)^{+}\cup(\bar{\gamma}_2)^{+}$ with $\hat{y} \in (0,\infty)$. 
\end{lemma}

Now, we are ready to prove Theorem \ref{FreeCLT}.
\begin{proof}[Proof of Theorem \ref{FreeCLT}]
	By the estimates in Propositions \ref{keyProp}, \ref{keyprop2}, and Lemma \ref{lem.062830}, we have
	\begin{align*}
		\frac{\alpha_t(z)}{1 + \alpha_t(z) + \beta_t(z)} = \frac{y_t - 1 -\omega_t(z)m_{\boxplus}(z) }{m_{\boxplus}(z)-1-2\omega_t(z)m_{\boxplus}(z)} +  O_{\prec}\left( \frac{p_t}{N^2} \right).
	\end{align*}
	Here  we also used the fact that $\alpha_t(z)$ is of order $O_{\prec}\left(p_t/N \right)$.
	This implies
	\begin{align*}
		&\left(z_1 - \sum_{s=1}^k\frac{\alpha_s(z_1)}{1+\alpha_s(z_1) + \beta_s(z_1)} \right)^{-1} = \left( z_1 - \sum_{s=1}^{k}\frac{y_s - 1 -\omega_s(z_1)m_{\boxplus}(z_1) }{m_{\boxplus}(z_1)-1-2\omega_s(z_1)m_{\boxplus}(z_1)}  \right)^{-1} + O_{\prec}\left( \frac{1}{N}   \right).
	\end{align*}
After some elementary algebraic operation (c.f., Section \ref{SimplifyV} in Appendix) , by using \eqref{sub2}, \eqref{sub3} and \eqref{F2}, we can get
	\begin{align}\label{FreeCLTfactor}
		\left( z_1 - \sum_{s=1}^{k}\frac{y_s - 1 -\omega_s(z_1)m_{\boxplus}(z_1) }{m_{\boxplus}(z_1)-1-2\omega_s(z_1)m_{\boxplus}(z_1)}  \right)^{-1} = -\frac{m'_{\boxplus}(z_1)}{m_{\boxplus}(z_1)}
	\end{align}
	which is bounded. This gives the leading factor of $\mathcal{K}(z_1,z_2)$. Applying Proposition \ref{keyprop2} and Lemma \ref{lem.062830}, we conclude
	\begin{align}\label{FreeCLTK1}
		&\frac{\E^{\chi} [ \partial_{z_2}\left( \tr Q_tG(z_1)P_tG(z_2) \right)]}{1 + \alpha_t(z_1) + \beta_t(z_1)}=\frac{\partial_{z_2}\mathcal{L}_t(z_1,z_2)}{m_{\boxplus}(z_1)-1-2\omega_t(z_1)m_{\boxplus}(z_1)} + O_{\prec}\left( \sqrt{\frac{p_t}{N^3}} \right).
	\end{align}
Here we used the fact that $\{m_{\boxplus}(z): z\in \gamma_1^0\}$ and $ \{m_{\boxplus}(z): z\in \gamma_2^0\}$ are well separated (c.f., Section \ref{Discussion on the contours} in Appendix).

  Similarly, we have by Proposition \ref{keyprop2} with $G(z_1)G(z_2)=(G(z_1)-G(z_2))/(z_1-z_2)$ 
	\begin{align}\label{FreeCLTK2}
	\frac{\E^{\chi} [ \partial_{z_2}\left( \tr Q_tG(z_1)G(z_2) \right) ]}{1 + \alpha_t(z_1) + \beta_t(z_1)} 
		=\frac{\partial_{z_2}\left( \frac{\omega_t(z_1)m_{\boxplus}(z_1) - \omega_t(z_2)m_{\boxplus}(z_2)}{z_1 - z_2}  \right)}{m_{\boxplus}(z_1)-1-2\omega_t(z_1)m_{\boxplus}(z_1)} + O_{\prec}\left( \sqrt{\frac{p_t}{N^3}} \right).
	\end{align}
	Plugging \eqref{FreeCLTfactor}, \eqref{FreeCLTK1} and \eqref{FreeCLTK2} into (\ref{def of K(zz)}), we can get
	\begin{align}
		\mathcal{K}(z_1,z_2) =& -\frac{m'_{\boxplus}(z_1)}{m_{\boxplus}(z_1)} \sum_{t=1}^{k}\left( \frac{\partial_{z_2}\mathcal{L}_t(z_1,z_2)-\partial_{z_2}\left( \frac{\omega_t(z_1)m_{\boxplus}(z_1) - \omega_t(z_2)m_{\boxplus}(z_2)}{z_1 - z_2}  \right)}{m_{\boxplus}(z_1)-1-2\omega_t(z_1)m_{\boxplus}(z_1)}  \right)  + O_{\prec}\left( \frac{1}{\sqrt{N}} \right). \label{06221215}
	\end{align}
Further simplification (c.f., Section \ref{SimplifyV} in Appendix) leads to  
	\begin{align}
		\mathcal{K}(z_1,z_2)
		= \sum_{t=1}^{k} \frac{\omega'_t(z_1)\omega'_t(z_2)}{(\omega_t(z_1) - \omega_t(z_2))^2} - \frac{(k-1)m'_{\boxplus}(z_1)m'_{\boxplus}(z_2)}{(m_{\boxplus}(z_1) - m_{\boxplus}(z_2))^2} -\frac{1}{(z_1 - z_2)^2} + O_{\prec}\left( \frac{1}{\sqrt{N}} \right). \label{06221214}
	\end{align}
Adding  the parts that $|\Im z| \le N^{-K}$ of the contours back to the integral of the main term in the RHS of (\ref{06221214})  (with negligible error)  gives the variance in Theorem \ref{FreeCLT}.

	For the expectation in Theorem \ref{FreeCLT}, we can estimate it as follows. Let
	\begin{align}\label{EtEFree}
		E_t(z) : = \E^{\chi} [ \Tr P_tG(z) ] - N(1+\omega_t(z)m_{\boxplus}(z)), \quad E(z) := \E^{\chi} [\Tr G(z) ] - Nm_{\boxplus}(z).
	\end{align}
	By Propositions \ref{keyProp} and \ref{keyprop2},  we immediately obtain $E(z) = O_{\prec}\left(1\right)$  and 
	$E_t(z) = O_{\prec}( \sqrt{p_t/N} )$.
	Our task is to give an explicit expression for $E(z)$ up to the first order. In the following estimation, since only one parameter $z$ is involved, we use the shorthand notation $G:=G(z)$ for brevity. 

	We start from the following lemma which is essentially the cumulant expansion of $\E^{\chi} [\Tr P_tG]$ with some error estimates.
	\begin{lemma}\label{LemmaEtrPG2}
\begin{align}\label{EtrPG2}
		\E^{\chi} [ \Tr P_tG ] =& \E^{\chi} [ \tr \left(\bbX_t\bbX_t' \right)^{-1}  - \tr \bbQ_t\bbG ]\E^{\chi} [  \Tr \bbG - \Tr \bbP_t\bbG  ] +\E^{\chi} [ \tr\bbQ_t\bbG\bbP_t\bbG-\tr\bbQ_t\bbG-\tr\bbQ_t\bbG^2 ] \notag \\
	&+\sum_{ij}^{(t)} \frac{\kappa_4^{t,j}}{{N}^{2}}\E^{\chi} \Big[  [\left(\bbX_t\bbX_t' \right)^{-1}]_{jj}^2\left(  [{I} - \bbP_t]_{ii} [(\bbP_t-{I})\bbG]_{ii} - [({I} - \bbP_t)\bbG]_{ii}^2 \right) \Big]+ O_{\prec}\left(\frac{p_t}{N^{\frac{7}{6}}} \right).
	\end{align}
	\end{lemma}
	The proof of Lemma \ref{LemmaEtrPG2} is given in Section \ref{Sec C} of Appendix. Next, if we start from  $\E^{\chi}[\Tr P_t] $ and $\E^{\chi}[\Tr P_tGP_t]$, write them as $\sum_{ij}^{(t)} \mathbb{E}^{\chi} X_{t,ji}(\cdots)_{ij}$  and then apply cumulant expansion w.r.t. $X_{t,ji}$'s, we can get the quantities $\E^{\chi} [ \Tr (\bbX_t\bbX_t' )^{-1} ] $ and $\E^{\chi} [\Tr \bbQ_t\bbG ] $ after the expansion. More specifically, by performing cumulant expansion up to $\kappa_{4}^{t,j}$ term on $\E^{\chi}[\Tr P_t] $ and $\E^{\chi}[\Tr P_tGP_t]$, and then using Propositions \ref{keyProp}, \ref{keyprop2} and \ref{DiagApproByFC} to replace the random quantities by deterministic estimates, we can actually get 
\begin{align}
	&\E^{\chi} [ \Tr (\bbX_t\bbX_t' )^{-1} ] = \frac{\E^{\chi}[\Tr P_t] }{1-y_t} + \frac{y_t}{(1-y_t)^2}+ \sum_{ij}^{(t)}\frac{\kappa_4^{t,j}}{N^2(1-y_t)} +{ O_{\prec}\left(\frac{p_t}{N^{\frac76}}\right)},\label{EtrXX}\\
	&\E^{\chi} [\Tr \bbQ_t\bbG ] = \frac{{\E^{\chi} [\Tr P_tG P_t ]}}{1-y_t-\frac{1}{N}} - \sum_{ij}^{(t)}\frac{\kappa_4^{t,j}(m_{\boxplus}(z)-1-\omega_t(z)m_{\boxplus}(z))}{N^2(1-y_t)\left(1-y_t-\frac{1}{N} \right)}+ { O_{\prec}\left(\frac{p_t}{N^{\frac76}}\right)}.\label{EtrQG}
\end{align}
The derivation of the above two equations is similar to (\ref{EtrPG2}), and thus the details are omitted. In addition, we observe that $\text{Tr} P_t=p_t$ and $\text{Tr} P_tGP_t=\text{Tr}P_tG$. 

Further using Proposition \ref{DiagApproByFC}, we have 
\begin{align}\label{k4E}
	&\E^{\chi} \Big[  [\left(\bbX_t\bbX_t' \right)^{-1}]_{jj}^2\left(  [{I} - \bbP_t]_{ii} [(\bbP_t-{I})\bbG]_{ii} - [({I} - \bbP_t)\bbG]_{ii}^2 \right) \Big] \notag \\
	=& \frac{1+\omega_t(z)m_{\boxplus}(z)-m_{\boxplus}(z)}{1-y_t} - \left(\frac{m_{\boxplus}(z)-1-\omega_t(z)m_{\boxplus}(z)}{1-y_t}\right)^2  + {O_{\prec}\left(\frac{p_t}{N^{\frac{3}{2}}}\right)}.
\end{align} 
Plugging \eqref{EtEFree}, \eqref{EtrXX}, \eqref{EtrQG} and \eqref{k4E} into \eqref{EtrPG2}, after some basic algebra, it turns out that the $\kappa_4^{t,j}$ terms are cancelled out, and we can get
\begin{align*}
	\frac{m_{\boxplus}(z)-1-2\omega_t(z)m_{\boxplus}(z)}{1-y_t} E_t(z)=& \frac{y_t - 1 - \omega_t(z)m_{\boxplus}(z)}{1-y_t}E(z) - \E^{\chi} [ \tr\bbQ_t\bbG^2 ]+ \E^{\chi} [ \tr\bbQ_t\bbG\bbP_t\bbG ] + {O_{\prec}\left(\frac{p_t}{N^{\frac{7}{6}}}\right)} .
\end{align*}
Dividing the coefficient of $E_t(z)$ on both sides, and then summing over $t$, we have
\begin{align*}
	\sum_{t=1}^{k} E_t(z) =& \sum_{t=1}^k \frac{y_t - 1 - \omega_t(z)m_{\boxplus}(z)}{m_{\boxplus}(z)-1-2\omega_t(z)m_{\boxplus}(z)}E(z) \\
	&- \sum_{t=1}^k\frac{\E^{\chi}[ (1-y_t) \tr\bbQ_t\bbG^2 ]- \E^{\chi} [(1-y_t)  \tr\bbQ_t\bbG\bbP_t\bbG ]}{m_{\boxplus}(z)-1-2\omega_t(z)m_{\boxplus}(z)} +{O_{\prec}\left( \frac{1}{N^{\frac{1}{6}}} \right)}.
\end{align*}
Notice that by the resolvent identity, we have
$
	\sum_{t=1}^kE_t (z)= zE(z).
$
Hence, 
\begin{align}
	 E(z) =& \left( z - \sum_{t=1}^{k} \frac{y_t - 1 -\omega_t(z)m_{\boxplus}(z) }{m_{\boxplus}(z)-1-2\omega_t(z)m_{\boxplus}(z)} \right)^{-1} \notag\\
	 &\times \sum_{t=1}^k\frac{\E^{\chi} [(1-y_t)  \tr\bbQ_t\bbG\bbP_t\bbG ] - \E^{\chi} [ (1-y_t) \tr\bbQ_t\bbG^2  ]}{m_{\boxplus}(z)-1-2\omega_t(z)m_{\boxplus}(z)} + {O_{\prec}\left( \frac{1}{N^{\frac{1}{6}}} \right)}\notag\\
	 =& -\frac{m'_{\boxplus}(z)}{m_{\boxplus}(z)}\sum_{t=1}^k\frac{\E^{\chi} [(1-y_t) \tr\bbQ_t\bbG\bbP_t\bbG ] - \E^{\chi} [ (1-y_t)\tr\bbQ_t\bbG^2  ]}{m_{\boxplus}(z)-1-2\omega_t(z)m_{\boxplus}(z)} + {O_{\prec}\left( \frac{1}{N^{\frac{1}{6}}} \right)} \label{0621100}
\end{align}
Using Proposition \ref{keyprop2}, and then performing some simplification (c.f., Section \ref{SimplifyV} in Appendix), we can obtain
\begin{align}
	E(z)  = \frac{1}{2}\bigg[ \sum_{t=1}^k\frac{\omega''_t(z)}{\omega_t'(z)}+(k-1)\left( \frac{2m_{\boxplus}'(z)}{m_{\boxplus}(z)}-\frac{m_{\boxplus}''(z)}{m_{\boxplus}'(z)} \right) \bigg] + {O_{\prec}\left( \frac{1}{N^{\frac{1}{6}}} \right)}. \label{0621101}
\end{align}
Therefore, by the definition of $E(z)$ in (\ref{EtEFree}),
\begin{align}
	\E^{\chi} [\Tr G ]  = Nm_{\boxplus}(z) +\frac{1}{2}\bigg[ \sum_{t=1}^k\frac{\omega''_t(z)}{\omega_t'(z)}+(k-1)\left( \frac{2m_{\boxplus}'(z)}{m_{\boxplus}(z)}-\frac{m_{\boxplus}''(z)}{m_{\boxplus}'(z)} \right) \bigg] + {O_{\prec}\left( \frac{1}{N^{\frac{1}{6}}} \right)}. \label{062850}
\end{align}
Adding  the part that $|\Im z| \le N^{-K}$ of the contour back to the integral of the main term in the RHS of (\ref{062850})  (with negligible error)  gives  the expectation in Theorem \ref{FreeCLT}.
\end{proof}

\begin{appendix}
\section{Additional preliminary and proofs of Lemmas in Section \ref{Sec. Preliminaries}}\label{Sec A}

The following cumulant expansion formula plays a central role in our computation, whose proof can be found in \cite{he2020mesoscopic}; also see for instance \cite{lytova2009central} and \cite{khorunzhy1996asymptotic} for earlier versions of this formula. 
 \begin{lemma}[Cumulant expansion] \label{cumulant expansion}
 	Let $f: \mathbb{R}\to \mathbb{C}$ be a smooth function, and denote by $f^{(k)}$ its $kth$ derivative. Then for every fixed $l \in \mathbb{N}$, we have
 	\begin{align}\label{CE}
 		\E [\xi f(\xi)] = \sum_{k=0}^l\frac{\kappa_{k+1}(\xi)}{k!}\E [f^{(k)}(\xi)] + \mathcal{R}_{l+1},
 	\end{align}
	assuming that all expectations in \eqref{CE} exist, where $\mathcal{R}_{l+1}$ is the remainder term (depending on $f$ and $\xi$), such that for any $t > 0$,
	\begin{align}\label{remainder}
	\mathcal{R}_{\ell+1}=&O(1) \cdot \mathbb{E}\left[|\xi|^{ \ell+2} \mathbf{1}_{|\xi|>t}\right]\cdot \sup _{x \in \mathbb{R}}\left|f^{(\ell+1)}(x)\right|  +O(1) \cdot \mathbb{E}\left[|\xi|^{\ell+2}\right] \cdot \sup _{|x| \leqslant t}\left|f^{(\ell+1)}(x)\right|.	
	\end{align} 
 \end{lemma}

We also need the following result on the spectrum of the sample covariance matrix.
\begin{lemma}\label{8002}
	Under Assumptions \ref{assum1} and \ref{assum2} , we have for all $t \in [\![ k]\!]$,
	\begin{align}
		\frac{1}{2}\left(1 - \sqrt{\frac{p_t}{N}} \right)^2 \leq \lambda_{\min}(X_tX_t') < \lambda_{\max}(X_tX_t')  \leq 2\left(1 + \sqrt{\frac{p_t}{N}} \right)^2 \label{4114}
	\end{align}
	with high probability.
\end{lemma}
\begin{proof}[Proof of Lemma \ref{8002}]
	In the regime that $p_t \ge N^{\epsilon}$ for some fixed small $\epsilon > 0$, (\ref{4114}) can be directly obtained from \cite{bloemendal2016principal}. 
	In the regime that $p_t \le N^{\epsilon}$ for sufficiently small $\epsilon > 0$, we write
	\begin{align*}
		X_tX_t' = \mathrm{I} + R,
	\end{align*} 
	where $R$ is a $p_t \times p_t$ matrix defined entrywise by $R_{ij} = x_{ti}x_{tj}' - \E x_{ti}x_{tj}'$ and $x_{ti}$ represents the $i$-th row of $X_t$. One can easily see that $|R_{ij}| \prec N^{-1}$ and $\| R\| \le \|R \|_{\mathrm{HS}} \prec p_t^2/N \prec N^{-1+2\epsilon}$. Therefore, by the Weyl's inequality, we have
	\begin{align*}
		\big|\lambda_l(X_tX_t') - 1\big| \prec N^{-1+2\epsilon}, \quad l \in [\![p_t ]\!],
	\end{align*}
	which implies (\ref{4114}). 
\end{proof}

Based on Lemma \ref{8002}, we have some preliminary high probability bounds on the spectral norm of some random matrices involved in the following discussions. 
\begin{lemma} \label{bound for covariance}
	Under Assumptions \ref{assum1} and \ref{assum2} , we have for all $t \in [\![ k]\!]$,
	\begin{align}
		 \|X_t \| \prec 1,\quad  \|(X_tX_t')^{-1} \| \prec 1.
	\label{8001}\end{align}
	Further, for any $z \in \mathbb{C}$ such that $\| G\|\prec 1$, we have
	\begin{align}
		|\tr G|\prec 1,\quad |\tr P_tG| \prec \frac{p_t}{N}, \quad |\tr Q_tG| \prec \frac{p_t}{N} ,\quad |\tr (X_tX_t')^{-1}| \prec \frac{p_t}{N}.
	\label{8003}\end{align}
\end{lemma}
\begin{proof}[Proof of Lemma \ref{bound for covariance}]
By Lemma \ref{8002}, we can obtain (\ref{8001}) directly. If $\| G\|  = O_{\prec}(1)$, (\ref{8003}) can be bounded as the following,
 \begin{align*}
 	&|\tr G| \le \|G \| \prec 1, \quad |\tr P_tG| \le \frac{p_t}{N} \| P_tG\| \prec \frac{p_t}{N} ,\\
 	&|\tr Q_tG |\le \frac{p_t}{N} \| Q_tG\|  \leq  \frac{p_t}{N} \| X_t \|^2\| \|(X_tX_t')^{-1} \|^2\| G\|  \prec \frac{p_t}{N} ,\notag\\
	& \tr (X_tX_t')^{-1} \le  \frac{p_t}{N} \| (X_tX_t')^{-1} \|  \prec \frac{p_t}{N}.
 \end{align*}
 Here we used the inequality $|\Tr (A)| \le r\|A \| $   for any matrix $A$ with rank $r$.
\end{proof}

We have the following properties of the truncation function $\Xi$ in (\ref{Xit}) which is a direct consequence of the definition and Lemma \ref{8002}.
\begin{lemma}\label{XiLemma}
Under the assumption of Theorem \ref{FreeCLT}, we have $\Xi = 1$ with high probability,  and for any fixed $s$, $\partial_{t,ji}^s\Xi = 0$ with high probability.
\end{lemma}
From the definition of $\Xi$, we also have the following lemma

\begin{lemma}\label{moreXi}
	Let
	$
	\mathsf{X}=(\mathsf{X}^{(N)}(u):  N \in \mathbb{N}, \ u \in \mathsf{U}^{(N)})
	$
	be family of  random variables, where $\mathsf{U}^{(N)}$ is a possibly $N$-dependent parameter set.	

There exists a constant $C>0$ such that  $|\mathsf{X}^{(N)}(u)| \leq N^{C}$ a.s. uniformly in $u$ for all sufficiently large $N$. Then for any fixed $n \ge 0$, we have $\E[ \mathsf{X}]=  \E[ \mathsf{X}\cdot \Xi^n ]+ O(N^{-D})$ for any fixed $D>0$. 
\end{lemma}
\begin{proof}
Notice that $
		\E[ \mathsf{X}]=  \E[ \mathsf{X}\cdot \Xi^n ] + \E[ \mathsf{X}\cdot (1 - \Xi^n) ]. 
$ By Cauchy-Schwarz, we have
\begin{align*}
	|\E[ \mathsf{X}\cdot (1 - \Xi^n) ]|   \le N^C\mathbb{P}(\Xi \neq 1) =O(N^{-D})
\end{align*}
for any large $D > 0$. Where in the last step, we used Lemma \ref{XiLemma}.

\end{proof}

We further notice that on the event $\{\Xi\neq 0\}$, by the definition of $\Xi$ the crude bounds hold
\begin{align*}
\|(X_tX_t')^{-1}\|\leq 2N^{K+1}, \qquad \|X_tX_t'\|\leq 2N^{K+1},\qquad t\in [\![k]\!].  
\end{align*}

Next, we give the proofs of Lemmas \ref{prop_prec} and \ref{Imomegatc}.
\begin{proof}[Proof of Lemma \ref{prop_prec}]
	Part (i) is obvious from Definition \ref{def.sd}. For any fixed $\varepsilon>0$, we have
	\begin{align}
	|\E \mathsf{X}_1| \leq \E |\mathsf{X}_1{\bf 1}(|\mathsf{X}_1|\leq  N^{\varepsilon}\Psi)|+\E |\mathsf{X}_1{\bf 1}(|\mathsf{X}_1|\geq  N^{\varepsilon}\Psi)|\leq N^{\varepsilon}\Psi+ N^{C}\Psi \mathbb{P}(|\mathsf{X}_1|\geq N^{\varepsilon}\Psi)=O(N^{\varepsilon}\Psi) \label{062801}
	\end{align}
 for sufficiently large $N\geq N_0(\epsilon)$. This proves part (ii). 
\end{proof}

\begin{proof}[Proof of Lemma \ref{Imomegatc}]
First, we have the following high probability bound on the spectral norm of $H$.
\begin{align*}
	\|H \| = \Big\|\sum_{t=1}^kP_t \Big\| \prec \Big\|\sum_{t=1}^{k}X_t'X_t \Big\| \prec 1,
\end{align*}
where in the second step, we used Lemma \ref{8002} to get that $\sum_{t=1}^kX_t'((X_tX_t')^{-1} - CI)X_t$ is semi-negative definite for some large $C > 0$. In the last step, we used the fact that the spectral norm of a $p \times N$ sample covariance matrix is bounded with high probability.
	Therefore, by the boundedness of the operator norm of $H$, $G(z)$ can be expanded as
	\begin{align*}
		G(z) =- \sum_{l=1}^{\infty}z^{-l}H^{l-1}, 
	\end{align*}
	as $\Im z \to \infty$. Therefore, from the definition of $\omega_t^c(z)$, we have
	\begin{align*}
		\omega_t^c(z) = z - \frac{\sum_{s\neq t}\tr P_sG(z)}{\tr G(z)} = z - \frac{\sum_{s\neq t}\left( \tr P_sz^{-1} + \tr P_sHz^{-2} \right)}{z^{-1}} + O(z^{-2}),
	\end{align*}
	as $\Im z \to \infty$. Taking imaginary part on both sides, we have 
	\begin{align*}
		\Im \omega_t^c(z) - \Im z = \frac{\Im z \sum_{s\neq t}\tr P_sH}{|z|^2} + O\left(|z|^{-2} \right),
	\end{align*}
	as $\Im z \to \infty$. 
	Since
$
		\sum_{s\neq t}\tr P_sH \geq  \sum_{s\neq t}\tr P_s = \sum_{s\neq t} p_s/N,
$
we can get
$
		\Im \omega_t^c(z) - \Im z \ge \frac{\Im z}{2|z|^2}\sum_{s\neq t}p_s/N,
$
	as $\Im z \to \infty$. Hence, when $\Im z$ is sufficiently large, we have 
$
		\Im \omega_t^c(z) - \Im z \ge c 
$
for some small constants $c > 0$	by Assumption \ref{assum2}.  This concludes the proof. 
\end{proof}

We further recall the  classical Newton-Kantorvich theorem \cite{ferreira2012kantorovich}.
	\begin{thm}[Newton-Kantorvich] Let $X,Y$ be Banach spaces, $\mathcal{S} \subset X$, and $g:\mathcal{S}  \mapsto Y$ a continuous function, continuously differentiable on $\mathrm{int}(\mathcal{S})$. Take $x_0 \in \mathrm{int}(\mathcal{S})$, $L,b >0$ and suppose that,
	\begin{itemize}
		\item[1.] $g'(x_0)$ is non-singular,
		\item[2.] $\|g^{\prime}\left(x_{0}\right)^{-1}\left[g^{\prime}(y)-g^{\prime}(x)\right]\| \leq L\|x-y\| \quad \text { for any } x, y \in \mathcal{S}$,
		\item[3.] $\|g^{\prime}\left(x_{0}\right)^{-1} g\left(x_{0}\right)\| \leq b$,
		\item[4.] $2bL < 1$.
	\end{itemize}
		Define 
		\begin{align*}
			t_{*}:=\frac{1-\sqrt{1-2 b L}}{L}, \quad t_{* *}:=\frac{1+\sqrt{1-2 b L}}{L}
		\end{align*}
		If $B[x_0, t_{*}] \subset \mathcal{S}$, then sequences $\{x_n\}$ generated by Newton's Method for solving $g(x) = 0$ with starting point $x_0$,
		\begin{align*}
			x_{n+1} = x_n -g'(x_n)^{-1} g(x_n), \quad n = 1,2,\cdots
		\end{align*}
		is well defined, is contained in $B(x_0, t_{*})$, converges to a point $x_{*} \in B[x_0, t_{*}]$ which is the unique zero of $g$ in $B[x_0, t_{*}]$ and 
		\begin{align*}
			\left\|x_{*}-x_{n+1}\right\| \leq \frac{1}{2}\left\|x_{*}-x_{n}\right\|, \quad n=0,1, \cdots
		\end{align*}
	\label{NewtonKantorvich}\end{thm}

\section{Proofs of results in Section \ref{Proofof1stLimit}}\label{Sec B}
\subsection{Proof of Lemma \ref{supportofmu}}
For any $t \in [\![k ]\!]$, define $\tilde{X}_t$ be an $\tilde{p}_t \times \tilde{N}$ Gaussian matrix with $\tilde{p}_t =  y_t\tilde{N}$, where we choose $\tilde{N}=\ell N$ so that $\tilde{p_t}=\ell p_t$ is an integer. In this part, we regard $N$ and also $k=k(N)$ as fixed, but send  $\ell $ to infinity independently of $N$. Using the large $\ell$ (and thus large $\tilde{N}$) limit behaviour of the ESD of the matrix $\tilde{H}$ defined below in (\ref{061590}), we can get some basic property of $\mu_\boxplus$.  Performing the singular value decomposition, we have
	\begin{align*}
		\tilde{X}_t = \tilde{U}'_t \big(\tilde{\Lambda}_t ,  0\big)\tilde{V}_t, \quad t \in [\![k ]\!],
	\end{align*}
	where $\tilde{U}_t$ and $\tilde{V}_t$ are $\tilde{p}_t$ dimensional and $\tilde{N}$ dimensional orthogonal matrices respectively. Since $\tilde{X}_t$ has i.i.d normal entries,  $\tilde{U}_t$ and $\tilde{V}_t$ are independent and Haar distributed. Let
	\begin{align}
		\tilde{H}:= \sum_{t=1}^{k} \tilde{X}_t'(\tilde{X}_t\tilde{X}_t')^{-1}\tilde{X}_t = \sum_{t=1}^k\tilde{V}_t'\left(I_{\tilde{p}_t} \oplus 0_{\tilde{N}-\tilde{p}_t}\right)\tilde{V}_t. \label{061590}
	\end{align}
	Notice that the ESD of $\tilde{V}_t'\big(I_{\tilde{p}_t} \oplus 0_{\tilde{N}-\tilde{p}_t}\big)\tilde{V}_t$ goes to $\mu_t$ when $\tilde{N}\to \infty$ via $\ell\to \infty$, and it is well known  that the ESD of $\tilde{H}$ converges weakly in probability to $\mu_{\boxplus}$; see  \cite{collins2011strong} for instance.  Therefore, in order to show (\ref{supp of mu}), it suffices to show that 
	\begin{align}
		\text{Spec}(\tilde{H})\setminus\{0\}\subset [a, b]  \label{0615100}
	\end{align} 
	with high probability (in $\tilde{N}$) for some strictly positive constants $a$ and $b$.   Since the non-zero eigenvalues of $\tilde{H}$ are the same as the following matrix
	\begin{align*}
		\tilde{\mathcal{B}} = {\rm diag}((\tilde{X}_t\tilde{X}_t')^{-\frac{1}{2}})_{t=1}^k \cdot \tilde{X}\tilde{X}' \cdot {\rm diag}((\tilde{X}_t\tilde{X}_t')^{-\frac{1}{2}})_{t=1}^k,
	\end{align*}
	where $\tilde{Y}$ is defined analogously to $Y$ in the definition of block correlation matrix. We have by the bounds of the smallest and largest eigenvalues  of sample covariance matrix (c.f. Lemma \ref{8002}), 
	\begin{align}
		\|\tilde{\mathcal{B}}\| 
		\le C\left( 1 + \sqrt{y}\right)^2 \max_t  \left(1 - \sqrt{y_t} \right)^{-1}
		\le  C'\max_t  \left(1 - \sqrt{y_t} \right)^{-1} \label{0624011}
	\end{align}
	and similarly,
	\begin{align}
		\| \tilde{\mathcal{B}}^{-1}\| 
		\le  C\left( 1 - \sqrt{y}\right)^{-2} \max_t  \left(1 + \sqrt{y_t} \right)
		\le  C'\left(1-\sqrt{y}\right)^2 \label{0624012}
	\end{align}
	with high probability (in $\tilde{N}$). 
	Here we used the assumption that $\hat{y} \in(0,1)$. 
	Note the above bounds hold almost surely. Since $y_t$ and $y$ are away from $1$, we have (\ref{0615100}). 	
	Combining (\ref{0624011}) with the fact that $\tilde{H}$ is semi-positive definite, we have (\ref{supp of mu2}). Hence, we complete the proof.

\subsection{Proof of Lemma \ref{Gbound}}
If $\hat{y} \in (0,1)$, similar to (\ref{0624011}) and (\ref{0624012}), we have
	\begin{align}
		&\|{\mathcal{B}}\| 
		\le C\left( 1 + \sqrt{y}\right)^2 \max_t  \left(1 - \sqrt{y_t} \right)^{-1}
		\le  C'\max_t  \left(1 - \sqrt{y_t} \right)^{-1} ,\label{06291015} \\
		&\| {\mathcal{B}}^{-1}\| 
		\le  C\left( 1 - \sqrt{y}\right)^{-2} \max_t  \left(1 + \sqrt{y_t} \right)
		\le  C'\left(1-\sqrt{y}\right)^2
	\end{align}
	with high probability. Since  $y_t$ and $y$ are away from $1$, we have
	\begin{align*}
		\mathrm{Spec}(H) \subset \{ 0\} \cup [a,b]
	\end{align*}
	with high probability (in $N$) for some strictly positive constants $a$ and $b$. Therefore, by the construction of contours in (\ref{contour12}), we have $\|G(z) \| = \max_i |\lambda_i(H)-z| \sim 1$ with high probability. The upper bound of $|\tr G(z)|$ can be obtained directly by the inequality $|\tr G(z)| \le \|G(z) \|$. For the lower bound, the discussion is divided into three cases. 
		
		(i) When $z \in \cup_{a=2}^5 \mathcal{C}_a$, we know that $\Im z > 0$. Hence, we have
		\begin{align*}
			|\tr G(z)| \ge \Im \tr G(z) = \frac{1}{N}\sum_{i=1}^N\frac{\Im z}{|\lambda_i - z|^2} \gtrsim 1
		\end{align*} 
		with high probability. Here in the last step we used the fact that $\mathrm{Spec}(H)$ is bounded with high probability.
		
		(ii) When $z \in \mathcal{C}_6$, with sufficiently large $\Re z$, we have
		\begin{align*}
			|\tr G(z)| \ge |\Re \tr G(z)| = \frac{1}{N}\sum_{i=1}^N\frac{\Re z - \lambda_i}{|\lambda_i - z|^2} \gtrsim 1
		\end{align*}
	
		(iii)When $z \in \mathcal{C}_1$, we have
		\begin{align*}
			|\tr G(z)| \ge \left(1-\frac{p}{N}\right)\frac{1}{|z|} - \frac{1}{N}\sum_{i=1}^p\frac{1}{\lambda_i - |z|} \gtrsim 1
		\end{align*}
		with high probability. Here in the last step we used the smallness of $|z|$ and the fact that $\lambda_p \ge a$ with high probability.
		
	If $\hat{y} \in (0,\infty)$, combining (\ref{06291015}) with the fact that $H$ is semi-positive definite, we have $\mathrm{Spec}(H) \subset [0,C]$ with high probability (in $N$) for some strictly positive constant $C$. Therefore, if $z \in \gamma_1\cup\gamma_2$ (sufficiently large contours, c.f. (\ref{contour12})), we still have $\|G(z) \| = \max_i |\lambda_i(H)-z| \sim 1$ with high probability. For the estimate of $|\tr G(z)|$, we only need to consider the case when $z \in \mathcal{C}_7$. This case is similar to case (ii) above. We have
		\begin{align*}
			|\tr G(z)| \ge |\Re \tr G(z)| = \frac{1}{N}\sum_{i=1}^N\frac{\lambda_i-\Re z}{|\lambda_i - z|^2} \gtrsim 1
		\end{align*}
		with high probability. Here in the last step we used the fact that $\mathrm{Spec}(H)$ is supported in the positive real line with high probability.

\subsection{Proof of Lemma \ref{Lemma error estimates}}
We will show the estimates of $\gamma_{t1}$ and $\sum_{t=1}^k\mathcal{W}_t\gamma_{t1}$ in detail, and  the estimates of other terms will be sketchy since it is simpler. 

We first consider $\gamma_{t1}$. For simplicity, we set
\begin{align*}
Z=\Tr P_t G-\left(\tr( X_t X_t')^{-1}-\tr Q_t G\right)(\Tr G-\Tr P_t G)
\end{align*}
and further  set 
\begin{align*}
\mathsf{Z}^{p,q}:= (Z\cdot\Xi)^p(\bar{Z}\cdot \Xi)^q,
\end{align*}
where $\Xi$ is defined in (\ref{Xit}).
By direct calculation, one can show that 
\begin{align}
	\partial_{t,ji} Z = [\mathcal{A}]_{ij}\label{061491}
\end{align} 
where
\begin{align}
	\mathcal{A} =& 2 W_t G( I- P_t) -2  W_t G P_t G( I- P_t)\notag\\
	&+ 2\left(\tr G-\tr P_t G\right)\left(( X_t X_t')^{-2} X_t  -W_t G Q_t +W_t G Q_t G( I- P_t)+( I- P_t) G X_t'( X_t W_t')^{-2} \right)\notag\\
	&+ 2\left(\tr( X_t X_t')^{-1}-\tr Q_t G\right)\left( W_t G( I- P_t)+ W_t G G( I- P_t)+W_t G P_t G( I- P_t) \right), \label{061493}
\end{align}
and it is easy to show that   $\| \mathcal{A} \|$ is bounded with high probability by Lemma \ref{bound for covariance}. 

Taking derivatives inductively and using Lemma \ref{bound for covariance}, one can easily check the following bound for any fixed order partial derivatives 
\begin{align*}
	|\partial_{t,ji}^s Z| = O_{\prec}(1), \quad s \in \mathbb{N}.
\end{align*}
For any fixed positive integer $n$, by the cumulant expansion, we have 
\begin{align}
\E \big[ \mathsf{Z}^{n,n} \big] 
=& \E^{\chi}\Big[ \left( \Tr P_t G-\left(\tr( X_t X_t')^{-1}-\tr Q_t G\right)\left(\Tr G-\Tr P_t G\right)\right)\mathsf{Z}^{n-1,n} \Big] \notag\\
=&\E^{\chi} \Big[  \sum_{ij}^{(t)}X'_{t,ij}[ W_t G]_{ji}\mathsf{Z}^{n-1,n} \Big]-\E^{\chi}\Big[ \left(\tr( X_t X_t')^{-1}-\tr Q_t G\right)(\Tr G-\Tr P_t G)\mathsf{Z}^{n-1,n}\Big] \notag\\
=&\frac{1}{N} \E^{\chi}\bigg[\Big(\sum_{ij}^{(t)}\partial_{t,ji} [ W_t G]_{ji}-\left(\tr( X_t X_t')^{-1}-\tr Q_t G\right)(\Tr G-\Tr P_t G) \Big) \mathsf{Z}^{n-1,n}\bigg] \notag\\
&+\frac{1}{N}\E^{\chi}\bigg[\Big(\sum_{ij}^{(t)}[ W_t G]_{ji} \partial_{t,ji} Z\Big)\Xi \mathsf{Z}^{n-2,n}\bigg] +  \frac{1}{N}\E^{\chi}\bigg[\Big(\sum_{ij}^{(t)}[ W_t G]_{ji} \partial_{t,ji} \bar{Z}\Big)\Xi\mathsf{Z}^{n-1,n-1}\bigg] \notag\\
&+\sum_{s\geq 2}^l O\left(\frac{1}{N^{\frac{s+1}{2}}}\right)\E^{\chi}\bigg[ \sum_{ij}^{(t)}\partial_{t,ji}^s\{[ W_t G]_{ji}\mathsf{Z}^{n-1,n}\} \bigg]\notag \\
&+ \sum_{s\geq1}^l O\left(\frac{1}{N^{\frac{s+1}{2}}}\right) \sum_{\substack{s_0+s_1 = s\\s_1\geq 1}}\E\Big[ \sum_{ij}^{(t)}\partial_{t,ji}^{s_0}\{[ W_t G]_{ji}\mathsf{Z}^{n-1,n}\} \partial_{t,ji}^{s_1}\Xi\Big]  + \E\bigg[ \sum_{ij}^{(t)} \mathcal{R}^{ij}_{l+1}\bigg]\notag\\
  =:& \mathsf{T}_1 + \mathsf{T}_2 + \mathsf{T}_3 + \mathsf{T}_4 + \mathsf{E}_1 + \mathsf{E}_2. \label{061490}
\end{align}
where 
\begin{align}
	\mathcal{R}^{ij}_{l+1} =& O(1) \cdot \sup _{X_{t,ji}\in \mathbb{R}}\left| \partial_{t,ji}^{l+1}\left\{ [ W_t G]_{ji}\Xi\mathsf{Z}^{n-1,n} \right\} \right| \cdot \mathbb{E}[|X_{t,ji}|^{l+2} \mathbf{1}_{|X_{t,ji}|>x}] \notag\\
	&+O(1) \cdot \mathbb{E}[|X_{t,ji}|^{l+2} ]\cdot \sup _{|X_{t,ji}| \leqslant x}\left| \partial_{t,ji}^{l+1}\left\{ [ W_t G]_{ji}\Xi \mathsf{Z}^{n-1,n} \right\} \right|. \label{0625041}
\end{align}
We first consider the error terms $\mathsf{E}_i, i = 1,2$. For $\mathsf{E}_1$, by the definition of $\Xi$ in (\ref{Xit}), we have $\partial_{t,ji}^{s_1}\Xi = 0$ with high probability. Together with the fact that $\partial_{t,ji}^{s_0}\{[ W_t G]_{ji}\mathsf{Z}^{n-1,n}\} \partial_{t,ji}^{s_1}\Xi$ has deterministic bound when $\Im z  > 0$, we have 
\begin{align}
	|\mathsf{E}_1| \prec N^{-D} \label{Error1}
\end{align}
 for any large $D$. For $\mathsf{E}_2$. Choosing $x = N^{-\frac12+\epsilon}$ with $\epsilon$ being any small positive constant, we have
\begin{align}
	\mathbb{E}\Big[|X_{t,ji}|^{l+2} \mathbf{1}_{|X_{t,ji}|>N^{-\frac12+\epsilon}}\Big] \le \sqrt{\E \Big[|X_{t,ji}|^{2l+4}\Big] } \sqrt{\mathbb{P}(|X_{t,ji}|>N^{-\frac12+\epsilon})} \lesssim N^{-D} \label{0625021}
\end{align}  
for any large $D > 0$. Due to the existence of $\Xi$ and $\Im z  > 0$, we have 
\begin{align}
	\sup _{X_{t,ji}\in \mathbb{R}}\left| \partial_{t,ji}^{l+1}\left\{ [ W_t G(z_1)]_{ji} \Xi \mathsf{Z}^{n-1,n} \right\} \right| \lesssim N^{\tilde{K}} \label{0625031}
\end{align}
for some fixed $\tilde{K}>0$. Combining (\ref{0625021}) and  (\ref{0625031}), we have the first term in (\ref{0625041}) can be bounded by $N^{-D}$ for any large $D > 0$. For the second term in (\ref{0625041}), by Lemma \ref{remainderLemma} and the crude bound $|\bar{Z}|,|Z| \prec p_t$ (c.f. Lemmas \ref{bound for covariance} and \ref{Gbound}), we have
\begin{align*}
	\sup _{|X_{t,ji}| \leqslant N^{-\frac12+\epsilon}}\left| \partial_{t,ji}^{l+1}\left\{ [ W_t G]_{ji} \Xi \mathsf{Z}^{n-1,n} \right\} \right| \prec p_t^{2n-1}
\end{align*}
Therefore, together with Assumption \ref{assum1}, we have
\begin{align}
	\mathbb{E}[|X_{t,ji}|^{l+2} ]\cdot \sup _{|X_{t,ji}| \leqslant x}\left| \partial_{t,ji}^{l+1}\left\{ [ W_t G]_{ji} \Xi\mathsf{Z}^{n-1,n}\right\} \right| \prec p_t^{2n-1}N^{-\frac{l+2}{2}}. \label{0625044}
\end{align}
Plugging (\ref{0625021}), (\ref{0625031}) and (\ref{0625044}) into (\ref{0625041}), we can get
\begin{align*}
	\mathcal{R}_{l+1}^{t,ji} \prec p_t^{2n-1}N^{-\frac{l+2}{2}}.
\end{align*}
Choosing $l > 4n + 1$ , we have 
\begin{align}
	| \mathsf{E}_2| \prec \sum_{ij}^{(t)} \frac{p_t^{2n-1}}{N^{2n+2}} \prec \left(\frac{p_t}{N}\right)^{n}.\label{Error2}
\end{align}

 We then turn to the main terms ($\mathsf{T}_i, i = 1,2,3,4$) in the RHS of (\ref{061490}). First, 
notice that
 \begin{align}
\frac{1}{N}\sum_{ij}^{(t)}\partial_{t,ji} [ W_t G]_{ji} =& \left(\tr( X_t X_t')^{-1}-\tr Q_t G\right)\left(\Tr G-\Tr P_t G \right) - \tr Q_tG - \tr Q_tGG + \tr Q_tGP_tG \notag\\
=&\left(\tr( X_t X_t')^{-1}-\tr Q_t G\right)\left(\Tr G-\Tr P_t G \right) + O_{\prec}\left( \frac{p_t}{N} \right), \label{061501}
 \end{align}
 which together with (\ref{061491}) gives
 \begin{align}
 	\frac{1}{N}\sum_{ij}^{(t)}[ W_t G]_{ji} \partial_{t,ji} Z = & \frac{1}{N}\sum_{ij}^{(t)} [ W_t G]_{ji} \mathcal{A}_{ij} = \tr  \mathcal{A}W_t G = O_{\prec}\left( \frac{p_t}{N} \right), \label{061502}
 \end{align}
 where  we used (\ref{061493}) from which it is easy to check  $|\Tr \mathcal{A}W_t G|  \le p_t \|  \mathcal{A}W_t G \| = O_{\prec}(p_t)$.
Plugging (\ref{061501}) and (\ref{061502}) (also its complex conjugate counterpart) into (\ref{061490}), we have
\begin{align}
	\mathsf{T}_1 + \mathsf{T}_2 + \mathsf{T}_3 = \E^{\chi} \bigg[O_{\prec}\left( \frac{p_t}{N} \right) \mathsf{Z}^{n-1,n} \bigg] + \E^{\chi} \bigg[O_{\prec}\left( \frac{p_t}{N} \right) \mathsf{Z}^{n-2,n} \bigg] +\E^{\chi} \bigg[O_{\prec}\left( \frac{p_t}{N} \right) \mathsf{Z}^{n-1,n-1} \bigg]. \label{062590}
\end{align}

 Next, we  estimate $\mathsf{T}_4$ in (\ref{061490}). 
First, similarly to the proof for Lemma \ref{moreXi}, we can easily show that 
\begin{align*}
	\E^{\chi}\bigg[ \sum_{ij}^{(t)}\partial_{t,ji}^s\{[ W_t G]_{ji}\mathsf{Z}^{n-1,n}\} \bigg] 
	= \E^{\chi} \bigg[ \sum_{ij}^{(t)}\partial_{t,ji}^s\{[ W_t G]_{ji}Z^{n-1}\bar{Z}^n\} \bigg] + O_{\prec}(N^{-D})
\end{align*}
for any fixed $D > 0$.  Further,  since the complex conjugate takes no effect in the remaining estimations, we drop the complex conjugate for simplicity. 
 Therefore,  we will work with $Z^{2n-1}$ instead of $\mathsf{Z}^{n-1,n}$ in the remaining derivation, for notational brevity. For any $s\geq 2$, we first write  
\begin{align*}
\frac{1}{N^{\frac{s+1}{2}}}\sum_{ij}^{(t)}\partial_{t,ji}^s\{[ W_t G]_{ji}Z^{2n-1}\}&=\frac{1}{N^{\frac{s+1}{2}}}\sum_{r=0}^{s}\sum_{ij}^{(t)}\sum_{\substack{s_0\geq 0,s_1,\cdots,s_r >0\\ s_0+s_1+\cdots+s_r=s}}\partial_{t,ji}^{s_0}[ W_t G ]_{ji}\partial_{t,ji}^{s_1} Z \cdots \partial_{t,ji}^{s_r}Z Z^{2n-1-r}\\
&=:\sum_{r=0}^s\mathcal{C}_{r,s} Z^{(2n-1-r)\vee 0}
\end{align*}
When $r=s$, i.e., $s_0 = s_1 - 1 = \cdots = s_r - 1 = 0$, with (\ref{061491}) and the fact that $\| \mathcal{A} \| \prec 1$, we have  
\begin{align}
|\mathcal{C}_{r,s}| \prec &  \frac{1}{N^{\frac{r+1}{2}}}\bigg|\sum_{ij}^{(t)}[W_t G]_{ji}\left([\mathcal{A}]_{ji}\right)^r\bigg| \prec  \frac{1}{N^{\frac{r+1}{2}}}  \sum_{ij}^{(t)} | [W_t G]_{ji}|| [\mathcal{A}]_{ji}|  \notag \\
    \lesssim  &\frac{1}{N^{\frac{r+1}{2}}} \sqrt{\sum_{ij}^{(t)} | [W_t G]_{ji}|^2} \sqrt{\sum_{ij}^{(t)} | [\mathcal{A}]_{ji}|^2 } =\frac{1}{N^{\frac{r+1}{2}}} \sqrt{\Tr W_t GG^{\ast} W_t'}\sqrt{\Tr \mathcal{A}\mathcal{A}^*}\prec  \frac{p_t}{N^{\frac{r+1}{2}}}. \label{062301}
\end{align}
 When $s-r=1$, i.e.,  $s_0 - 1= s_1 - 1 = \cdots = s_r - 1 = 0$ or there exists an $i$ such that $s_0 = s_1 - 1 = \cdots = s_i - 2 = \cdots = s_r - 1 = 0$, we have
 \begin{align*}
 	&|\mathcal{C}_{r,s}| \prec \frac{1}{N^{\frac{r+2}{2}}} \bigg|\sum_{ij}^{(t)}\partial_{t,ji}[W_t G]_{ji}\left([\mathcal{A}]_{ji}\right)^r \bigg| + \frac{1}{N^{\frac{r+2}{2}}} \bigg|\sum_{ij}^{(t)}[W_t G]_{ji}\partial_{t,ji}[\mathcal{A}]_{ji} \left([\mathcal{A}]_{ji}\right)^{r-1}  \bigg|\\
 &\prec \frac{1}{N^{\frac{r+2}{2}}} \sum_{ij}^{(t)} \left(|[\mathcal{A}]_{ji} |  +   |[W_t G]_{ji} |\right)
 	\lesssim \frac{1}{N^{\frac{r+1}{2}}} \left(\sqrt{p_t \Tr \mathcal{A}\mathcal{A}^*}  + \sqrt{p_t \Tr W_t GG^*W_t' } \right)\prec \frac{p_t}{N^{\frac{r+1}{2}}},
 \end{align*}
 where in the third step we used Cauchy-Schwarz. 
When $s-r\geq 2$, we have two cases, $r = 0$ and $r > 0$. For $r > 0$,
\begin{align*}
	|\mathcal{C}_{r,s}| \prec \frac{1}{N^{\frac{r+3}{2}}} \bigg|\sum_{ij}^{(t)}\sum_{\substack{s_0\geq 0,s_1,\cdots,s_r >0\\ s_0+s_1+\cdots+s_r=s}}\partial_{t,ji}^{s_0}[ W_t G ]_{ji}\partial_{t,ji}^{s_1} Z \cdots\partial_{t,ji}^{s_r}Z   \bigg|\prec \frac{p_t}{N^{\frac{r+1}{2}}},
\end{align*}
and for $r = 0$
\begin{align*}
	|\mathcal{C}_{0,s}| \prec&\frac{1}{N^{\frac{3}{2}}} \bigg|\sum_{ij}^{(t)} \partial_{t,ji}^{2}[ W_t G ]_{ji} \bigg|+ \frac{1}{N^{2}}\bigg|\sum_{ij}^{(t)} \partial_{t,ji}^{3}[ W_t G ]_{ji}\bigg| + \frac{1}{N^{\frac{5}{2}}}\bigg|\sum_{ij}^{(t)} \partial_{t,ji}^{4}[ W_t G ]_{ji}\bigg| + \cdots \\
	= & \frac{1}{N^{\frac{3}{2}}}\bigg|\sum_{ij}^{(t)}\partial_{t,ji}^{2}[ W_t G ]_{ji} \bigg|+ O_{\prec}\left( \frac{p_t}{N} \right).
\end{align*}
Notice that all terms in $ \partial_{t,ji}^{2}[ W_t G ]_{ji}$ contain at least one off-diagonal entry, which can be denoted by $\mathcal{M}_{ij}$ in general, say. Hence, the bound 
\begin{align*}
 \frac{1}{\sqrt{N}}\sum_{ij}^{(t)} |[\mathcal{M}]_{ij}|\leq  \frac{1}{\sqrt{N}}\sqrt{\sum_{ij}^{(t)} |[\mathcal{M}]_{ij}|^2 Np_t} = \sqrt{p_t\Tr \mathcal{M}\mathcal{M}^*} \prec p_t
\end{align*}
implies
$
	|\mathcal{C}_{0,s}|   \prec p_t/N.
$

Combining the above estimates, we can get
\begin{align*}
	\E \big[ \mathsf{Z}^{n,n}\big] = & \E^{\chi} \bigg[O_{\prec}\left( \frac{p_t}{N} \right) \mathsf{Z}^{n-1,n} \bigg] + \E^{\chi} \bigg[O_{\prec}\left( \frac{p_t}{N} \right) \mathsf{Z}^{n-2,n} \bigg] +\E^{\chi} \bigg[O_{\prec}\left( \frac{p_t}{N} \right) \mathsf{Z}^{n-1,n-1} \bigg] \\
	&+\sum_{r = 1}^{2n-1} \E^{\chi}\bigg[ O_{\prec}\left(  \frac{p_t}{N^{\frac{r+1}{2}}}\right) |Z|^{2n-r-1} \bigg] +  O_{\prec}\left( \left(\frac{p_t}{N}\right)^{n} \right) \\
	= & \E^{\chi} \bigg[O_{\prec}\left( \frac{p_t}{N} \right) \mathsf{Z}^{n-1,n} \bigg] + \E^{\chi} \bigg[O_{\prec}\left( \frac{p_t}{N} \right) \mathsf{Z}^{n-2,n} \bigg] +\E^{\chi} \bigg[O_{\prec}\left( \frac{p_t}{N} \right) \mathsf{Z}^{n-1,n-1} \bigg] \\
	&+\sum_{r = 1}^{2n-1} \E^{\chi}\bigg[ O_{\prec}\left(  \frac{p_t}{N^{\frac{r+1}{2}}}\right) |Z\cdot \Xi|^{2n-r-1} \bigg] +  O_{\prec}\left( \left(\frac{p_t}{N}\right)^{n} \right)
\end{align*}
where in the last step, we substitute $Z\cdot\Xi$ for $Z$, by a similar argument as the proof of Lemma \ref{moreXi}, up to an $O_{\prec}(N^{-D})$ error which can be absorbed by $O_{\prec}\left( \left(\frac{p_t}{N}\right)^{n} \right)$.

Applying Young's inequality for each term in the RHS of the above equation, we have
\begin{align}
	\E \big[ \mathsf{Z}^{n,n} \big] 
	\le &  \frac{1}{2n}\E^{\chi} \bigg[ O_{\prec}\left(\frac{p_t}{N} \right)^{2n}(\log N )^{2n} \bigg] + \frac{2n-1}{2n\left( \log N \right)^{\frac{2n}{2n-1}}}\E \big[  \mathsf{Z}^{n,n} \big] \notag \\
	&+ \sum_{r = 1}^{2n-1}  \left( \frac{r+1}{2n}\E^{\chi} \bigg[ O_{\prec}\left(\frac{p_t}{N^\frac{r+1}{2}} \right)^{\frac{2n}{r+1}} (\log N)^{\frac{2n}{r+1}}\bigg] + \frac{2n-r-1}{2n\left(\log N \right)^{\frac{2n}{2n-r-1}}}\E\big[  \mathsf{Z}^{n,n}\big]  \right) \notag \\
	= &   O_{\prec}\left( \left(\frac{p_t}{N}\right)^{n} \right)  + \sum_{r = 0}^{2n-1} \frac{2n-r-1}{2n\left(\log N \right)^{\frac{2n}{2n-r-1}}}\E\big[  \mathsf{Z}^{n,n}\big], \label{YoungIneq}
\end{align}
where in the last step we absorbed the logarithmic factor into $O_\prec(\cdot)$ by definition, and we also noticed that the quantities which give us  the $O_\prec(\cdot)$ terms in the first two lines all have crude deterministic bound due to the existence of the smooth cutoff function $\chi$, and thus the stochastic bound can also be applied when one take expectation. 
Absorbing the term with $\mathsf{Z}^{n,n}$ in the RHS into LHS in the above equation, we get
\begin{align}\label{PGMoment}
	\E \Big[ \mathsf{Z}^{n,n} \Big]\prec \left( \frac{p_t}{N}\right)^n .
\end{align}
Using Markov inequality, we obtain \eqref{trPG} from \eqref{PGMoment}, due to the arbitrariness of the fixed $n$.

If we replace the definition of $Z$ by 
\begin{align*}
	Z = \Tr P_t - (1-y_t)\Tr (X_tX_t')^{-1},
\end{align*}
write $\Tr P_t = \sum_{ij}^{(t)}X_{t,ji}[W_t]_{ji}$ and then apply cumulant expansion w.r.t $X_{t,ji}$'s
, similarly,  we can obtain the estimate of $\gamma_{t2}$. Here at the end we used the trivial fact that $\Tr P_t=p_t$. The details are omitted.

Again, if we replace the definition of $Z$ by 
\begin{align*}
	Z = \Tr P_tGP_t - (1-y_t)\Tr Q_tG,
\end{align*}
write $\Tr P_tGP_t = \sum_{ij}^{(t)}X_{t,ji}[W_tGP_t]_{ji}$ and then apply cumulant expansion w.r.t $X_{t,ji}$'s
, similarly,  we can obtain the estimate of $\gamma_{t3}$. Here at the end we used the fact that $\Tr P_tGP_t=\Tr P_t^2G=\Tr P_tG$. The details are also omitted.

Next, we consider the estimate of $\sum_{t=1}^k\mathcal{W}_t\gamma_{t1}$.  
By Lemma \ref{8002}, we have, 
\begin{align*}
&\lambda_{\max}((X_tX_t')^{-1})\leq C\left(1-\sqrt{\frac{p_t}{N}}\right)^{-2}, \\
&\lambda_{\max}(X_t'X_t), \quad  \lambda_{\max}(X_tAX_t')\leq C \left(1+\sqrt{\frac{p_t}{N}}\right)^2\leq 4C, 
\end{align*}
with high probability (c.f. Definition \ref{def.high-probab}),  where $A$ is any Hermitian matrix with bounded spectral norm $C$. Then  with high probability we have that $\sum_{t=1}^k X_t'[(X_tX_t')^{-2}-aI]X_t$ is a negative definite matrix for sufficiently large $a>0$. In addition, because $Q_t$ and $X_tX_t'$ are positive semi-definite matrices, we have 
\begin{align}\label{SumQt}
 \Big\|\sum_{t=1}^kQ_t\Big\|\leq C\Big\|\sum_{t=1}^kX_tX_t'\Big\|=O_\prec(1).
\end{align}
Define
		\begin{align*}
Z_k=\sum_{t=1}^k \mathcal{W}_t\Tr P_t G-\sum_{t=1}^k\mathcal{W}_t\left(\tr( X_t X_t')^{-1}-\tr Q_t G\right)(\Tr G-\Tr P_t G)
\end{align*}
and further set
\begin{align*}
	\mathsf{Z}_k^{q,p} := (Z_k\cdot\Xi)^p(\bar{Z}_k\cdot\Xi)^q.
\end{align*}
Then for any fixed $n$, applying the cumulant expansion formula in Lemma \ref{cumulant expansion},  we have
\begin{align}
&\E\big[\mathsf{Z}_k^{n,n}\big]\notag\\
&=\E^{\chi} \bigg[\sum_{t=1}^k\alpha_t\sum_{ij}^{(t)}X'_{t,ij}[ W_t G]_{ji}\mathsf{Z}_k^{n-1,n}\bigg]-\E^{\chi}\bigg[\sum_{t=1}^k\alpha_t \left(\tr( X_t X_t')^{-1}-\tr Q_t G\right)(\Tr G-\Tr P_t G)\mathsf{Z}_k^{n-1,n}\bigg] \notag\\
&=\E^{\chi}\bigg[ \sum_{t=1}^k\sum_{ij}^{(t)}\frac{1}{N}\partial_{t,ji}\{\alpha_t[ W_t G]_{ji}\mathsf{Z}_k^{n-1,n}\}\bigg]-\E^{\chi}\bigg[\sum_{t=1}^k \alpha_t\left(\tr( X_t X_t')^{-1}-\tr Q_t G\right)(\Tr G-\Tr P_t G)\mathsf{Z}_k^{n-1,n}\bigg] \notag\\
&+\sum_{s\geq 2}^lO\left(\frac{1}{N^{\frac{s+1}{2}}}\right)\E^{\chi}\bigg[\sum_{t=1}^k\sum_{ij}^{(t)}\partial_{t,ji}^s \{\alpha_t[ W_t G]_{ji}\mathsf{Z}_k^{n-1,n}\}\bigg] +O_{\prec}(1). \label{062102}
\end{align}
where  the error $O_{\prec}(1)$ can be verified similarly to (\ref{061490}) by choosing $l$ sufficiently large.
By direct calculation, we have
\begin{align*}
	\partial_{t,ji} Z_k =&\sum_{d=1}^{k}\tr P_d G \Big((\partial_x \mathcal{F}) (\partial_{t,ji} \Tr P_d G) +  (\partial_y\mathcal{F}) (\partial_{t,ji} \Tr G)\Big) +\sum_{d=1}^{k}\mathcal{W}_d(\partial_{t,ji} \Tr P_d G)\\
	&-\sum_{d=1}^{k} \Big((\partial_x\mathcal{F}) (\partial_{t,ji} \Tr P_d G)+(\partial_y\mathcal{F}) (\partial_{t,ji} \Tr G)\Big)\left(\tr( X_d X_d')^{-1}-\tr Q_d G\right)\left(\tr G-\tr P_d G\right)\\
	&+2\sum_{d=1}^k \mathcal{W}_d\left(\tr(X_dX_d')^{-1}-\tr Q_d G\right)[ W_t G( I- P_d) G( I- P_t)]_{ji}\\
	&-2\sum_{d=1}^k \mathcal{W}_d[ W_t G Q_d G( I- P_t)]_{ji}\left(\tr G-\tr P_d G\right)+2\mathcal{W}_t( W_t'([ X_t X_t')^{-1}]_{ij}+[ W_t G Q_t]_{ji}\\
	&-[( X_t X_t')^{-2} X_t G( I- P_t)]_{ji})\left(\tr G-\tr P_t G\right)-2\mathcal{W}_t[ W_t G( I- P_t)]_{ji}\left(\tr( X_t X_t')^{-1}-\tr Q_t G\right)
\end{align*}
Similarly to (\ref{061491}), one can check that, there exists a matrix $\mathcal{A}_{k}\in \mathbb{C}^{N\times p_t}$ with $\|\mathcal{A}_k\| = O_{\prec}(1)$, s.t.
\begin{align}
	\partial_{t,ji} Z_k = [\mathcal{A}_{k}]_{ij}. \label{062150}
\end{align}
With the boundedness of $\mathcal{W}_t$ and its partial derivatives, the boundedness of $\|\mathcal{A}_k\|$ is benefited from either the spectral norm of the summation of $Q_d$ is bounded or $\sum_{s=1}^k \tr M = O_{\prec}(1)$ when $M$ has $O_{\prec}(1)$ operator norm and of rank $p_t$.
For example, for the fourth term of $\partial_{t,ji} Z_k$, we have
\begin{align*}
	\Big\|\sum_{d=1}^k\mathcal{W}_d(\tr(X_dX_d')^{-1}-\tr Q_d G) W_t G( I- P_d) G( I- P_t)\Big\| 
	\prec \sum_{d=1}^k\left|\tr(X_dX_d')^{-1}-\tr Q_d G\right| \prec \sum_{s=1}^k \frac{p_t}{N} = O_{\prec}(1).
\end{align*}
And for the fifth term of $\partial_{t,ji} Z_k$, we have
\begin{align*}
	\Big\|\sum_{d=1}^k\mathcal{W}_d W_t G Q_d G( I- P_t)\left(\tr G-\tr P_d G\right)\Big\|
	 &\prec \Big\|\sum_{d=1}^k W_t G Q_d G( I- P_t)\Big\| + \Big\|\sum_{s=1}^k W_t G Q_d G( I- P_t)\tr P_d G\Big\|\notag\\
	 &\prec\Big\|\sum_{d=1}^kQ_d\Big\| +\Big|\sum_{d=1}^k \tr P_d G\Big| =O_{\prec}(1).
\end{align*}
Here we used \eqref{SumQt} to bound the spectral norm of the summation of $Q_d$. The other terms in the formula of $\partial_{t,ji}  Z_k$ can be checked similarly.

Therefore, we have the technical estimates
\begin{align*}
	&\frac{1}{N}\sum_{t=1}^k\sum_{ij}^{(t)}[ W_t G]_{ji}\partial_{t,ji}Z_k  = \sum_{t=1}^k\tr W_tG\mathcal{A}_k =  O_{\prec}\left(1\right), \\\
	&\frac{1}{N^{3/2}}\sum_{t=1}^k\sum_{ij}^{(t)}\partial_{t,ji}Z_k  \le  \frac{1}{N^{3/2}} \sum_{t=1}^k \sqrt{\sum_{ij}^{(t)}|[\mathcal{A}_{k}]_{ij}|^2 Np_t } =\sum_{t=1}^k  \frac{\sqrt{p_t\Tr \mathcal{A}_k\mathcal{A}_k^*}}{N} = O_{\prec}\left(1\right),
\end{align*}
and the same bounds hold when $Z_k$ is replaced by its complex conjugate $\bar{Z}_k$. 

Applying the above two estimates,  together with (\ref{061501}), we have 
\begin{align}
	&\E^{\chi}\bigg[ \sum_{t=1}^k\sum_{ij}^{(t)}\frac{1}{N}\partial_{t,ji}\{\mathcal{W}_t[ W_t G]_{ji}\mathsf{Z}_k^{n-1,n}\}\bigg]-\E^{\chi}\bigg[\sum_{t=1}^k \mathcal{W}_t\left(\tr( X_t X_t')^{-1}-\tr Q_t G\right)(\Tr G-\Tr P_t G)\mathsf{Z}_k^{n-1,n}\bigg] \notag\\
	=& \E^{\chi}\bigg[ \sum_{t=1}^k\sum_{ij}^{(t)}\frac{\mathcal{W}_t}{N}\partial_{t,ji}\{[ W_t G]_{ji}\}\mathsf{Z}_k^{n-1,n}\bigg]-\E^{\chi}\bigg[\sum_{t=1}^k \mathcal{W}_t\left(\tr( X_t X_t')^{-1}-\tr Q_t G\right)(\Tr G-\Tr P_t G)\mathsf{Z}_k^{n-1,n}\bigg] \notag\\
	&+ \E^{\chi}\bigg[ \sum_{t=1}^k\sum_{ij}^{(t)}\frac{1}{N}[ W_t G]_{ji}(\partial_x\mathcal{F}) (\partial_{t,ji} \Tr P_t G)\mathsf{Z}_k^{n-1,n}\bigg]+\E^{\chi}\bigg[ \sum_{t=1}^k\sum_{ij}^{(t)}\frac{1}{N}[ W_t G]_{ji}(\partial_y\mathcal{F}) (\partial_{t,ji} \Tr G)\mathsf{Z}_k^{n-1,n}\bigg]\notag\\
	& +\E^{\chi}\bigg[ \sum_{t=1}^k\sum_{ij}^{(t)}\frac{\mathcal{W}_t}{N}[ W_t G]_{ji}\left(\partial_{t,ji}Z_k\right) \mathsf{Z}_k^{n-2,n}\bigg]+\E^{\chi}\bigg[ \sum_{t=1}^k\sum_{ij}^{(t)}\frac{\mathcal{W}_t}{N}[ W_t G]_{ji}\left(\partial_{t,ji}\bar{Z}_k\right) \mathsf{Z}_k^{n-1,n-1}\bigg] + O_{\prec}(N^{-D}) \notag\\
	=&\E^{\chi}\bigg[\sum_{t=1}^k\mathcal{W}_t\left(\tr Q_t G-\tr Q_t G( I- P_t) G \right)\mathsf{Z}_k^{n-1,n}\bigg] + \E^{\chi}\bigg[ \sum_{t=1}^k\sum_{ij}^{(t)}\frac{1}{N}[ W_t G]_{ji}(\partial_x\mathcal{F}) (\partial_{t,ji} \Tr P_t G)\mathsf{Z}_k^{n-1,n}\bigg] \notag\\
	&+ \E^{\chi}\bigg[ \sum_{t=1}^k\sum_{ij}^{(t)}\frac{1}{N}[ W_t G]_{ji}(\partial_y\mathcal{F}) (\partial_{t,ji} \Tr G)\mathsf{Z}_k^{n-1,n}\bigg]+\E^{\chi} \bigg[O_{\prec}(1)\mathsf{Z}_k^{n-2,n} \bigg] + \E ^{\chi}\bigg[O_{\prec}(1)\mathsf{Z}_k^{n-1,n-1} \bigg] + O_{\prec}(N^{-D}) \notag\\
	=&\E^{\chi} \bigg[O_{\prec}(1)\mathsf{Z}_k^{n-1,n} \bigg] +\E ^{\chi}\bigg[O_{\prec}(1)\mathsf{Z}_k^{n-2,n} \bigg] + \E^{\chi} \bigg[O_{\prec}(1)\mathsf{Z}_k^{n-1,n-1} \bigg], \label{062101}
\end{align} 
where in the first step we replaced $Z_k\cdot \Xi$ by $Z_k$ (the error is negligible by Lemma \ref{moreXi}) and absorbed the terms containing the derivatives of $\Xi$ into the error (c.f. Lemma \ref{XiLemma}). In the last step we also used (\ref{8003}),  the boundedness of the partial derivatives of $\mathcal{F}(x,y)$, and
\begin{align*}
	\sum_{ij}^{(t)}[W_tG]_{ji}(\partial_{t,ji}\Tr P_tG) = O_{\prec}(p_t),\quad \sum_{ij}^{(t)}[W_tG]_{ji}(\partial_{t,ji}\Tr G) = O_{\prec}(p_t).
\end{align*}
Hence, (\ref{062101}) gives the estimate for the first two terms in the RHS of (\ref{062102}).
%
Then we turn to the third term in the RHS of (\ref{062102}). Since the complex conjugate takes no effect in the remaining estimations, we drop the complex conjugate for simplicity. Specifically,  we will work with $Z_k^{2n-1}$ instead of $\mathsf{Z}_k^{n-1,n}$ in the case $s\geq 2$, for notational brevity.  Here we also drop the function $\Xi$ first and put it back at the last step, up to negligible error.

For $s>2$, we have 
\begin{align*}
& \frac{1}{N^{\frac{s+1}{2}}}\sum_{t=1}^k\sum_{ij}^{(t)}\partial_{t,ji}^s \{\mathcal{W}_t[ W_t G]_{ji}Z_k^{2n-1}\} \\
&= \frac{1}{N^{\frac{s+1}{2}}}\sum_{r=0}^{2n-1}\sum_{t=1}^k\sum_{ij}^{(t)}\sum_{\substack{s_0\geq 0,s_1,\cdots,s_r >0\\ s_0+s_1+\cdots+s_r=s}}(\partial_{t,ji}^{s_0}\mathcal{W}_t[ W_t G ]_{ji})\partial_{t,ji}^{s_1} Z_k \cdots\partial_{t,ji}^{s_r}Z_k Z_k^{2n-1-r}\\
&= \sum_{r=0}^{2n-1}\sum_{t=1}^k\sum_{ij}^{(t)}O_{\prec}\left(\frac{1}{N^{2}}\right) Z_k^{2n-1-r} = \sum_{r=0}^{2n-1}O_{\prec}(1) Z_k^{2n-1-r},
\end{align*}
where the Green function entries such as $G_{ab}$ and $(W_tG)_{ab}$ are all bounded by $O_\prec(1)$ crudely. 
For $s=2$, we have  
\begin{align*}
&\frac{1}{N^{\frac{3}{2}}}\sum_{t=1}^k\sum_{ij}^{(t)}\partial_{t,ji}^2 \{\mathcal{W}_t[ W_t G]_{ji}Z_k^{2n-1}\} \\
&=\frac{1}{N^{\frac{3}{2}}}\sum_{t=1}^k\sum_{ij}^{(t)}\Big( \mathcal{W}_t(\partial_{t,ji}^2 [ W_t G]_{ji})  +2\left(\partial_{t,ji} \mathcal{W}_t\right) \left(\partial_{t,ji} [ W_t G]_{ji}\right) +  \left(\partial^2_{t,ji} \mathcal{W}_t\right)  [ W_t G]_{ji}\Big)     Z_k^{2n-1} \\
&+ \frac{1}{N^{\frac{3}{2}}}\sum_{t=1}^k\sum_{ij}^{(t)}\Big(\mathcal{W}_t[ W_t G]_{ji}\left(\partial_{t,ji}^2 Z_k\right)+2\mathcal{W}_t\left(\partial_{t,ji} [ W_t G]_{ji}\right)\left(\partial_{t,ji} Z_k\right) +2\left(\partial_{t,ji} \mathcal{W}_t\right)  [ W_t G]_{ji}\left(\partial_{t,ji} Z_k\right)   \Big) Z_k^{2n-2}\\
&+ \frac{2}{N^{\frac{3}{2}}}\sum_{t=1}^k\sum_{ij}^{(t)}  \mathcal{W}_t [ W_t G]_{ji}\left(\partial_{t,ji} Z_k\right)^2 Z_k^{2n-3} =  O_{\prec}(1) Z_k^{2n-1}+O_{\prec}(1) Z_k^{2n-2}+O_{\prec}(1) Z_k^{2n-3} ,
\end{align*}
where we used the fact that $\partial_{t,ji} \mathcal{W}_t , \partial_{t,ji}^2[ W_t G]_{ji}$ and $\partial_{t,ji} Z_k$ all generate at least one off-diagonal entry, and then we can use the Cauchy-Schwarz inequality to bound the summation over $i,j$ of these off-diagonal entries by tracial quantities.

Combining the above cases,  and substituting $Z_k\cdot \Xi$ for $Z_k$, we can get
\begin{align*}
	\E\big[ \mathsf{Z}_k^{n,n} \big] \le \sum_{r=0}^{2n-1}\E \big[ O_{\prec}(1)|Z_k\cdot \Xi|^{2n-1-r} \big] + O_{\prec}(N^{-D}).
\end{align*}
for any large $D > 0$.
 Applying Young's inequality as we did in the estimate of $\gamma_{t1}$, we can get
\begin{align}\label{strPGMoment}
	\E\big[ \mathsf{Z}_k^{n,n} \big]  = O_{\prec} \left( 1 \right).
\end{align}
Using Markov's inequality, we obtain the estimate of $\sum_{t=1}^k\mathcal{W}_t\gamma_{t1}$ from \eqref{strPGMoment}.

If we replace the definition of $Z_k$ by 
\begin{align*}
	Z_k = \sum_{t=1}^k\mathcal{W}_t\frac{1}{1-y_t}\tr P_t- \sum_{t=1}^k\mathcal{W}_t\tr( X_t X_t')^{-1},
\end{align*}
and apply the cumulant expansion again, similarly,  we can obtain the estimate of $\sum_{t=1}^k\mathcal{W}_t\gamma_{t2}$. 
Also, if we replace the definition of $Z_k$ by 
\begin{align*}
	Z_k = \sum_{t=1}^k\mathcal{W}_t\frac{1}{1-y_t}\tr P_tGP_t- \sum_{t=1}^k\mathcal{W}_t\tr Q_tG,
\end{align*}
and apply the cumulant expansion again, similarly,  we can obtain the estimate of $\sum_{t=1}^k\mathcal{W}_t\gamma_{t3}$. Here we used the fact that $\Tr P_tGP_t=\Tr P_t^2G=\Tr P_tG$. The details are omitted.  This completes the proof of Lemma 5.4.

\subsection{Proof of Lemma \ref{Flbound}}
Recall
	\begin{align*}
		F_{\mu_t}(\omega_t) = \omega_t -y_t + \frac{y_t(1-y_t)}{1-y_t-\omega_t}
	\end{align*}
	and also the contour defined in (\ref{062033}). 
	Taking derivatives with respect to $\omega_t$, we have
	\begin{align*}
		F'_{\mu_t}(\omega_t) = 1 +  \frac{y_t(1-y_t)}{(1-y_t-\omega_t)^2},\qquad F^{(n)}_{\mu_t}(\omega_t) = \frac{y_t(1-y_t)n!}{(1-y_t-\omega_t)^{n+1}}, \quad n \ge 2.
	\end{align*}
	For $z \in (\bar{\gamma}_1^0)^+\cup (\bar{\gamma}_2^0)^+$ with $\hat{y} \in (0,1)$, since $\mu_{\boxplus}$ has point mass at the origin,  we know that when $|z| \to 0$, by the definition of the Stieltjes transform, $|m_\boxplus(z)| \to \infty$. Recall that 
	\begin{align}
		m_\boxplus(z) = \frac{y_t}{1-\omega_t(z)} - \frac{1-y_t}{\omega_t(z)}, \quad t \in [\![k]\!], \label{062040}
	\end{align}
	we have $|\omega_t(z)| \to 0$ as $|z| \to 0$, which can be checked starting from $z=-\infty$ and applying monotonicity. More specifically,  when $z\to-\infty$ along the real line, we have $m_\boxplus(z)\to 0+$ as well.  By (\ref{062040}), in this case, we have either $\omega_t(z)\to -\infty$ or $\omega_t(z)\to 1-y_i$. We first show that the latter is impossible. This can be checked by taking $z=-E+\mathrm{i}\eta$ with increasing $\eta$ and sufficiently large $E$. When $E$ is sufficiently large, $m_\boxplus(z)$ is sufficiently small. The smallness of $m_\boxplus(z)$ will not  change when one increase $\eta$. However, if $\omega_t(z)$ is close to $1-y_t$ when $z=-E$ for large $E$, by increasing $\eta$, it is impossible for the RHS of (\ref{062040}) keep being small due to $\Im \omega_t(z)\gtrsim  \eta$.  This shows $\omega_t(z)\to -\infty$ if $z\to- \infty$. Then starting from $-\infty$, we increase $z$ to $0-$. Due to monotonicity, $\omega_t(z) \to 0-$ when $z\to 0-$ by (\ref{062040}) and the fact that $\mu_\boxplus$ has a point mass at $0$. Then we shall further argue that when $z\to 0+$, $\omega_t(z)\to 0+$ as well. This can be done by showing that when $z$ goes from $-\varepsilon$ to $\varepsilon$ via the arc $|z|=\varepsilon$ clockwise, it is impossible for $\omega_t(z)$ to go from $-\tilde{\varepsilon}$ to somewhere close to $1$. Since on the arc $|z|=\varepsilon$, $m_\boxplus(z)$ will always be large in magnitude. But it is possible for the RHS of  (\ref{062040}) keep being large when  $\omega_t(z)$ go from $-\tilde{\varepsilon}$ to somewhere close to $1$. Therefore, for $z\in \mathcal{C}_1(\epsilon_{11},\epsilon_{21})\cup\mathcal{C}_1(\epsilon_{12},\epsilon_{22})$, by choosing $\epsilon_{11}$ and $\epsilon_{12}$ sufficiently small, and then performing Taylor expansion around $\omega_t = 0$, we have
	\begin{align*}
		&|\omega_t| \sim 1, \qquad  F_{\mu_t}(\omega_t) = \frac{\omega_t}{1-y_t} + O(|\omega_t|^2) \sim 1, \\
		&F'_{\mu_t}(\omega_t)=1 + \frac{y_t}{1-y_t} + O(|\omega_t|) =  \frac{1}{1-y_t} + O(|\omega_t|) \sim 1,\\
		&F^{(n)}_{\mu_t}(\omega_t) = \frac{y_tn!}{(1-y_t)^{n-1}} + O(\left|y_t\omega_t \right| ) \sim y_t, \quad n \geq  2,
		\end{align*}
		and 
		\begin{align*}
		&1 - \frac{1}{k-1}\sum_{t=1}^{k}\frac{1}{F'_{\mu_t}(\omega_t)} =1 - \frac{1}{k-1}\sum_{t=1}^{k}\frac{1}{F'_{\mu_t}(0) + O\left(|F''_{\mu_t}(0)\omega_t|\right)}\\
		&= 1 - \frac{1}{k-1}\sum_{t=1}^{k}\frac{1}{F'_{\mu_t}(0) } + \frac{1}{k-1}\sum_{t=1}^{k}O\left(y_t|\omega_t|\right) = \frac{y-1 + O\left(|\omega_t|\right) }{k-1}  \sim \frac{1}{k}.
	\end{align*}
	For $z \in \mathcal{C}_2(\epsilon_{11},\epsilon_{21},M_{11})$ with $\hat{y} \in (0,1)$, since $\mu_{\boxplus}$ is supported on the nonnegative half of  real line and $z$ lies outside the support of $\mu_{\boxplus}$ (c.f. Lemma \ref{supportofmu}), we have $\Im \omega_t = 0$ when $\Im z = 0$. By the continuity of $\omega_t$, if we choose $\epsilon_{21}$ to be sufficiently close to $\epsilon_{11}$ so that $\Im z=\sqrt{\epsilon_{11}^2-\epsilon_{21}^2}$ is sufficiently small on $\mathcal{C}_2(\epsilon_{11},\epsilon_{21},M_{11})$, we also have  that $\Im \omega_t$ is sufficiently  small. According to previous discussion, if we choose $z=x+\mathrm{i} 0+$ for $x\in (-\infty,0)$, $m_\boxplus(z)=m_{\mu_t}(\omega_t(z))$ is monotonic in $x$, and thus $\omega_t(z)=\Re \omega_t(z)+\mathrm{i}0+$ and $\Re \omega_t(z)$ goes from $-\infty$ to $0-$ when $x$ goes from $-\infty$ to $0-$. Hence $|1-y_t-\Re\omega_t(x+\mathrm{i}0+)|\geq 1-y_t$.  Then, by continuity (c.f. (\ref{080711})), if we choose $\epsilon_{21}$ to be sufficiently close to $\epsilon_{11}$ so that $\Im z=\sqrt{\epsilon_{11}^2-\epsilon_{21}^2}$ is sufficiently small, we also have $|1-y_t\Re\omega_t(z)|\geq \frac{1}{2}(1-y_t)$ when $z\in \mathcal{C}_2$.  
	
Since $\Im z > 0$ and $|z|$ is bounded, by the continuity of $\omega_t(z)$ and definition of Stieltjes transform, we have
\begin{align*}
	|\omega_t(z)| \sim 1,\qquad |F_{\mu_t}(\omega_t)| = -\frac{1}{|m_{\boxplus}(z)|} \sim 1.
\end{align*}
Then we can perform Taylor expansion around $\Im \omega_t$, 
	\begin{align*}
		&F'_{\mu_t}(\omega_t)=1 + \frac{y_t(1-y_t)}{(1-y_t-\Re \omega_t)^2} + O(|\Im \omega_t|)   \sim 1,\\
		&F^{(n)}_{\mu_t}(\omega_t) =  \frac{y_tn!}{(1-y_t-\Re \omega_t)^{n-1}} + O(\left|\Im\omega_t \right| ) \sim y_t, \quad n \ge 2.
	\end{align*}	
	Further,  we lower bound the last quantity in Lemma \ref{Flbound} by its imaginary part, which reads as
	\begin{align*}
		&\bigg|1 - \frac{1}{(k-1)}\sum_{t=1}^{k}\frac{1}{F'_{\mu_t}(\omega_t)}\bigg| \ge \bigg|  \frac{1}{(k-1)}\sum_{t=1}^{k}\frac{\Im F'_{\mu_t}(\omega_t)}{|F'_{\mu_t}(\omega_t)|^2} \bigg|\\
		&\gtrsim\bigg|\frac{1}{(k-1)}\sum_{t=1}^{k} \frac{2y_t(1-y_t)(1-y_t-\Re(\omega_t))\Im(\omega_t)}{|1-y_t-\omega_t|^4} \bigg| \gtrsim \frac{1}{k},
	\end{align*}
	where we used the bound $|1-y_t-\Re\omega_t(z)|\geq \frac{1}{2}(1-y_t)$ when $z\in \mathcal{C}_2(\epsilon_{11},\epsilon_{21},M_{11})$. The case of $z \in \mathcal{C}_2(\epsilon_{12},\epsilon_{22},M_{12})$ with $\hat{y} \in (0,1)$ is similar.

	For $z \in (\bar{\gamma}_1)^+\cup (\bar{\gamma}_2)^+$ with $\hat{y} \in (0, \infty)$, by choosing $M_{11}, M_{21}, M_{12}$ and $M_{22}$ sufficiently large, we have $|\omega_t|$ sufficiently large. Therefore,
	\begin{align*}
		&|\omega_t| \sim 1,\qquad |F_{\mu_t}(\omega_t)| = |\omega_t - y_t + O(|\omega_t|^{-1})| \sim |\omega_t| \sim 1, \\
		&|F'_{\mu_t}(\omega_t)|=\left|1 +  O(y_t|\omega_t|^{-1}) \right|  \sim 1,\qquad |F^{(n)}_{\mu_t}(\omega_t)| =  O(y_t\left|\omega_t \right|^{-n} ) \sim y_t, \quad n \ge 2.\\
		&\bigg|1 - \frac{1}{k-1}\sum_{t=1}^{k}\frac{1}{F'_{\mu_t}(\omega_t)}\bigg| = \bigg| 1 +  \frac{1}{k-1}\sum_{t=1}^{k}\frac{1}{1 + O(y_t|\omega_t|^{-1})} \bigg|=\bigg| \frac{-1 + O(|\omega_t|^{-1})}{k-1} \bigg| \gtrsim \frac{1}{k}.
	\end{align*}
	We finish the proof by combining the above results.

\section{Proofs of results in Section \ref{ProofofCLT}}\label{Sec C}
\subsection{Proof of Proposition \ref{keyprop2}}
For any $t \in [\![ k ]\!]$, by the definition of $\omega_t^c(z)$, we have
$
	\tr P_t G(z) = 1 + \omega_t^c(z) m_N(z) 
$
which together with Proposition \ref{keyProp} gives the a priori bound
\begin{align}\label{trPGGap1}
	&\left| \tr P_t G(z) -  \left(1 + \omega_t(z)m_{{\boxplus}}(z)\right) \right| \prec  \left| \omega_t^c(z) - \omega_t(z) \right| + \left| m_N(z) - m_{{\boxplus}}(z) \right| 
	= O_{\prec}\left(\frac{1}{N} \right).
\end{align}
Notice that this error bound  can be improved as follows
\begin{align*}
	&\tr P_tG(z) =\left( \frac{y_t}{1-y_t} - \frac{1}{1-y_t}\tr P_tG(z) \right)\left(\tr\bbG(z)-\tr\bbP_t\bbG(z)\right)+ O_{\prec} \left( \sqrt{\frac{p_t}{N^3}} \right) \\
	&=\left( \frac{y_t}{1-y_t} - \frac{1}{1-y_t}\tr P_tG(z) \right)\left(m_{\boxplus}(z) -1 - \omega_t(z)m_{\boxplus}(z)\right) + O_{\prec} \left( \sqrt{\frac{p_t}{N^3}} \right),
\end{align*}
where we used the estimates of $\gamma_{t2}$ and $\gamma_{t3}$ (c.f., Lemma \ref{Lemma error estimates}) in the first step, and Proposition \ref{keyProp} and \eqref{trPGGap1} in the second step.
Solving the above equation for $\tr P_tG(z)$, we can obtain,
\begin{align*}
	\tr P_tG(z) = 1 + \omega_t(z)m_{{\boxplus}}(z)  + O_{\prec} \left( \sqrt{\frac{p_t}{N^3}} \right).
\end{align*}
This proves the first estimate of (\ref{080411}).
Starting from the estimate of $\gamma_{t3}$, for any $t \in [\![ k ]\!]$, we have
\begin{align*}
	\tr Q_tG(z) =&\frac{1}{1-y_t}\tr P_t G(z) + O_{\prec} \left( \sqrt{\frac{p_t}{N^3}} \right) = \frac{1+ \omega_t(z)m_{\boxplus}(z)}{1-y_t} + O_{\prec} \left( \sqrt{\frac{p_t}{N^3}} \right).
\end{align*}

For the estimates of two-points functions of Green functions, we first define
	\begin{align*}
		Z_1 =& \tr\bbP_t\bbG(z_1)\bbP_t\bbG(z_2) -\left(\tr(\bbX_t\bbX_t')^{-1}-\tr \bbQ_t\bbG(z_1) \right)\left( \tr\bbG(z_1)\bbP_t\bbG(z_2)- \tr\bbP_t\bbG(z_1)\bbP_t\bbG(z_2) \right)\\
		&+\left( \tr \bbQ_t\bbG(z_1)\bbP_t\bbG(z_2) - \tr \bbQ_t\bbG(z_1) \right)\left( \tr \bbG(z_2) - \tr\bbP_t\bbG(z_2) \right),\\
			Z_2 =& \tr P_tG(z_1)P_tG(z_2)P_t - \frac{1-y_t}{N}\Tr Q_tG(z_1)P_tG(z_2),
	\end{align*}
	and then follow the same argument as we did in the estimate of $\gamma_{t1}$,  one can get
	\begin{align}
		Z_1, Z_2 = O_{\prec}\left(\sqrt{\frac{p_t}{N^3}}\right). \label{trPGPGZ1}
	\end{align}
	The details are omitted.
	Therefore,
	\begin{align}
			\tr\bbP_t\bbG(z_1)\bbP_t\bbG(z_2) =&\left(\tr(\bbX_t\bbX_t')^{-1}-\tr \bbQ_t\bbG(z_1) \right)\left( \tr\bbG(z_1)\bbP_t\bbG(z_2)- \tr\bbP_t\bbG(z_1)\bbP_t\bbG(z_2) \right)\notag\\
		&-\left( \tr \bbQ_t\bbG(z_1)\bbP_t\bbG(z_2) - \tr \bbQ_t\bbG(z_1) \right)\left( \tr \bbG(z_2) - \tr\bbP_t\bbG(z_2) \right) + O_{\prec}\left(\sqrt{\frac{p_t}{N^3}}\right)\notag \\
		=&\left(\tr(\bbX_t\bbX_t')^{-1}-\tr \bbQ_t\bbG(z_1) \right)\left(\frac{ \tr\bbP_t\bbG(z_1) - \tr\bbP_t\bbG(z_2)}{z_1-z_2}- \tr\bbP_t\bbG(z_1)\bbP_t\bbG(z_2) \right)\notag\\
		&-\left( \tr \bbQ_t\bbG(z_1)\bbP_t\bbG(z_2) - \tr \bbQ_t\bbG(z_1) \right)\left( \tr \bbG(z_2) - \tr\bbP_t\bbG(z_2) \right) + O_{\prec}\left(\sqrt{\frac{p_t}{N^3}}\right). \label{5001}
	\end{align}
	Interchanging $z_1$ and $z_2$ in the above equation, we have
	\begin{align}
		\tr\bbP_t\bbG(z_2)\bbP_t\bbG(z_1) =&\left(\tr(\bbX_t\bbX_t')^{-1}-\tr \bbQ_t\bbG(z_2) \right)\left(\frac{ \tr\bbP_t\bbG(z_2) - \tr\bbP_t\bbG(z_1)}{z_2-z_1}- \tr\bbP_t\bbG(z_2)\bbP_t\bbG(z_1) \right)\notag\\
		&-\left( \tr \bbQ_t\bbG(z_2)\bbP_t\bbG(z_1) - \tr \bbQ_t\bbG(z_2) \right)\left( \tr \bbG(z_1) - \tr\bbP_t\bbG(z_1) \right) + O_{\prec}\left(\sqrt{\frac{p_t}{N^3}}\right). \label{5002}
	\end{align}
	Replacing $\tr Q_tG(z_1)P_tG(z_2)$ and $\tr Q_tG(z_2)P_tG(z_1)$ by $\tr P_tG(z_1)P_tG(z_2)$ and $\tr P_tG(z_2)P_tG(z_1)$ respectively, and then combining (\ref{5001}) and (\ref{5002}), we get
	\begin{align*}
	\tr P_tG(z_1)P_tG(z_2)&=\frac{(\tr P_tG(z_1)-\tr P_tG(z_2))^2}{(z_1-z_2)(\tr G(z_1)-\tr G(z_2))}\\
	&+\frac{\tr P_tG(z_2)\tr G(z_1)-\tr P_tG(z_1)\tr G(z_2)}{\tr G(z_1)-\tr G(z_2)}+O_{\prec}\left(\sqrt{\frac{p_t}{N^3}}\right).
	\end{align*}
By the choice of $\gamma^0_1$ and $\gamma^0_2$, together with the estimates $\tr G(z_1) - \tr G(z_2) = m_{\boxplus}(z_1) - m_{\boxplus}(z_2) + O_{\prec}(N^{-1})$, we see that $\tr G(z_1) \neq \tr G(z_2)$ with high probability. Further using the estimates of the tracial quantities we already obtained, we get the estimate of $\tr P_tG(z_1)P_tG(z_2)$. 
Finally, using the estimate of $\tr P_tG(z_1)P_tG(z_2)$ together with the estimation of $Z_2$, we can obtain the estimate of $\tr Q_tG(z_1)P_tG(z_2)$. This completes the proof of (\ref{0807120}).
	
	Since for any function $F$ which is independent of $z$, we have 
$
		\tr F\bbG(z)^2 = \partial_z \left(\tr FG(z) \right),
$	
	together with \eqref{trPGPGZ1}, we have
	\begin{align*}
		\tr\bbP_t\bbG(z)\bbP_t\bbG(z) =& \left(\tr(\bbX_t\bbX_t')^{-1}-\tr \bbQ_t\bbG(z) \right)\left( \tr\bbG(z)\bbP_t\bbG(z)- \tr\bbP_t\bbG(z)\bbP_t\bbG(z) \right)\notag\\
		&-\left( \tr \bbQ_t\bbG(z)\bbP_t\bbG(z) - \tr \bbQ_t\bbG(z) \right)\left( \tr \bbG(z) - \tr\bbP_t\bbG(z) \right) + O_{\prec}\left(\sqrt{\frac{p_t}{N^3}}\right) \notag\\	
		=& \left(\tr(\bbX_t\bbX_t')^{-1}-\tr \bbQ_t\bbG(z) \right)\left( \partial_z(\tr\bbP_t\bbG(z)) -  \tr\bbP_t\bbG(z)\bbP_t\bbG(z) \right)\notag\\
		&-\left( \tr \bbQ_t\bbG(z)\bbP_t\bbG(z) - \tr \bbQ_t\bbG(z) \right)\left( \tr \bbG(z) - \tr\bbP_t\bbG(z) \right) + O_{\prec}\left(\sqrt{\frac{p_t}{N^3}}\right).
	\end{align*}
	Using the estimation of $Z_2$ with $z_1=z_2=z$, and then by the deterministic approximations (c.f. (\ref{080411}) and Proposition \ref{keyProp}), we can obtain \eqref{080412}.

\subsection{Proof of Proposition \ref{DiagApproByFC}}\label{Sec Proof of Proposition 6.2}
We should first have some concentration results of the diagonal entries, which are stated in the following lemma.
\begin{lemma}\label{LemPGii}
		For any $t \in [\![k]\!]$, if $\hat{y} \in (0,1)$, for any $z \in (\bar{\gamma}^0_1)^{+}\cup(\bar{\gamma}^0_2)^{+}$, we have
	\begin{align}
		&[P_t]_{ii} - y_t = O_{\prec}\left(\frac{1}{\sqrt{N}}\right), \label{Pii} \\
		&[ P_t G]_{ii}-\left(\tr( X_t X_t')^{-1}-\tr Q_t G\right)( G_{ii}-[ P_t G]_{ii})=O_{\prec}\left(\frac{1}{\sqrt{N}}\right), \label{PGii} \\
		&[\bbP_t\bbG\bbP_t]_{ii}-\big(\tr(\bbX_t\bbX_t')^{-1}-\tr\bbQ_t\bbG\big)\big([\bbG\bbP_t]_{ii}-[\bbP_t\bbG\bbP_t]_{ii}\big)
		-\left(1-[\bbP_t]_{ii}\right)\tr\bbQ_t\bbG=O_{\prec}\left(\frac{1}{\sqrt{N}}\right)\label{PGPii},\\
		&1 - (1-y_t)[(X_tX_t')^{-1}]_{jj} = O_{\prec}\left(\frac{1}{\sqrt{N}}\right), \label{XXjj} \\
		&(1-y_t)[\bbW_t\bbG\bbW_t']_{jj}-\left(\tr\bbG-\tr\bbP_t\bbG\right)\left([\left(\bbX_t\bbX_t'\right)^{-1}]_{jj}-[\bbW_t\bbG\bbW'_t]_{jj}\right) =O_{\prec}\left(\frac{1}{\sqrt{N}}\right)\label{WGWjj}.
	\end{align}
	The same estimates hold when $z \in (\bar{\gamma}_1)^{+}\cup(\bar{\gamma}_2)^{+}$ with $\hat{y} \in (0,\infty)$.
\end{lemma}
However, if we directly apply the above lemma to obtain the concentration results of $[G]_{ii}$ and $[PG]_{ii}$, the error term will blow up for large $k$. Therefore,  we have to further explore  the fluctuation averaging effect due to the  summation over $k$. The resulting concentration behaviour can be described by the following lemma.
\begin{lemma}\label{LemSumDiag}
	Let $\mathcal{W}_t, t \in [\![k]\!]$ be a collection of deterministic complex numbers. Suppose that   $\sup_t| \mathcal{W}_t| <  C$ for some constant $C$. If $\hat{y} \in (0,1)$, for any $z \in (\bar{\gamma}^0_1)^{+}\cup(\bar{\gamma}^0_2)^{+}$, we have
	\begin{align}
		\sum_{t=1}^k \mathcal{W}_t [ P_t G]_{ii}-\sum_{t=1}^k \mathcal{W}_t \left(\tr( X_t X_t')^{-1}-\tr Q_t G\right)( G_{ii}-[ P_t G]_{ii}) =O_{\prec}\left(\frac{1}{\sqrt{N}}\right). \label{sPGii}	
	\end{align}
	The same estimates hold when $z \in (\bar{\gamma}_1)^{+}\cup(\bar{\gamma}_2)^{+}$ with $\hat{y} \in (0,\infty)$.
\end{lemma}
The proof of Lemma \ref{LemPGii} and Lemma \ref{LemSumDiag} will be postponed to Section \ref{Sec Proof of Lemma C}.
With the aid of the  two lemmas above, we can prove Proposition \ref{DiagApproByFC}. Firstly, (\ref{XXjj}) directly implies the first estimate of (\ref{QGiiGap}).
	For any $z \in (\bar{\gamma}_1)^{+}\cup(\bar{\gamma}_2)^{+}$ with $\hat{y} \in (0,\infty)$ or $z \in (\bar{\gamma}_1)^{+}\cup(\bar{\gamma}_2)^{+}$ with $\hat{y} \in (0,\infty)$, taking 
 $
 \mathcal{W}_t = -(1-y_t)/(\omega_t(z) m_{\boxplus}(z))
 $
 in \eqref{sPGii} and rearranging terms, we have
 \begin{align}\label{sPGii1}
 	\sum_{t=1}^k & \frac{1 + \tr( X_t X_t')^{-1}-\tr Q_t G}{ -\omega_t(z) m_{\boxplus}(z)(1-y_t)^{-1} } [ P_t G]_{ii} -\sum_{t=1}^k \frac{\tr( X_t X_t')^{-1}-\tr Q_t G}{ -\omega_t(z) m_{\boxplus}(z)(1-y_t)^{-1} }[G]_{ii}  = O_{\prec}\left(\frac{1}{\sqrt{N}}\right). 
 \end{align}
 From (\ref{PGii}), it is easy to check 
 \begin{align}
 |[PG]_{ii}|\prec \frac{p_t}{N}+\frac{1}{\sqrt{N}}. \label{062030}
 \end{align} 
For the first term in the LHS of \eqref{sPGii1},  using (\ref{062030}), Lemma \ref{Lemma error estimates} and  Proposition \ref{keyprop2}, one can check 
 \begin{align*}
 	\sum_{t=1}^k \frac{1 + \tr( X_t X_t')^{-1}-\tr Q_t G}{ -\omega_t(z) m_{\boxplus}(z)(1-y_t)^{-1} } [ P_t G]_{ii} = &\sum_{t=1}^k [ P_t G]_{ii}  + O_{\prec}\left( \frac{1}{\sqrt{N}}\right) = 1 + z[G]_{ii} + O_{\prec}\left( \frac{1}{\sqrt{N}}\right) .
 \end{align*} 

For the second term in the LHS of \eqref{sPGii1}, similarly, we have 
\begin{align*}
	\sum_{t=1}^k \frac{\tr( X_t X_t')^{-1}-\tr Q_t G}{-\omega_t(z) m_{\boxplus}(z)(1-y_t)^{-1}}[G]_{ii}  
	= \sum_{t=1}^k \frac{1+\omega_t(z) m_{\boxplus}(z)-y_t}{\omega_t(z) m_{\boxplus}(z)} [G]_{ii} + O_{\prec}\left( \frac{1}{\sqrt{N}} \right).
\end{align*}

Combining the above estimations, \eqref{sPGii1} boils down to
\begin{align*}
	1 + z[G]_{ii} = \sum_{t=1}^k \frac{1+\omega_t(z) m_{\boxplus}(z)-y_t}{\omega_t(z) m_{\boxplus}(z)} [G]_{ii} + O_{\prec}\left( \frac{1}{\sqrt{N}} \right), 
	\end{align*}
	which further gives
	\begin{align}\label{GiiGap}
	[G]_{ii} = \left(z - \sum_{t=1}^k \frac{1+\omega_t(z) m_{\boxplus}(z)-y_t}{\omega_t(z) m_{\boxplus}(z)}  \right)^{-1} + O_{\prec}\left( \frac{1}{\sqrt{N}} \right) = m_{\boxplus}(z) + O_{\prec}\left( \frac{1}{\sqrt{N}} \right).
\end{align}
Therefore, by the deterministic approximation of the tracial quantities (c.f., Proposition \ref{keyprop2} and the estimate of $\gamma_{t2}$ in Lemma \ref{Lemma error estimates}) and (\ref{PGii}), we have
\begin{align}\label{PGiiGap1}
	[P_tG]_{ii} &= \frac{\tr( X_t X_t')^{-1} -\tr Q_tG  }{1 + \tr( X_t X_t')^{-1} -\tr Q_tG}[G]_{ii} +  O_{\prec}\left( \frac{1}{\sqrt{N}} \right)\notag\\
	&= \frac{1 + \omega_t(z) m_{\boxplus}(z) - y_t}{\omega_t(z) m_{\boxplus}(z)} m_{\boxplus}(z) + O_{\prec}\left( \frac{1}{\sqrt{N}} \right) 
	= 1 + \omega_t(z) m_{\boxplus}(z) + O_{\prec}\left( \frac{1}{\sqrt{N}} \right).
\end{align}
By applying similar arguments, we can also prove the second estimate of \eqref{PGPiiGap} and \eqref{QGiiGap}. The details are omitted. Hence, we complete the proof of Proposition \ref{DiagApproByFC}.

\subsection{Proofs of Lemma \ref{Error in CLT}}\label{Sec Proofs of Lemma 6.4}
Before we prove Lemma \ref{Error in CLT}, we give some preliminary lemmas, whose proofs  will be postponed to Section \ref{Sec Proof of Lemma C}.

The following lemma provide some estimates on some $W_t$ related quantities.
\begin{lemma}\label{sumij}
Under the assumption of Theorem \ref{FreeCLT}, if $\hat{y} \in (0,1)$, for any fixed $z \in (\bar{\gamma}^0_1)^{+}\cup(\bar{\gamma}^0_2)^{+}$,  we have
\begin{align}
&\Big|\frac{1}{\sqrt{N}}\sum_{i=1}^{N}[ W_t]_{ji}\Big| \prec \frac{1}{N^{\frac{1}{6}}},\label{SUMw}\\
	&\Big|\frac{1}{\sqrt{N}}\sum_{i=1}^{N}[ W_t G(z)]_{ji} \Big|\prec \frac{1}{N^{\frac{1}{6}}} , \qquad \Big|\frac{1}{\sqrt{N}}\sum_{i=1}^{N}[ W_t G(z) P_t]_{ji}\Big|\prec \frac{1}{N^{\frac{1}{6}}}. \label{SUMWGP}
\end{align}
The same bounds hold for $z \in (\bar{\gamma}_1)^{+}\cup(\bar{\gamma}_2)^{+}$ with $\hat{y} \in (0,\infty)$. 
\end{lemma}
We remark here that similarly to (\ref{06281152}), the estimates (\ref{SUMw}) and (\ref{SUMWGP}) still hold if we take derivatives w.r.t to $z$ for the quantities in the LHS.

For brevity, in the sequel, for any random variable $\xi$, we use the notation
\begin{align*}
\la\xi\ra:=\xi-\mathbb{E}\xi.
\end{align*}

\begin{lemma}\label{EtrXXtrQG}
Under the assumption of Theorem \ref{FreeCLT}, if $\hat{y} \in (0,1)$, for any fixed $z \in (\bar{\gamma}^0_1)^{+}\cup(\bar{\gamma}^0_2)^{+}$,  we have
\begin{align}
		&\E^{\chi} \bigg[  \bla\Tr (X_tX_t')^{-1} \bra {\rm e}^{\mathrm{i}xL_{N}(f)} \bigg] = O_{\prec}\left( \frac{p_t}{N^{\frac76}} \right), \label{0616101} \\
		&\E^{\chi} \bigg[ \bla\Tr Q_tG(z) \bra {\rm e}^{\mathrm{i}xL_{N}(f)} \bigg] =\frac{1}{1-y_t} \E^{\chi} \bigg[ \bla \Tr P_tG(z) \bra {\rm e}^{\mathrm{i}xL_{N}(f)} \bigg] + O_{\prec}\left( \frac{p_t}{N^{\frac76}} \right). \label{0616102} \\
		&\frac{1}{N^2}\sum_{ij}^{(t)} \left|\partial_{t,ji}^2\Tr \bbG(z)\right| = O_{\prec}\left( \frac{p_t}{N^{\frac{3}{2}}} \right). \label{062520}
\end{align}
The same bounds hold for $z \in (\bar{\gamma}_1)^{+}\cup(\bar{\gamma}_2)^{+}$ with $\hat{y} \in (0,\infty)$. 
\end{lemma}

Next, we present a lemma concerning the perturbations of the spectrum of $G, (X_tX_t')^{-1}$ and $X_tX_t'$, which will be useful for the estimation of the remainder term in \eqref{remainder}.
\begin{lemma}\label{remainderLemma}
	Under the assumption of Theorem \ref{FreeCLT}, let $X_t^{(ij)} :=  X_t^{(ij)} (x)$  be the matrix with the single entry $X_{t,ij}$ replaced by a generic argument $x$ in $X_t$. We also define $P_t^{(ij)} := P_t^{(ij)}(x)$, $H^{(t,ij)} := H^{(t,ij)}(x)$ and  $G^{(t,ij)} := G^{(t,ij)}(x)$ as the analogue of $P_t, H$ and $G$ where a single entry $X_{t,ij}$ is replaced by  $x$. 
	Then we have, for any  fixed (but small) $\epsilon > 0$, we have the following estimates uniformly in $|x| \le N^{-\frac12+\epsilon}$, 
	\begin{align}
		& \Big\|X_t^{(ij)}\big( X_t^{(ij)} \big)'\Big\| \prec 1,\quad   \Big\|\Big(X_t^{(ij)}\big( X_t^{(ij)} \big)'\Big)^{-1}\Big\|\prec1,\quad \big\|G^{(t,ij)}\big\|\prec 1. \label{4113}
	\end{align}
\end{lemma}

With the help of the above Lemmas, we can prove Lemma 6.4 as follows.
\begin{proof}[Proof of Lemma \ref{Error in CLT}]
	We start with the estimate of  the error terms $\mathsf{E}_1$, $\mathsf{E}_2$, and $\mathsf{E}_3$. 
For $\mathsf{E}_1$, using Lemma \ref{XiLemma} and the fact that $\partial_{t,ji}^{s_0}\left\{[W_tG(z_1)]_{ji} \bla {\rm e}^{\mathrm{i}x\la L_N^1(f)\ra }\bra \right\} \partial_{t,ji}^{s_1}\Xi$ has deterministic upper bound when $\Im z_1 \ge N^{-K}$, we have $|\mathsf{E}_1| \prec N^{-D}$ for any fixed large $D > 0$. For $\mathsf{E}_2$, due to the existence of the truncation function $\Xi$ and the fact $\Im z$ is greater than $N^{-K}$, we have the deterministic upper bound for $\partial^s_{t,ji}\left\{ [W_tG(z_1)]_{ji} \bla {\rm e}^{\mathrm{i}x\la L_N^1(f)\ra }\bra\right\} \cdot \Xi$. This together with the high probability bound $|\partial^s_{t,ji}\left\{ [W_tG(z_1)]_{ji} \bla {\rm e}^{\mathrm{i}x\la L_N^1(f)\ra }\bra\right\}| \prec 1$ leads to $|\mathsf{E}_2| \prec p_t/N^{3/2}$. Lastly, we estimate $\mathsf{E}_3$. First,  for any fixed (but small) $\epsilon>0$,  by Cauchy Schwarz, we have
\begin{align}
	\mathbb{E}\Big[|X_{t,ji}|^{l+2} \mathbf{1}_{|X_{t,ji}|>N^{-\frac12+\epsilon}}\Big] \le \sqrt{\E \Big[|X_{t,ji}|^{2l+4}\Big] } \sqrt{\mathbb{P}\big(|X_{t,ji}|>N^{-\frac12+\epsilon}\big)} \lesssim N^{-D} \label{062502}
\end{align}  
for any large $D > 0$, in light of Assumption \ref{assum1}. Due to the existence of $\Xi$ and $\Im z \ge N^{-K}$, we have 
\begin{align}
	\sup _{X_{t,ji}\in \mathbb{R}}\left| \partial_{t,ji}^{l+1}\left\{ [ W_t G(z_1)]_{ji}\bla {\rm e}^{\mathrm{i}x\la L_N^1(f)\ra }\bra \Xi \right\} \right| \lesssim N^{O(K)}.\label{062503}
\end{align}
 Combining (\ref{062502}) and  (\ref{062503}), we have the first term in (\ref{062501}) can be bounded by $N^{-D}$ for any large $D > 0$. Here our $D$ may vary from line to line.  For the second term in (\ref{062501}), by Lemma \ref{remainderLemma}, we have
\begin{align*}
	\sup _{|X_{t,ji}| \leqslant N^{-\frac12+\epsilon}}\left| \partial_{t,ji}^{l+1}\left\{ [ W_t G]_{ji}\bla {\rm e}^{\mathrm{i}x\la L_N^1(f)\ra }\bra \Xi\right\} \right| \prec 1
\end{align*}
Therefore, together with Assumption \ref{assum1}, we have
\begin{align}
	\mathbb{E}[|X_{t,ji}|^{l+2} ]\cdot \sup _{|X_{t,ji}| \leqslant N^{-\frac12+\epsilon}}\left| \partial_{t,ji}^{l+1}\left\{ [ W_t G]_{ji}\bla {\rm e}^{\mathrm{i}x\la L_N^1(f)\ra }\bra \Xi\right\} \right| \prec N^{-\frac{l+2}{2}}. \label{062504}
\end{align}
Plugging (\ref{062502}), (\ref{062503}) and (\ref{062504}) into (\ref{062501}), we can get
\begin{align*}
	\mathcal{R}_{l+1}^{t,ji} \prec N^{-\frac{l+2}{2}}.
\end{align*}
Observe that, in addition to the above high probability bound, we also have the crude deterministic bound of $\mathcal{R}_{l+1}^{t,ji}$ due to the existence of $\Xi$ and the fact $\Im z\geq N^{-K}$. Hence, 
choosing $l \ge 4$ , we have $|\mathsf{E}_3| \prec p_t/N^2$.

Therefore, we have
\begin{align*}
	\mathsf{E}_1 + \mathsf{E}_2  + \mathsf{E}_3= O_{\prec}\left( \frac{p_t}{N^{\frac32}} \right).
\end{align*}

In the sequel, we estimate $\mathsf{I}_{ta}$'s one by one. 
We start with $\mathsf{I}_{t1}$. By direct calculation, we have 
\begin{align}
	\mathsf{I}_{t1} 
	=& \E^{\chi}\Big[  \left(\tr(\bbX_t \bbX'_t)^{-1} - \tr \bbQ_t\bbG(z_1)\right)\left( \Tr\bbG(z_1) - \Tr\bbP_t\bbG(z_1) \right)\bla {\rm e}^{\mathrm{i}x\la L_N^1(f)\ra }\bra \Big] \notag\\
	&+\E^{\chi}\Big[  \left(\tr\bbQ_t\bbG(z_1)\bbP_t\bbG(z_1) - \tr \bbQ_t\bbG^2(z_1) - \tr \bbQ_t\bbG(z_1) \right) \bla {\rm e}^{\mathrm{i}x\la L_N^1(f)\ra }\bra\Big] \notag\\
	=&  \E^{\chi}\Big[ \left(\tr(\bbX_t \bbX'_t)^{-1} - \tr \bbQ_t\bbG(z_1) \right)\left( \Tr\bbG(z_1) - \Tr\bbP_t\bbG(z_1) \right) \bla{\rm e}^{\mathrm{i}x\la L_N^1(f)\ra }\bra\Big]+ O_{\prec}\left(\sqrt{\frac{p_t}{N^3}}\right), \label{061301}
\end{align}
where in the last step, we used Proposition \ref{keyprop2}. 
Using the trivial identity $\mathbb{E}[\zeta\la\xi\ra]=\mathbb{E}[\la\zeta\ra\xi]$, the main term in $\mathsf{I}_{t1}$ can be further expanded as
\begin{align}
	&\E^{\chi}\Big[ \left(\tr(\bbX_t \bbX'_t)^{-1} - \tr \bbQ_t\bbG(z_1) \right)\left( \Tr\bbG(z_1) - \Tr\bbP_t\bbG(z_1)\right)  \bla {\rm e}^{\mathrm{i}x\la L_N^1(f)\ra }\bra \Big] \notag\\
	&=\E^{\chi}\Big[ \tr(\bbX_t \bbX'_t)^{-1} - \tr \bbQ_t\bbG(z_1)\Big] \E\Big[\left( \Tr\bbG(z_1) - \Tr\bbP_t\bbG(z_1)  \right) \bla{\rm e}^{\mathrm{i}x\la L_N^1(f)\ra }\bra\Big] \notag\\
	&\qquad+ \E \big[ \Tr \bbG(z_1) - \Tr \bbP_t\bbG(z_1) \big]\E^{\chi} \Big[ \left( \tr(\bbX_t \bbX'_t)^{-1} - \tr \bbQ_t\bbG(z_1)  \right)\bla{\rm e}^{\mathrm{i}x\la L_N^1(f)\ra }\bra\Big] \notag\\
	&\qquad+ \E^{\chi} \Big[ \bla \tr(\bbX_t \bbX'_t)^{-1} - \tr \bbQ_t\bbG(z_1) \bra \bla \Tr\bbG(z_1) - \Tr\bbP_t\bbG(z_1)  \bra \bla{\rm e}^{\mathrm{i}x\la L_N^1(f)\ra }\bra \Big]+O_\prec(N^{-D}) \notag\\
	&=: \mathsf{I}_{t11} +  \mathsf{I}_{t12} +  \mathsf{I}_{t13}+O_\prec(N^{-D}),\label{061302}
\end{align}
for any fixed $D>0$. 
With the notation in (\ref{def of alphabeta}), we have
\begin{align*}
	\mathsf{I}_{t11} &= \alpha_t \E\Big[ \left( \Tr\bbG(z_1) - \Tr\bbP_t\bbG(z_1)  \right) \bla {\rm e}^{\mathrm{i}x\la L_N^1(f)\ra }\bra \Big]\notag\\
	&=\alpha_t\E^{\chi}\Big[ \left( \Tr\bbG(z_1) - \Tr\bbP_t\bbG(z_1)  \right) \bla {\rm e}^{\mathrm{i}x\la L_N^1(f)\ra }\bra \Big]+ O_\prec(N^{-D}).
\end{align*}

For $\mathsf{I}_{t12}$, using Lemma \ref{EtrXXtrQG}, we can similarly write 
\begin{align*}
	 \mathsf{I}_{t12}
	 =&-\beta_t \E^{\chi} \Big[\Tr \bbP_t\bbG(z_1)  \bla{\rm e}^{\mathrm{i}x\la L_N^1(f)\ra }\bra\Big] + O_{\prec}\left( \frac{p_t}{N^{\frac76}} \right).
\end{align*}

For $\mathsf{I}_{t13}$, by Propositions \ref{keyProp} and \ref{keyprop2},
\begin{align*}
	|\mathsf{I}_{t13}| =&\left|\E^{\chi} \Big[ \bla \tr(\bbX_t \bbX'_t)^{-1} - \tr \bbQ_t\bbG(z_1) \bra \bla \Tr\bbG(z_1) - \Tr\bbP_t\bbG(z_1)  \bra\bla {\rm e}^{\mathrm{i}xL_N^1(f)}\bra\Big]\right| \\
	 \le&  \E^{\chi} \Big[\left| \bla  \tr(\bbX_t \bbX'_t)^{-1} - \tr \bbQ_t\bbG(z_1) \bra\right| \left| \bla \Tr\bbG(z_1) - \Tr\bbP_t\bbG (z_1) \bra \right|\Big] \prec \sqrt{\frac{p_t}{N^3}}.
\end{align*}
Therefore, plugging the above estimates together with (\ref{061302}) into (\ref{061301}) arrives at 
\begin{align}
	\mathsf{I}_{t1} 
	=& \alpha_t \E^{\chi}\Big[ \bla  \Tr\bbG(z_1) \bra {\rm e}^{\mathrm{i}x\la L_N^1(f)\ra }\Big] - (\alpha_t + \beta_t)\E^{\chi} \Big[ \bla  \Tr \bbP_t\bbG(z_1)  \bra {\rm e}^{\mathrm{i}x\la L_N^1(f)\ra }\Big]+O_{\prec}\left( \sqrt{\frac{p_t}{N^3}} \right). \label{061350}
\end{align}

Next, we estimate $\mathsf{I}_{t2}$. From (\ref{062810}) we see that the difference $L_N^1(f)-L_N^2(f)$ is $O_\prec(N^{-K+1})$. Further, we also have deterministic upper bound $N^{O(K)}$  for $L_N^1(f)-L_N^2(f)$ from the definition. It is easy to check that these bounds still apply if we take derivative for $L_N^1(f)-L_N^2(f)$ w.r.t. $X_{t,ij}$'s. Hence, we can perform the replacement of $	L_N^1(f)$ by $	L_N^2(f)$ in any part of the integrand. The purpose to use such replacement is to avoid possible singularities of the integrand. Particularly, $\mathsf{I}_{t2}$ can be written as
\begin{align*}
\mathsf{I}_{t2}=&\frac{\mathrm{i}x}{N}\sum_{ij}^{(t)}\E^{\chi} \Big[  [\bbW_t\bbG(z_1)]_{ji} \left(\partial_{t,ji} L_N^1(f) \right){\rm e}^{\mathrm{i}x\la L_N^1(f)\ra }\Big] \\
=& \frac{\mathrm{i}x}{N}\sum_{ij}^{(t)}\E^{\chi} \Big[  [\bbW_t\bbG(z_1)]_{ji} \left(\partial_{t,ji} L_N^2(f) \right){\rm e}^{\mathrm{i}x\la L_N^1(f)\ra }\Big] + O_{\prec}(N^{-\frac{K}{2}})\\
=&-\frac{x}{2\pi }\oint_{\bar{\gamma}^0_2}\E^{\chi} \bigg[ \frac{1}{N}\sum_{ij}^{(t)}[\bbW_t\bbG(z_1)]_{ji}\left(\partial_{t,ji}\Tr\bbG(z_2)\right){\rm e}^{\mathrm{i}x\la L_N^1(f)\ra }\bigg] f(z_2) {\rm d}z_2  + O_{\prec}(N^{-\frac{K}{2}})\\
=& \frac{x}{\pi}\oint_{\bar{\gamma}^0_2}\E^{\chi}\bigg[  \Big( \partial_{z_2}\left(\tr\bbQ_t\bbG(z_1)\bbP_t\bbG(z_2)\right)-\partial_{z_2}\left( \tr\bbQ_t\bbG(z_1)\bbG(z_2)\right)\Big){\rm e}^{\mathrm{i}x\la L_N^1(f)\ra }\bigg]f(z_2){\rm d}z_2 + O_{\prec}(N^{-\frac{K}{2}}), 
\end{align*}
where  in the third equality, the term containing $\partial_{t,ji} \Xi$ is absorbed into the error. 

Moreover,  by the concentration in Proposition \ref{keyprop2} and the fact $G(z_1)G(z_2)=(G(z_1)-G(z_2))/(z_1-z_2)$, we have 

 \begin{align}
 	\mathsf{I}_{t2}= \frac{x}{\pi}\oint_{\bar{\gamma}^0_2}\E^{\chi}\partial_{z_2}\Big[   \tr\bbQ_t\bbG(z_1)\bbP_t\bbG(z_2)- \tr\bbQ_t\bbG(z_1)\bbG(z_2)\Big]f(z_2){\rm d}z_2 \E \big[{\rm e}^{\mathrm{i}x\la L_N^1(f)\ra }\big]
	+O_{\prec}\left( \sqrt{\frac{p_t}{N^3}} \right). \label{061351}
 \end{align}

Next, we estimate $\mathsf{I}_{t3}$. After taking derivatives, we notices that all terms in $\mathsf{I}_{t3}$ can be expressed in terms of the following factors
\begin{align}
&[(X_tX_t')^{-1}]_{jj},\quad  [W_t{\bbG}\bbW'_t]_{jj}, \quad [I-\bbP_t]_{ii}, \quad [{\bbG}]_{ii}, \quad [\bbP_t{\bbG}]_{ii}, \quad [\bbP_t{\bbG}\bbP_t]_{ii}\notag\\
&  [\bbW_t]_{ji}, \quad [\bbW_t{\bbG}]_{ji}, \quad [\bbW_t{\bbG}\bbP_t]_{ji}, \quad [P_tG^2W_t']_{ij}, \quad [G^2W_t']_{ij}. \label{061310}
\end{align} 
 We call a factor with $ii$ or $jj$ index as a {\it diagonal entry} and a factor with $ij$ or $ji$ index as an { \it off-diagonal entry}, in the sequel. 
It is easy to check that all terms in $\mathsf{I}_{t3}$ either contain  one off-diagonal entry together with two diagonal entries or contain three off-diagonal entries. For simplicity, we denote by  $A_{ij}B_{ij}C_{ij}$ a generic term in the latter case and by $D_{ij}E_{ii}F_{jj}$ a generic term in the former case, where $A,\ldots, D$-entries are chosen from (\ref{061310}). Then it suffices to bound 
\begin{align}
\frac{1}{N^{\frac32}}\E^{\chi} \bigg[ \sum_{ij}^{(t)}\kappa_3^{t,j}A_{ij}B_{ij}C_{ij} \bigg],\qquad  \frac{1}{N^{\frac32}}\E^{\chi} \bigg[ \sum_{ij}^{(t)}\kappa_3^{t,j}D_{ij}E_{ii}F_{jj} \bigg]. \label{061315}
\end{align}
For the first term, we have
\begin{align*}
	\bigg|\frac{1}{N^{\frac32}}\E^{\chi} \bigg[ \sum_{ij}^{(t)}\kappa_3^{t,j}A_{ij}B_{ij}C_{ij} \bigg]\bigg|\prec & \frac{1}{N^{\frac32}}\E^{\chi} \bigg[ \sum_{ij}^{(t)}| A_{ij}|| B_{ij}|\bigg]\leq\frac{1}{N^{\frac32}}\E^{\chi} \bigg[ \sum_{ij}^{(t)}(| A_{ij}|^2 +| B_{ij}|^2) \bigg] \prec \frac{p_t}{N^{\frac32}}.
\end{align*}

For the second term in (\ref{061315}), we first replace the diagonal entries by their deterministic estimates, using Proposition \ref{DiagApproByFC}. Let $E$ and $F$ be the deterministic approximation of $E_{ii}$ and $F_{jj}$, respectively. Then we have 
\begin{align*}
	&\bigg|\frac{1}{N^{\frac32}}\E^{\chi} \bigg[ \sum_{ij}^{(t)}\kappa_3^{t,j}D_{ij}E_{ii}F_{jj} \bigg]\bigg| \leq  \bigg|\frac{E}{N^{\frac32}}\E^{\chi} \bigg[ \sum_{ij}^{(t)}\kappa_3^{t,j}D_{ij}F_{jj} \bigg]\bigg| + \frac{|\kappa_3^{t,j}|}{N^{\frac32}}\E^{\chi} \bigg[ \sum_{ij}^{(t)}|D_{ij}| |F_{jj}| |E-E_{ii}|\bigg] \\
	\prec & \bigg|\frac{E}{N^{\frac32}}\E^{\chi} \bigg[ \sum_{ij}^{(t)}\kappa_3^{t,j}D_{ij}F_{jj} \bigg]\bigg| + \frac{1}{N^{2}}\E^{\chi} \bigg[ \sum_{ij}^{(t)}\left|D_{ij}\right| \bigg]= \bigg|\frac{E}{N^{\frac32}}\E^{\chi} \bigg[ \sum_{ij}^{(t)}\kappa_3^{t,j}D_{ij}F_{jj} \bigg]\bigg| + O_{\prec}\left( \frac{p_t}{N^{\frac32}}\right).
\end{align*}
Further estimating  $F_{jj}$ by $F$, we can get
\begin{align*}
	\bigg|\frac{1}{N^{\frac32}}\E^{\chi} \bigg[ \sum_{ij}^{(t)}\kappa_3^{t,j}D_{ij}E_{ii}F_{jj} \bigg]\bigg| \prec & \bigg|\frac{EF}{N^{\frac32}}\E^{\chi} \bigg[ \sum_{ij}^{(t)}\kappa_3^{t,j}D_{ij} \bigg]\bigg| +O_{\prec}\left(\frac{p_t}{N^{\frac32}}\right) \prec \frac{p_t}{N^{\frac76}},
\end{align*}
where we used Lemma \ref{sumij} in the last step.

Combining the above estimates, we obtain
\begin{align}
\mathsf{I}_{t3} = O_{\prec}\left(\frac{p_t}{N^{\frac76}}\right). \label{061352}
\end{align}

Finally, we estimate  $\mathsf{I}_{t4}$. For brevity, we further write 
\begin{align*}
\mathsf{I}_{t4}=: \mathsf{I}_{t41} + \mathsf{I}_{t42} + \mathsf{I}_{t43} + \mathsf{I}_{t44}.
\end{align*}
where 
\begin{align*}
\mathsf{I}_{t4b}:={3\choose b-1}\sum_{ij}^{(t)}\frac{\kappa_4^{t,j}}{6{N}^{2}}\E^{\chi} \Big[ \partial^{4-b}_{t,ji}[\bbW_t\bbG]_{ji}\partial^{b-1}_{t,ji}\bla {\rm e}^{\mathrm{i}x\la L_N^1(f)\ra }\bra\Big], \quad b=1,2,3,4. 
\end{align*}

We first consider the simplest one $\mathsf{I}_{t42}$. Again, due to Lemma \ref{XiLemma}, we can neglect the terms with derivatives of $\Xi$. Notice that  the other terms in $\partial^2_{t,ji} [\bbW_t{\bbG} ]_{ji}\partial_{t,ji} {\rm e}^{\mathrm{i}x\la L_N^1(f)\ra }$ contains at least two off-diagonal entries. Let $A_{ij}$ and $B_{ij}$ be two of them and bound the other two terms simply by $O_\prec (1)$, it suffices to have the following bound for $\mathsf{I}_{t42}$
\begin{align*}
	\frac{1}{N^2}\E^{\chi} \bigg[ \sum_{ij}^{(t)}\kappa_4^{t,j}|A_{ij}  | |B_{ij} |\bigg] 
	 \lesssim  \frac{1}{N^2}\E^{\chi} \bigg[ \sum_{ij}^{(t)}|A_{ij}  |^2 + |B_{ij} |^2\bigg] = O_{\prec}\left( {\frac{p_t}{N^{2}}}\right).
\end{align*}
This gives
\begin{align}
	\mathsf{I}_{t42} = O_{\prec}\left( {\frac{p_t}{N^2}}\right). \label{061340}
\end{align}
Similar arguments lead to
\begin{align}
	\mathsf{I}_{t44} = O_{\prec}\left( {\frac{p_t}{N^2}}\right). \label{061341}
\end{align}
For $\mathsf{I}_{t41}$, we first rewrite 
\begin{align}
\mathsf{I}_{t41}=\sum_{ij}^{(t)}\frac{\kappa_4^{t,j}}{6{N}^{2}}\E^{\chi} \Big[\bla \partial^{3}_{t,ji}[\bbW_t\bbG]_{ji}\bra {\rm e}^{\mathrm{i}x\la L_N^1(f)\ra }\Big] \label{061320}
\end{align} 
After straightforward calculation, we notice that the terms either contain at least two off-diagonal entries as factors, or all four factors are diagonal entries. In the former case, 
 we can bound them similarly to $\mathsf{I}_{t42}$ and $\mathsf{I}_{t44}$. In the latter case, 2 diagonal factors will be from the set $\mathcal{I}$ and the other two are from $\mathcal{J}$ defined below
\begin{align*}
	\mathcal{I} = \{ [{\bbG}]_{ii}, [\bbP_t{\bbG}]_{ii},[\bbP_t{\bbG}\bbP_t]_{ii} ,[I-\bbP_t]_{ii}\},\quad \mathcal{J}= \{[\bbX_t\bbX_t']_{jj}, [\bbW_t{\bbG}\bbW'_t]_{jj} \}.
\end{align*}
\noindent
In light of (\ref{061320}), for the contribution of  the terms with four diagonal entries to $\mathsf{I}_{t41}$, it suffices to bound the following type of terms  
\begin{align*}
	\frac{1}{N^2}\E^{\chi}\bigg[\Big\la \sum_{ij}^{(t)}\kappa_4^{t,j}A_{ii}B_{ii}C_{jj}D_{jj}\Big\ra {\rm e}^{\mathrm{i}x\la L_N^1(f)\ra }  \bigg],
\end{align*}
where $A_{ii}, B_{ii}\in \mathcal{I}$ and $C_{jj}, D_{jj} \in \mathcal{J}$. From the concentration results for diagonal entries (Propposition \ref{DiagApproByFC}), we know that there exists deterministic estimate $A$, such that $|A_{ii}-A|\prec1/\sqrt{N}$, and thus
\begin{align*}
	&\frac{1}{N^2}\E^{\chi}\bigg[\Big\la \sum_{ij}^{(t)}\kappa_4^{t,j}A_{ii}B_{ii}C_{jj}D_{jj}\Big\ra {\rm e}^{\mathrm{i}x\la L_N^1(f)\ra }  \bigg] =\frac{A}{{N}^2}\E^{\chi} \bigg[ \Big\la\sum_{ij}^{(t)}\kappa_4^{t,j}B_{ii}C_{jj}D_{jj}\Big\ra {\rm e}^{\mathrm{i}x\la L_N^1(f)\ra } \bigg] + O_{\prec}\left(\frac{p_t}{N^{\frac32}} \right).
\end{align*}
Similarly, by applying the concentration results of $B_{ii}, C_{jj}, D_{jj}$ repeatedly, we have
\begin{align*}
	\frac{1}{N^2}\E^{\chi}\bigg[\Big\la \sum_{ij}^{(t)}\kappa_4^{t,j}A_{ii}B_{ii}C_{jj}D_{jj}\Big\ra {\rm e}^{\mathrm{i}x\la L_N^1(f)\ra }  \bigg] = O_{\prec}\left(\frac{p_t}{N^{\frac32}} \right),
\end{align*}
which gives
\begin{align}
	\mathsf{I}_{t41} = O_{\prec}\left(\frac{p_t}{N^{\frac32}} \right). \label{061342}
\end{align}
Lastly, we consider $\mathsf{I}_{t43}$. Notice that
\begin{align}
\mathsf{I}_{t43}
=&\sum_{ij}^{(t)}\frac{\kappa_4^{t,j}}{2N^2}\E^{\chi} \Big[ \left( \partial_{t,ji}[\bbW_t{\bbG(z_1)}]_{ji} \right)\left({\rm i}x\partial_{t,ji}^2 L_N^1(f)+({\rm i} x\partial_{t,ji} L_N^1(f))^2\right){\rm e}^{\mathrm{i}x\la L_N^1(f)\ra } \Big]. \label{061331}
\end{align}
For the term with $({\rm i} x\partial_{t,ji} L_N^1(f))^2$, we have 
\begin{align}
&\bigg|\sum_{ij}^{(t)} \kappa_4^{t,j}\E^{\chi}\Big[ \left( \partial_{t,ji}[\bbW_t{\bbG(z_1)}]_{ji} \right)\left(\partial_{t,ji} L_N^1(f)\right)^2\Big]\bigg| \notag\\
=& \bigg|\frac{ 1}{2\pi i}\sum_{ij}^{(t)}\oint_{\bar{\gamma}_1^0}\oint_{\bar{\gamma}_1^0}\kappa_4^{t,j}\E^{\chi} \Big[ \partial_{t,ji}[\bbW_t{\bbG(z_1)}]_{ji} \left( \partial_{t,ji} \Tr{\bbG(z_2)}\right) \left( \partial_{t,ji} \Tr{\bbG(z_3)}\right)\Big]f(z_2)f(z_3){\rm d}z_2{\rm d}z_3 \bigg| + O_{\prec}(N^{-k}) \notag\\
\prec&\sum_{ij}^{(t)}\oint_{\bar{\gamma}_1^0}\oint_{\bar{\gamma}_1^0}\E^{\chi} \Big[ 
 \left| \partial_{t,ji} \Tr{\bbG(z_2)}\right| \left| \partial_{t,ji} \Tr{\bbG(z_3)}\right|\Big]{\rm d}z_2{\rm d}z_3 \label{061330}
\end{align}
where in the last step we bounded the term $\partial_{t,ji}[\bbW_t{\bbG(z_1)}]_{ji}$ by $O_\prec(1)$ which is easy to check after one computes the derivative. Further, notice that $\partial_{t,ji} \Tr{\bbG(z)}$ is a linear combination of terms 
\begin{align*}
[P_t{\bbG^2(z)}\bbW_t]_{ij},\quad [{\bbG^2(z)}\bbW_t]_{ij}
\end{align*}
Let $\mathcal{E}_{ij}(z)$ and $\mathcal{F}_{ij}(z)$ be either of the above two. Then we can bound trvially
\begin{align*}
\sum_{ij}^{(t)}|\mathcal{E}_{ij}(z_2)||\mathcal{F}_{ij}(z_3)|\leq \sum_{ij}^{(t)}|\mathcal{E}_{ij}(z_2)|^2+\sum_{ij}^{(t)}|\mathcal{F}_{ij}(z_3)|^2.
\end{align*}
Notice that both of the two sums are tracial quantities which can be simply bounded by $O_\prec(p_t)$. Hence, we conclude
\begin{align}
\text{(\ref{061330})}\prec p_t.  \label{061332}
\end{align}
Further, for the term with ${\rm i}x\partial_{t,ji}^2 L_N^1(f)$ in (\ref{061331}),  we have 
\begin{align}
&\bigg|\sum_{ij}^{(t)}\kappa_4^{t,j}\E^{\chi} \Big[ \left( \partial_{t,ji}[\bbW_t{\bbG(z_1)}]_{ji} \right) (\partial^2_{t,ji} L_N^1(f)){\rm e}^{\mathrm{i}x\la L_N^1(f)\ra }  \Big]\bigg| \notag\\ =&\bigg|\frac{1}{2\pi }\sum_{ij}^{(t)}\oint_{\bar{\gamma}_1^0} \kappa_4^{t,j}\E^{\chi} \Big[ \partial_{t,ji}[\bbW_t{\bbG(z_1)}]_{ji}\left( \partial_{t,ji}^2 \Tr{\bbG(z_2)}\right){\rm e}^{\mathrm{i}x\la L_N^1(f)\ra }  \Big]f(z_2){\rm d}z_2\bigg| \notag\\
	\prec&  \sum_{ij}^{(t)} \oint_{\bar{\gamma}_1^0}\E^{\chi}\Big[ \left|\partial_{t,ji}^2 \Tr{\bbG(z_2)}\right|\Big] {\rm d}z_2 \prec p_tN^{\frac12}. \label{061333}
\end{align}
Here in the last step, we used (\ref{062520}). Plugging (\ref{061332}) and (\ref{061333}) into (\ref{061331}), we can conclude 
\begin{align}
\mathsf{I}_{t43}=O_\prec\left(\frac{p_t}{N^\frac32} \right). \label{061343}
\end{align}

In summary, with (\ref{061340}),(\ref{061341}),(\ref{061342}) and (\ref{061343}) we have
\begin{align}
		\mathsf{I}_{t4}= O_{\prec}\left(\frac{p_t}{N^\frac32} \right). \label{061353}
\end{align}

\end{proof}

\subsection{Proofs of Lemma \ref{lem.062830}}
Recall the definitions of $\alpha_t$ and $\beta_t$ from (\ref{def of alphabeta}).
		Using Propositions \ref{keyProp} and \ref{keyprop2} together with the estimate of $\gamma_{t2}$ in Lemma \ref{Lemma error estimates}, we have
	\begin{align}
		\alpha_t(z) 
		=& \frac{y_t}{1-y_t} - \frac{1 + \omega_t(z)m_{\boxplus}(z)}{1-y_t} + O_{\prec}\left( \sqrt{\frac{p_t}{N^3}} \right), \notag\\
\beta_t(z) =& \frac{m_{\boxplus}(z)}{1-y_t} - \frac{1 + \omega_t(z)m_{\boxplus}(z)}{1-y_t} + O_{\prec}\left( \frac{1}{N} \right). \label{061403}
	\end{align}
	Therefore, we have
	\begin{align*}
		{1 + \alpha_t(z) + \beta_t(z)} = \frac{ m_{\boxplus}(z)-1-2\omega_t(z)m_{\boxplus}(z)}{1-y_t} + O_{\prec}\left(\frac{1}{N} \right).
	\end{align*}
	Next, we prove (\ref{lowerboundof1plusalphatbetat}) when $z$ lies in our contours. For $z \in (\bar{\gamma}_1)^{+}\cup(\bar{\gamma}_2)^{+}$, we have $|z|$ sufficiently large and therefore $|m_{\boxplus}(z)|$ sufficiently small. By the continuity of $\omega_t(z)$ and (\ref{m2}), we see that $|\omega_t(z)|$ is sufficiently large.  In this case, we notice that
	\begin{align*}
		m_{\boxplus}(z)-1-2\omega_t(z)m_{\boxplus}(z) = m_{\boxplus}(z) -1  - 2\left( \frac{\omega_t(z)y_t}{1-\omega_t(z)}  - (1-y_t)\right) \to 1 
	\end{align*}
	as $|z| \to \infty$. By the continuity of $m_{\boxplus}(z)$ and $\omega_t(z)$, we have
$
		|m_{\boxplus}(z)-1-2\omega_t(z)m_{\boxplus}(z) | > 0
$
	for sufficiently large $|z|$. This proves (\ref{lowerboundof1plusalphatbetat}) for $z$ on $\mathcal{C}$. For $z \in \mathcal{C}_1(\epsilon_{11},\epsilon_{21})\cup \mathcal{C}_1(\epsilon_{12},\epsilon_{22})$ with $\hat{y}\in (0,1)$ (c.f. (\ref{062033})), since we know that when $|z| \to 0$, due to the singularity at $0$,
	\begin{align*}
		|m_{\boxplus}(z)-1-2\omega_t(z)m_{\boxplus}(z)| = \Big|m_{\boxplus}(z) -1  - 2\left( \frac{\omega_t(z)y_t}{1-\omega_t(z)}  - (1-y_t)\right)\Big| \to \infty, 
	\end{align*}
	we can also get 
	$
		|m_{\boxplus}(z)-1-2\omega_t(z)m_{\boxplus}(z) | > 0.
	$
	For $z \in \mathcal{C}_2(\epsilon_{11},\epsilon_{21},M_{11})\cup\mathcal{C}_2(\epsilon_{12},\epsilon_{22},M_{12})$ with $\hat{y}\in (0,1)$ (c.f. (\ref{062033})), we first notice that
	\begin{align*}
		m_{\boxplus}(z)-1-2\omega_t(z)m_{\boxplus}(z) =\frac{m'_\boxplus(z)}{\omega_t'(z)\omega_t(z)(1- \omega_t(z))}.
	\end{align*}
	Then it suffices to lower bound $|m'_\boxplus(z)|$ for $z \in \mathcal{C}_2(\epsilon_{11},\epsilon_{21},M_{11})\cup\mathcal{C}_2(\epsilon_{12},\epsilon_{22},M_{12})$ with $\hat{y} \in (0,1)$ (c.f. (\ref{062033})). Since
	\begin{align*}
		|m'_\boxplus(z)| \ge \Im m'_\boxplus(z) = \int \frac{2(\lambda - \Re z)\Im z}{|\lambda - z|^4} {\rm d}\mu_{\boxplus} > 0,
	\end{align*}
	we can get the desired result.
	
\subsection{Proofs of Lemma \ref{LemmaEtrPG2}}
We start from the cumulant expansion of $\E^{\chi} [\Tr P_tG(z)]$ as follows
	\begin{align*}
		\E ^{\chi}[ \Tr P_t G(z) ] =& \sum_{ij}^{(t)}\frac{1}{N}\E^{\chi} [\partial_{t,ji}[W_tG(z)]_{ji}] +  \sum_{ij}^{(t)}\frac{\kappa_3^{t,j}}{2N^{3/2}}\E^{\chi} [\partial^2_{t,ji}[W_tG(z)]_{ji} ]\\
		&+ \sum_{ij}^{(t)}\frac{\kappa_4^{t,j}}{6N^2}\E^{\chi} [\partial^{3}_{t,ji}[W_tG(z)]_{ji} ] + O_{\prec}\left(  \frac{p_t}{N^{\frac{3}{2}}} \right)=: J_{t1} + J_{t2} + J_{t3} + O_{\prec}\left(  \frac{p_t}{N^{\frac{3}{2}}} \right).
	\end{align*}
	Here the error term can be bounded similarly to the estimates of $\mathsf{E}_i, i=1,2,3$ in the proof of Lemma \ref{Error in CLT}, we omit the details.
		
For $J_{t1}$, by direct calculation, we have
	\begin{align}
		&J_{t1} = \E^{\chi} [( \tr \left(\bbX_t\bbX_t' )^{-1}  - \tr \bbQ_t\bbG\right)\left( \Tr \bbG - \Tr \bbP_t\bbG \right) ]+\E^{\chi} [ \tr\bbQ_t\bbG\bbP_t\bbG-\tr\bbQ_t\bbG-\tr\bbQ_t\bbG^2 ] \notag\\
		&= \E^{\chi} [( \tr \left(\bbX_t\bbX_t' )^{-1}  - \tr \bbQ_t\bbG\right)\left( \Tr \bbG - \Tr \bbP_t\bbG \right) \Xi]+\E^{\chi} [ \tr\bbQ_t\bbG\bbP_t\bbG-\tr\bbQ_t\bbG-\tr\bbQ_t\bbG^2 ] + O_{\prec}(N^{-D}) \notag\\
	&= \E^{\chi} [ \tr \left(\bbX_t\bbX_t' \right)^{-1}  - \tr \bbQ_t\bbG]\E^{\chi} [  \Tr \bbG - \Tr \bbP_t\bbG ] +\E^{\chi} [ \tr\bbQ_t\bbG\bbP_t\bbG-\tr\bbQ_t\bbG-\tr\bbQ_t\bbG^2 ]+ O_{\prec}\left( \sqrt{\frac{p_t}{N^3}} \right). \label{062550}
	\end{align}
Here  in the last step we used the estimate of $\gamma_{t2}$ in Lemma \ref{Lemma error estimates} and Propositions \ref{keyProp} and \ref{keyprop2}.

	For $J_{t2}$, similar to the arguments in the estimate  of $\mathsf{I}_{t3}$ (c.f. (\ref{061310})-(\ref{061352})), we can get
	\begin{align*}
		J_{t2} = O_{\prec}\left( \frac{p_t}{N^{\frac{7}{6}}} \right).
	\end{align*}
	For $J_{t3}$, we only need to consider the terms consisting of four diagonal entries, and the other terms with at least two off-diagonal entries  can be bounded similarly to (\ref{061340}). Specifically, we will have 
	\begin{align*}
		J_{t3} =& \sum_{ij}^{(t)}\frac{\kappa_4^{t,j}}{{N}^{2}}\E^{\chi} \Big[  [\left(\bbX_t\bbX_t' \right)^{-1}]_{jj}^2\left(  [{I} - \bbP_t]_{ii} [(\bbP_t-{I})\bbG]_{ii} - [({I} - \bbP_t)\bbG]_{ii}^2 \right) \Big] + O_{\prec}\left( {\frac{p_t}{N^2}} \right).
	\end{align*}
Combining the above estimates of $J_{t1}$, $J_{t2}$ and $J_{t3}$, we get
	\begin{align}\label{EtrPG2}
		&\E^{\chi} [ \Tr P_tG ] = \E^{\chi} [ \tr \left(\bbX_t\bbX_t' \right)^{-1}  - \tr \bbQ_t\bbG ]\E^{\chi} [  \Tr \bbG - \Tr \bbP_t\bbG  ]+\E^{\chi} [ \tr\bbQ_t\bbG\bbP_t\bbG-\tr\bbQ_t\bbG-\tr\bbQ_t\bbG^2 ] \notag \\
	&+\sum_{ij}^{(t)} \frac{\kappa_4^{t,j}}{{N}^{2}}\E^{\chi} \Big[  [\left(\bbX_t\bbX_t' \right)^{-1}]_{jj}^2\left(  [{I} - \bbP_t]_{ii} [(\bbP_t-{I})\bbG]_{ii} - [({I} - \bbP_t)\bbG]_{ii}^2 \right) \Big]+ O_{\prec}\left(\frac{p_t}{N^{\frac{7}{6}}} \right).
	\end{align}

\section{Proofs of Lemmas in Sections \ref{Sec Proof of Proposition 6.2}-\ref{Sec Proofs of Lemma 6.4}}\label{Sec Proof of Lemma C}
\subsection{Proof of Lemma \ref{LemPGii}}

With certain abuse of notation, we recycle the notation $Z$ to denote 
\begin{align*}
Z=[\bbP_t\bbG]_{ii}-\left(\tr(\bbX_t\bbX_t')^{-1}-\tr\bbQ_t\bbG\right)\left(\bbG_{ii}-[\bbP_t\bbG]_{ii}\right),
\end{align*}
in this proof. 
Further, we set 
\begin{align*}
	\mathsf{Z}^{q,p} := (Z\cdot\Xi)^{p}(\bar{Z}\cdot\Xi)^{q}.
\end{align*}
Then we have 
\begin{align}
\E \big[ \mathsf{Z}^{n,n} \big]
=& \E^{\chi} \Big[ [\bbP_t\bbG]_{ii}\mathsf{Z}^{n-1,n} \Big]-\E^{\chi} \Big[ \left(\tr(\bbX_t\bbX_t')^{-1}-\tr\bbQ_t\bbG\right)([\bbG]_{ii}-[\bbP_t\bbG]_{ii}) \mathsf{Z}^{n-1,n} \Big] \notag\\
=&\sum_{j=1}^{p_t}\E^{\chi}\Big[ \frac{1}{N}\left(\partial_{t,ji} [\bbW_t\bbG]_{ji}\right) \mathsf{Z}^{n-1,n} \Big]
-\E^{\chi} \Big[ \left(\tr(\bbX_t\bbX_t')^{-1}-\tr\bbQ_t\bbG\right)([\bbG]_{ii}-[\bbP_t\bbG]_{ii}) \mathsf{Z}^{n-1,n} \Big] \notag\\
&+\sum_{j=1}^{p_t}\E^{\chi} \Big[\frac{1}{N} [\bbW_t\bbG]_{ji} \left((n-1)\left( \partial_{t,ji}Z \right)\mathsf{Z}^{n-2,n} + n \left( \partial_{t,ji}\bar{Z} \right)\mathsf{Z}^{n-1,n-1}\right) \Big] \notag\\
&+\sum_{s\geq 2}^lO\left(\frac{1}{N^{\frac{s+1}{2}}}\right)\sum_{j=1}^{p_t}\E^{\chi} \Big[ \partial^s_{t,ji}\left\{[\bbW_t\bbG]_{ji}\mathsf{Z}^{n-1,n}\right\} \Big] + O_{\prec}(\frac{1}{N}).
\label{PGiimoment3}
\end{align}
where the error $O_{\prec}(N^{-1})$ can be verified similarly to (\ref{061490}) by choosing $l$ sufficiently large.
By direct calculation,
\begin{align}
	\sum_{j=1}^{p_t}\frac{1}{N}\partial_{t,ji}[\bbW_t\bbG]_{ji} = \left(\tr(\bbX_t\bbX_t')^{-1}-\tr\bbQ_t\bbG\right)([\bbG]_{ii}-[\bbP_t\bbG]_{ii}) + O_{\prec}\left( \frac{1}{N} \right). \label{062152}
\end{align}
Further, similarly to (\ref{061491}), via direct calculation, we can write  $\partial_{t,ji} Z = [\mathcal{A}]_{ij}$ for some matrix $\mathcal{A}\in \mathbb{C}^{N\times p_t}$ with $\|\mathcal{A}\|\prec 1$. The details are omitted.  Then it is easy to show 
\begin{align}
	\sum_{j=1}^{p_t}\frac{1}{N}[\bbW_t\bbG]_{ji}\partial_{t,ji}Z = O_{\prec}\left( \frac{1}{N} \right). \label{062151}
\end{align}
The same bound also holds when $Z$ is replaced by $\bar{Z}$. 
Hence, plugging (\ref{062152}) and (\ref{062151}) to \eqref{PGiimoment3}, we have 
\begin{align}\label{PGiimoment2}
	\E \big[ \mathsf{Z}^{n,n} \big]=& \E^{\chi} \Big[ O_{\prec}\left( \frac{1}{N} \right)\mathsf{Z}^{n-1,n} \Big] + \E^{\chi} \Big[ O_{\prec}\left( \frac{1}{N} \right)\mathsf{Z}^{n-2,n} \Big] +\E^{\chi} \Big[ O_{\prec}\left( \frac{1}{N} \right)\mathsf{Z}^{n-1,n-1} \Big]\notag\\
	&+ \sum_{s\geq 2}^lO\left(\frac{1}{N^{\frac{s+1}{2}}}\right)\sum_{j=1}^{p_t}\E^{\chi} \Big[ \partial^s_{t,ji}\left\{[\bbW_t\bbG]_{ji}\mathsf{Z}^{n-1,n}\right\} \Big]+ O_{\prec}\left( \frac{1}{N} \right)
\end{align}

For the case $s\geq 2$, since the complex conjugate takes no effect in our estimations, we drop the complex conjugate for convenience.  Here we also drop the function $\Xi$ first and put it back at the last step, up to negligible error.  Hence, for the fourth term of (\ref{PGiimoment2}), we have
\begin{align*}
	&\sum_{s\geq 2}^lO\left(\frac{1}{N^{\frac{s+1}{2}}}\right)\sum_{j=1}^{p_t}\E^{\chi} \Big[ \partial^s_{t,ji}\left\{[\bbW_t\bbG]_{ji}Z^{2n-1}\right\} \Big]= \sum_{r=0}^{2n-1}\E^{\chi}\Big[ C_{r}Z^{2n-1-r} \Big]
\end{align*}
where
\begin{align*}
	&C_r = \sum_{s\ge2}^l\sum_{\substack{s_0,\dots,s_r>0\\   s = s_0+\dots+s_r}} O\left(\frac{1}{N^{\frac{s+1}{2}}}\right)\sum_{j=1}^{p_t}(\partial^{s_0}_{t,ji} [\bbW_t\bbG]_{ji})\partial^{s_1}_{t,ji}Z\cdots\partial^{s_r}_{t,ji}Z \\
	&=O\left( \frac{1}{N^{\frac{r+1}{2}}} \right) \sum_{s\ge2}^l\sum_{\substack{s_0,\dots,s_r>0\\   s = s_0+\dots+s_r}} O\left(\frac{1}{N^{\frac{s-r}{2}}} \right)\sum_{j=1}^{p_t}(\partial^{s_0}_{t,ji} [\bbW_t\bbG]_{ji})\partial^{s_1}_{t,ji}Z\cdots\partial^{s_r}_{t,ji}Z=: O\left( \frac{1}{N^{\frac{r+1}{2}}}  \right) \sum_{s\ge2}^l\mathcal{C}_{r,s}
\end{align*}

Case $1$: when $s - r = 0$, we have $s_0 = s_1 - 1 = \cdots s_{r} - 1 = 0$, and $\mathcal{C}_{r,s}$ can be written as,
\begin{align*}
	\mathcal{C}_{r,s} = \sum_{j=1}^{p_t} [\bbW_t\bbG]_{ji}\partial_{t,ji} Z\cdots\partial_{t,ji} Z.
\end{align*}
Since the product above has at least two off-diagonal entries, we have
$
	|\mathcal{C}_{r,s}| \prec 1.
$

Case $2$: when $s - r = 1$, we have $s_0 = s_1 =  \cdots =s_r = 1$ or $s_0 = s_1 - 1 = \cdots = s_i - 2 = \cdots s_r - 1 = 0$ for an $i$. Now $\mathcal{C}_{r,s}$ becomes
\begin{align*}
	\mathcal{C}_{r,s} = O\left(\frac{1}{N^\frac12}\right)\sum_{j=1}^{p_t} (\partial^{s_0}_{t,ji} [\bbW_t\bbG]_{ji})\partial^{s_1}_{t,ji}Z\cdots\partial^{s_r}_{t,ji}Z. 
\end{align*}
We have at least one off-diagonal entry in both cases for the above product. Therefore, 
\begin{align*}
	\mathcal{C}_{r,s} \prec \frac{1}{N^\frac12}\sqrt{{p_t\sum_{j=1}^{p_t}|[\mathcal{M}]_{ij}|^2}} =  \frac{1}{N^\frac12} \sqrt{{p_t[\mathcal{M}\mathcal{M}^*]_{ii}}} \prec {\sqrt{\frac{p_t}{N}}}.
\end{align*}
Here $[\mathcal{M}]_{ij}$ is the off-diagonal entry either from $\bbW_t\bbG$ or the derivative of $Z$.

Case $3$: when $s - r \geq 2$, we can use the trivial bound
\begin{align*}
	\mathcal{C}_{r,s} = O\left(\frac{1}{N^{\frac{s-r}{2}}}\right)\sum_{j=1}^{p_t} (\partial^{s_0}_{t,ji} [\bbW_t\bbG]_{ji})\partial^{s_1}_{t,ji}Z\cdots\partial^{s_r}_{t,ji}Z  = O_\prec\left({\frac{p_t}{N}}\right).
\end{align*}

 Combining all these cases, we get
$
 	|C_r| \prec  N^{-\frac{r+1}{2}}.
$
 
Substituting $Z\cdot \Xi$ for $Z$ and then applying Young's inequality to \eqref{PGiimoment2}, we can obtain,
\begin{align}\label{PGiimoment}
	\E \big[\mathsf{Z}^{n,n} \big]= O_\prec\left( N^{-n}\right).
\end{align}
By Markov inequality, we obtain \eqref{PGii} from \eqref{PGiimoment}.

For the proof of \eqref{Pii} and \eqref{PGPii}-\eqref{WGWjj}, we only give the constructions of $Z$. 

Define
	\begin{align*}
		Z=[\bbP_t]_{ii}-\tr(\bbX_t\bbX_t')^{-1}(1-[P_t]_{ii}),
	\end{align*}
Then following the same argument in the proof of \eqref{PGii}, we have
\begin{align*}
	[\bbP_t]_{ii}-\tr(\bbX_t\bbX_t')^{-1}(1-[P_t]_{ii}) = O_{\prec}\left(\frac{1}{\sqrt{N}}\right).
\end{align*}
Plugging (\ref{trXX}) into the above estimate and then solving for $[P_t]_{ii}$, we can obtain (\ref{Pii}).

Define
	\begin{align*}
		Z=[\bbP_t\bbG\bbP_t]_{ii}-\left(\tr(\bbX_t\bbX_t')^{-1}-\tr\bbQ_t\bbG\right)([\bbG\bbP_t]_{ii}-[\bbP_t\bbG\bbP_t]_{ii})
		-(1-[\bbP_t]_{ii})\tr\bbQ_t\bbG,
\end{align*}
Then following the same argument in the proof of \eqref{PGii}, we can obtain \eqref{PGPii}.

Define
	\begin{align*}
		Z=[X_tX_t'(X_tX_t')^{-1}]_{jj}-(1-y_t)[(X_tX_t')^{-1}]_{jj},
\end{align*}
Then following similar argument with index $j$ replaced by $i$ in the proof of \eqref{PGii}, one can show
\begin{align*}
	|Z| \prec N^{-\frac12}.
\end{align*}
Using the fact that $[X_tX_t'(X_tX_t')^{-1}]_{jj} = 1$, we can obtain \eqref{XXjj}.

Define 
	\begin{align*}
Z_1&=[\bbX_t\bbG\bbW'_t]_{jj}-\left(\tr\bbG-\tr\bbP_t\bbG\right)\left([\left(\bbX_t\bbX_t'\right)^{-1}]_{jj}-[\bbW_t\bbG\bbW'_t]_{jj}\right),\\
Z_2&=[\bbX_t\bbP_t\bbG\bbW'_t]_{jj}-(1-y_t)[\bbW_t\bbG\bbW'_t]_{jj}.
\end{align*}
Then following similar argument with index $j$ replaced by $i$ in the proof of \eqref{PGii}, one can show
\begin{align*}
	|Z_1| \prec N^{-\frac12}, \quad |Z_2| \prec N^{-\frac12}.
\end{align*}
Using the fact that 
$
	\bbX_t\bbG\bbW'_t = \bbX_t\bbP_t\bbG\bbW'_t,
$
and then subtracting $Z_2$ from $Z_1$, we can obtain
\begin{align*}
	(1-y_t)[\bbW_t\bbG\bbW_t']_{jj}-\left(\tr\bbG-\tr\bbP_t\bbG\right)\left([\left(\bbX_t\bbX_t'\right)^{-1}]_{jj}-[\bbW_t\bbG\bbW'_t]_{jj}\right) 
		=O_{\prec}(N^{-\frac12}).
\end{align*} 

\subsection{Proof of Lemma \ref{LemSumDiag}}
We show the proof of \eqref{sPGii} with $\mathcal{W}_t = 1$. The proof can be easily generalised to the case with bounded $\mathcal{W}_t$. With certain abuse of notation, we recycle the notation $Z$ to denote

Define
\begin{align*}
Z=\sum_{t=1}^k[ P_t G]_{ii}-\sum_{t=1}^k\left(\tr( X_t X_t')^{-1}-\tr Q_t G\right)( G_{ii}-[ P_t G]_{ii}),
\end{align*}
and set
\begin{align*}
	\mathsf{Z}^{p,q} = (Z\cdot\Xi)^p(\bar{Z}\cdot\Xi)^q.
\end{align*}
Then by the cumulant expansion, we have 
\begin{align}
\E\big[ \mathsf{Z}^{n,n} \big] =&\E^{\chi}\bigg[ \sum_{t=1}^k\sum_{j=1}^{p_t} X'_{t,ij}[ W_t G]_{ji}\mathsf{Z}^{n-1,n}\bigg] -\E^{\chi} \bigg[ \sum_{t=1}^k\left(\tr( X_t X_t')^{-1}-\tr Q_t G\right)( G_{ii}-[ P_t G]_{ii})\mathsf{Z}^{n-1,n}\bigg] \notag\\
=&\E^{\chi}\bigg[ \sum_{t=1}^k\sum_{j=1}^{p_t}\frac{1}{N}\partial_{t,ji} \{[ W_t G]_{ji}\mathsf{Z}^{n-1,n}\} \bigg] -\E^{\chi}\bigg[ \sum_{t=1}^k \left(\tr( X_t X_t')^{-1}-\tr Q_t G\right)( G_{ii}-[ P_t G]_{ii})\mathsf{Z}^{n-1,n}\bigg] \notag\\
&+\E^{\chi} \bigg[ \sum_{s\geq 2}^l\sum_{t=1}^k\sum_{j=1}^{p_t}O\left(\frac{1}{N^{(s+1)/2}}\right)\partial_{t,ji}^s\{ [ W_t G]_{ji}\mathsf{Z}^{n-1,n}\}\bigg] + O_{\prec}\left( \frac{1}{N} \right) \notag\\
=:&M_1-M_2+M_3+O_{\prec}\left( \frac{1}{N} \right),  \label{062310}
\end{align}
where  the error $O_{\prec}(N^{-1})$ can be verified similarly to (\ref{061490}) by choosing $l$ sufficiently large.

Again, similarly to (\ref{061491}), via direct calculation, we can write  $\partial_{t,ji} Z = [\mathcal{A}]_{ij}$ for some matrix $\mathcal{A}\in \mathbb{C}^{N\times p_t}$ with $\|\mathcal{A}\|\prec 1$, and further $\mathcal{A}$ is a finite sum of terms of the form $a_t [\mathcal{M}_t]_{ij}$ with $a_t = O_{\prec}(1)$ and $\|\sum_{t=1}^k\mathcal{M}_t\mathcal{M}_t^*\|\prec 1$. For instance, one such $\mathcal{M}_t$ in $\mathcal{A}$ is $P_tG(\sum_{d=1}^k b_dQ_d)GW_t'$ for some $|b_d|\prec 1$. 

Since the complex conjugate takes no effect in the remaining estimations, we drop the complex conjugate for simplicity. Specifically,  we will work with $Z^{2n-1}$ instead of $\mathsf{Z}^{n-1,n}$ in the remaining derivation, for notational brevity.  Here again we drop the function $\Xi$ first and put it back at the last step, up to negligible error. 
\begin{align}
	M_1-M_2=&\E^{\chi}\bigg[ \sum_{t=1}^k\left(-\frac{1}{N}[ Q_t G]_{ii}-\frac{1}{N}[ Q_t G^2]_{ii} \right) Z^{2n-1}\bigg]+\E^{\chi}\bigg[\sum_{t=1}^k\sum_{j=1}^{p_t}\frac{1}{N}[W_tGP_t]_{ji}[W_tG]_{ji} Z^{2n-1}\bigg]\notag \\
	&+\E^{\chi}\bigg[\sum_{t=1}^k\sum_{j=1}^{p_t}\frac{O(1)}{N}[ W_t G]_{ji} \left( \partial_{t,ji} Z\right) Z^{2n-2}\bigg] + O_{\prec}(N^{-D}). \label{062153}
\end{align}
for any large $D > 0$.
Firstly, we have 
\begin{align}
&\sum_{t=1}^k\left(-\frac{1}{N}[ Q_t G]_{ii}-\frac{1}{N}[ Q_t G^2]_{ii} \right)\le \frac{1}{N} \left\| \sum_{t=1}^kQ_t G \right\| +\frac{1}{N} \left\| \sum_{t=1}^kQ_t G^2\right\|=O_{\prec}\left(\frac{1}{N}\right), \label{062190}
\end{align}
and 
\begin{align}
\sum_{t=1}^k\sum_{j=1}^{p_t}\frac{1}{N}[W_tGP_t]_{ji}[W_tG]_{ji}\leq \frac{1}{N}\sqrt{\sum_{t=1}^k[P_tG^{\ast}Q_tGP_t]_{ii}}\sqrt{\sum_{t=1}^k[G^{\ast}Q_tG]_{ii}}=O_{\prec}\left(\frac{1}{N}\right).\label{062191}
\end{align}
Here we used (\ref{SumQt}) and 
\begin{align}
\Big|\sum_{t=1}^k[P_tG^{\ast}Q_tGP_t]_{ii}\Big|\leq \Big\|\sum_{t=1}^kP_tG^{\ast}Q_tGP_t\Big\| \leq  \Big\|\sum_{t=1}^kP_t\Big\| \max_d\|G^{\ast}Q_dG\|\prec \|H\|\prec 1.  \label{062160}
\end{align}

As we mentioned,  all the terms in $\partial_{t,ji} Z$ (or $[\mathcal{A}]_{ji}$) is of the format $a_t [\mathcal{M}_t]_{ij}$ with $a_t = O_{\prec}(1)$. Since
\begin{align*}
	\sum_{t=1}^k \sum_{j=1}^{p_t}\frac{1}{N}[ W_t G]_{ij} [a_t\mathcal{M}_t]_{ij} \prec \frac{1}{N} \sqrt{\sum_{t=1}^k [G^*Q_tG]_{ii}}\sqrt{\sum_{t=1}^k [\mathcal{M}_t^*\mathcal{M}_t]_{ii} }.
\end{align*}
For the mentioned example  $\mathcal{M}_t = P_tG(\sum_{d=1}^k b_dQ_d)GW_t'=:P_tG\widetilde{Q}GW_t'$, similarly to (\ref{062160}), we have
\begin{align*}
	\sum_{t=1}^k [\mathcal{M}_t^*\mathcal{M}_t]_{ii} = \sum_{t=1}^k[P_tGQ_dGQ_tG^{*}Q_dG^*P_t]_{ii} \le  \left\| \sum_{t=1}^kP_tG\widetilde{Q}GQ_tG^{*}\widetilde{Q}^*G^*P_t \right\| \prec \left\| H \right\| \prec 1.
\end{align*}
The same bound applies to other terms in $\mathcal{A}$. 
 Hence, we have
\begin{align}
\sum_{t=1}^k\sum_{j=1}^{p_t}\frac{1}{N}[ W_t G]_{ji}\partial_{t,ji} Z = O_{\prec}\left(\frac{1}{N}\right). \label{062192}
\end{align}
Therefore, plugging (\ref{062190}), (\ref{062191}) and (\ref{062192}) into (\ref{062153}) we have
\begin{align*}
	M_1-M_2=&\E^{\chi} \bigg[ O_{\prec}\left(\frac{1}{N}\right) Z^{2n-1} \bigg] + \E^{\chi} \bigg[ O_{\prec}\left(\frac{1}{N}\right) Z^{2n-2} \bigg].
\end{align*}
Next, for $M_3$ in the RHS of (\ref{062310}), i.e, $s\geq 2$ terms, we rewrite 
\begin{align*}
& \frac{1}{N^{\frac{s+1}{2} }} \sum_{t=1}^k\sum_{j=1}^{p_t}\partial_{t,ji}^s\{ [ W_t G]_{ji}Z^{2n-1}\}\notag\\
&= \frac{1}{N^{\frac{s+1}{2} }}\sum_{r=1}^{s} \sum_{t=1}^k\sum_{j=1}^{p_t}\sum_{\substack{s_0\geq 0,s_1,\cdots,s_r> 0\\s_0+s_1+\cdots+s_r=s}}(\partial_{t,ji}^{s_0}[ W_t G]_{ji})\partial_{t,ji}^{s_1}Z \cdots \partial_{t,ji}^{s_r}Z Z^{(2n-1-r)\vee 0}\notag\\
&=: \frac{1}{N^{\frac{s+1}{2} }}\sum_{r=1}^{s} \mathfrak{b}_{s,r} Z^{(2n-1-r)\vee 0} .
\end{align*}
When $s=r$, i.e., $s_0=s_1-1=\cdots=s_r-1=0$, we have 
\begin{align*}
\frac{1}{N^{\frac{s+1}{2} }}\mathfrak{b}_{s,r}
=\frac{1}{N^{\frac{r+1}{2} }}\sum_{t=1}^k\sum_{j=1}^{p_t}[ W_t G]_{ji}\partial_{t,ji}Z\cdots\partial_{t,ji} Z\prec\frac{1}{N^{\frac{r+1}{2} }}\sum_{t=1}^k\sum_{j=1}^{p_t}|[ W_t G]_{ji}||[\mathcal{A}]_{ji}| =O_{\prec}\left(\frac{1}{N^{\frac{r+1}{2} }}\right).
\end{align*}
When $s=r+1$, i.e.,  $s_0+s_1-1+\cdots+s_r-1=1$, we have
\begin{align*}
\frac{1}{N^{\frac{s+1}{2} }}\mathfrak{b}_{s,r}=& \frac{O(1)}{N^{\frac{r+2}{2} }}\sum_{t=1}^k\sum_{j=1}^{p_t}\left((\partial_{t,ji}[ W_t G]_{ji})\partial_{t,ji} Z\cdots \partial_{t,ji} Z+ [ W_t G]_{ji}\partial^2_{t,ji} Z\cdots\partial_{t,ji} Z\right)\\
\prec&\frac{1}{N^{\frac{r+1}{2} }}\sum_{t=1}^k\sum_{j=1}^{p_t}\frac{1}{\sqrt{N}}|[ W_tG]_{ji}| +  \frac{1}{N^{\frac{r+1}{2} }}\sum_{t=1}^k\sum_{j=1}^{p_t}\frac{1}{\sqrt{N}}|[ \mathcal{A}]_{ji}|\\
\leq& \frac{1}{N^{\frac{r+1}{2} }} \sqrt{\sum_{t=1}^k[G^*Q_tG]_{ii}\sum_{t=1}^k \frac{p_t}{N}}  + \frac{1}{N^{\frac{r+1}{2} }} \sqrt{\sum_{t=1}^k[\mathcal{A}^*\mathcal{A}]_{ii}\sum_{t=1}^k \frac{p_t}{N}}= O_{\prec}\left(\frac{1}{N^{\frac{r+1}{2} }}\right).
\end{align*}
When $s-r\geq 2$, we have 
\begin{align*}
\frac{1}{N^{\frac{s+1}{2} }}\mathfrak{b}_{s,r}\prec \frac{1}{N^{\frac{r+1}{2} }}\sum_{t=1}^k\sum_{j=1}^{p_t}\frac{1}{N} 
=O_{\prec}\left(\frac{1}{N^{\frac{r+1}{2} }}\right).
\end{align*}
In summary, we have the coefficient of $Z^{2n-1-r}$ is of order $O_{\prec}(N^{-(r+1)/2})$. Substituting $Z\cdot \Xi$ for $Z$  and then applying Young's inequality as we did in the proof of \eqref{trPG}, we can get
\begin{align}\label{sPGiiMoment}
	\E \big[ \mathsf{Z}^{n,n}\big] = O_{\prec} \left( N^{-n} \right).
\end{align}
Using Markov's inequality, we obtain \eqref{sPGii} from \eqref{sPGiiMoment}.

\subsection{Proof of Lemma \ref{sumij}}
We prove the first estimate in \eqref{SUMWGP}. To this end, we introduce
\begin{align}
	&Z_1 = \sum_{i=1}^{N}[ X_t G]_{ji} - \tr P_t G\sum_{i=1}^{N}[ W_t G]_{ji} + \tr G\sum_{i=1}^{N}[ W_t G]_{ji}, \quad Z_2 = \sum_{i=1}^{N}[ X_t P_t G]_{ji} - (1-\frac{p_t}{N})\sum_{i=1}^{N}[ W_t G]_{ji} \label{061601}
\end{align}
and we further set
\begin{align*}
	\mathsf{Z}_1^{p,q} = (Z_1\cdot\Xi)^p(\bar{Z}_1\cdot\Xi)^q, \quad \mathsf{Z}_2^{p,q} = (Z_2\cdot\Xi)^p(\bar{Z}_2\cdot\Xi)^q.
\end{align*}
Our aim is to estimate $Z_1$ and $Z_2$, and via these estimates we can get an estimate of $\sum_{ij}^{(t)}[W_tG]_{ji}$ at the end. 
For any fixed $n$, and sufficiently large $l$, following similar remainder estimates as (\ref{061490}), we have
\begin{align*}
	\E \big[ \mathsf{Z}_1^{n,n} \big] =& \E^{\chi} \bigg[ \sum_{i=1}^{N}\sum_{d=1}^N[ X_t]_{jd} [ G]_{di}\mathsf{Z}_1^{n-1,n}\bigg] + \E^{\chi} \bigg[\tr (I-P_t) G\sum_{i=1}^{N}[ W_t G]_{ji}\mathsf{Z}_1^{n-1,n}\bigg]\\
	=&\frac{1}{N}\E^{\chi} \bigg[\sum_{i=1}^{N}\sum_{d=1}^N(\partial_{t,jd}[G]_{di})\mathsf{Z}_1^{n-1,n}\bigg]+ \E^{\chi}\bigg[\tr(I-P_t) G\sum_{i=1}^{N}[ W_t G]_{ji}\mathsf{Z}_1^{n-1,n}\bigg]\\
	&+ \frac{O(1)}{N}\E^{\chi} \bigg[\sum_{i=1}^{N}\sum_{d=1}^N [G]_{di}\left(\partial_{t,jd}Z_1\right) \mathsf{Z}_1^{n-2,n}\bigg]+  \frac{O(1)}{N}\E^{\chi} \bigg[\sum_{i=1}^{N}\sum_{d=1}^N [G]_{di}\left(\partial_{t,jd}\bar{Z}_1\right) \mathsf{Z}_1^{n-1,n-1}\bigg]\\
	&+\sum_{s\ge 2}^l O\left( \frac{1}{N^{\frac{s+1}{2}}} \right)\E^{\chi}\bigg[\sum_{i=1}^{N}\sum_{d=1}^N\partial^s_{t,jd}\left\{ [ G]_{di}\mathsf{Z}_1^{n-1,n} \right\}\bigg] + O_{\prec}(1)\notag\\
	=:& T_{11} + T_{12} + T_{13} + T_{14} + T_{15} + + O_{\prec}(1).
\end{align*}
By direct calculation, we have
\begin{align*}
	&\frac{1}{N} \sum_{i=1}^{N}\sum_{d=1}^N(\partial_{t,jd}[G]_{di})+ \left( \tr (I-P_t) G\right)\sum_{i=1}^{N}[ W_t G]_{ji} =\frac{1}{N}\sum_{i=1}^{N}[ G P_t G W'_t -  G^2 W'_t]_{ij} \\
	& \leq   \sqrt{\frac{p_t}{N}} \left\|G P_t G W'_t -  G^2 W'_t  \right\| = O_{\prec}\left(\sqrt{\frac{p_t}{N}} \right).
\end{align*}
This gives
\begin{align*}
	T_{11} + T_{12}= \E^{\chi}\bigg[ O_{\prec}\left(\sqrt{\frac{p_t}{N}} \right)\mathsf{Z}_1^{n-1,n}\bigg].
\end{align*}

Further, we have
\begin{align}
\partial_{t,jd}Z_1=&\sum_{u=1}^{N}1_{jj}G_{du}-\sum_{u=1}^{N}[X_tG(I-P_t)]_{jd}[W_tG]_{ju}-\sum_{u=1}^{N}[X_tGW_t]_{jj}[(I-P_t)G]_{du} \notag\\
&-\frac{2}{N}([(I-P_t)GW_t']_{dj}+[(I-P_t)G(I-P_t)GW_t']_{dj})\sum_{u=1}^{N}[W_tG]_{ju} \notag\\
&-\sum_{u=1}^{N}[W_t]_{jd}[W_tG]_{ju}-\sum_{u=1}^{N}[W_tG(I-P_t)]_{jd}[W_tG]_{ju}-\sum_{u=1}^{N}[W_tGW_t]_{jj}[(I-P_t)G]_{du}. \label{0615101}
\end{align}
By applying Cauchy-Schwarz inequality on each term, and using the boundedness of $\| G\|, \|X_t \|$ and $\|(X_tX_t')^{-1} \|$, one can easily get
\begin{align}
	\partial_{t,jd}Z_1 = O_{\prec}(\sqrt{N}). \label{jdZ1}
\end{align}

By direct calculation, for $s > 1$, we can also find that $\partial^s_{t,jd}Z_1$ can be bounded by a finite sum of the terms with the following form,
\begin{align}
 \sum_{u=1}^N(|[A]_{ju}| + |[B]_{du}|)\cdot (|[D]_{jj}| + |[E]_{dj}|) + \frac{1}{N}\sum_{u=1}^{N}[F]_{uj} = O_{\prec}(\sqrt{N})
\end{align}
where  $A$, $B$, $D$, $E$ and $F$ are some matrices bounded by $O_\prec(1)$ in spectral norm. Here we used Cauchy-Schwarz for the $u$-sum and $v$-sum. This gives, for and fixed $s \ge 1$,
\begin{align}
	\partial^s_{t,jd}Z_1 = O_{\prec}(\sqrt{N}). \label{jdsZ1}
\end{align}

In addition, by (\ref{0615101}), it is easy to check that $\frac{1}{N}\sum_{i=1}^N\sum_{d=1}^N[ G]_{di}\partial_{t,jd}Z_1$ is a linear combination of the terms of the following forms
\begin{align*}
	\frac{1}{N}\sum_{i=1}^{N}\sum_{d=1}^N[ G]_{di}\sum_{u=1}^{N}A_{jj}B_{du}, \qquad\frac{1}{N}\sum_{i=1}^{N}\sum_{d=1}^N[ G]_{di}\sum_{u=1}^{N}C_{jd}D_{ju}, \qquad  \frac{1}{N}\sum_{i=1}^{N}\sum_{d=1}^N[ G]_{di}E_{jd}\sum_{u=1}^{N}F_{ju}
\end{align*}
for some matrices $A,B,C,D,E,F$ with bounded operator norm with high probability. Further, we can bound each term as the following,
\begin{align}
&\frac{1}{N}\sum_{i=1}^{N}\sum_{d=1}^N[ G]_{di}\sum_{u=1}^{N}A_{jj}B_{du}=\frac{1}{N}\sum_{i,u=1}^N[B' G]_{ui}A_{jj}=O_{\prec}(1) \notag\\
& \frac{1}{N}\sum_{i=1}^{N}\sum_{d=1}^N[ G]_{di}\sum_{u=1}^{N}C_{jd}D_{ju}=\frac{1}{N}\sum_{i=1}^{N}[CG]_{ji}\sum_{u=1}^{N}D_{ju}=O_{\prec}(1) \notag\\
&\frac{1}{N}\sum_{i=1}^{N}\sum_{d=1}^N[ G]_{di}E_{jd}\sum_{u=1}^{N}F_{uj}=\frac{1}{N}\sum_{i=1}^{N}[EG]_{ji}\sum_{u=1}^{N}F_{ju}=O_{\prec}(1), \label{0615150}
\end{align}
where we used 
$\big|\sum_{iu}[B' G]_{ui} \big|= \big|\mathds{1}' B' G \mathds{1}\big| \le N\|B' G \|,
$
and similar bounds apply to all the other sums in (\ref{0615150}).
Therefore,
\begin{align*}
T_{13}=\E^{\chi}\big[O_{\prec}\left(1 \right)\mathsf{Z}_1^{n-2,n}\big],\quad T_{14}=\E^{\chi}\big[O_{\prec}\left(1 \right)\mathsf{Z}_1^{n-1,n-1}\big]
\end{align*}

Next, we  consider $T_{15}$. Since the complex conjugate takes no effect in the remaining estimations, we drop the complex conjugate for simplicity in this local part. Specifically,  we will work with $Z_1^{2n-1}$ instead of $\mathsf{Z}_1^{n-1,n}$ in the following derivation, for notational brevity.  Here we also drop the $\Xi$-factors first and put it back at the last step.

For $s \ge 2$, we rewrite
\begin{align*}
	&\frac{1}{N^{\frac{s+1}{2}}} \E^{\chi} \bigg[ \sum_{i=1}^{N}\sum_{d=1}^N\partial^s_{t,jd}\left\{ [ G]_{di}Z_1^{2n-1} \right\} \bigg]\\
	=&\frac{1}{N^{\frac{s+1}{2}}}\sum_{r=0}^s\E^{\chi} \bigg[\sum_{\substack{s_0 + s_1 + \cdots +s_r = s \\ s_0\geq 0, s_1,s_2,\cdots,s_r>0}}\sum_{i=1}^{N}\sum_{d=1}^N\partial^{s_0}_{t,jd}[ G]_{di}\partial^{s_1}_{t,jd}Z_1\cdots\partial^{s_r}_{t,jd}Z_1Z_1^{(2n-1-r)\vee 0}\bigg].
\end{align*}
\noindent
When $r = 0$, the coefficient of $Z_1^{2n-1}$ can be written as
\begin{align*}
	\sum_{s\ge 2}^l O\left(\frac{1}{N^{\frac{s+1}{2}}}\right)\sum_{i=1}^{N}\sum_{d=1}^N\partial_{t,jd}^s[ G]_{di}.
\end{align*}
By direct calculation, for $s \ge 0$, we find that $\sum_{i=1}^{N}\sum_{d=1}^N|\partial_{t,jd}^s[ G]_{di}|$ can be bounded by a finite sum of the terms with the following form,
\begin{align}
 \sum_{i=1}^{N}\sum_{d=1}^N|[A]_{di}| + \sum_{i=1}^{N}\sum_{d=1}^N|[B]_{ji}| = O_{\prec}\left( N^{\frac32}\right), \label{jdGdi}
\end{align}
where  $A$ and $B$ are some matrix bounded by $O_\prec(1)$ in operator norm. Here we used Cauchy-Schwarz for the $d$-sum and $i$-sum for the first and second terms respectively.  Hence, we have
\begin{align*}
	\sum_{s \ge 2}^lO\left(\frac{1}{N^{\frac{s+1}{2}}}\right)\sum_{i=1}^{N}\sum_{d=1}^N\partial_{t,jd}^s[ G]_{di} = O_{\prec}( 1).
\end{align*}
When $r = 1$, the coefficient of $Z_1^{2n-2}$ can be bounded using (\ref{jdZ1}) and (\ref{jdGdi}). It reads
\begin{align*}
&\sum_{s \ge 2}^lO\left(\frac{1}{N^{\frac{s+1}{2}}}\right)\sum_{i=1}^{N}\sum_{d=1}^N\sum_{s_0+s_1=s}\partial_{t,jd}^{s_0}[ G]_{di}\partial_{t,jd}^{s_1}Z_1\prec \sum_{s \ge 2}^l \frac{\sqrt{N}}{N^{\frac{3}{2}}}\sum_{i=1}^{N}\sum_{d=1}^N\sum_{s_0=1}^s\left|\partial_{t,jd}^{s_0}[ G]_{di}\right|  
\prec \sqrt{N},
\end{align*}
When $r \ge 2$, the coefficient of $Z_1^{2n - 1 - r}$ can be written as
\begin{align*}
\sum_{s \ge 2}^l\frac{1}{N^{\frac{s+1}{2}}}\sum_{i=1}^{N}\sum_{d=1}^N\partial_{t,jd}^{s_0}[ G]_{di}\partial_{t,jd}^{s_1}Z_1 \cdots \partial_{t,jd}^{s_r} Z_1 
	\prec\sum_{s \ge 2}^l \frac{N^{\frac{r}{2}}p_t^{r}}{N^{\frac{s+1}{2}}} \sum_{i=1}^{N}\sum_{d=1}^N\left| \partial_{t,jd}^{s_0}[ G]_{di} \right| \prec  N.
\end{align*}

Therefore, substituting $Z_1 \cdot \Xi$ for $Z_1$ (up to negligible error), then by Young's inequality, we can obtain
$
	\E \big[\mathsf{Z}_1^{n,n}\big] = O_{\prec}\left(  N^{\frac{2n}{3}}\right).
$
By similar arguments, we can also obtain
$
		\E \big[\mathsf{Z}_2^{n,n}\big] = O_{\prec}\left(  N^{\frac{2n}{3}}\right).
$
Hence, by Markov inequality, we have 
$
|Z_1|, |Z_2|\prec N^{\frac13}. 
$
Recall the definition in (\ref{061601}) and notice the fact $X_tG=X_tP_tG$.  Hence, taking the difference between $Z_1$ and $Z_2$ leads to
\begin{align*}
	\left( \tr  G + 1 - \tr  P_t G - \frac{p_t}{N} \right)\sum_{i=1}^{N}[ W_t G]_{ji} = O_{\prec}( N^{\frac{1}{3}} ).
\end{align*}
The coefficient of the sum can be bounded below as follows. By Propositions \ref{keyProp} and \ref{keyprop2}, we have
\begin{align*}
	\tr  G + 1 - \tr  P_t G - \frac{p_t}{N} =& m_{\boxplus}(z)(1-\omega_t(z)) - \frac{p_t}{N}+ O_{\prec}\left(\frac{1}{N}\right) =\frac{\omega_t(z)-1}{\omega_t(z)}\left(1- \frac{p_t}{N}\right) + O_{\prec}\left(\frac{1}{N}\right).
\end{align*}
Further, notice
\begin{align*}
	\left| \frac{\omega_t(z)-1}{\omega_t(z)} \right|\left(1 - \frac{p_t}{N} \right) > c\left(1 - \frac{p_t}{N}\right), 
\end{align*}
for some positive constants $c$ when $z \in (\bar{\gamma}_1^0)^{+} \cup  (\bar{\gamma}_2^0)^{+} $ with $\hat{y} \in (0,1)$ and $z \in (\bar{\gamma}_1)^{+} \cup  (\bar{\gamma}_2)^{+}$ with $\hat{y} \in (0,\infty)$. Here we need Lemma \ref{Flbound} to get $|\omega_t(z)|, |m_{\mu_t}(\omega_t(z))| \sim 1$, and then from the fact that $ |m_{\mu_t}(\omega_t(z))| \sim 1$ we can obtain $|\omega_t(z) - 1| \sim 1$ since $m_{\mu_t}(\omega_t(z)) = y_t/(1-\omega_t(z)) - (1-y_t)/\omega_t(z)$. Therefore,  we can obtain
\begin{align}
	\frac{1}{\sqrt{N}}\sum_{i=1}^{N}[ W_t G]_{ji} = O_{\prec}\left( \frac{1}{N^{\frac{1}{6}}} \right). \label{061620}
\end{align}

Define
\begin{align*}
	Z = \sum_{i=1}^{N}[W_t']_{ij} + \tr (X_tX_t')^{-1}\sum_{i=1}^{N}[W_t']_{ij}.
\end{align*}
Then similar to the proof of the first estimate in \eqref{SUMWGP}, we can get $|Z| \prec N^{\frac13}$. This details are omitted. Using the estimate of $\gamma_{t2}$ in Lemma \ref{Lemma error estimates}, we can obtain (\ref{SUMw}).

For  the second estimate in \eqref{SUMWGP}, we consider the two terms separately. For the first term,  we can write $\sum_{i=1}^{N}[ P_t G W_t']_{ij}$ as $\sum_{i=1}^{N}[ X_t'W_t G W_t']_{ij}$ and conduct the cumulant expansion w.r.t. $X_t'$-entries, and the remaining estimate is similar to the first estimate in  \eqref{SUMWGP}. For the second  term $\sum_{i=1}^{N}(\tr( X_t X_t')^{-1}-\tr Q_t G)[(I-P_t) G W_t']_{ij} + \tr Q_tG[W_t']_{ij}$, we can use  (\ref{061620}), (\ref{SUMw}), the estimate of $\gamma_{t2}$ in Lemma \ref{Lemma error estimates} and Proposition \ref{keyprop2} directly. The details are omitted. Therefore, we completed the proof of Lemma \ref{sumij}.

\subsection{Proof of Lemma \ref{EtrXXtrQG}}
We start from the following identity, 
\begin{align*}
	\E^{\chi} \Big[\Tr\bbP_t \bla {\rm e}^{\mathrm{i} xL_N^1(f)}\bra \Big] = 0
\end{align*}
since $\Tr\bbP_t=p_t$ almost surely. 
By the cumulant expansion, we have
\begin{align}
	0&=\E^{\chi} \Big[  \Tr\bbP_t\bla {\rm e}^{ixL_N^1(f)}\bra \Big]  = \E^{\chi}\Big[ \sum_{ij}^{(t)}\bbX_{t,ij}[\bbW_t]_{ji} \bla {\rm e}^{\mathrm{i}x\la L_N^1(f)\ra }\bra\Big] =\frac{1}{N}\E^{\chi}\Big[ \sum_{ij}^{(t)}(\partial_{t,ji}[\bbW_t]_{ji})\bla {\rm e}^{\mathrm{i} xL_n(f)}\bra\Big]\notag\\
&+\frac{1}{N}\E^{\chi}\Big[\sum_{ij}^{(t)}[\bbW_t]_{ji}\partial_{t,ji}{\rm e}^{\mathrm{i}x\la L_N^1(f)\ra }\Big] +\frac{1}{{N}^{3/2}}\sum_{ij}^{(t)}\kappa_3^{t,j}\sum_{a=0}^2{2\choose a} \E^{\chi} \Big[  (\partial^a_{t,ji}[\bbW_t]_{ji}) \partial^{2-a}_{t,ji}\bla {\rm e}^{\mathrm{i}x\la L_N^1(f)\ra } \bra\Big] \notag\\
&+\frac{1}{{N}^{2}}\sum_{ij}^{(t)}\kappa_4^{t,j}\sum_{a=0}^3{3\choose a}\E^{\chi}\Big[ (\partial^a_{t,ji}[\bbW_t]_{ji})\partial^{3-a}_{t,ji}\bla {\rm e}^{\mathrm{i}x\la L_N^1(f)\ra }\bra\Big]
+ O_{\prec}\left({\frac{p_t}{N^{\frac32}}}  \right) \notag\\
&=:\mathcal{P}_{t1}+\mathcal{P}_{t2}+\mathcal{P}_{t3}+\mathcal{P}_{t4} + O_{\prec}\left({\frac{p_t}{N^{\frac32}}}\right), \label{0616100}
\end{align}
where the remainder terms are estimated similarly to the estimates of $\mathsf{E}_{a}, a=1,2,3$ in Lemma \ref{Error in CLT}. The estimation for $\mathcal{P}_{t3}$ and $\mathcal{P}_{t4}$ are also similar to $\mathsf{I}_{t3}$ and $\mathsf{I}_{t4}$ in (\ref{061352}) and (\ref{061353}).  From which we can obtain,
\begin{align}
	{\mathcal{P}_{t3}+\mathcal{P}_{t4} =O_{\prec}\left(\frac{p_t}{N^{7/6}}\right)}. \label{Pt3Pt4}
\end{align}
Next we estimate $\mathcal{P}_{t1}$ and $\mathcal{P}_{t2}$. For $\mathcal{P}_{t1}$, we have
\begin{align}
	\mathcal{P}_{t1} =& \left(1-y_t-{\frac{1}{N}}\right)\E^{\chi} \Big[  \Tr \left( \bbX_t\bbX_t' \right)^{-1}\bla {\rm e}^{\mathrm{i}x\la L_N^1(f)\ra }\bra\Big]
	=\left(1-y_t\right)\E^{\chi} \Big[ \Tr \left( \bbX_t\bbX_t' \right)^{-1}\bla {\rm e}^{\mathrm{i}x\la L_N^1(f)\ra }\bra\Big] \notag\\
	& + \E^{\chi} \Big[ \bla\tr \left( \bbX_t\bbX_t' \right)^{-1}\bra {\rm e}^{\mathrm{i}x\la L_N^1(f)\ra }\Big] 
	=\left(1-y_t\right)\E^{\chi} \Big[ \Tr \left( \bbX_t\bbX_t' \right)^{-1}\bla {\rm e}^{\mathrm{i}x\la L_N^1(f)\ra }\bra\Big]  + O_{\prec}\left(\sqrt{\frac{p_t}{N^3}}\right), \label{Pt1}
\end{align}
where in the last step we used (\ref{trXX}). For $\mathcal{P}_{t2}$, we have 
\begin{align}
	\mathcal{P}_{t2} =& \frac{\mathrm{i}x}{N}\E^{\chi} \Big[ \sum_{ij}^{(t)} [\bbW_t]_{ji}\left( \partial_{t,ji}L_N^1(f)\right)\bla {\rm e}^{\mathrm{i}x\la L_N^1(f)\ra }\bra\Big]\notag \\ =& \frac{-x}{2\pi N}\sum_{ij}^{(t)}\oint_{\bar{\gamma}_1^0}\E^{\chi}  \Big[[\bbW_t]_{ji}\left(\partial_{t,ji}\Tr\bbG(\tilde{z})\right)\bla{\rm e}^{\mathrm{i}x\la L_N^1(f)\ra }\bra\Big] f(\tilde{z}){\rm d}\tilde{z}\notag\\
	=& \frac{-x}{2\pi N}\sum_{ij}^{(t)}\oint_{\bar{\gamma}_1^0}\E^{\chi}  \Big[[\bbW_t]_{ji}\left(2[W_tG^2(\tilde{z})]_{ji} - 2[W_tG^2(\tilde{z})P_t]_{ji}\right)\bla{\rm e}^{\mathrm{i}x\la L_N^1(f)\ra }\bra\Big] f(\tilde{z}){\rm d}\tilde{z} \notag\\ 
	=& \frac{-x}{2\pi N}\oint_{\bar{\gamma}_1^0}\E^{\chi}  \Big[\left(2\tr G^2(\tilde{z})Q_t - 2\tr P_tG^2(\tilde{z})Q_t  \right)\bla{\rm e}^{\mathrm{i}x\la L_N^1(f)\ra }\bra\Big] f(\tilde{z}){\rm d}\tilde{z}  = 0.\label{Pt2}
\end{align}
Hence, by plugging \eqref{Pt1}, \eqref{Pt2} and \eqref{Pt3Pt4} into (\ref{0616100}), we obtain
\begin{align*}
	0 =& \E^{\chi}  \Big[ \Tr\bbP_t   \bla {\rm e}^{\mathrm{i}x\la L_N^1(f)\ra } \bra \Big] = (1-y_t)\E^{\chi}  \Big[ \Tr \left( \bbX_t\bbX_t' \right)^{-1} \bla {\rm e}^{\mathrm{i}x\la L_N^1(f)\ra }\bra\Big] + {O_{\prec}\left( \frac{p_t}{N^{7/6}} \right)}.
\end{align*}
This verifies (\ref{0616101})
Next, we can start with another identity,
$
		\E ^{\chi} \Big[ \Tr\bbP_t\bbG\bbP_t \bla {\rm e}^{ixL_n(f)} \bra \Big] = \E^{\chi}  \Big[ \Tr\bbP_t\bbG\bla {\rm e}^{ixL_n(f)} \bra \Big].
$
By performing cumulant expansion at the LHS of the above identity, and then following the same argument as the estimations of $\mathcal{P}_{tb}$, $b= 1,2,3,4$, we can obtain (\ref{0616102}). 

Lastly, we consider (\ref{062520}). By direct calculation, we have
	\begin{align*}
		&\sum_{ij}^{(t)} \left| \partial_{t,ji}^2\Tr \bbG(z)\right|=2\sum_{ij}^{(t)} \left| \partial_{z} \left\{ \partial_{t,ji} \left([\bbW_t\bbG(z)(I-\bbP_t)]_{ji}  \right) \right\}\right| \\
		=&\sum_{ij}^{(t)} \left| \partial_{z} \left\{  \left( [(\bbX_t\bbX_t')^{-1}]_{jj} -[\bbW_t{\bbG}\bbW'_t]_{jj}\right) [(I-\bbP_t){\bbG}(I-\bbP_t)]_{ii}-[\bbW_t{\bbG}\bbW'_t]_{jj}[I-\bbP_t]_{ii} \right\}\right| \\
		&+ \text{sum of off-diagonal entries} \\
		=& \sum_{ij}^{(t)} \left| \partial_{z} \left\{  \left( [(\bbX_t\bbX_t')^{-1}]_{jj} -[\bbW_t{\bbG}\bbW'_t]_{jj}\right) [(I-\bbP_t){\bbG}(I-\bbP_t)]_{ii}-[\bbW_t{\bbG}\bbW'_t]_{jj}[I-\bbP_t]_{ii} \right\}\right|+ O_{\prec}\left( p_t\right)\\
		=& \sum_{ij}^{(t)} \left| \partial_{z} \left\{ \frac{N}{p_t} \left( \tr(\bbX_t\bbX_t')^{-1} -\tr Q_tG\right) (\tr G - \tr P_tG)-\frac{N(1-y_t)}{p_t} \tr Q_tG\right\}\right| + O_{\prec}(\sqrt{N}p_t) = O_{\prec}(\sqrt{N}p_t). 
	\end{align*}
	where we used the tracial quantities to replace the $ii$ and $jj$ entries (c.f. Propositions \ref{keyProp}, \ref{keyprop2} and \ref{DiagApproByFC}) such that the error $O_{\prec}(\sqrt{N}p_t) = \sum_{ij}^{(t)}O_{\prec}(1/\sqrt{N})$ emerges here. In the last step, we used the estimate of $\gamma_{t1}$ in Lemma \ref{Lemma error estimates}. Here we also omitted the expression of the sum of off-diagonal entries, whose estimate can be done similarly to (\ref{062301}).
This completes the proof of Lemma \ref{EtrXXtrQG}.

\subsection{Proof of Lemma \ref{remainderLemma}}
 Let $\Delta_1 : =(x - X_{t,ij})E_{ij}$. Here $E_{ij}$ represents the matrix with $ij$ entry being $1$ and the other entries are $0$. We have
	\begin{align*}
		X_t^{(ij)}\left( X_t^{(ij)} \right)' = (X_t+\Delta_1)(X_t' + \Delta_1')=: X_tX_t' + \tilde{\Delta}_1. 
	\end{align*}
	The spectral norm of the matrix $\tilde{\Delta}_1$ can be bounded as follows,
	\begin{align*}
		\| \tilde{\Delta}_1\| \lesssim \|X_t \|\|\Delta_1 \| + \|\Delta_1 \|^2 \prec  |x| +  \big|X_{t, ij}\big|
	\end{align*}
	where we used Lemma \ref{bound for covariance}. 
	 Therefore, when $|x| \leq  N^{-\frac12+\epsilon}$, by Weyl's inequality, we have
	\begin{align*}
		\Big|\sigma_l\Big(X_t^{(ij)}\Big( X_t^{(ij)} \Big)' \Big) -\sigma_l\Big(X_tX_t' \Big)\Big| \leq \| \tilde{\Delta}_1\| \prec N^{-\frac12+\epsilon} , \quad l=1,\cdots,p_t,
	\end{align*}
	which implies the first and second estimates in (\ref{4113}). 
	Similarly, let
$
		\Delta_2 := \big(X_t^{(ij)}( X_t^{(ij)})' \big)^{-1} - (X_tX_t')^{-1}.
$
	 We have
	\begin{align*}
		P^{(ij)}_t =& (X_t' + \Delta_1')\left( (X_tX_t')^{-1} + \Delta_2 \right)(X_t+\Delta_1)  = :P_t + \tilde{\Delta}_2.
	\end{align*}
	Similarly,  when $|x| \leq  N^{-1/2+\epsilon}$, it is easy to check
$
		\|\tilde{\Delta}_2 \| \prec N^{-\frac12+\epsilon}.	
$
	Therefore, by the Weyl's inequality, we have
	\begin{align*}
		\big|\sigma_l(H-z\mathrm{I}) - \sigma_l(H^{(t,ij)}-z\mathrm{I}) \big| \leq \|\tilde{\Delta}_2 \| \prec N^{-\frac12+\epsilon}, \quad l = 1,\cdots,N
	\end{align*}
	when $|x| \leq  N^{-1/2+\epsilon}$. This completes the proof of (\ref{4113}).

\section{Proof of Theorem \ref{1stLimit} (Cases 1 $\&$ 2)}\label{Sec Proof of 1.9}
\subsection{Proof of case 1}
 Our goal is to show that, for each fixed $z \in \mathbb{C}^{+}$,
\begin{align}\label{traceGGap}
	\max_t|\omega_t^c(z)-\omega_t(z)|, \quad\left|m_N(z)- m_{\boxplus}(z)\right| \prec \frac{1}{N}.
\end{align}
To this end,  we will first establish a perturbed system of \eqref{Phi2} for $\omega_t^c(z)$'s, and then show that \eqref{traceGGap} holds for $z$ with sufficiently large $\Im z$. Finally, by a continuity argument, we will show that \eqref{traceGGap} holds for each fixed $z \in \mathbb{C}^{+}$. For notational simplicity, we write $G := G(z)$,  and $m_N := m_N(z) $ for short, whenever there is no confusion. 

By Lemma \ref{Lemma error estimates}, we have
	\begin{align}
		&\tr(\bbX_t\bbX_t')^{-1} = \frac{y_t}{1-y_t} + O_{\prec}\left( {\sqrt{\frac{p_t}{N^3}}} \right).\label{trXX} \\
		&\tr \bbP_t\bbG = (1-y_t)\tr \bbQ_t\bbG + O_{\prec}\left( {\sqrt{\frac{p_t}{N^3}}} \right), \label{trPGP}\\
		&\tr \bbP_t\bbG  = \left(\tr(\bbX_t\bbX_t')^{-1}-\tr\bbQ_t\bbG \right)\left(\tr\bbG-\tr\bbP_t\bbG\right) 
		+ O_{\prec}\left( {\sqrt{\frac{p_t}{N^3}}} \right). \label{trPG}
	\end{align}

Plugging \eqref{trXX} and (\ref{trPGP}) into \eqref{trPG}, and using (\ref{Appsub1}), 
we obtain,  for any $t \in [\![ k ]\!]$,
\begin{align*}
	1 + \omega_t^c(z)m_N - \left(\frac{y_t}{1-y_t} -\frac{1}{1-y_t}(1 + \omega_t^c(z)m_N)\right)(m_N-1-\omega_t^c(z)m_N) = O_{\prec}\left( {\sqrt{\frac{p_t}{N^3}}} \right),
\end{align*}
where we used the trivial bound $|\tr \bbG|, |\tr\bbP_t\bbG|\leq 1/\Im z$. 
Rearranging the terms, we have
\begin{align*}
	m_N + \frac{\omega_t^c(z)m_N(m_N-1-\omega_t^c(z)m_N)}{1-y_t} = O_{\prec}\left( {\sqrt{\frac{p_t}{N^3}}} \right).
\end{align*}
Multiplying both sides by $1-y_t$, and then dividing both sides by 
\begin{align}
\theta_t := m_N\omega_t^c(z)(1-\omega_t^c(z)), \label{def of theta_t}
\end{align} 
we have
\begin{align*}
	m_N = \frac{y_t}{1-\omega_t^c(z)} - \frac{1-y_t}{\omega_t^c(z)} +  O_{\prec}\left(\frac{1}{|\theta_t|} {\sqrt{\frac{p_t}{N^3}}} \right)= -\frac{1}{F_{\mu_t}(\omega_t^c(z))} + O_{\prec}\left(\frac{1}{|\theta_t|} {\sqrt{\frac{p_t}{N^3}}} \right).
\end{align*}
Using \eqref{Appsub2}, we can get
\begin{align*}
	(k-1)F_{\mu_t}(\omega_t^c(z))- \omega_1^c(z) - \omega_2^c(z) - \cdots - \omega_k^c(z) + z = O_{\prec}\left(\frac{|F^2_{\mu_t}(\omega_t^c(z))|}{|\theta_t|} {\frac{k-1}{N}\sqrt{\frac{p_t}{N}}} \right), \quad z \in \mathbb{C}^{+}.
\end{align*}
Since $k$ is fixed in case 1,  the $k$-dependence of the RHS can be neglected, i.e., 
\begin{align*}
	(k-1)F_{\mu_t}(\omega_t^c(z))- \omega_1^c(z) - \omega_2^c(z) - \cdots - \omega_k^c(z) + z = O_{\prec}\left(\frac{|F^2_{\mu_t}(\omega_t^c(z))|}{|\theta_t|} {\frac{1}{N}} \right), \quad z \in \mathbb{C}^{+}.
\end{align*}
The above system of equations for all $t \in [\![k ]\!]$ forms a perturbed version of \eqref{Phi2}. Here we keep the factor $|F^2_{\mu_t}(\omega_t^c)| / |\theta_t|$ whose bound is a priori unknown. It will be  proved that this factor is  of order $1$ during the stability analysis of the perturbed system. Based on this perturbed system of \eqref{Phi2}, we can proceed to the next step.

\noindent $\bullet${\it Sufficiently large $\Im z$.} We start with the regime when $\Im z>0$ is sufficiently large. 
Recall that by the definition of the reciprocal Stieltjes transform of $\mu_t$, $t \in [\![k ]\!]$, we have for any $\omega \in \mathbb{C}^{+}$, 
\begin{align}\label{Fmut}
	F_{\mu_t}(\omega) =& -\frac{1}{m_{\mu_t}(\omega)} = \omega -y_t + \frac{y_t(1-y_t)}{1-y_t-\omega}.
\end{align}
Taking derivatives with respect to $\omega$ gives
\begin{align}
	&F'_{\mu_t}(\omega) = 1 + \frac{y_t(1-y_t)}{(1-y_t-\omega)^2},\label{Fpmut} \\
	&F^{(n)}_{\mu_t}(\omega) =\frac{y_t(1-y_t)n!}{(1-y_t-\omega)^{n+1}}, \quad n \ge 2. \label{Fnmut}
\end{align}
To prove (\ref{traceGGap}) for $z$ with sufficiently large $\Im z$, we first show the following stability result for the system (\ref{Phi1}) at $(\omega_1^c(z),\cdots, \omega_k^c(z))$ for sufficiently large $\Im z$ . 
\begin{lemma}\label{GammaUp1}
	There exists sufficiently large $\eta_0$, such that for any $z \in \mathbb{C}^{+}$ with $\Im z \ge \eta_0$, we have 
	\begin{align*}
			\Gamma_{\mu_1,\cdots,\mu_k}(\omega^c_1(z), \cdots,\omega^c_k(z)) \le C.
	\end{align*}
	where $C$ is a strictly positive constant independent of $\Im z$ and $N$.
\end{lemma}
\begin{proof}[Proof of Lemma \ref{GammaUp1}]	
	Recall the definition in (\ref{DPhi}) and (\ref{6001}), we have by resolvent identity,
	\begin{align}
		(\mathrm{D}\Phi)^{-1}(\omega_1^c(z),\cdots,\omega_k^c(z))  = \mathcal{D}^{-1}(\omega^c) + \frac{\mathcal{D}^{-1}(\omega^c)\mathds{1}\mathds{1}'\mathcal{D}^{-1}(\omega^c)}{1-\mathds{1}\mathcal{D}^{-1}(\omega^c)\mathds{1}'}. \label{0617105}
	\end{align}
	Therefore,
	\begin{align*}
		\Gamma_{\mu_1,\cdots,\mu_k}(\omega^c_1(z), \cdots,\omega^c_k(z)) \le \| \mathcal{D}^{-1}(\omega^c) \|_{(\infty,\infty)} + \left\| \frac{D^{-1}(\omega^c)\mathds{1}\mathds{1}'\mathcal{D}^{-1}(\omega^c)}{1-\mathds{1}'\mathcal{D}^{-1}(\omega^c)\mathds{1}} \right\|_{(\infty,\infty)}.
	\end{align*}
	For the first term, by Lemma \ref{Imomegatc}, we have for sufficiently large $\Im z$, $\Im \omega_t^c$ ($t \in [\![ k]\!]$) are also large. As a result,
	\begin{align*}
		\| \mathcal{D}^{-1}(\omega^c) \|_{(\infty,\infty)} =& \max_t\frac{1}{|(k-1)F_{\mu_t}'(\omega^c_t(z))|} 
		= \max_t (k-1)^{-1}\left|1 + \frac{y_t(1-y_t)}{(1-y_t-\omega_t^c(z))^2}\right|^{-1}\\
		=& \max_t (k-1)^{-1}\left|1 + O(|\Im \omega_t^c(z)|^{-2}) \right|^{-1} 
		\le C
	\end{align*}
	as $\Im z \to \infty$. For the second term, since 
$
		\left\|\mathcal{D}^{-1}(\omega^c)\mathds{1}\mathds{1}'\mathcal{D}^{-1}(\omega^c) \right\|_{(\infty,\infty)}  \leq  k \left\|\mathcal{D}^{-1}(\omega^c)\right\|^2_{(\infty,\infty)} \le C,
$
	it suffices to show 
$
		\left|1-\mathds{1}'\mathcal{D}^{-1}(\omega^c)\mathds{1}\right| > c
$
	for some strictly positive constant $c$ which is independent of $\Im z$ and $N$. We have indeed
	\begin{align*}
		\left|1-\mathds{1}'\mathcal{D}^{-1}(\omega^c)\mathds{1}\right| =& \Big|1 - \frac{1}{k-1}\sum_{t=1}^k\frac{1}{F_{\mu_t}'(\omega^c_t(z))} \Big|=\Big|1 - \frac{1}{k-1}\sum_{t=1}^k\left(1 + O(|\Im \omega_t^c(z)|^{-2}) \right)^{-1} \Big| \\     
		=&\Big|\frac{-1}{k-1}+ O(\max_t|\Im \omega_t^c(z)|^{-2})  \Big| > \frac{1}{2(k-1)}	.
	\end{align*}
	 when $\Im z$ is sufficiently large. This completes the proof of Lemma \ref{GammaUp1}.
\end{proof}
Next, we continue the proof of (\ref{traceGGap}) in the regime when $\Im z>0$ is sufficiently large. From \eqref{Fnmut}, we also have for $\omega\in \mathbb{C}^+$ with large $\Im \omega$, 
	\begin{align}
		|F_{\mu_t}^{(2)}(\omega)| = O\left( |\omega|^{-3} \right). 
	\end{align}
	Hence the matrix of second derivatives of $\Phi$ given by
	\begin{align*}
		&\mathrm{D}^2\Phi\left(\omega_1,\cdots,\omega_k \right):= \left( \frac{\partial^2\Phi}{\partial\omega_1^2}, \cdots, \frac{\partial^2\Phi}{\partial\omega_k^2}  \right)= {\rm diag}\Big( (k-1)F_{\mu_{t}}^{(2)}\left(\omega_{t}\right)\Big)_{t=1}^k
	\end{align*}
	satisfies $\|\mathrm{D}^2\Phi\left({\omega}^c_1(z),\cdots,{\omega}^c_k(z) \right) \| = O\left( |\Im z|^{-3} \right)$, as $\Im z \to \infty$. For any $t \in [\![k ]\!]$, denote by
	\begin{align}
		r_t(z) := (k-1)F_{\mu_t}(\omega_t^c(z))- \omega_1^c(z) -  \cdots - \omega_k^c(z) + z = O_{\prec}\left(\frac{|F^2_{\mu_t}(\omega_t^c(z))|}{|\theta_t|} {\frac{1}{N}} \right), \label{def of rt}
	\end{align} 
	and $r(z) := (r_1(z), \cdots, r_k(z))^{T}$. As $\Im z \to \infty$, by Lemma \ref{Imomegatc}, we have with high probability
	\begin{align}
		\frac{|\theta_t|}{|F^2_{\mu_t}(\omega_t^c(z))|} = \left| \frac{m_N(y_t-1+\omega_t^c(z))^2}{\omega_t^c(z)(1-\omega_t^c(z))}\right|  =|m_N|\left| \frac{(y_t-1+\omega_t^c(z))^2}{\omega_t^c(z)(1-\omega_t^c(z))}\right| \ge \frac{1}{2}|m_N| \gtrsim \frac{1}{\Im z}. \label{6010}
	\end{align}
	Therefore, we can obtain
$
		|r_t(z)| \prec |\Im z|/N
$
for all   $t \in [\![k ]\!]$.
	
	Hence, choosing $\eta_0 > 0$ sufficiently large, and using Lemma \ref{GammaUp1} with the above bound we can achieve that
	\begin{align*}
		s_0 :=&  \Gamma_{\mu_1,\cdots,\mu_k}(\omega_1^c(z), \cdots,\omega_k^c(z)) \|\mathrm{D}^2\Phi\left({\omega}^c_1(z),\cdots,{\omega}^c_k(z) \right) \|_{(\infty,\infty)} \| (\mathrm{D}\Phi)^{-1}\cdot\Phi (\omega_1^c(z), \cdots,\omega_k^c(z))  \|_{\infty} \\
		\le & C\|\mathrm{D}^2\Phi\left({\omega}^c_1(z),\cdots,{\omega}^c_k(z) \right) \|_{(\infty,\infty)} \|(\mathrm{D}\Phi)^{-1}\left({\omega}^c_1(z),\cdots,{\omega}^c_k(z) \right) \|_{(\infty,\infty)}\|r(z) \|_{\infty}\prec  \frac{1}{N|\Im z|^{2}},
	\end{align*}
which implies that $s_0<\frac12$  with high probability on the domain $\{z \in \mathbb{C}^{+}: \Im z \ge \eta_0 \}$. By the Newton-Kantorovich theorem (Theorem \ref{NewtonKantorvich}) with $$b \equiv \Gamma_{\mu_1,\cdots,\mu_k}(\omega_1^c(z), \cdots,\omega_k^c(z)) \|\mathrm{D}^2\Phi\left({\omega}^c_1(z),\cdots,{\omega}^c_k(z) \right) \|_{(\infty,\infty)},$$ and $L \equiv \| (\mathrm{D}\Phi)^{-1}\cdot\Phi (\omega_1^c(z), \cdots,\omega_k^c(z))  \|_{\infty}$,  we have that there are for every such $z$ unique $\hat{\omega}_t(z)$'s satisfying $$\Phi_{\mu_1,\cdots,\mu_k} (\hat{\omega}_1(z),\cdots,\hat{\omega}_k(z),z) =0,$$ with
	\begin{align}\label{Gap2}
		\left|{\omega}^c_{t}(z)-\hat{\omega}_{t}(z)\right| \leq &\frac{1 - \sqrt{1 - 2s_0}}{s_0} \| (\mathrm{D}\Phi)^{-1} (\omega_1^c(z), \cdots,\omega_k^c(z))  \|_{(\infty,\infty)}\|r(z) \|_{\infty}\notag \\
		\leq &  C \|r(z) \|_{\infty}  \prec  {\frac{|\Im z|}{N}}.
	\end{align}
	Finally, using Lemma \ref{Imomegatc},  we note that $\Im \hat{\omega}_t(z) = \Im \hat{\omega}_t(z) - \Im{\omega}^c_t(z) + \Im {\omega}^c_t(z) \ge \Im z $ for $\Im z \ge \eta_0$ with high probability as $N \to \infty$. It further follows that 
	$\Gamma_{\mu_1,\cdots,\mu_k} (\hat{\omega}_1(z),\cdots,\hat{\omega}_k(z)) \neq 0$ for all $z \in \mathbb{C}^{+}$ with $\Im z \ge \eta_0$. Thus $\hat{\omega}_t(z)$ is analytic on the domain $\{ z \in \mathbb{C}^{+}: \Im z \ge \eta_0 \}$ since $F_{\mu_t}$ is. Finally, using \eqref{Fmut} with $\omega = \hat{\omega}_t(z)$, we see that
	\begin{align} 
		\lim_{ \eta \to \infty}\frac{\Im \hat{\omega}(\mathrm{i}\eta)}{\mathrm{i}\eta} = 1.
	\end{align}
	Thus by the uniqueness claim in Proposition \ref{freeaddprop}, $\hat{\omega}_t(z)$ agrees with $\omega_t(z)$ on the domain $\{ z \in \mathbb{C}^{+}: \Im z \ge \eta_0 \}$. Therefore, \eqref{Gap2} implies that
	\begin{align}\label{Gap3}
		\left|{\omega}^c_{t}(z)-{\omega}_{t}(z)\right| \prec \frac{1}{N},
	\end{align}
	for  all fixed $z \in \mathbb{C}^{+}$ with $\Im z \ge \eta_0$. Therefore, subtracting \eqref{Appsub2} from \eqref{sub3}, we have
	\begin{align}\label{DPer}
		\sum_{t=1}^{k}(\omega_t^c(z) - \omega_t(z)) = -\frac{k-1}{m_N(z)} + \frac{k-1}{m_\boxplus(z)} = (k-1)\frac{m_\boxplus(z) - m_N(z)}{m_N(z)m_\boxplus(z)}.
	\end{align}
	As a result,
	\begin{align}\label{initial}
		\left|m_N(z) - m_\boxplus(z) \right| \prec \frac{1}{k-1} \sum_{t=1}^{k}|\omega_t^c(z) - \omega_t(z)| \prec \frac{1}{N} ,
	\end{align}
	for  all fixed $z \in \mathbb{C}^{+}$ with $\Im z \ge \eta_0$. 
	
	Next, taking \eqref{initial} as an input, we can use the continuity argument to obtain the bound for each fixed $z \in \mathbb{C}^{+}$.\\
	
\noindent
$\bullet$ {\it Any fixed $z\in \mathbb{C}^+$.}
For any fixed $z \in \mathbb{C}^{+}$, we have the following stability of the system (\ref{Phi2}) at $(\omega_1(z),\cdots,\omega_k(z))$, which is an extension of Lemma \ref{GammaUp1}.
\begin{lemma}\label{GammaUP2}
		For any fixed $k\in \mathbb{N}$, and  any fixed $z \in \mathbb{C}^{+}$, let $\omega_t(z), t \in [\![k ]\!]$ be the subordination functions. Then there is a strictly positive constant $C$ independent of  $N$  such that
	\begin{align}
		\Gamma_{\mu_1,\cdots,\mu_k}(\omega_1(z),\omega_2(z), \cdots,\omega_k(z)) \le C, \label{8005}
	\end{align}
	and
	\begin{align}
		|\omega_t'(z)| \le C, \quad t \in [\![k ]\!].\label{8006}
	\end{align}
	Here the constant $C$ may depend on $ z$.
\end{lemma}
\begin{proof}[Proof of Lemma \ref{GammaUP2}]
	To prove (\ref{8005}), by Cramer's rule, it suffices to show that
$
		|\det (\mathrm{D}\Phi)| > c
$
	for some constant $c>0$.  By basic algebra, we have 
	\begin{align}
	|\det (\mathrm{D}\Phi)|=&\bigg|\sum_{r=0}^k (r-1)(k-1)^{k-r-1}\sum_{i_1\neq \ldots\neq i_r} f_{i_1}\cdots f_{i_r}\bigg| \notag\\
	\geq & (k-1)^{k-1}-\sum_{r=1}^k (r-1)(k-1)^{k-r-1}\sum_{i_1\neq \ldots\neq i_r} |f_{i_1}|\cdots |f_{i_r}| \label{061550}
	\end{align}
	where $f_t: = \left(k-1\right)\left( F_{\mu_t}'(\omega_t(z)) - 1\right)$. 
	
By \eqref{sub2} and \eqref{sub3}, we have,
\begin{align*}
	\left|f_t\right| = (k-1)\left| \frac{\Im F_{\mu_t}(\omega_t(z)) - \Im \omega_t(z)}{\Im \omega_t(z)} \right| = \frac{\sum_{s\neq t}\Im \omega_s(z) - (k-2)\Im \omega_t(z) }{\Im \omega_t(z)} - \frac{\Im z}{\Im \omega_t(z)}, \quad t \in [\![ k]\!].
\end{align*}
For any fixed $z\in \mathbb{C}^+$, we have by $\Im \omega_t(z) \ge \Im z > 0$,
\begin{align*}
	\left|f_t\right| < \frac{\sum_{s\neq t}\Im \omega_s(z) - (k-2)\Im \omega_t(z) }{\Im \omega_t(z)} =: a_t, \quad t \in [\![ k]\!].
\end{align*} 
This together with (\ref{061550}) implies 
\begin{align*}
	\left|\det \left( \mathrm{D}\Phi
\right) \right| 
> &(k-1)^{k-1}-\sum_{r=1}^k (r-1)(k-1)^{k-r-1}\sum_{t_1\neq \ldots\neq t_r} a_{t_1}\cdots a_{t_r}=\det\Big({\rm diag}(a_t+k-1)_{t=1}^k-\mathds{1}\mathds{1}' \Big)\notag\\
=&\frac{1}{\prod_{t}^k\Im \omega_ t(z)}\det
\begin{bmatrix}
    \sum_{t\neq1}\Im \omega_t(z) \ & -\Im \omega_1(z) &  \dots  & -\Im \omega_1(z) \\
    -\Im \omega_2(z) & \sum_{t\neq 2}\Im \omega_t(z) &  \dots  & -\Im \omega_2(z) \\
    \vdots & \vdots & \ddots & \vdots  \\
    -\Im \omega_k(z) & -\Im \omega_k(z) &  \dots  & \sum_{t\neq k}\Im \omega_t(z)
\end{bmatrix}= 0,
\end{align*}
where the last step follows from the fact that the matrix has linearly dependent columns. Since here our $z$ and $k$ are fixed, it is easy to show the positiveness of $|\det({\rm D}\Phi)|$ is effective, i.e., $|\det({\rm D}\Phi)|>c$ for some positive constant $c$ which may depend on $k$ and $z$. This completes the proof of (\ref{8005}).

The estimates in (\ref{8006}) follow by differentiating the equation (\ref{Phi2}) with respect to $z$, we get
\begin{align}
	\mathrm{D}\Phi \cdot \omega'(z) = \mathds{1}, \label{80007}
\end{align}
where $\omega'(z) := (\omega_1'(z),\cdots,\omega_k'(z))$. Together with (\ref{8005}), we get (\ref{8006}) by inverting (\ref{80007}).

\end{proof}

To prove \eqref{traceGGap} for each fixed $z \in \mathbb{C}^{+}$, we start from $z_0=\Re z+\mathrm{i} \eta_0$ with sufficiently large $\eta_0>\Im z$, and decrease the imaginary part step by step with a step size $N^{-2}$, so that $z_0$ goes to $z$ eventually after $O(N^2)$ steps. We aim to show that the bound \eqref{traceGGap} remains hold after each step, by using the continuity of the subordination functions and the Stieltjes transforms. In the sequel we show the details of this continuity argument for the first step. The remaining steps are the same.  Let  $z_1=\Re z+\mathrm{i}(\eta_0 - N^{-2})$.

In the sequel, till Section \ref{s.case 2}, all the omitted $z$-variables are $z_1$.  Notice that for the perturbed system
	\begin{align*}
		(k-1)F_{\mu_t}(\omega_t^c(z_1)) - \omega_1^c(z_1)  - \cdots -\omega_k^c(z_1) + z_1 = r_t(z_1).
	\end{align*}
	Let $\Omega_t:= \omega_t^c(z_1) - \omega_t(z_1)$, performing Taylor expansion for $F_{\mu_t}(\omega_t^c(z_1))$ around $\omega_t(z_1)$ we get
	\begin{align*}
		(k-1)F'_{\mu_t}(\omega_t(z_1))\Omega_t - \Omega_1 - \Omega_2 - \cdots -\Omega_k = r_t(z_1) - (k-1)\sum_{n\ge 2}\frac{1}{n!}F_{\mu_t}^{(n)}(\omega_t(z_1))\Omega_t^n.
	\end{align*}
	Then with (\ref{DPhi}) we have
	\begin{align*}
		\Omega = (\mathrm{D}\Phi)^{-1} \cdot r(z_1) -\sum_{n\ge 2}\frac{k-1}{n!} (\mathrm{D}\Phi)^{-1}\cdot \Omega_F^n
	\end{align*}
	where $\Omega:=(\Omega_1,\ldots, \Omega_k)$ and $\Omega_F^n := \left( F_{\mu_1}^{(n)}(\omega_1(z_1))\Omega_1^n, \cdots,  F_{\mu_k}^{(n)}(\omega_k(z_1))\Omega_k^n\right)^{\mathrm{T}}$.  Taking $\ell_{\infty}$ norm  on both sides, and using Lemma \ref{GammaUP2} and \eqref{Fnmut}, we have
	\begin{align}
		\|\Omega \|_{\infty} \le C\|r(z_1) \|_{\infty} + C\sum_{n \ge 2} \max_t\frac{\|\Omega \|_{\infty}^n}{|1 - y_t - \omega_t(z_1)|^n}. \label{090909}
	\end{align}
	Since $\Im \omega_t(z) \ge \Im z > 0$, we have $|1 - y_t - \omega_t(z_1)| \ge |\Im \omega_t(z_1)| > 0$. Also, by the definition of $\omega_t^c(z)$ (c.f., (\ref{Appsub1})), we have
	\begin{align*}
		|(\omega_t^c(z))| = \left|\frac{\tr G(z)\tr P_tG^2(z) - (\tr P_tG(z) -1) \tr G^2(z)}{(\tr G(z))^2}\right| \le C
	\end{align*}
	with high probability. Here we used the fact that $|\tr G(z)| \sim 1$ when $\Im z > 0$ and $\|H\|$ is bounded with high probability (c.f., Lemma \ref{Gbound}).
	Therefore, together with (\ref{Gap3}) at $z_0$ and the continuity of $\omega_t$ (c.f. (\ref{8006})), we have for $z_1$ , 
	\begin{align}
		|\omega_t^c(z_1) - \omega_t(z_1)| \prec \frac{1}{N} + \frac{1}{N^2}.  \label{bound by continuity}
	\end{align}
	More precisely, by the definition of the stochastic domination, (\ref{bound by continuity}) implies that  for any large $D > 0$ and small $\epsilon > 0$, there exists an high probability event $\mathcal{E}_1 \equiv \mathcal{E}(z_1,\epsilon,D)$, satisfying $\mathbb{P}(\mathcal{E}_1^c) < N^{-D}$, such that on the event $\mathcal{E}_1$, we have 
	\begin{align*}
		|\omega_t^c(z_1) - \omega_t(z_1)| \le N^{\epsilon}(N^{-1} + N^{-2}) \le 2N^{-1+\epsilon},
	\end{align*}
	for any small $\epsilon > 0$. Together with the fact that $\Im \omega_t(z_1) > 0$, (\ref{090909}) can be rewritten as
	\begin{align}
				\|\Omega \|_{\infty} \le C\|r(z_1) \|_{\infty} + C\|\Omega \|_{\infty}^2 \quad \text{on}\quad \mathcal{E}_1. \label{090908}
	\end{align}
	As $N \to \infty$, we can absorb the quadratic term into the left hand side, which gives
	\begin{align*}
		\|\Omega \|_{\infty} \le C\|r(z_1) \|_{\infty} \quad \text{on}\quad \mathcal{E}_1.
	\end{align*}
	Recall the definition of $r(z)$ in (\ref{def of rt}), we have there exists an high probability event $\mathcal{E}_2$, such that
	\begin{align}
		\|r(z_1) \|_{\infty} \le N^{\epsilon}\left(\frac{|F^2_{\mu_t}(\omega_t^c(z_1))|}{|\theta_t|} {\frac{1}{N}} \right) \label{090907}
	\end{align}
	Further notice that, on the event $\mathcal{E}_1$, we have
	\begin{align}
		\frac{|\theta_t(z_1)|}{|F^2_{\mu_t}(\omega_t^c(z_1))|} = |m_N(z_1)|\left| \frac{(y_t-1+\omega_t^c(z_1))^2}{\omega_t^c(z_1)(1-\omega_t^c(z_1))}\right| = |m_{\boxplus}(z_1)|\left| \frac{(y_t-1+\omega_t(z_1))^2}{\omega_t(z_1)(1-\omega_t(z_1))}\right| + O\left( \frac{1}{N^{1-\epsilon}}\right). \label{090906}
	\end{align}
	Together with the facts that
$
		|y_t-1+\omega_t(z_1)| \ge \Im \omega_t(z_1) >0$ with $ |m_{\boxplus}(z_1)| \ge \Im m_{\boxplus}(z_1) > 0$, 
	we know 
$
		|\theta_t(z_1)|/|F^2_{\mu_t}(\omega_t^c(z_1))|> c > 0 ,
$
	on the event $\mathcal{E}_1$.
	
	Therefore, combining (\ref{090908}), (\ref{090907}) and (\ref{090906}), we have on the event $\mathcal{E}_1\cap\mathcal{E}_2$,
	\begin{align}
		\|\Omega \|_{\infty} \le CN^{1-\epsilon}.  \label{090900}
	\end{align}
	 Repeating the above procedure for $O(N^2)$ steps, we will generate a series of events, i.e., $\mathcal{E}_1,\cdots,\mathcal{E}_{CN^2}$. Therefore, we have on the event $\bigcap_{i=1}^{O(N^2)}\mathcal{E}_i$, (\ref{090900}) still holds. Since all these events are high probability events, we obtain
	\begin{align}
		\|\Omega \|_{\infty} \prec  {\frac{1}{N}}. \label{fine bound of omega}
	\end{align}
	Using \eqref{DPer}, we can again have
	\begin{align}
		|m_N(z_1) - m_{\boxplus}(z_1)| \prec \frac{1}{N}.  \label{071310}
	\end{align}
	
 Now, further by \cite{bercovici1993free}, we know  that $\mu_1\boxplus\cdots\boxplus\mu_k\Rightarrow \mu_1^\infty\boxplus\cdots\boxplus\mu_k^\infty$ when $k$ is fixed. By Stieltjes continuity theorem, this implies $m_{\boxplus}(z)\to m_{\boxplus}^\infty(z)$ for any fixed $z\in \mathbb{C}^+$, where $m_{\boxplus}^\infty(z)$ is the Stieltjes transform of  the limiting measure  $\mu_1^\infty\boxplus\cdots \boxplus \mu_k^\infty$. This together with $|m_N(z) - m_{\boxplus}(z)| \prec \frac{1}{N}$ implies the convergence of $m_N(z)$ to $m_{\boxplus}^\infty(z)$ with high probability. Using the Stieltjes continuity theorem in orther direction, we can conclude that $\mu_N$ converges weakly in probability to $\mu_1^\infty\boxplus\cdots \boxplus\mu_k^\infty$.

\subsection{Proof of case 2}\label{s.case 2} Recall the crude bounds
 in Lemma {\ref{bound for covariance}}. Further, by setting $\mathcal{W}_t = 1$ in Lemma \ref{Lemma error estimates}, we have,
	\begin{align}
	&\sum_{t=1}^k \Big(\tr P_tG- \left(\tr( X_t X_t')^{-1}-\tr Q_t G\right)\left(\tr G-\tr P_tG\right)\Big) =O_{\prec}\left(\frac{1}{{N}}\right), \label{strPG1} \\
			&\sum_{t=1}^k\Big(\frac{1}{1-y_t}\tr P_t- \tr( X_t X_t')^{-1}\Big)=O_{\prec}\left(\frac{1}{{N}}\right), \label{strP1} \\
		&\sum_{t=1}^k \Big(\frac{1}{1-y_t}\tr P_t G P_t-\tr Q_t G\Big)=O_{\prec}\left(\frac{1}{{N}}\right). \label{strPGP1}	
	\end{align}
We note here in (\ref{strPGP1}), $\tr P_tGP_t$ is actually $\tr P_tG$. However, we keep this form since only starting with $\tr P_tGP_t$, and then applying the cumulant expansion we can see an (almost) algebraic cancellation with $\tr Q_tG$.

Applying Lemma {\ref{bound for covariance}} to absorb the $O_{\prec}(p_t/N)$ terms into error term in \eqref{strPG1}, we arrive at 
\begin{align*}
	\sum_{t=1}^k \tr P_tG-\sum_{t=1}^k \left(\tr( X_t X_t')^{-1}-\tr Q_t G\right)\tr G  = O_{\prec}\left( \frac{p_{\max}}{N} \right).
\end{align*}
Plugging \eqref{strP1} and \eqref{strPGP1} into above equation with the trivial fact that $\tr P_tGP_t = \tr P_tG$, we have
\begin{align*}
	\sum_{t=1}^k \tr P_tG-\sum_{t=1}^k \left(\frac{y_t}{1-y_t} -\frac{1}{1-y_t}\tr P_tG\right)\tr G  = O_{\prec}\left( \frac{p_{\max}}{N} \right).
\end{align*}
In light of the smallness of $y_t$'s in case 2 of Theorem \ref{1stLimit}, we can also rewrite the above as
\begin{align*}
	\sum_{t=1}^k \tr P_tG-\sum_{t=1}^k \left(y_t -\tr P_tG\right)\tr G  = O_{\prec}\left( \frac{p_{\max}}{N} \right).
\end{align*}
Further using the trivial identity
$
	\sum_{t=1}^k \tr P_tG = \tr HG = 1 + z\tr G,
$
we arrive at
\begin{align*}
	1 + (z - y +1)m_N + zm_N^2 = O_{\prec}\left( \frac{p_{\max}}{N} \right).
\end{align*}
By the stability of the quadratic equations, we can conclude that
\begin{align*}
	|m_N(z) - m_{y}(z)| \prec  \frac{p_{\max}}{N}.
\end{align*}
Here $m_{y}(z)$ is the Stieljes transform of the \textit{Machenko-Pastur law} $\mu_{mp,y}$ (c.f. Theorem \ref{1stLimit}) which satisfying the following quadratic equation
\begin{align}
	1 + (z - y +1)m_{y}(z) + zm^2_{y}(z) = 0, \quad z \in \mathbb{C}^{+}. \label{071000}
\end{align}
This immediately implies the weak convergence (in probability) of  $\mu_N$ to $\mu_{mp,\hat{y}}$,  in light of the well known fact that $\mu_{mp,y}$ converges to $\mu_{mp,\hat{y}}$ weakly.

\section{Proofs of Corollaries \ref{coroschott} and \ref{corowilk}}\label{Sec Proof of 1.14}
\subsection{Simplification of variance and expectation: Proofs of (\ref{06221214}) and (\ref{0621101})}\label{SimplifyV}
We first  show the derivation of (\ref{FreeCLTfactor}). Recall Proposition \ref{freeaddprop} in the main text, we have
\begin{align*}
	m_{\boxplus}(z) = m_{\mu_t}(\omega_t(z)),\quad  t\in [\![k]\!], \quad \omega_1(z) + \omega_2(z) +\cdots +\omega_k(z) = z - \frac{k-1}{m_\boxplus(z)}.
\end{align*}
For brevity, we omit $z$ from the notation in the sequel. Taking derivative w.r.t $z$, we have
\begin{align}
	&m_{\boxplus}'=\frac{y_t\omega_t'}{(1-\omega_t)^2} + \frac{(1-y_t)\omega_t'}{\omega_t^2},\quad  t\in [\![k]\!],\label{06221220} \\
	&\omega_1' + \omega_2' + \cdots + \omega_k' = 1 + \frac{(k-1)m_{\boxplus}'}{m_{\boxplus}^2}. \label{06221221}
\end{align}
 Solving $\omega_t'$ in (\ref{06221220}) and then plugging the results into (\ref{06221221}), we have
 \begin{align}
 	m_{\boxplus}' \left(\sum_{t=1}^{k}\frac{\omega_t^2(1-\omega_t)^2}{y_t\omega_t^2 + (1-y_t)(1-\omega_t)^2}  \right) = 1 + \frac{(k-1)m_{\boxplus}'}{m_{\boxplus}^2}. \label{06221222}
 \end{align}
 Notice that
 \begin{align*}
 	y_t\omega_t^2 + (1-y_t)(1-\omega_t)^2 =& \omega_t^2+(1-y_t)(1-2\omega_t) = \omega_t^2 + \omega_t(1-m_{\boxplus}(1-\omega_t))(1-2\omega_t)\\
 	=& \omega_t(1-\omega_t)(1 - m_\boxplus +2m_\boxplus\omega_t ),
 \end{align*}
 then (\ref{06221222}) becomes
 \begin{align}
 	m_{\boxplus}' \left(-\mathfrak{u}+\mathfrak{n} \right) = 1 + \frac{(k-1)m_{\boxplus}'}{m_{\boxplus}^2}, \label{06221223}
 \end{align}
where we denote
\begin{align*}
	\mathfrak{u} := \sum_{t=1}^k\frac{\omega_t}{m_\boxplus -1- 2m_\boxplus\omega_t}, \quad \mathfrak{n}:= \sum_{t=1}^k\frac{\omega_t^2}{m_\boxplus -1- 2m_\boxplus\omega_t}.
\end{align*}

By direct calculation, we have
\begin{align*}
	\mathfrak{u} =\sum_{t=1}^k\frac{\omega_t}{m_\boxplus} - \sum_{t=1}^k\left(\frac{\omega_t}{m_\boxplus} - \frac{\omega_t}{m_\boxplus -1-2\omega_tm_\boxplus}  \right) = \sum_{t=1}^k\frac{\omega_t}{m_\boxplus} + \frac{1}{m_\boxplus}\mathfrak{u} + 2\mathfrak{n},
\end{align*}
which gives
\begin{align*}
	\left(1-\frac{1}{m_\boxplus} \right)\mathfrak{u} = \sum_{t=1}^k\frac{\omega_t}{m_\boxplus} + 2\mathfrak{n} = \frac{zm_{\boxplus}-(k-1)}{m_\boxplus^2} + 2\mathfrak{n}.
\end{align*}
Combining the above equation with (\ref{06221223}), we have
\begin{align*}
	-\frac{1}{m_\boxplus'} -\frac{k-1}{m_\boxplus^2} =\mathfrak{u} - \mathfrak{n} = \frac{1}{m_\boxplus}\mathfrak{u} + \mathfrak{n} + \frac{zm_{\boxplus}-(k-1)}{m_\boxplus^2} + 2\mathfrak{n}.
\end{align*}
Rearranging the terms, we arrive at
\begin{align*}
	\frac{1}{m_\boxplus}\mathfrak{u} + \mathfrak{n} + \frac{1}{m_\boxplus}z = -\frac{1}{m_\boxplus'}.
\end{align*}
This implies
\begin{align*}
	\left( z - \sum_{t=1}^{k}\frac{y_t - 1 -\omega_tm_{\boxplus} }{m_{\boxplus}-1-2\omega_tm_{\boxplus}} \right)^{-1} =  \left( z + \sum_{t=1}^{k}\frac{\omega_t +m_\boxplus\omega_t^2 }{m_{\boxplus}-1-2\omega_tm_{\boxplus}} \right)^{-1} =\left(z + \mathfrak{u} + m_\boxplus\mathfrak{n} \right)^{-1} = \frac{-m_\boxplus'}{m_\boxplus}.
\end{align*}

Next we show the proof of (83). For brevity, we denote $m_{\boxplus i}:=m_{\boxplus}(z_i), i = 1,2$, and $\omega_{ti}:=\omega_t(z_i), i = 1,2, t \in [\![ k ]\!]$ in the following derivation. Define
\begin{align*}
	V_t(z_1,z_2):= -\frac{m'_{\boxplus1}}{m_{\boxplus1}} \frac{1 + \frac{(\omega_{t1}m_{\boxplus1}-\omega_{t2}m_{\boxplus2})^2}{(z_1-z_2)(m_{\boxplus1} - m_{\boxplus2})} + \frac{\omega_{t1} - \omega_{t2}}{\frac{1}{m_{\boxplus1}}-\frac{1}{m_{\boxplus2}}} }{m_{\boxplus1}-1-2\omega_{t1}m_{\boxplus1}}
\end{align*}
From (\ref{06221221}), we have
\begin{align}
	\frac{m_{\boxplus1}'}{m_{\boxplus1}(m_{\boxplus1} - 1 - 2\omega_{t1}m_{\boxplus1})} = -\frac{\omega_{t1}'}{m_{\boxplus1}\omega_{t1}(1-\omega_{t1})}. \label{0621110}
\end{align}
Using (\ref{0621110}), we obtain
\begin{align}
	V_t(z_1,z_2) =& \frac{m'_{\boxplus1}}{m_{\boxplus1}}\frac{(\omega_{t1} - \omega_{t2})m_{\boxplus1}m_{\boxplus2}}{(m_{\boxplus1}-1-2\omega_{t1}m_{\boxplus1})(m_{\boxplus1}-m_{\boxplus2})} -\frac{m'_{\boxplus1}}{m_{\boxplus1}(m_{\boxplus1}-1-2\omega_{t1}m_{\boxplus1})} \notag\\
	&+ \frac{\omega_{t1}'}{m_{\boxplus1}}\frac{(\omega_{t1}m_{\boxplus1}-\omega_{t2}m_{\boxplus2})(\omega_{t1}m_{\boxplus1}-\omega_{t2}m_{\boxplus2}-m_{\boxplus1}+m_{\boxplus2})}{\omega_{t1}(1-\omega_{t1})(z_1-z_2)(m_{\boxplus1}-m_{\boxplus2})}\label{06221255} .
\end{align}
By some elementary algebra with (\ref{0621110}) and the fact $\omega_{t1} - m_{\boxplus1}(\omega_{t1}-\omega_{t1}^2) = \omega_{t2} - m_{\boxplus2}(\omega_{t2}-\omega_{t2}^2) = 1-y_t$, we can rewrite (\ref{06221255}) as 
\begin{align*}
	V_t(z_1,z_2) =& \frac{m'_{\boxplus1}}{m_{\boxplus1}} \frac{m_{\boxplus2}(\omega_{t1} - \omega_{t2})}{(z_1-z_2)(m_{\boxplus1}-m_{\boxplus2})} - \frac{m_{\boxplus2}}{m'_{\boxplus1}m_{\boxplus1}(m_{\boxplus1}-m_{\boxplus2})} + \frac{\omega'_{t1}}{\omega_{t1} - \omega_{t2}} - \frac{\omega'_{t1}}{z_1-z_2}.
\end{align*}
Summing over $t$ and then taking derivative w.r.t $z_2$, we obtain (83).

Next, we consider the simplification of the expectation. Recall from the main text the expression of $E(z)$,
\begin{align}
	E(z) = \frac{m_{\boxplus}'(z)}{m_{\boxplus}(z)}\sum_{t=1}^k \frac{\E \Big[\tr \bbP_t\bbG\bbP_t\bbG \Big] - \E\Big[\tr \bbP_t\bbG^2  \Big] }{m_{\boxplus}(z)-1-2\omega_t(z)m_{\boxplus}(z)} + O_{\prec}\left(\frac{1}{N^{\frac{1}{6}}} \right). \label{0621112}
\end{align}
For brevity, we omit $z$ from the notations in the sequel. 

Further, by Proposition 6.1, we have 
\begin{align}
	&\E \Big[\tr \bbP_t\bbG\bbP_t\bbG \Big] = \frac{(1+\omega_tm_{\boxplus})(m_{\boxplus}-1-\omega_tm_{\boxplus}) - (m_{\boxplus}\omega_t^2+\omega_t)(\omega_tm_{\boxplus})'}{m_{\boxplus}-1-2m_{\boxplus}\omega_t} + O_{\prec}\left( {\sqrt{\frac{p_t}{N^3}} }\right),\notag\\
	&\E\Big[\tr \bbP_t\bbG^2  \Big] = (\omega_tm_{\boxplus})' + O_{\prec}\left( {\sqrt{\frac{p_t}{N^3}} }\right). \label{0621111}
\end{align}

Therefore, substituting (\ref{0621110}) with $m_{\boxplus1}$  replaced by  $m_{\boxplus}$ and (\ref{0621111}) in (\ref{0621112}), we have  for any fixed $t \in [\![k]\!]$,
\begin{align*}
	E(z)
	=& \sum_{t=1}^k\frac{\omega_t'}{m_{\boxplus}} \frac{1}{\omega_t(1-\omega_t)}\bigg[ \frac{(1+\omega_tm_{\boxplus})(m_{\boxplus}-1-\omega_tm_{\boxplus}) - (m_{\boxplus}\omega_t^2+\omega_t)(\omega_tm_{\boxplus})'}{m_{\boxplus}-1-2m_{\boxplus}\omega_t} -  (\omega_tm_{\boxplus})' \bigg]+ O_{\prec}\left(\frac{1}{N^{\frac{1}{6}}} \right)\\
	=&\sum_{t=1}^k\frac{\omega_t'}{m_{\boxplus}} \frac{1}{\omega_t(1-\omega_t)}\bigg[1 - \omega_t'  - \frac{\omega_t'}{m_{\boxplus}}\left( m_{\boxplus}^2 + m_{\boxplus}\frac{(\omega_t-1)(\omega_t'm_{\boxplus} + m_{\boxplus}'\omega_t) + \omega_t'}{\omega_t} \right)  \bigg]+ O_{\prec}\left(\frac{1}{N^{\frac{1}{6}}} \right)
\end{align*}
Further using
\begin{align*}
	\frac{(\omega_t-1)(\omega_t'm_{\boxplus} + m_{\boxplus}'\omega_t) + \omega_t'}{\omega_t} = -m_{\boxplus}\omega_t',
\end{align*}
we have
\begin{align*}
	E(z)
	=&\sum_{t=1}^k\frac{\omega_t'}{m_{\boxplus}} \frac{1}{\omega_t(1-\omega_t)}\bigg[1 - \omega_t'  - \frac{\omega_t'}{m_{\boxplus}}\left( m_{\boxplus}^2 - m_{\boxplus}^2\omega_t' \right)  \bigg]+ O_{\prec}\left(\frac{1}{N^{\frac{1}{6}}} \right) \\
	=&\sum_{t=1}^k\frac{\omega_t'}{m_{\boxplus}} \frac{1}{\omega_t(1-\omega_t)}\bigg[(1 - \omega_t')(2m_{\boxplus}\omega_t + 1 - m_{\boxplus})\\&\qquad + (1 - \omega_t')(m_{\boxplus}-2m_{\boxplus}\omega_t) - \frac{\omega_t'}{m_{\boxplus}}\left( m_{\boxplus}^2 - m_{\boxplus}^2\omega_t' \right)  \bigg]+O_{\prec}\left(\frac{1}{N^{\frac{1}{6}}} \right)\\
	=&\sum_{t=1}^k \frac{m_{\boxplus}'}{m_{\boxplus}}(1-\omega_t') + \frac{\omega_t'}{m_{\boxplus}} \frac{(1-\omega_t')m_{\boxplus}}{m_{\boxplus}'\omega_t(1-\omega_t)}\bigg[ m_{\boxplus}' - 2\omega_tm_{\boxplus}'-\omega_t'm_{\boxplus} \bigg]+ O_{\prec}\left(\frac{1}{N^{\frac{1}{6}}} \right)\\
	=&\sum_{t=1}^k\frac{m_{\boxplus}'}{m_{\boxplus}}(1-\omega_t') + \frac{\omega_t'}{m_{\boxplus}} \frac{(1-\omega_t')m_{\boxplus}}{2m_{\boxplus}'}\left( \frac{m_{\boxplus}-1-2m_{\boxplus}\omega_t}{\omega_t(1-\omega_t)} \right)'+ O_{\prec}\left(\frac{1}{N^{\frac{1}{6}}} \right)\\
	=&\sum_{t=1}^k\frac{m_{\boxplus}'}{m_{\boxplus}}(1-\omega_t') - \frac{\omega_t'}{m_{\boxplus}} \frac{(1-\omega_t')m_{\boxplus}}{2m_{\boxplus}'}\left(\frac{m_{\boxplus}'}{\omega_t'} \right)'+ O_{\prec}\left(\frac{1}{N^{\frac{1}{6}}} \right)\\
	=&\sum_{t=1}^k(1-\omega_t')\left( \frac{m_{\boxplus}'}{m_{\boxplus}} - \frac{m''_{\mu}}{2m'_{\mu}} \right) + \frac{\omega''_t}{2\omega'_t} - \frac{\omega''_t}{2}+O_{\prec}\left(\frac{1}{N^{\frac{1}{6}}} \right)\\
	=& \frac{1}{2}\bigg[ \sum_{t=1}^k\frac{\omega''_t}{\omega_t'}+(k-1)\left( \frac{2m'_{\boxplus}}{m_{\boxplus}}-\frac{m_{\boxplus}''}{m_{\boxplus}'} \right) \bigg]+O_{\prec}\left(\frac{1}{N^{\frac{1}{6}}} \right).
\end{align*}
where in the third step we used (\ref{0621110}). The completes the proof of (90).

\subsection{Proof of Corollary \ref{coroschott}}
\subsubsection{Expectation for Schott's statistics: $f(x) = x^2$}
Notice that we can first rewrite
\begin{align}
	a_{x^2} =& -\frac{1}{4\pi i}\oint_{\gamma} z^2 \frac{\partial \log\left(\frac{\prod_{t=1}^k\omega_t'(z)}{[(-1/m_{\boxplus})']^{(k-1)}}\right)}{\partial z} {\rm d}z \notag \\
	=&-\frac{1}{4\pi i}\oint_\gamma z^2 \sum_{t=1}^k \frac{\partial \log \left( \frac{\omega_t'm_{\boxplus}^2}{m_{\boxplus}'} \right)}{\partial z} {\rm d}z-\frac{1}{4\pi i}\oint_\gamma z^2 \sum_{t=1}^k \frac{\partial \log \left( \frac{m_{\boxplus}'}{m_{\boxplus}^2} \right)}{\partial z} {\rm d}z. \label{062203}
\end{align}

By (\ref{sub1}) and (\ref{F2}), we can get
\begin{align}\label{Eomega}
	\omega_t = \frac{m_{\boxplus}-1 - \sqrt{(m_{\boxplus}+1)^2-4y_tm_{\boxplus}}}{2m_{\boxplus}},
\end{align}
where the square root takes the branch that $\Im \sqrt{w}>0$ when $\Im w>0$. 
Here $\omega_t$ and $m_{\boxplus}$ are functions of $z$. Taking derivative with respect to $z$, we have,
\begin{align}\label{Eomegap}
	\frac{\omega_t'}{m_{\boxplus}'} = \frac{1}{2m_{\boxplus}^2} + \frac{m_{\boxplus}+1-2y_tm_{\boxplus}}{2m_{\boxplus}^2\sqrt{(1+m_{\boxplus})^2-4y_tm_{\boxplus}}}
\end{align}
Summing over index $t$ for \eqref{Eomega} and using (\ref{sub3}), we have
\begin{align}\label{zofm}
	z = -\frac{1}{m_{\boxplus}} + \sum_{t=1}^k \frac{m_{\boxplus}+1 - \sqrt{(m_{\boxplus}+1)^2-4y_tm_{\boxplus}}}{2m_{\boxplus}}.
\end{align} 
Taking derivative with respect to $z$, we have,
\begin{align}\label{Emp}
	\frac{m'_{\boxplus}}{m_{\boxplus}^2} = \frac{1}{1 - \sum_{t=1}^k\left(\frac{1}{2} - \frac{m_{\boxplus}+1-2y_tm_{\boxplus}}{2\sqrt{(1+m_{\boxplus})^2-4y_tm_{\boxplus}}} \right)}.
\end{align}
Therefore, substituting \eqref{Eomegap} and \eqref{Emp} into (\ref{062203}), we have
\begin{align}
	a_{x^2} =& -\frac{1}{4\pi i}\oint_\gamma z^2 \sum_{t=1}^k \frac{\partial \log \left( \frac{1}{2} + \frac{m_{\boxplus}+1-2y_tm_{\boxplus}}{2\sqrt{(1+m_{\boxplus})^2-4y_tm_{\boxplus}}} \right)}{\partial z} {\rm d}z \\&+\frac{1}{4\pi i}\oint_\gamma z^2 \sum_{t=1}^k \frac{\partial \log \left( 1 - \sum_{t=1}^k\left(\frac{1}{2} - \frac{m_{\boxplus}+1-2y_tm_{\boxplus}}{2\sqrt{(1+m_{\boxplus})^2-4y_tm_{\boxplus}}} \right) \right)}{\partial z} {\rm d}z.\label{063001}
\end{align}
By choosing $\gamma$ so that $|z|$ is sufficiently large on $\gamma$, we have that $m_{\boxplus}$ is sufficiently small. Then we see that the log functions are analytic on $\gamma$. Further, we introduce the shorthand notations
\begin{align}
	g_t(m_\boxplus):=\frac{m_{\boxplus}+1-2y_tm_{\boxplus}}{\sqrt{(1+m_{\boxplus})^2-4y_tm_{\boxplus}}} ,\qquad \ell_t(m_\boxplus):=   \sqrt{(m_{\boxplus}+1)^2-4y_tm_{\boxplus}}. \label{062236}
\end{align}
Then, for (\ref{063001}), performing integration by parts, and using (\ref{zofm}), (\ref{Emp}) and ${\rm d}z = m_{\boxplus}'dm_{\boxplus}$, we have
\begin{align}
	a_{x^2} =& -\frac{1}{2\pi i} \sum_{t=1}^k\oint_{\gamma_m} z\log \left( \frac{1 + g_t(m_\boxplus)}{2}  \right) \frac{1}{m_{\boxplus}'}{\rm d}m_{\boxplus} \notag\\
	&+\frac{1}{2\pi i} \oint_{\gamma_m} z\log \left( 1 - \sum_{t=1}^k\frac{1- g_t(m_\boxplus)}{2} \right)\frac{1}{m_{\boxplus}'}{\rm d}m_{\boxplus} \notag\\
	=&-\frac{1}{2\pi i} \sum_{t=1}^k\oint_{\gamma_m}  \left( -\frac{1}{m_{\boxplus}^3} + \sum_{s=1}^k \frac{m_\boxplus+1-\ell_s(m_\boxplus)}{2m_{\boxplus}^3} \right) \log \left( \frac{1 + g_t(m_\boxplus)}{2} \right) \left(1 - \sum_{s=1}^k\frac{1 - g_t(m_\boxplus) }{2}  \right){\rm d}m_{\boxplus} \notag\\
	&+\frac{1}{2\pi i} \oint_{\gamma_m} \left( -\frac{1}{m_{\boxplus}^3} + \sum_{s=1}^k \frac{m_\boxplus+1-\ell_s(m_\boxplus)}{2m_{\boxplus}^3}\right) \log \left( 1 - \sum_{t=1}^k\frac{1- g_t(m_\boxplus)}{2} \right)\left(1 - \sum_{s=1}^k\frac{1 - g_s(m_\boxplus) }{2}  \right){\rm d}m_{\boxplus} \notag\\
	=&: I_1 + \sum_{s=1}^k I_{2,s} + \sum_{s=1}^k I_{3,s} +\sum_{i=1}^k\sum_{j=1}^kI_{4,ij} + J_1 + \sum_{s=1}^k J_{2,s} + \sum_{s=1}^k J_{3,s} +\sum_{i=1}^k\sum_{j=1}^kJ_{4,ij}, \label{062210}
\end{align}
where $\gamma_m$ is the contour of $m$ after change of variable, and 
\begin{align*}
	&I_1 = \frac{1}{2\pi i}  \sum_{t=1}^k\oint_{\gamma_m} \frac{1}{m_{\boxplus}^3}\log \left( \frac{1+ g_t(m_\boxplus)}{2}  \right)\left(1-\frac{k}{2} \right)dm_{\boxplus},\\
	&I_{2,s} = -\frac{1}{2\pi i} \sum_{t=1}^k \oint_{\gamma_m}  \frac{m_\boxplus+1-\ell_s(m_\boxplus)}{2m_{\boxplus}^3} \log \left( \frac{1+ g_t(m_\boxplus)}{2} \right)\left(1-\frac{k}{2} \right){\rm d}m_{\boxplus},\\
	&I_{3,s} = \frac{1}{2\pi i}  \sum_{t=1}^k\oint_{\gamma_m} \frac{g_s(m_\boxplus)}{m_{\boxplus}^3}\log \left( \frac{1+ g_t(m_\boxplus)}{2} \right){\rm d}m_{\boxplus},\\
	&I_{4,ij} = -\frac{1}{2\pi i}  \sum_{t=1}^k\oint_{\gamma_m}\frac{(m_\boxplus+1-\ell_i(m_\boxplus))g_j(m_\boxplus)}{2m_{\boxplus}^3}\left( \frac{1+ g_t(m_\boxplus)}{2} \right){\rm d}m_{\boxplus},\\
	&J_1 = -\frac{1}{2\pi i}  \oint_{\gamma_m} \frac{1}{m_{\boxplus}^3}\log \left( 1 - \sum_{t=1}^k\frac{1- g_t(m_\boxplus)}{2} \right)\left(1-\frac{k}{2} \right){\rm d}m_{\boxplus},\\
	&J_{2,s} = \frac{1}{2\pi i} \sum_{t=1}^k \oint_{\gamma_m} \frac{m_\boxplus+1-\ell_s(m_\boxplus)}{2m_{\boxplus}^3}  \log \left( 1 - \sum_{t=1}^k\frac{1- g_t(m_\boxplus)}{2} \right)\left(1-\frac{k}{2} \right){\rm d}m_{\boxplus},\\
	&J_{3,s} = -\frac{1}{2\pi i}  \sum_{t=1}^k\oint_{\gamma_m} \frac{g_s(m_\boxplus)}{m_{\boxplus}^3}\log \left( 1 - \sum_{t=1}^k\frac{1- g_t(m_\boxplus)}{2} \right){\rm d}m_{\boxplus},\\
	&J_{4,ij} = \frac{1}{2\pi i}  \sum_{t=1}^k\oint_{\gamma_m}   \frac{(m_\boxplus+1-\ell_i(m_\boxplus))g_j(m_\boxplus)}{2m_{\boxplus}^3} \log \left( 1 - \sum_{t=1}^k\frac{1- g_t(m_\boxplus)}{2}  \right){\rm d}m_{\boxplus}.
	\end{align*}
Note that the contour $\gamma_m$ of $m_\boxplus$ will not enclose the pole $0$ more than once when $z$ goes through $\gamma$. This is guaranteed by the fact that $\Im m_\boxplus$ has the same sign as $\Im z$.
When $|z|$ is sufficiently large, $|m_{\boxplus}|$ is sufficiently small, and thus the only pole inside ${\gamma_m}$ is $m_{\boxplus} = 0$. Hence,  by the residue theorem, we have
\begin{align*}
	&I_1 = \sum_{t=1}^{k}y_t(1-y_t),\quad I_{3,s} = \frac{1}{2}\sum_{t=1}^{k}y_t(1-y_t),\\
	&J_1  = -\left(1-\frac{k}{2} \right) \sum_{t=1}^{k}y_t(1-y_t),\quad J_{3,s} = -\frac{1}{2}\sum_{t=1}^{k}y_t(1-y_t),
\end{align*} 
\begin{align*}
	I_{2,s} =I_{4,ij} = J_{2,s} =J_{4,ij} = 0,
\end{align*} 
Plugging the above results into (\ref{062210}) yields
\begin{align}
	a_{x^2} = 0. \label{062231}
\end{align}

According to (\ref{062230}), what remains is the calculation of  $N\int_{\mathbb{R}}x^2d\mu_\boxplus (x)$. For this term,  we can compute it using the relation between the moments and the free cumulants \cite{mingo2017free}. Let $m_k(\mu)$ and $c_k(\mu)$ be the $k$-th moment and $k$-th free cumulant of $\mu$, respectively. Then we have 
\begin{align*}
	c_1(\mu_\boxplus) = m_1(\mu_\boxplus), \quad c_2(\mu_\boxplus) = -m_1^2(\mu_\boxplus) + m_2(\mu_\boxplus).
\end{align*}
Also, we have
\begin{align*}
	c_j(\mu_\boxplus) = c_j(\mu_1) + c_j(\mu_2) + \cdots + c_j(\mu_k).
\end{align*}
Hence, we can obtain
\begin{align*}
	&c_1(\mu_\boxplus) = c_1(\mu_1) + c_1(\mu_2) + \cdots + c_1(\mu_k) = \sum_{t=1}^k \frac{p_t}{N},\\
	&c_2(\mu_\boxplus) = c_2(\mu_1) + c_2(\mu_2) + \cdots + c_2(\mu_k) = \sum_{t=1}^k \frac{p_t}{N} \left(1 - \frac{p_t}{N} \right).
\end{align*}
Therefore,
\begin{align*}
	N\int_{\mathbb{R}}x^2d\mu_\boxplus (x) = N m_2(\mu_\boxplus) = \sum_{i \neq j}^k \frac{p_i p_j}{N} + \sum_{t=1}^k p_t,
\end{align*}
which together with (\ref{062231}) gives
\begin{align*}
	\E \Tr H^2 = \sum_{i \neq j}^k \frac{p_i p_j}{N} + \sum_{t=1}^k p_t.
\end{align*}
\subsubsection{Variance for Schott's statistics: $f(x) = x^2$}
Recall $\sigma_f$ from (\ref{072201}). Denoted by $\omega_{ti} := \omega_t(z_i), i = 1,2, t \in [\![ k]\!]$ and $m_{\boxplus i} := m_{\boxplus}(z_i), i=1,2$, we have
\begin{align*}
\sigma_{x^2}=&-\frac{1}{2\pi^2}\oint_{\gamma_1}\oint_{\gamma_2}  z_1^2z_2^2 \bigg[\sum_{t=1}^{k} \frac{\omega_{t1} - \omega_{t2}}{(\omega_{t1}-\omega_{t2})^2}-\frac{1}{(z_1-z_2)^2}-(k-1)\frac{m'_{\boxplus1}m'_{\boxplus2}}{(m_{\boxplus1}-m_{\boxplus2})^2}\bigg]{\rm d}z_2{\rm d}z_1\\
	=&:\sum_{t=1}^k K_{1,t} + K_{2} + K_3
\end{align*}
We first consider $K_2$, 
\begin{align}
	K_2 =\frac{1}{2\pi^2}\oint_{\gamma_1}\oint_{\gamma_2} \frac{z_1^2z_2^2}{(z_1-z_2)^2}{\rm d}z_2{\rm d}z_1= 0. \label{062240}
\end{align}
For $K_{1,t}$, using \eqref{Eomega} and \eqref{Eomegap}, we have
\begin{align*}
	K_{1,t} =& -\frac{1}{2\pi^2}\oint_{\gamma_{m1}}\oint_{\gamma_{m2}}   z_1^2 z_2^2  \frac{\omega'_{t1}\omega'_{t2}}{m_{\boxplus1}'m_{\boxplus2}'(\omega_{t1}-\omega_{t2})^2}{\rm d}m_{\boxplus2}{\rm d}m_{\boxplus1} \\
	=&  -\frac{1}{2\pi^2}\oint_{\gamma_{m1}}\oint_{\gamma_{m2}}    \frac{z_1^2 z_2^2\left( 1 +g_t(m_{\boxplus1})\right)\left( 1 + g_t(m_{\boxplus2} \right) }{\left(m_{\boxplus1}\left(1+\ell_t(m_{\boxplus2})\right) -m_{\boxplus2}\left(1+\ell_t(m_{\boxplus1})\right) \right)^2}{\rm d}m_{\boxplus2}{\rm d}m_{\boxplus1},  
\end{align*}
where $g_t$ and $\ell_t$ are defined in (\ref{062236}), and $\gamma_{mi}$ is the contour of $m_{\boxplus i}$ after change of variable. For the inner integral,
\begin{align*}
	&\oint_{\gamma_{m2}}   z_1^2\frac{1 + g_t(m_{\boxplus2}) }{\left(m_{\boxplus2}\left(1+\ell_t(m_{\boxplus1})\right) -m_{\boxplus1}\left(1+\ell_t(m_{\boxplus2})\right) \right)^2}{\rm d}m_{\boxplus2} \\
	&=\oint_{\gamma_{m2}}   \frac{ \left( 1 + g_t(m_{\boxplus2})\right) \left( -\frac{1}{m_{\boxplus2}} + \sum_{s=1}^k \frac{m_{\boxplus2}+1-\ell_s(m_{\boxplus2})}{2m_{\boxplus2}} \right)^2 }{\left(m_{\boxplus2}\left(1+\ell_t(m_{\boxplus1})\right) -m_{\boxplus1}\left(1+\ell_t(m_{\boxplus2})\right) \right)^2}{\rm d}m_{\boxplus2} 
	=: \hat{K}_{1,t}(m_{\boxplus1})
\end{align*}
Setting $|z_2|$ sufficiently large, which makes $m_{\boxplus2} = 0$ be the only singular point inside $\gamma_{m2}$, by the residue theorem, we can check via tedious but elementary calculation
\begin{align*}
	\hat{K}_{1,t}(m_{\boxplus1}) =2\pi \mathrm{i} \left( \frac{2y_t-1-2y}{2m_{\boxplus1}^2} + \frac{1 + \ell_t(m_{\boxplus1})}{2m_{\boxplus1}^3} \right).
\end{align*}
Plugging it into $K_{1,t}$, we have 
\begin{align*}
	K_{1,t} =&-\frac{1}{2\pi^2}\oint_{\gamma_{m1}} z_2^2\left( 1 + g_t(m_{\boxplus1}) \right) \hat{K}_{1,t}(m_{\boxplus1}) {\rm d}m_{\boxplus1} \notag\\
	=& \frac{1}{\pi \mathrm{i}} \oint_{\gamma_{m1}} z_2^2  \left( \frac{2y_t-1-2y}{2m_{\boxplus1}^2} + \frac{1 + \ell_t(m_{\boxplus1})}{2m_{\boxplus1}^3} \right)\left( 1 + g_t(m_{\boxplus1}) \right){\rm d}m_{\boxplus1} \\
	=& \frac{1}{\pi \mathrm{i}} \oint_{\gamma_{m1}} \left(-\frac{1}{m_{\boxplus1}} + \sum_{s=1}^k \frac{m_{\boxplus1}+1-\ell_s(m_{\boxplus1})}{2m_{\boxplus1}} \right)^2  \left( \frac{2y_t-1-2y}{2m_{\boxplus1}^2} + \frac{1 + \ell_t(m_{\boxplus1})}{2m_{\boxplus1}^3} \right)\left( 1 + g_t(m_{\boxplus1}) \right){\rm d}m_{\boxplus1} 
\end{align*}
Performing residue calculations again and summing over $t$, we arrive at
\begin{align}
\sum_{t=1}^kK_{1,t} 
=&44\sum_{s=1}^ky_s^2 - 72\sum_{s=1}^ky_s^3 + 36\sum_{s=1}^ky_s^4 - 8y - 48k\sum_{s=1}^ky_s^2+ 80k\sum_{s=1}^ky_s^3 - 40k\sum_{s=1}^ky_s^4 \notag\\& + 8ky + 48\sum_{s=1}^ky_s^2y - 32\sum_{s=1}^ky_s^3y + 4k\left(\sum_{s=1}^ky_s^2\right)^2 + 20ky^2 + 8ky^3+ 8\sum_{s=1}^ky_s^2y^2 \notag\\& - 16y^2 - 8y^3 - 8k\sum_{s=1}^ky_s^2y^2 - 56k\sum_{s=1}^ky_s^2y + 32k\sum_{s=1}^ky_s^3y. \label{062241}
\end{align}

Finally, we calculate $K_3$. First we rewrite 
\begin{align*}
	K_3 = \frac{(k-1)}{2\pi^2}\oint_{\gamma_{m1}}\oint_{\gamma_{m2}}   \frac{z_1^2 z_2^2}{(m_{\boxplus2}-m_{\boxplus1})^2}{\rm d}m_{\boxplus2}{\rm d}m_{\boxplus1}. 
\end{align*}
For the inner integral, we have
\begin{align*}
	\oint_{\gamma_{m2}}   \frac{z_2^2 }{(m_{\boxplus2}-m_{\boxplus1})^2}{\rm d}m_{\boxplus2} =& \oint_{\gamma_{m2}}   \frac{1 }{(m_{\boxplus2}-m_{\boxplus1})^2}\left( -\frac{1}{m_{\boxplus2}} + \sum_{s=1}^k \frac{m_{\boxplus2}+1-\ell_s(m_{\boxplus2})}{2m_{\boxplus2}} \right)^2 {\rm d}m_{\boxplus2}.
\end{align*}

Performing residue calculation, we get
\begin{align*}
	\oint_{\gamma_{m2}}   \frac{z_2^2 }{(m_{\boxplus2}-m_{\boxplus1})^2}{\rm d}m_{\boxplus2} = 2\pi \mathrm{i}\left(\frac{2}{m_{\boxplus1}^3} - \frac{2y}{m_{\boxplus1}^2}\right).
\end{align*}
Hence,
\begin{align*}
	K_3 =& \frac{(1-k)}{\pi i}\oint_{\gamma_{m1}} z_1^2\left(\frac{2}{m_{\boxplus1}^3} - \frac{2y}{m_{\boxplus1}^2}\right){\rm d}m_{\boxplus1} \notag\\
	=& \frac{(1-k)}{\pi i}\oint_{\gamma_{m1}} \left(\frac{2}{m_{\boxplus1}^3} - \frac{2y}{m_{\boxplus1}^2}\right)\left( -\frac{1}{m_{\boxplus1}} + \sum_{s=1}^k \frac{m_{\boxplus1}+1-\ell_s(m_{\boxplus1})}{2m_{\boxplus1}} \right)^2{\rm d}m_{\boxplus1}.
\end{align*}
Again, by residue theorem, we obtain
\begin{align}
	K_3 =&(1 - k)\Big(80\sum_{s=1}^ky_s^3 - 48\sum_{s=1}^ky_s^2 - 40\sum_{s=1}^ky_s^4 + 8y \notag\\&- 48\sum_{s=1}^ky_s^2y + 32\sum_{s=1}^ky_s^3y - 8\sum_{s=1}^ky_s^2y^2 + 4\left(\sum_{s=1}^ky_s^2\right)^2 + 20y^2 + 8y^3\Big). \label{062242}
\end{align}
Combining (\ref{062240}), (\ref{062241}) and (\ref{062242}), we get
\begin{align*}
	\sigma_{x^2} = 4\sum_{i \neq j}y_iy_j(1-y_i)(1-y_j).
\end{align*}
This completes the calculation for Schott's statistics.

\subsection{Proof of Corollary \ref{corowilk}}
\subsubsection{Expectation for Wilks' statistics: $f(x) =\log(x)$}
Recall $a_f$ from (\ref{062201}) and the contour $\gamma^0_1$ and $\gamma_2^0$ defined in (\ref{contour12}). Let $\gamma^0_{mi}$ and $\gamma^0_{wi}$ be the corresponding contour of $m_{\boxplus}$ and $\omega_t$ respectively. We have
\begin{align*}
	 &N\int_{\mathbb{R}} \log (x) d\mu_{\boxplus} + a_{\log(x)}=-\frac{N}{2\pi \mathrm{i}}\oint_{\gamma^0_1} \log (z)m_\boxplus(z) {\rm d}z-\frac{1}{4\pi \mathrm{i}}\oint_{\gamma^0_1} \log(z) {\rm d}\;{ \log\left(\frac{\prod_{t=1}^k\omega_t'}{[(-1/m_{\boxplus})']^{(k-1)}}\right)} \\
	 =&-\frac{N}{2\pi \mathrm{i}}\oint_{\gamma_{m1}^0} \log(z(m_{\boxplus})) \frac{m_{\boxplus}}{m'_{\boxplus}} {\rm d}m_{\boxplus}-\frac{1}{4\pi \mathrm{i}}\oint_{\gamma^0_1}  \log(z(m_\boxplus)) {\rm d}\;{ \log\left(\frac{\prod_{t=1}^k\omega_t'}{[(-1/m_{\boxplus})']^{(k-1)}}\right)}  =: NL_1 + L_2.
\end{align*}
Recall the definition of $g_t$ and $\ell_t$ in (\ref{062236}). Further using (\ref{zofm}) and (\ref{sub3}), we have
\begin{align}
	L_1 =&-\frac{1}{2\pi \mathrm{i}}\oint_{\gamma_{m1}^0} \log \left(1- \sum_{t=1}^{k}\frac{m_{\boxplus}+1-\ell_t(m_{\boxplus})}{2} \right) \frac{m_{\boxplus}}{m'_{\boxplus}} {\rm d}m_{\boxplus} + \frac{1}{2\pi \mathrm{i}}\oint_{\gamma_{m1}^0}\log \left(-m_{\boxplus} \right)m_{\boxplus}{\rm d}z\notag \\
	 =&-\frac{1}{2\pi \mathrm{i}}\oint_{\gamma_{m1}^0}\log \left(1- \sum_{t=1}^{k}\frac{m_\boxplus+1-\ell_t(m_{\boxplus})}{2} \right) \frac{m_{\boxplus}}{m'_{\boxplus}} {\rm d}m_{\boxplus} +\sum_{t=1}^{k} \frac{1}{2\pi \mathrm{i}}\oint_{\gamma^0_{w1}} \log \left(-m_{\boxplus} \right)m_{\boxplus}{\rm d}\omega_t \notag\\
	 &+ \frac{k-1}{2\pi \mathrm{i}}\oint_{\gamma_{m1}^0} \frac{\log \left(-m_{\boxplus} \right)}{m_\boxplus}{\rm d}m_{\boxplus}=:L_{1,1} + \sum_{t=1}^{k}L_{1,2t} + L_{1,3}. \label{082210}
\end{align} 
We first have $L_{1,3} = 0$ since the contour $\gamma_m$ doesn't enclose $0$, and the integrand is analytic inside the contour. Then we consider $L_{1,1}$. Let
\begin{align*}
	h_1(m_{\boxplus}) :=& \log \left(1- \sum_{t=1}^{k}\frac{m+1-\ell_t(m_{\boxplus})}{2} \right) \frac{m_{\boxplus}}{m_{\boxplus}'} =\frac{1}{m_{\boxplus}}\left(1-\sum_{t=1}^{k}\frac{1-g_t(m_{\boxplus})}{2} \right) \log \left(1- \sum_{t=1}^{k}\frac{m_{\boxplus}+1-\ell_t(m_{\boxplus})}{2} \right).
\end{align*}
We have by the residue at infinity,
\begin{align*}
	L_{1,1}=\mathrm{Res}\left( h_1(m_\boxplus),\infty\right)+\mathrm{Res}\left( h_1(m_\boxplus),0\right) = -\mathrm{Res}\left(\frac{1}{m_{\boxplus}^2}h_1\left(\frac{1}{m_{\boxplus}} \right),0 \right)+\mathrm{Res}\left(h_1(m_{\boxplus}),0 \right).
\end{align*}
By direct calculation, we have
\begin{align*}
	\mathrm{Res}\left(\frac{1}{m_{\boxplus}^2}h_1\left(\frac{1}{m_{\boxplus}} \right),0 \right) = (1-y)\log(1-y), \quad \mathrm{Res}\left(h_1(m_{\boxplus}),0 \right) = 0.
\end{align*}
Therefore,
\begin{align}
	L_{1,1} =  -(1-y)\log(1-y). \label{082205}
\end{align}
For $L_{1,2t}$, we have
\begin{align*}
	L_{1,2t} =& \frac{1}{2\pi \mathrm{i}}\oint_{\gamma_w} \log \left(-m_{\boxplus} \right)\frac{y_t}{1-\omega_t}{\rm d}\omega_t - \frac{1}{2\pi \mathrm{i}}\oint_{\gamma_w} \log \left(-m_{\boxplus} \right)\frac{1-y_t}{\omega_t}{\rm d}\omega_t \\
	=& \frac{1}{2\pi \mathrm{i}}\oint_{\gamma_w} \log \left( \frac{1-y_t}{\omega_t}-\frac{y_t}{1-\omega_t} \right)\frac{y_t}{1-\omega_t}{\rm d}\omega_t - \frac{1}{2\pi \mathrm{i}}\oint_{\gamma_w} \log \left(-m_{\boxplus} \right)\frac{1-y_t}{\omega_t}{\rm d}\omega_t
\end{align*}
For the first integral, we can compute it based on the contour of $\omega_t$. By the previous analysis, we know that as $|z| \to 0$, $|\omega_t| \to 0$, and  as $|z| \to \infty$, $|\omega_t| \to \infty$. Therefore, the contour of $\omega_t$ will enclose two branch points of the $\log$ term, $\omega_t = 1$ and $\omega_t = 1-y_t$, but not $0$. Thus
\begin{align*}
	&\frac{1}{2\pi \mathrm{i}}\oint_{\gamma_w} \log \left( \frac{1-y_t}{\omega_t}-\frac{y_t}{1-\omega_t} \right)\frac{y_t}{1-\omega_t}{\rm d}\omega_t\\
	=&\frac{1}{2\pi \mathrm{i}}\oint_{\gamma_w} \log \left( \frac{1-y_t-\omega_t}{1-\omega_t} \right)\frac{y_t}{1-\omega_t}d\omega_t -\frac{1}{2\pi \mathrm{i}}\oint_{\gamma_w} \log \left( \frac{1}{\omega_t} \right)\frac{y_t}{1-\omega_t}{\rm d}\omega_t\\
	=&\frac{1}{2\pi \mathrm{i}}\oint_{\gamma_w} \log \left( \frac{1-y_t-\omega_t}{1-\omega_t} \right)\frac{y_t}{1-\omega_t}{\rm d}\omega_t.
\end{align*}  
Here in last step, the second term in $L_{1,2t}$ equals to 0 since it only has one pole $\omega_t = 1$ inside the contour, and the residue at this pole is 0. 
By the residue at infinity, the remaining term can be calculated as follows. Let 
\begin{align*}
	h_2(\omega_t) = \log \left( \frac{1-y_t-\omega_t}{1-\omega_t} \right)\frac{y_t}{1-\omega_t},
\end{align*} 
then
\begin{align*}
	\frac{1}{2\pi \mathrm{i}}\oint_\gamma 	h_2(\omega_t) {\rm d}\omega_t	= \mathrm{Res}\left( \frac{1}{\omega_t^2}h_2\left(\frac{1}{\omega_t} \right),0 \right) = 0.
\end{align*}

As a result,
\begin{align*}
	L_{1,2t} =&-\frac{1}{2\pi \mathrm{i}}\oint_{\gamma_w} \log \left(-m_{\boxplus} \right)\frac{1-y_t}{\omega_t}{\rm d}\omega_t \\
	=&-\frac{1}{2\pi \mathrm{i}}\oint_{\gamma_w} \log \left( \frac{1-y_t-\omega_t}{1-\omega_t} \right)\frac{1-y_t}{\omega_t}{\rm d}\omega_t - \frac{1}{2\pi \mathrm{i}}\oint_{\gamma_w} \log \left( \frac{1}{\omega_t} \right)\frac{1-y_t}{\omega_t}{\rm d}\omega_t
\end{align*}
The second term vanished since the integrand has antiderivative
\begin{align*}
	\frac{1}{2}\left(\log(\omega_t) \right)^2,
\end{align*}
which is single valued along the contour. For the first term, let
\begin{align*}
	h_3(\omega_t) = \log \left( \frac{1-y_t-\omega_t}{1-\omega_t} \right)\frac{1-y_t}{\omega_t},
\end{align*}
we have by the residue at infinity,
\begin{align*}
	-\frac{1}{2\pi \mathrm{i}}\oint_{\gamma_w} h_3(\omega_t) {\rm d}\omega_t =& -\mathrm{Res}\left( \frac{1}{\omega_t^2}h_3\left(\frac{1}{\omega_t}\right) ,0  \right) + \mathrm{Res}\left( h_3(\omega_t),0 \right)= (1-y_t)\log(1-y_t),
\end{align*}
which gives
\begin{align}
	L_{1,2t} = (1-y_t)\log(1-y_t). \label{082201}
\end{align}
Plugging (\ref{082205}) and (\ref{082201}) into (\ref{082210}), we have
\begin{align*}
	L_1 = -(1-y)\log(1-y) + \sum_{t=1}^k (1-y_t)\log(1-y_t).
\end{align*}
Next, we calculate $L_2$. We rewrite 
\begin{align*}
	L_2 =& -\frac{1}{4\pi \mathrm{i}}\oint_{\gamma^0_1} \log(z(m_{\boxplus})) {\rm d}\;{ \log\left(\prod_{t=1}^k\frac{\omega_t'm_{\boxplus}^2}{m_{\boxplus}'}\right)}-\frac{1}{4\pi \mathrm{i}}\oint_{\gamma^0_1}  \log(z(m_{\boxplus})) {\rm d}\;{ \log\left( \frac{m_{\boxplus}'}{m_{\boxplus}^2} \right)}\\
	=& -\sum_{t=1}^{k} \frac{1}{4\pi \mathrm{i}}\oint_{\gamma^0_1}  \log(z(m_{\boxplus})) {\rm d}\;\log \left(\frac{\omega_t'm_{\boxplus}^2}{m'_{\boxplus}}\right)-\frac{1}{4\pi \mathrm{i}}\oint_{\gamma^0_1}  \log(z(m_{\boxplus})) {\rm d}\;{ \log\left( \frac{m_{\boxplus}'}{m_{\boxplus}^2} \right)}=: \sum_{t=1}^{k}L_{2,1t} + L_{2,2}.
\end{align*}
For $L_{2,1t}$, we have by \eqref{Eomegap},
\begin{align*}
	L_{2,1t} = \frac{1}{4\pi \mathrm{i}}\oint_{\gamma_m} \frac{z'(m_{\boxplus})}{z(m_{\boxplus})} \log \left( \frac{1+g_t(m_{\boxplus})}{2}  \right){\rm d}m_{\boxplus}.
\end{align*}
Let
\begin{align*}
	h_4(m_{\boxplus}) = \frac{z'(m_{\boxplus})}{z(m_{\boxplus})} \log \left( \frac{1 + g_t(m_{\boxplus})}{2} \right),
\end{align*}
then by the residue at infinity, we have
\begin{align*}
	L_{2,1t}=\frac{1}{4\pi \mathrm{i}}\oint_{\gamma_m} h_4(m_{\boxplus}) {\rm d}m_{\boxplus} = \frac{1}{2}\mathrm{Res}\left( \frac{1}{m_{\boxplus}^2}h_4\left(\frac{1}{m_{\boxplus}} \right),0 \right) - \frac{1}{2}\mathrm{Res}\left( h_4(m_{\boxplus}),0 \right).
\end{align*}
By direct calculation, we can get
\begin{align*}
	\mathrm{Res}\left( \frac{1}{m_{\boxplus}^2}h_4\left(\frac{1}{m_{\boxplus}} \right),0 \right) = -\log(1-y_t), \quad \mathrm{Res}\left( h_4(m_{\boxplus}),0 \right) = 0.
\end{align*}
Therefore,
\begin{align*}
	L_{2,1t} = -\frac{1}{2}\log(1-y_t).
\end{align*}
For $L_{2,2}$, we have by \eqref{Emp},
\begin{align*}
	L_{2,2} = -\frac{1}{4\pi \mathrm{i}}\oint_{\gamma_m} \frac{z'(m_{\boxplus})}{z(m_{\boxplus})} \log \left( 1 - \sum_{t=1}^{k}
	\frac{1-g_t(m_{\boxplus})}{2} \right){\rm d}m_{\boxplus}.
\end{align*}
Let
\begin{align*}
	h_5(m_{\boxplus}) = \frac{z'(m_{\boxplus})}{z(m_{\boxplus})} \log \left( 1 - \sum_{t=1}^{k}
	\frac{1-g_t(m_{\boxplus})}{2} \right),
\end{align*}
the residue at infinity, we have
\begin{align*}
	L_{2,2}=-\frac{1}{4\pi \mathrm{i}}\oint_{\gamma_m} h_5(m_{\boxplus}) {\rm d}m_{\boxplus} = -\frac{1}{2}\mathrm{Res}\left( \frac{1}{m_{\boxplus}^2}h_5\left(\frac{1}{m_{\boxplus}} \right),0 \right) + \frac{1}{2}\mathrm{Res}\left(h_5(m_{\boxplus}),0 \right).
\end{align*}
By direct calculation, we can get
\begin{align*}
	\mathrm{Res}\left( \frac{1}{m_{\boxplus}^2}h_5\left(\frac{1}{m_{\boxplus}} \right),0 \right) = -\log(1-y), \quad \mathrm{Res}\left( h_5(m_{\boxplus}),0 \right) = 0.
\end{align*}
Therefore,
\begin{align*}
	L_{2,2} = \frac{1}{2}\log (1-y).
\end{align*}
Combining the above results, we can obtain
\begin{align*}
	L_2 = \frac{1}{2}\log (1-y) - \frac{1}{2}\log (1-y_t).
\end{align*}
Then we have the final result of the expectation of the Wilks' statistics, which is 
\begin{align*}
	N\int_{\mathbb{R}} \log (x) {\rm d}\mu_{\boxplus} + a_{\log(x)}=\sum_{t=1}^{k}\left(N-p_t-\frac{1}{2}\right)\log(1-y_t) - \left(N-Ny- \frac{1}{2}\right)\log(1-y) . 
\end{align*}

\subsubsection{Variance for Wilks' statistics: $f(x) = \log(x)$}
Using the same notation as the case $f(x) = x^2$, recall that
\begin{align*}
\sigma_{\log(x)}=&-\frac{1}{2\pi^2}\oint_{\gamma_1}\oint_{\gamma_2}  \log(z_1)\log(z_2) \bigg[\sum_{t=1}^{k} \frac{\omega_{t1} - \omega_{t2}}{(\omega_{t1}-\omega_{t2})^2}-\frac{1}{(z_1-z_2)^2}-(k-1)\frac{m'_{\boxplus1}m'_{\boxplus2}}{(m_{\boxplus1}-m_{\boxplus2})^2}\bigg]{\rm d}z_2{\rm d}z_1\\
	=&:\sum_{t=1}^k K_{1,t} + K_{2} + K_3
\end{align*}
For our analysis, it is convenient to take $\gamma_{m1}$ and $\gamma_{m2}$ be the contour of $m_{\boxplus1}$ and $m_{\boxplus2}$, respectively. For Wilks' statistics, we only care about the positive eigenvalues of $H$. Therefore, the contour of $z$ should enclose those positive eigenvalues but not $0$. In other words, the contours of $m_{\boxplus1}$ and $m_{\boxplus2}$ should not enclose $0$. In addition, by choosing the parameters in $\gamma_i$'s suitably, we can have that $\gamma_{m1}$ encloses $\gamma_{m2}$. Further, in this section, we view $z$ and $\omega_t$ be functions of $m_{\boxplus}$, and all the derivatives are taken w.r.t $m_{\boxplus}$.

We first consider $K_3$,
\begin{align*}
	K_3 = \frac{k-1}{2\pi^2}\oint_{\gamma_{m1}}\oint_{\gamma_{m2}}  \frac{\log (z(m_{\boxplus1})) \log (z(m_{\boxplus2}))}{(m_{\boxplus1}-m_{\boxplus2})^2}{\rm d}m_{\boxplus2}{\rm d}m_{\boxplus1}.
\end{align*}
For the inner integral, we have
\begin{align*}
	\oint_{\gamma_{m2}}  \frac{\log (z(m_{\boxplus2}))}{(m_{\boxplus1}-m_{\boxplus2})^2}{\rm d}m_{\boxplus2}  = \oint_{\gamma_{m2}}\frac{z'(m_{\boxplus2})}{z(m_{\boxplus2})(m_{\boxplus2} - m_{\boxplus1})} {\rm d}m_{\boxplus2}.
\end{align*}
Let
\begin{align*}
	h_6(m_{\boxplus2}) = \frac{z'(m_{\boxplus2})}{z(m_{\boxplus2})(m_{\boxplus2} - m_{\boxplus1})}
\end{align*}
be the integrand of the contour integral. Instead of calculating the residue inside $\gamma_{m2}$, again we turn to compute the residue at infinity, to simplify the calculation. Since $\gamma_{m2}$ does not enclose $0$ and $\gamma_{m1}$ encloses $\gamma_{m2}$, we have
\begin{align*}
	\oint_{\gamma_{m2}} h_6(m_{\boxplus2}) {\rm d}m_{\boxplus2} =& -2\pi \mathrm{i} \left(\mathrm{Res}\left(h_6(m_{\boxplus2}),\infty \right) +  \mathrm{Res}\left(h_6(m_{\boxplus2}),m_{\boxplus1} \right) + \mathrm{Res}\left(h_6(m_{\boxplus2}),0 \right)  \right) \\
	=&2\pi \mathrm{i} \left(\mathrm{Res}\left(\frac{1}{m_{\boxplus2}^2}h_6\left(\frac{1}{m_{\boxplus2}}\right),0\right) -  \mathrm{Res}\left(h_6(m_{\boxplus2}),m_{\boxplus1} \right) - \mathrm{Res}\left(h_6(m_{\boxplus2}),0 \right)\right)
\end{align*}
Using \eqref{zofm}, by direct calculation, we can get
\begin{align*}
	\mathrm{Res}\left(\frac{1}{m_{\boxplus2}^2}h_6\left(\frac{1}{m_{\boxplus2}}\right),0\right) = 0, \quad \mathrm{Res}\left(h_6(m_{\boxplus2}),m_{\boxplus1} \right) = \frac{z'(m_{\boxplus1})}{z(m_{\boxplus1})},\quad \mathrm{Res}\left(h_6(m_{\boxplus2}),0 \right) = \frac{1}{m_{\boxplus1}}.
\end{align*}
Therefore,
\begin{align*}
	\oint_{\gamma_{m2}} h_6(m_{\boxplus2}) {\rm d}m_{\boxplus2} = -2\pi \mathrm{i}  \left( \frac{z'(m_{\boxplus1})}{z(m_{\boxplus1})} + \frac{1}{m_{\boxplus1}}  \right).
\end{align*}
Then we can calculate the outer integral,
\begin{align*}
	K_3 =& -\frac{(k-1)\mathrm{i}}{\pi} \oint_{\gamma_{m1}}\log (z(m_{\boxplus1})) \left( \frac{z'(m_{\boxplus1})}{z(m_{\boxplus1})} + \frac{1}{m_{\boxplus1}}  \right) {\rm d}m_{\boxplus1} \\
	=& -\frac{(k-1)\mathrm{i}}{\pi} \oint_{\gamma_{m1}}\log \left( -1+\sum_{t=1}^k\frac{m_{\boxplus1}+1-\ell_t(m_{\boxplus1})}{2} \right) \left( \frac{z'(m_{\boxplus1})}{z(m_{\boxplus1})} + \frac{1}{m_{\boxplus1}}  \right) {\rm d}m_{\boxplus1}\\
	&+\frac{(k-1)\mathrm{i}}{\pi} \oint_{\gamma_{m1}}\log \left( m_{\boxplus1} \right) \left( \frac{z'(m_{\boxplus1})}{z(m_{\boxplus1})} + \frac{1}{m_{\boxplus1}}  \right) {\rm d}m_{\boxplus1}.
\end{align*}
Notice that
\begin{align*}
	\frac{z'(m_{\boxplus1})}{z(m_{\boxplus1})} + \frac{1}{m_{\boxplus1}} =\left( \log \left( -1+\sum_{t=1}^k\frac{m_{\boxplus1}+1-\ell_t(m_{\boxplus1})}{2} \right) \right)', 
\end{align*}
then first integral is zero because the integrand has antiderivative
\begin{align*}
	\frac{1}{2}\left( \log \left( -1+\sum_{t=1}^k\frac{m_{\boxplus1}+1-\ell_t(m_{\boxplus1})}{2} \right) \right)^2  = \frac12( \log(z(m_{\boxplus1})m_{\boxplus1}))^2,
\end{align*}
which is analytic along the contour. For the second integral, we have
\begin{align*}
	&\frac{(k-1)\mathrm{i}}{\pi} \oint_{\gamma_{m1}}\log \left( m_{\boxplus1} \right) \left( \frac{z'(m_{\boxplus1})}{z(m_{\boxplus1})} + \frac{1}{m_{\boxplus1}}  \right) {\rm d}m_{\boxplus1} \\
	=&\frac{(1-k)\mathrm{i}}{\pi} \oint_{\gamma_{m1}}\frac{1}{m_{\boxplus1}}\log \left( -1+\sum_{t=1}^k\frac{m_{\boxplus1}+1-\ell_t(m_{\boxplus1})}{2} \right)   {\rm d}m_{\boxplus1}
\end{align*}
Let
\begin{align*}
	h_7(m_{\boxplus1}) := \frac{1}{m_{\boxplus1}}\log \left( -1+\sum_{t=1}^k\frac{m_{\boxplus1}+1-\ell_t(m_{\boxplus1})}{2} \right).
\end{align*}
Applying the technique of residue at infinity, we have
\begin{align*}
	&\frac{(1-k)\mathrm{i}}{\pi} \oint_{\gamma_{m1}}h_7(m_{\boxplus1}){\rm d}m_{\boxplus1} = 2(k-1)\left(\mathrm{Res}\left(\frac{1}{m_{\boxplus1}^2}h_7\left(\frac{1}{m_{\boxplus1}}\right),0\right)  - \mathrm{Res}\left(h_7(m_{\boxplus1}),0 \right)\right)
\end{align*} 
By direct calculation, we get
\begin{align*}
	\mathrm{Res}\left(\frac{1}{m_{\boxplus1}^2}h_7\left(\frac{1}{m_{\boxplus1}}\right),0\right) = \log (-1+y), \quad \mathrm{Res}\left(h_7(m_{\boxplus1}),0 \right) = \log (-1).
\end{align*}
Therefore,
\begin{align*}
	K_3 = 2(k-1)\log(1-y).
\end{align*}
Next, we calculate $K_{1,t}$,
\begin{align*}
	K_{1,t} = -\frac{1}{2\pi^2}\oint_{\gamma_{m1}}\oint_{\gamma_{m2}}  \frac{\log(z(m_{\boxplus2}))\log(z(m_{\boxplus1}))}{(\omega_{t1}-\omega_{t2})^2} {\rm d}\omega_{t2} {\rm d}\omega_{t1}
\end{align*}
For the inner inegral,
\begin{align*}
	\oint_{\gamma_{w2}}  \frac{\log(z(m_{\boxplus2}))}{(\omega_{t1}-\omega_{t2})^2} {\rm d}\omega_{t2} =&  \oint_{\gamma_{m2}}\frac{z'(m_{\boxplus2})}{z(m_{\boxplus2})(\omega_{t2} - \omega_{t1})} {\rm d}m_{\boxplus2}.
\end{align*}
Let
\begin{align*}
	h_8(m_{\boxplus2}) = \frac{z'(m_{\boxplus2})}{z(m_{\boxplus2})(\omega_{t2} - \omega_{t1})}
\end{align*}
be the integrand. We have
\begin{align*}
	\oint_{\gamma} h_8(m_{\boxplus2}) {\rm d}m_{\boxplus2} = 2\pi\mathrm{i}\left( \mathrm{Res}\left( \frac{1}{m_{\boxplus2}^2}h_8\left(\frac{1}{m_{\boxplus2}} \right),0 \right) - \mathrm{Res}\left(h_8(m_{\boxplus2}) , m_{\boxplus1} \right) - \mathrm{Res}\left(h_8(m_{\boxplus2}) , 0 \right) \right).
\end{align*}
Using \eqref{Eomega} and \eqref{zofm}, by direct calculation, we can get
\begin{align*}
	\mathrm{Res}\left( \frac{1}{m_{\boxplus2}^2}h_8\left(\frac{1}{m_{\boxplus2}} \right),0 \right) = \frac{1}{\omega_{t1}}, \quad \mathrm{Res}\left(h_8(m_{\boxplus2}) , m_{\boxplus1} \right) = \frac{z'(m_{\boxplus1})}{z(m_{\boxplus1})\omega'_{t1}},\quad \mathrm{Res}\left(h_8(m_{\boxplus2}) , 0 \right)  = 0.
\end{align*}
Therefore,
\begin{align*}
	\oint_{\gamma} g_3(m_{\boxplus1}) {\rm d}m_{\boxplus1}  = 2\pi\mathrm{i} \left(  \frac{1}{\omega_{t1}} -\frac{z'(m_{\boxplus1})}{z(m_{\boxplus1})\omega'_{t1}}  \right).
\end{align*}
Then we can calculate the outer integral,
\begin{align*}
	K_{1,t} = -\frac{\mathrm{i}}{\pi}\oint_{\gamma_{w1}}\frac{ \log (z(m_{\boxplus1})) }{\omega_{t1}} {\rm d}\omega_{t1} + \frac{\mathrm{i}}{\pi}\oint_{\gamma_{m1}} \log (z(m_{\boxplus1}))\frac{z'(m_{\boxplus1})}{z(m_{\boxplus1})} {\rm d}m_{\boxplus1}
\end{align*}
The second integral is zero because the integrand has antiderivative
\begin{align*}
	\frac{1}{2}\left(\log (z(m_{\boxplus1}))\right)^2,
\end{align*}
which is single valued along the contour. For the first integral, we have
\begin{align*}
	&-\frac{\mathrm{i}}{\pi}\oint_{\gamma_{w1}}\frac{ \log (z(m_{\boxplus1})) }{\omega_{t1}} {\rm d}\omega_{t1}  = -\frac{\mathrm{i}}{\pi}\oint_{\gamma_{m1}}\log (z(m_{\boxplus1})) {\rm d}\log (\omega_{t1}) \\
	=& \frac{\mathrm{i}}{\pi}\oint_{\gamma_{m1}}\frac{z'(m_{\boxplus1})}{z(m_{\boxplus1})}\log \left( \frac{m_{\boxplus1}-1-\ell_t(m_{\boxplus1})}{2} \right) {\rm d}m_{\boxplus1} + \frac{\mathrm{i}}{\pi}\oint_{\gamma_{m1}}\frac{\log (z(m_{\boxplus1}))}{m_{\boxplus1}}  {\rm d}m_{\boxplus1}.
\end{align*}
These two integral can be calculated using residue at infinity now. Let
\begin{align*}
	h_9(m_{\boxplus1}) = \frac{z'(m_{\boxplus1})}{z(m_{\boxplus1})}\log \left( \frac{m_{\boxplus1}-1-\ell_t(m_{\boxplus1})}{2} \right), \quad h_{10}(m_{\boxplus1}) =  \frac{\log (z(m_{\boxplus1}))}{m_{\boxplus1}}.
\end{align*}
Then we have
\begin{align*}
	\frac{\mathrm{i}}{\pi}\oint_{\gamma_{m1}}h_9(m_{\boxplus1}) {\rm d}m_{\boxplus1} = -2\mathrm{Res}\left(\frac{1}{m_{\boxplus1}^2} h_9\left(\frac{1}{m_{\boxplus1}}\right),0 \right) + 2\mathrm{Res}\left( h_9(m_{\boxplus1}),0 \right).
\end{align*}
By direct calculation, we can get
\begin{align*}
	\mathrm{Res}\left(\frac{1}{m_{\boxplus1}^2} h_9\left(\frac{1}{m_{\boxplus1}}\right),0 \right) = -\log(-1 + y_t), \quad \mathrm{Res}\left( h_9(m_{\boxplus1}),0 \right) = -\log(-1).
\end{align*}
Therefore,
\begin{align*}
	\frac{\mathrm{i}}{\pi}\oint_{\gamma_{m1}}h_9(m_{\boxplus1}) {\rm d}m_{\boxplus1} = 2\log(1-y_t).
\end{align*}
The calculation of integral of $g_5(m_{\boxplus1})$ is identical to $g_2(m_{\boxplus1})$, thus we have
\begin{align*}
	\frac{\mathrm{i}}{\pi}\oint_{\gamma_{m1}}h_{10}(m_{\boxplus1})  {\rm d}m_{\boxplus1} = -2\log(1+y).
\end{align*}
As a result,
\begin{align*}
	K_{1,t} = 2\log(1-y_t)-2\log(1+y).
\end{align*}
Combining the results of $K_{1,t}$ and $K_3$, we have
\begin{align*}
	\sigma_{\log(x)} = -2\log(1+y) + 2\sum_{t=1}^k \log(1-y_t).
\end{align*}

\section{Proof of Theorem \ref{MPCLT}}\label{ProofofMPCLT}
The proof of Theorem \ref{MPCLT} is also based on Theorem \ref{GeneralCLT}. We need the following Lemmas to replace the stochastic quantities by the deterministic quantities.

\begin{lemma}\label{trPGMP}
Let ${p_{\max}} \le N^{1/2-\epsilon}$ for any given (small) constant $\epsilon > 0$.	If $\hat{y} \in (0,1)$, for any fixed $z \in (\bar{\gamma}^0_1)^{+}\cup(\bar{\gamma}^0_2)^{+}$,  we have
	\begin{align}
		&\left|\tr \bbG(z) - m_{y}(z) \right| = O_{\prec}\left(\frac{p_{\max}}{N} \right), \label{trGGapMP}\\
		&\left| \tr \bbP_t\bbG(z) - \frac{y_tm_y(z)}{1+m_y(z)} \right|= O_{\prec}\left(  \sqrt{\frac{p_t}{N^3}}  \vee\frac{p_tp_{\max}}{N^2} \right),\quad t \in [\![ k]\!],\label{trPGGapMP}\\
		&\left|\tr \bbQ_t\bbG(z) - \frac{y_tm_y(z)}{1+m_y(z)}  \right|= O_{\prec}\left(  \sqrt{\frac{p_t}{N^3}}  \vee\frac{p_tp_{\max}}{N^2} \right), \quad t \in [\![ k]\!].\label{trQGGapMP}
	\end{align}
The same bounds hold for $z \in (\bar{\gamma}_1)^{+}\cup(\bar{\gamma}_2)^{+}$ with $\hat{y} \in (0,\infty)$.  
\end{lemma}
\begin{proof}[Proof of Lemma \ref{trPGMP}]
	The proof of (\ref{trGGapMP}) is identical to the proof in Section \ref{s.case 2}. All the estimates remain valid if $\|G(z) \| = O_{\prec}(1)$ which is guaranteed by Lemma \ref{Gbound}. Hence, we omit the details.
	
	Next, we prove (\ref{trPGGapMP}) based on (\ref{trGGapMP}) and (\ref{trXX})-(\ref{trPG}). Recall (\ref{trPG}), we have
	\begin{align*}
		\tr P_tG(z) =& (\tr (X_tX_t')^{-1} - \tr Q_tG(z))(\tr G(z) - \tr P_tG(z)) + O_{\prec}\left(\sqrt{\frac{p_t}{N^{3}}} \right)\\
		=& (\tr (X_tX_t')^{-1} - \tr Q_tG(z))\tr G(z) + O_{\prec}\left( \left(\frac{p_t}{N}\right)^2\vee \sqrt{\frac{p_t}{N^{3}}} \right) \\
		=& \left( y_t - \tr P_tG(z) \right)m_y(z) + O_{\prec}\left(\frac{p_tp_{\max}}{N^2}\vee \sqrt{\frac{p_t}{N^{3}}} \right)
	\end{align*} 
	Solving the above equation for $\tr P_tG(z)$, we can obtain (\ref{trPGGapMP}). Using (\ref{trPGGapMP}) together with (\ref{trPGP}), we have (\ref{trQGGapMP}). This completes the proof of  Lemma \ref{trPGMP}.
\end{proof}

\begin{lemma}\label{trPGQGMP}
Let ${p_{\max}} \le N^{1/2-\epsilon}$ for any given (small) constant $\epsilon > 0$. If $\hat{y} \in (0,1)$, for any fixed $z_1 \in (\bar{\gamma}^0_1)^{+}$ and $z_2 \in (\bar{\gamma}^0_2)^+$,  we have
	\begin{align}
		&\left| \tr \bbP_t\bbG(z_1)\bbP_t\bbG(z_2) - \frac{y_tm_y(z_1)m_y(z_2)}{(1+m_y(z_1))(1+m_y(z_2))} \right| 
		= O_{\prec}\left(\sqrt{\frac{p_t}{N^3}} \vee\frac{p_tp_{\max}}{N^2}  \right), \label{trPGPGGapMP}\\
		&\left| \tr \bbQ_t\bbG(z_1)\bbP_t\bbG(z_2) - \frac{y_tm_y(z_1)m_y(z_2)}{(1+m_y(z_1))(1+m_y(z_2))}  \right| 
		= O_{\prec}\left( \sqrt{\frac{p_t}{N^3}} \vee\frac{p_tp_{\max}}{N^2}  \right)\label{trQGPGGapMP}.
	\end{align}
The same bounds hold for $z_1 \in (\bar{\gamma}_1)^{+}$ and $z_2 \in (\bar{\gamma}_2)^{+}$ with $\hat{y} \in (0,\infty)$. 
\end{lemma}
\begin{proof}[Proof of Lemma \ref{trPGQGMP}]
	The proof is similar to the proof of Proposition \ref{keyprop2}. We start with (\ref{trPGPGZ1}), replacing $\tr Q_tG(z_1)P_tG(z_2)$ by $\tr P_tG(z_1)P_tG(z_2)$ and then absorbing the $O_{\prec}(p_t^2/N^2)$ terms into the error, we have
	\begin{align*}
		\tr P_tG(z_1)P_tG(z_2) + (\tr P_tG(z_1)P_tG(z_2)-\tr Q_tG(z_1))\tr G(z_2) =  O_{\prec}\left( \left(\frac{p_t}{N}\right)^2\vee \sqrt{\frac{p_t}{N^{3}}} \right).
	\end{align*}
	Using Lemma \ref{trPGMP} to replace the stochastic quantities into deterministic quantities we obtain (\ref{trPGPGGapMP}). Then (\ref{trQGPGGapMP}) follows immediately by the second estimates of (\ref{trPGPGZ1}). This completes the proof of Lemma \ref{trPGQGMP}.
\end{proof}

In addition to the concentration results for the tracial quantities, we also have the following concentration results for the diagonal entries of the random matrices. 
\begin{lemma}\label{DiagApproByMP}
Let ${p_{\max}} \le N^{1/2-\epsilon}$ for any given (small) constant $\epsilon > 0$. 	If $\hat{y} \in (0,1)$, for any fixed $z \in (\bar{\gamma}^0_1)^{+}\cup(\bar{\gamma}^0_2)^{+}$,  we have	
	\begin{align}
		&\left| [G(z)]_{ii} - m_{y}(z) \right| = O_{\prec}\left(  \frac{1}{\sqrt{N} }\right), \label{GiiGapMP}\\
		&\left| [\bbP_t\bbG(z)]_{ii} -\frac{y_tm_y(z)}{1+m_y(z)}\right|= O_{\prec}\left( \frac{1}{\sqrt{N}} \right),\quad t \in [\![ k]\!],\label{PGiiGapMP}\\
		&\left| [\bbP_t\bbG(z)\bbP_t]_{ii} -\frac{y_tm_y(z)}{1+m_y(z)}  \right|= O_{\prec}\left( \frac{1}{\sqrt{N}}  \right), \quad t \in [\![ k]\!],\label{PGPiiGapMP}\\
		&\left| [\bbW_t\bbG(z)\bbW'_t]_{jj} - \frac{m_y(z)}{1+m_y(z)} \right|= O_{\prec}\left(  \frac{1}{\sqrt{N}} \right), \quad t \in [\![ k]\!].\label{QGiiGapMP}
	\end{align}
The same bounds hold for $z \in (\bar{\gamma}_1)^{+}\cup(\bar{\gamma}_2)^{+}$ with $\hat{y} \in (0,\infty)$.
\end{lemma}
\begin{proof}[Proof of Lemma \ref{DiagApproByMP}]
	By Propositions \ref{keyProp} and \ref{DiagApproByFC}, we can obtain
	\begin{align*}
		|G_{ii}(z) - \tr G(z)| = O_{\prec}\left(\frac{1}{\sqrt{N}} \right).
	\end{align*}
	Replacing $\tr G(z)$ by Lemma \ref{trPGMP}, together with the condition that $p_{\max} \le N^{1/2-\epsilon}$ for some small $\epsilon > 0$, we can obtain (\ref{GiiGapMP}). Similarly, we can prove (\ref{PGiiGapMP}), (\ref{PGPiiGapMP}) and (\ref{QGiiGapMP}). This concludes the proof of Lemma \ref{DiagApproByMP}.
\end{proof}

With the aid of these lemmas, we begin the proof of Theorem \ref{MPCLT}.
Recall that 
	\begin{align*}
		\alpha_t(z) =\E^{\chi} [ \tr( X_t X_t')^{-1} -\tr Q_tG(z)],\qquad
		\beta_t(z) =\frac{1}{1-y_t} \E^{\chi} [\tr G(z) - \tr P_tG(z) ].
	\end{align*}
	Using Lemma \ref{trPGMP}, we have
	\begin{align*}
		\alpha_t(z) = \frac{y_t}{1-y_t} - \frac{y_tm_{y}(z)}{1+m_y(z)} +O_{\prec}\left( \sqrt{\frac{p_t}{N^3}}\vee\frac{p_tp_{\max}}{N^2}  \right),
	\end{align*}
	and similarly,
	\begin{align*}
		\beta_t(z) = \frac{m_{y}(z)}{1-y_t} - \frac{y_tm_{y}(z)}{1+m_y(z)}+ O_{\prec}\left( \sqrt{\frac{p_t}{N^3}}\vee\frac{p_tp_{\max}}{N^2}  \right)
	\end{align*}
	Therefore,
	\begin{align*}
		\frac{\alpha_t(z)}{1 + \alpha_t(z) + \beta_t(z)} =& \frac{\frac{y_t}{1-y_t} - \frac{y_tm_{y}(z)}{1+m_y(z)}  }{\frac{1}{1-y_t}+ \frac{m_{y}(z)}{1-y_t} - \frac{y_tm_{y}(z)}{1+m_y(z)}} +  O_{\prec}\left( \sqrt{\frac{p_t}{N^3}}\vee\frac{p_tp_{\max}}{N^2}  \right)\\
		=& \frac{y_t}{(1+m_y(z))^2} + O_{\prec}\left( \sqrt{\frac{p_t}{N^3}}\vee\frac{p_tp_{\max}}{N^2}  \right),
	\end{align*}
	where we used Taylor expansion and then ignored all the $O_{\prec}(p_t^2/N^2)$ terms.
	Therefore,
	\begin{align*}
		&\left(z_1 - \sum_{s=1}^k\frac{\alpha_s(z_1)}{1+\alpha_s(z_1) + \beta_s(z_1)} \right)^{-1}  = \left( z_1 - \frac{y}{(1+m_y(z_1))^2}  \right)^{-1} + O_{\prec}\left( \frac{1}{\sqrt{N}}   \right)
	\end{align*}
	
	Using the quadratic equation of $m_y(z_1)$ (c.f. (\ref{071000})), we can get
	\begin{align}\label{MPCLTfactor}
		\left( z_1 - \frac{y}{(1+m_y(z_1))^2}  \right)^{-1} = -\frac{m'_{y}(z_1)}{m_{y}(z_1)}.
	\end{align}

	Also, by Lemma \ref{trPGQGMP}, we have,
	\begin{align*}
		&\E^{\chi} \left[ \partial_{z_2}\left( \tr Q_tG(z_1)P_tG(z_2) \right) \right] =\partial_{z_2}\left(\frac{y_tm_y(z_1)m_y(z_2)}{(1+m_y(z_1))(1+m_y(z_2))} \right) + O_{\prec}\left( \sqrt{\frac{p_t}{N^3}}\vee\frac{p_tp_{\max}}{N^2}  \right).
	\end{align*}
	Therefore,
	\begin{align}\label{MPCLTK1}
		&\frac{\E^{\chi} \left[ \partial_{z_2}\left( \tr Q_tG(z_1)P_tG(z_2) \right) \right]}{1 + \alpha_t(z_1) + \beta_t(z_1)} = \frac{\partial_{z_2}\left(\frac{y_tm_y(z_1)m_y(z_2)}{(1+m_y(z_1))(1+m_y(z_2))} \right)}{1 + m_y(z_1)} + O_{\prec}\left( \sqrt{\frac{p_t}{N^3}}\vee\frac{p_tp_{\max}}{N^2}  \right).
	\end{align}
	Similarly, we can obtain
	\begin{align}\label{MPCLTK2}
		&\frac{\E^{\chi} \left[ \partial_{z_2}\left( \tr Q_tG(z_1)G(z_2) \right) \right]}{1 + \alpha_t(z_1) + \beta_t(z_1)} 
		=\frac{\E^{\chi} \left[ \partial_{z_2}\left( \frac{\tr Q_tG(z_1)}{z_1-z_2} - \frac{\tr Q_tG(z_2)}{z_1-z_2}  \right) \right]}{1 + \alpha_t(z_1) + \beta_t(z_1)}\notag \\
	=&\frac{\partial_{z_2}\left( \frac{\frac{y_tm_y(z_1)}{1+m_y(z_1)} -\frac{y_tm_y(z_2)}{1+m_y(z_2)} }{z_1 - z_2}  \right)}{1 + m_y(z_1)} + O_{\prec}\left( \sqrt{\frac{p_t}{N^3}}\vee\frac{p_tp_{\max}}{N^2}  \right).
	\end{align}
	Combining \eqref{MPCLTfactor}, \eqref{MPCLTK1} and \eqref{MPCLTK2}, we can get
	\begin{align*}
		&\mathcal{K}(z_1,z_2) = \frac{m'_{y}(z_1)}{m_{y}(z_1)} \sum_{t=1}^{k}\frac{\partial_{z_2}\left(\frac{y_tm_y(z_1)m_y(z_2)}{(1+m_y(z_1))(1+m_y(z_2))} \right)}{1 + m_y(z_1)}  -  \frac{m'_{y}(z_1)}{m_{y}(z_1)} \sum_{t=1}^{k}\frac{\partial_{z_2}\left( \frac{\frac{y_tm_y(z_1)}{1+m_y(z_1)} -\frac{y_tm_y(z_2)}{1+m_y(z_2)} }{z_1 - z_2}  \right)}{1 + m_y(z_1)}+ O_{\prec}\left( \frac{1}{\sqrt{N}} \right).
	\end{align*}
	By some further calculation, we can simplify the above expression,
	\begin{align}
		&\mathcal{K}(z_1,z_2)= \frac{m'_y(z_1)m'_y(z_2)}{(m_y(z_1)-m_y(z_2))^2} - \frac{1}{(z_1-z_2)^2}-\frac{ym'_y(z_1)m'_y(z_2)}{(1+m_y(z_1))^2(1+m_y(z_2))^2}+ O_{\prec}\left( \frac{1}{\sqrt{N}} \right). \label{0629100}
	\end{align}
	Adding  the parts that $|\Im z| \le N^{-K}$ of the contours back to the integral of the main term in the RHS of (\ref{0629100})  (with negligible error)  gives the variance in Theorem \ref{MPCLT}. 
	
	For the expectation, we can calculate it as follows. Let
	\begin{align}\label{EtE}
		E_t(z) : = \E^{\chi} \left[ \tr P_tG(z) \right] - \frac{Ny_tm_y(z)}{1+m_y(z)}, \quad E(z) := \E^{\chi} \left[\tr G(z) \right] - Nm_{y}(z).
	\end{align}
	By Lemma \ref{trPGMP}, we immediately obtain
	\begin{align*}
		E_t(z) = O_{\prec}\left( \sqrt{\frac{p_t}{N}} \vee \frac{p_tp_{\max}}{N} \right), \quad E(z) = O_{\prec}\left({p_{\max}}\right).
	\end{align*}
	Our task is to give an explicit expression for $E(z)$. In the following estimation, since only one parameter $z$ is involved, we use the shorthand notation $G:=G(z)$ for brevity. 
	
	Following the same procedure as we did in the calculation of expectation in Theorem \ref{FreeCLT}, we start from the cumulant expansion of $\tr P_tG$,
	\begin{align*}
		\E ^{\chi}\left[ \tr P_t G \right] =& \frac{1}{N}\sum_{ij}^{(t)}\E^{\chi} \left[\partial_{t,ji}[W_tG]_{ji} \right] +  \sum_{ij}^{(t)}\frac{\kappa_3^{t,j}}{2N^{\frac32}}\E^{\chi} \left[\partial^2_{t,ji}[W_tG]_{ji} \right]+ \sum_{ij}^{(t)}\frac{\kappa_4^{t,j}}{6N^2}\E ^{\chi}\left[\partial^{3}_{t,ji}[W_tG]_{ji} \right] +  O_{\prec}\left(\frac{p_t}{N^{\frac{3}{2}}} \right).
		\end{align*}
		Here the error term can be bounded similarly to the error estimates of $\mathsf{E}_a, a=1,2,3$ in Lemma \ref{Error in CLT}, we omit the details.
		With a slight abuse of notations, we use $J_{t1}, J_{t2}$ and $J_{t3}$ to denote each term, which reads
	\begin{align*}
		\E^{\chi} \left[ \tr P_t G \right]=&: J_{t1} + J_{t2} + J_{t3} + O_{\prec}\left(\frac{p_t}{N^{\frac{3}{2}}} \right).
	\end{align*}
	For $J_{t1}$, similar to (\ref{062550}), we have
	\begin{align*}
		J_{t1} =& \E^{\chi} \left[ \tr \left(\bbX_t\bbX_t' \right)^{-1}  - \tr \bbQ_t\bbG\right]\E^{\chi} \left[  \Tr \bbG - \Tr \bbP_t\bbG  \right]+ \E^{\chi} \left[ \tr\bbQ_t\bbG\bbP_t\bbG -\tr\bbQ_t\bbG^2-\tr Q_tG\right] +  O_{\prec}\left(\frac{p_t}{N^{\frac{3}{2}}} \right).
	\end{align*}
	For $J_{t2}$, we can apply similar arguments as we did in the calculation of $\mathsf{I}_{t3}$ in the proof of Lemma \ref{Error in CLT}, which gives
	\begin{align*}
		J_{t2} = O_{\prec}\left( \frac{p_t}{N^{\frac76}} \right).
	\end{align*}
	For $J_{t3}$, we only need to consider the terms consisting of four diagonal entries, which reads,
	\begin{align*}
		J_{t3} = \sum_{ij}^{(t)}\frac{\kappa_4^{t,j}}{{N}^{2}}\E^{\chi} \Big[  [\left(\bbX_t\bbX_t' \right)^{-1}]_{jj}^2\left(  [{I} - \bbP_t]_{ii} [(\bbP_t-{I})\bbG]_{ii} - [({I} - \bbP_t)\bbG]_{ii}^2 \right) \Big] + O_{\prec}\left( {\frac{p_t}{N^2}} \right).
	\end{align*}
	Here the error term comes from the estimation for the summation of off-diagonal entries, which is identical to the estimation of $\mathsf{I}_{t42}$ in the proof of Lemma \ref{Error in CLT}. Combining $J_{t1}$, $J_{t2}$ and $J_{t3}$, we get
	\begin{align}\label{EtrPG2MP}
		\E^{\chi} \left[ \Tr P_tG \right] = & \E^{\chi} \left[ \tr \left(\bbX_t\bbX_t' \right)^{-1}  - \tr \bbQ_t\bbG\right]\E^{\chi} \left[  \Tr \bbG - \Tr \bbP_t\bbG  \right]+ \E^{\chi} \left[ \tr\bbQ_t\bbG\bbP_t\bbG -\tr\bbQ_t\bbG^2-\tr Q_tG\right] \notag \\
	&+ \sum_{ij}^{(t)}\frac{\kappa_4^{t,j}}{{N}^{2}}\E^{\chi} \Big[  [\left(\bbX_t\bbX_t' \right)^{-1}]_{jj}^2\left(  [{I} - \bbP_t]_{ii} [(\bbP_t-{I})\bbG]_{ii} - [({I} - \bbP_t)\bbG]_{ii}^2 \right) \Big] +O_{\prec}\left( \frac{p_t}{N^{\frac76}} \right).
	\end{align}  
	Performing cumulant expansion up to $\kappa_{4}^{t,j}$ term on $\E^{\chi}[\Tr P_t] $ and $\E^{\chi}[\Tr P_tGP_t]$, and then using Lemma \ref{trPGMP} and Lemma \ref{DiagApproByMP} to replace the random quantities by deterministic quantities, together with the condition that ${p_{\max}} \le N^{1/2-\epsilon}$ for some small $\epsilon > 0$, we have
\begin{align}\
	&\E^{\chi} \left[\Tr \left(\bbX_t\bbX_t' \right)^{-1} \right] = \frac{\E^{\chi}[\Tr P_t]}{1-y_t} + \frac{y_t}{(1-y_t)^2}+  \sum_{ij}^{(t)}\frac{\kappa_4^{t,j} }{N^2(1-y_t)} +{ O_{\prec}\left(\frac{p_t}{N^{\frac76}}\right)},\label{EtrXXMP}\\
	&\E^{\chi} \left[\Tr \bbQ_t\bbG\right] = \frac{\E^{\chi}[\Tr P_tGP_t]}{1-y_t-\frac{1}{N}} - \sum_{ij}^{(t)}\frac{\kappa_4^{t,j}m_y(z)}{N^2(1-y_t)^2}+ { O_{\prec}\left(\frac{p_t}{N^{\frac76}}\right)},\label{EtrQGMP}
\end{align}
and 
\begin{align}\label{k4EMP}
	&\sum_{ij}^{(t)}\frac{\kappa_4^{t,j}}{{N}^{2}}\E^{\chi} \bigg[  [\left(\bbX_t\bbX_t' \right)^{-1}]_{jj}^2\left(  [{I} - \bbP_t]_{ii} [(\bbP_t-{I})\bbG]_{ii} - [({I} - \bbP_t)\bbG]_{ii}^2 \right) \bigg]\notag\\
	 =& -\sum_{ij}^{(t)}\frac{\kappa_4^{t,j}\left(m_y(z) + m_y^2(z) \right)}{N^2(1-y_t)^2} + { O_{\prec}\left(\frac{p_t}{N^{\frac32}}\right)}.
\end{align}
Plugging \eqref{EtE}, \eqref{EtrXXMP}, \eqref{EtrQGMP} and \eqref{k4EMP} into \eqref{EtrPG2MP}, after some basic algebra, $\kappa_4^{t,j}$ terms will be cancelled out, and we get
\begin{align*}
&\frac{1}{N}\frac{E_t^2(z)}{1-y_t}-E_t(z)\left(1+\frac{m_y(z)}{1-y_t}+\frac{y_t(1-m_y(z))}{(1-y_t)(1+m_y(z))}+\frac{E(z)}{N(1-y_t)}\right)\\
&+\frac{Ny_t^2m_y^2(z)}{(1+m_y(z))^2(1-y_t)}+\frac{y_t E(z)}{(1+m_y(z))(1-y_t)}-\partial_z \frac{y_tm_y(z)}{1+m_y(z)}+\frac{y_tm_y^2(z)}{(1+m_y(z))^2}=O_{\prec}\left( \frac{p_t}{N^{\frac76}} \right).
\end{align*}
And we know that
\begin{align*}
	&\frac{1}{N}\frac{E_t^2(z)}{1-y_t} = O_{\prec}\left( \frac{p_t^2{p_{\max}^2}}{N^3} \right), \qquad 
	\frac{1}{N}\frac{E_t(z)E(z)}{1-y_t} = O_{\prec}\left( \frac{p_t{p_{\max}^2}}{N^2} \right), \\
	&y_tE_t(z)  = O_{\prec}\left( \frac{p_t^2{p_{\max}}}{N^2}  \right),\qquad 
	y_t^2E(z) = O_{\prec}\left( \frac{p_t^2{p_{\max}}}{N^2}  \right).
\end{align*}
Together with the condition that $p_{\max} \le N^{1/2-\epsilon}$ for some small $\epsilon > 0$, we have
\begin{align*}
	&-E_t(z)\left(1+m_y(z)\right)+\frac{Ny_t^2m_y^2(z)}{(1+m_y(z))^2}+\frac{y_t E(z)}{(1+m_y(z))}-\partial_z \frac{y_tm_y(z)}{1+m_y(z)}+\frac{y_tm_y^2(z)}{(1+m_y(z))^2}=O_{\prec}\left( \frac{p_t}{N^{7/6}} \right).
\end{align*}
Dividing the coefficient of $E_t$ on both sides, and then summing over t, we have
\begin{align*}
	\sum_{t=1}^{k} E_t(z) =& \sum_{t=1}^k \frac{Ny_t^2m_y^2(z)}{(1+m_y(z))^2} + \frac{yE}{(1+m_y(z))^2}  - \partial_z \frac{ym_y(z)}{(1+m_y(z))^2} +\frac{ym_y^2(z)}{(1+m_y(z))^3} 
\end{align*}
Notice that by the resolvent identity, we have
\begin{align*}
	\sum_{t=1}^kE_t(z) = zE(z).
\end{align*}
Hence, 
\begin{align}
	 E(z) =& \left( z - \frac{y}{(1+m_y(z))^2} \right)^{-1} \left(\sum_{t=1}^k \frac{Ny_t^2m_y^2(z)}{(1+m_y(z))^2} - \partial_z \frac{ym_y(z)}{(1+m_y(z))^2} +\frac{ym_y^2(z)}{(1+m_y(z))^3}  \right)+ O_{\prec}\left( \frac{1}{N^\frac{1}{6}} \right)\notag\\
	 =&-\sum_{t=1}^k\frac{Ny_t^2m_y(z) m_y'(z)}{(1+m_y(z))^3}+\frac{y(m_y'(z))^2}{m_y(z)(1+m_y(z))^3}-\frac{ym_y(z)m_y'(z)}{(1+m_y(z))^3} +O_{\prec}\left( \frac{1}{N^\frac{1}{6}} \right). \label{0629101}
\end{align}
	Adding  the parts that $|\Im z| \le N^{-K}$ of the contours back to the integral of the main term in the RHS of (\ref{0629101})  (with negligible error) gives the expectation in Theorem \ref{MPCLT}.

\section{Removing the sample mean: Proof of Theorem \ref{RemoveSampleMean}}\label{ProofofRMS}
In this section, we discuss the case when sample mean does not equal to $0$. To subtract the sample mean, we first introduce the following projection matrix,
\begin{align}
	A := \mathrm{I} - \frac{1}{N}\mathds{1}\mathds{1}'. \label{070410}
\end{align}
Therefore, the sample matrix becomes $\tilde{X}_t := X_tA$. Similarly,
let 
\begin{align*}
	&\tilde{P}_t := AX_t'(X_tAX_t')^{-1}X_tA,\quad \tilde{Q}_t := AX_t'(X_tAX_t')^{-2}X_tA,\quad\tilde{W}_t := (X_tAX_t')^{-1}X_tA.
\end{align*}
We would like to study the LSS of 
\begin{align*}
	\tilde{H} := \sum_{t=1}^{k} \tilde{P}_t.
\end{align*}
The corresponding Green function is denoted by 
\begin{align*}
	\tilde{G}\equiv\tilde{G}(z) := (\tilde{H}-z)^{-1}.
\end{align*}
The proof of Theorem \ref{RemoveSampleMean} follows almost the same procedure as the proof of Theorem \ref{FreeCLT} with minor modification. Therefore, we only show the key steps, i.e, (i) how the $N-1$ factor shows up, (ii) how we do the cumulant expansion for terms like $[\tilde{P}_t\tilde{G}]_{ii}$, which does not have $X_t$ as its leading factor, (iii) some additional technical treatments. 
%
%
%
%

Notice that for all $t \in [\![k]\!]$, we have
\begin{align*}
	 \tr \tilde{P}_t = \tr X_t'(X_tAX_t')^{-1}X_tA,\quad
	 \tr \tilde{P}_t\tilde{G}\tilde{P}_t =\tr X_t'(X_tAX_t')^{-1}X_tA\tilde{G}\tilde{P}_t,\quad
	 \tr \tilde{P}_t\tilde{G} =\tr X_t'(X_tAX_t')^{-1}X_tA\tilde{G}A.
\end{align*}
Using the recursive moment estimates similar to the proof of Lemma \ref{Lemma error estimates}, one can show that
\begin{align}
& \tr \tilde{P}_t = \tr (X_tAX_t')^{-1}\left( \tr A - \tr \tilde{P}_t \right)+O_{\prec}\left( \sqrt{\frac{p_t}{N^3}} \right)\label{trPSM},\quad t \in [\![k]\!],\\
	&\tr \tilde{P}_t\tilde{G}\tilde{P}_t = \tr \tilde{Q}_t\tilde{G}A \left( \tr A - \tr \tilde{P}_t \right) +O_{\prec}\left( \sqrt{\frac{p_t}{N^3}} \right),\quad t \in [\![k]\!],\label{trPGPSM}\\
	&\tr \tilde{P}_t\tilde{G} = \left(\tr (X_tAX_t')^{-1} - \tr \tilde{Q}_t\tilde{G}A  \right)\left( \tr \tilde{G}A - \tr \tilde{P}_t\tilde{G}A \right) +O_{\prec}\left( \sqrt{\frac{p_t}{N^3}} \right),\quad t \in [\![k]\!].\label{trPGSM}
\end{align}
By \eqref{trPGPSM} and \eqref{trPSM} together with the trivial identity $\tr A = (N - 1)/N, \tr \tilde{P}_t = p_t/N$, we have for all $t \in [\![k]\!]$,
\begin{align*}
	&\tr (X_tAX_t')^{-1} = \frac{\tilde{y}_t}{1-\tilde{y}_t} + O_{\prec}\left( \sqrt{\frac{p_t}{N^3}} \right),\quad\tr \tilde{Q}_t\tilde{G}A = \frac{1}{(1-\tilde{y}_t)}\frac{\Tr \tilde{P}_tGA}{N-1} +O_{\prec}\left( \sqrt{\frac{p_t}{N^3}} \right).
\end{align*}
where $\tilde{y}_t := p_t / (N-1)$. Plugging the above two estimates into \eqref{trPGSM}, we get
\begin{align}
	&\frac{\Tr \tilde{P}_t\tilde{G}A}{N-1} = \left( \frac{\tilde{y}_t}{1-\tilde{y_t}} - \frac{1}{(1-\tilde{y}_t)}\frac{\Tr \tilde{P}_t\tilde{G}A }{N-1}\right)\left( \frac{\Tr \tilde{G}A}{N-1} - \frac{\Tr \tilde{P}_t\tilde{G}A }{N-1}\right) + O_{\prec}\left( \sqrt{\frac{p_t}{N^3}} \right).\quad t \in [\![k]\!]. \label{070602}
\end{align}
Define the new approximation subordination functions $\tilde{\omega}_t^c(z)$ as 
\begin{align}
	\tilde{\omega}_t^c(z) = z - \frac{\sum_{s \neq t}\frac{\Tr \tilde{P}_s\tilde{G}A}{N-1}}{\tilde{m}_N(z)},\quad t \in [\![k]\!], \label{070401}
\end{align}
where $\tilde{m}_N(z) := \Tr \tilde{G}A / (N-1)$. Therefore, by the definition of $\tilde{\omega}_t^c(z)$ and the identity $\sum_{t=1}^k \Tr \tilde{P}_t\tilde{G}A = \Tr \tilde{H}\tilde{G}A = \Tr (\tilde{H}-z)\tilde{G}A +z\Tr \tilde{G}A = N-1+z\Tr \tilde{G}A $, we have  
\begin{align*}
	\sum_{t=1}^k \tilde{\omega}_t^c(z) = z - \frac{k-1}{\tilde{m}_N(z)}.
\end{align*}
Another direct consequence from (\ref{070401}) is that we can rewrite $\Tr \tilde{P}_t\tilde{G}A / (N-1)$ as $ 1 + \tilde{\omega}_t^c(z)\tilde{m}_N(z)$ for all $t \in [\![ k ]\!]$, which together with (\ref{070602}) gives
\begin{align}
	1 + \tilde{\omega}_t^c(z)\tilde{m}_N(z) = \left( \frac{\tilde{y}_t}{1-\tilde{y_t}} - \frac{1 +\tilde{\omega}_t^c(z)\tilde{m}_N(z)  }{1-\tilde{y}_t}\right)\left( \tilde{m}_N(z) - (1 +\tilde{\omega}_t^c(z)\tilde{m}_N(z) )\right) + O_{\prec}\left( \sqrt{\frac{p_t}{N^3}} \right),\quad t \in [\![k]\!]. \label{070404}
\end{align}
We see that the above equation is almost identical to (\ref{070403}) in the proof of Proposition \ref{keyProp}, but with $y_t$, $\omega_t^c(z)$ and $m_N(z)$ replaced by $\tilde{y}_t$, $\tilde{\omega}_t^c(z)$ and $\tilde{m}_N(z)$, respectively. As a result,  (\ref{070401}) and (\ref{070404}) will form a perturbed system of (\ref{Phi2}) with $\mu_t\equiv \tilde{\mu}_t := \tilde{y}_t\delta_{\{1\}}+(1-\tilde{y}_t)\delta_{\{0\}}, t \in [\![ k]\!]$. By the stability analysis as we did in the proof of Proposition \ref{keyProp}, we get
\begin{align*}
	\left|\tilde{\omega}_t^c(z) - \tilde{\omega}_t(z)\right| = O_{\prec}\left( \frac{1}{N} \right), t \in [\![ k]\!],
\end{align*}
where $\tilde{\omega}_t(z)$ is the solution of \eqref{Phi2} with $\mu_t\equiv \tilde{\mu}_t$. This further implies 
\begin{align*}
	\left|\tilde{m}_N(z)- \tilde{m}_{\boxplus}(z) \right| = O_{\prec}\left( \frac{1}{N} \right).
\end{align*}
where 
$
	\tilde{m}_{\boxplus}(z) := m_{\mu_t}(\tilde{\omega_t}(z)), t\in [\![k]\!]
$
is the Stieltjes transform of $\tilde{\mu}_{\boxplus}:= \tilde{\mu}_1\boxplus\cdots\boxplus\tilde{\mu}_k$. 

Next, we explain how we do the cumulant expansion for terms like $[\tilde{P}_t\tilde{G}]_{ii}$, which does not have $X_t$ as its leading factor and we cannot use the cyclicity of trace. We take the estimate for $\sum_{t=1}^{k}[\tilde{P}_t\tilde{G}]_{ii}$ as an example. Since
\begin{align}
	\sum_{t=1}^{k}[\tilde{P}_t\tilde{G}]_{ii} =& \sum_{t=1}^{k}[X_t'(X_tAX_t')^{-1}X_tA\tilde{G}]_{ii} - \frac{1}{N}\sum_{t=1}^{k}[\mathds{1}\mathds{1}'X_t'(X_tAX_t')^{-1}X_tA\tilde{G}]_{ii} \notag \\
	=& \sum_{t=1}^{k}[X_t'(X_tAX_t')^{-1}X_tA\tilde{G}]_{ii} -\frac{1}{N}\sum_{t=1}^{k}\sum_{u=1}^{N}[X_t'(X_tAX_t')^{-1}X_tA\tilde{G}]_{ui} \label{070406}
\end{align}
For the first term we can perform cumulant expansion as usual. For the second term, let $M :=\sum_{t=1}^{k}X_t'(X_tAX_t')^{-1}X_t,$
we have
\begin{align*}
	\Big|\frac{1}{N}\sum_{t=1}^{k}\sum_{j=1}^{N}[X_t'(X_tAX_t')^{-1}X_tA\tilde{G}]_{ji}\Big|   \le  \frac{1}{N} \sqrt{\sum_{u=1}^{N}|[MA\tilde{G}]_{ui}|^{2}N} = \frac{1}{\sqrt{N}}\sqrt{[\tilde{G}^{*}AMMA\tilde{G}]_{ii}} \prec  \frac{1}{\sqrt{N}},
\end{align*}
where in the first step we used Cauchy-schwarz, and in the last step we used the fact that $\|M \|, \| \tilde{G}A\| \prec 1$.
Therefore, the second term in (\ref{070406}) can be viewed as the error term. Thus, similar to the proof of Lemma \ref{LemSumDiag}, we have for any $\alpha_t$ satisfies $\sup_t |\alpha_t| < C$,
\begin{align}
	\sum_{t=1}^{k}\alpha_t[\tilde{P}_t\tilde{G}]_{ii} = \sum_{t=1}^k\alpha_t(\tr (X_tAX_t')^{-1} - \tr \tilde{Q}_t\tilde{G})([\tilde{G}A]_{ii} - [\tilde{P}_t\tilde{G}]_{ii}) + O_{\prec}\left( \frac{1}{\sqrt{N}} \right). \label{070409}
\end{align} 
We see that from (\ref{070409}), we have $[\tilde{G}A]_{ii}$ instead of $[\tilde{G}]_{ii}$. This is another technical thing that we have to deal with. Actually, $[\tilde{G}A]_{ii}$ and  $[\tilde{G}]_{ii}$ share the same approximation, up to negligible error. In general, we have the following lemma.
\begin{lemma}\label{070412}
	For any $m\in \mathbb{N}$, let $U \in \mathbb{C}^{m \times N}$ be a random matrix with $\|U \| \prec 1$. Then we have 
	\begin{align*}
		|[UA]_{ui} - [U]_{ui}| \prec N^{-\frac{1}{2}}.
	\end{align*}
	for any $u \in [\![ m ]\!]$ and $i \in [\![ N ]\!]$.
\end{lemma}
\begin{proof}
	By the definition of $A$ in (\ref{070410}), we have
	\begin{align*}
		|[UA]_{ui} - [U]_{ui}| = \frac{1}{N}|[U\mathds{1}\mathds{1}']_{ui}| = \frac{1}{N}\Big| \sum_{i=1}^N[U]_{ui}\Big| \le \frac{1}{N}\sqrt{\sum_{i=1}^N|[U]_{ui}|^2N} =\frac{1}{\sqrt{N}}\sqrt{[UU']_{uu}} \prec  \frac{1}{\sqrt{N}},
	\end{align*}
	where in the third step we used Cauchy-schwarz, and in the last step we used the condition that $\|U \|\prec 1$.
\end{proof}
By Lemma \ref{070412} together with the fact that $\sum_{t=1}^k\alpha_t(\tr (X_tAX_t')^{-1} - \tr \tilde{Q}_t\tilde{G}) = O_{\prec}(1)$ , (\ref{070409}) can be rewritten as
\begin{align*}
	\sum_{t=1}^{k}\alpha_t[\tilde{P}_t\tilde{G}]_{ii} = \sum_{t=1}^k\alpha_t(\tr (X_tAX_t')^{-1} - \tr \tilde{Q}_t\tilde{G})([\tilde{G}]_{ii} - [\tilde{P}_t\tilde{G}]_{ii}) + O_{\prec}\left( \frac{1}{\sqrt{N}} \right). 
\end{align*}
This gives $| [\tilde{G}]_{ii} - \tilde{m}_{\boxplus}(z) |\prec N^{-1/2}$, and the other estimates for diagonal entries can be obtain similarly. 

The last technical issue we would like to point out is that in the estimation of $\mathsf{I}_{t1}$ in (\ref{061301}), the term $\tr Q_tG^2$ is rewritten as $\partial_z\tr Q_tG$. Then by Cauchy's integral formula, we can use the estimate for $\tr Q_tG$ to obtain the estimate for $\tr Q_tG^2$. However, in the current case, this term becomes $\tr \tilde{Q}_t\tilde{G}A\tilde{G}$, which prevent us from using the identity $\tilde{G}^2 = \partial_z\tilde{G}$. This issue can be handled by the following lemma.
\begin{lemma}\label{070420}
	For any bounded $\mathcal{W}_t\in \mathbb{C}$ such that $\sup_t |\mathcal{W}_t| < C$ for some strictly positive constant $C$, $z_1$ and $z_2$ such that $\|\tilde{G}(z_1) \|, \|\tilde{G}(z_2) \| \prec 1$, we have
	\begin{align*}
		\sum_{t=1}^{k} \mathcal{W}_t\tr \tilde{Q}_t\tilde{G}(z_1)A\tilde{G}(z_2) = \sum_{t=1}^{k} \mathcal{W}_t\tr \tilde{Q}_t\tilde{G}(z_1)\tilde{G}(z_2) + O_{\prec}\left( \frac{1}{N} \right).
	\end{align*}
\end{lemma}
\begin{proof}
	By the definition of $A$ in (\ref{070410}), we have
	\begin{align*}
		\sum_{t=1}^{k} \mathcal{W}_t\tr \tilde{Q}_t\tilde{G}(z_1)A\tilde{G}(z_2) - \sum_{t=1}^{k}\mathcal{W}_t\tr \tilde{Q}_t\tilde{G}(z_1)\tilde{G}(z_2) = \frac{1}{N}\sum_{t=1}^{k}\mathcal{W}_t\tr \tilde{Q}_t\tilde{G}(z_1)\mathds{1}\mathds{1}'\tilde{G}(z_2)
	\end{align*}
	Denoted by $\tilde{M} := \sum_{t=1}^k\mathcal{W}_t\tilde{Q}_t$. Similar to (\ref{SumQt}), we have $\|\tilde{M} \| \prec 1$. Therefore,
	\begin{align*}
		\frac{1}{N}\sum_{t=1}^{k}\mathcal{W}_t\tr \tilde{Q}_t\tilde{G}(z_1)\mathds{1}\mathds{1}'\tilde{G}(z_2) = \frac{1}{N^2}\sum_{i,u}^N[\tilde{G}(z_2)\tilde{M}\tilde{G}(z_1)]_{ui}\le  \frac{1}{N} \|\tilde{G}(z_2)\tilde{M}\tilde{G}(z_1) \| \prec  \frac{1}{N},
	\end{align*}
	which completes the proof.
\end{proof}
The foregoing lemma indicates that the error of dropping the projection $A$ after summation over $t$ is still negligible. Therefore, we can again use the identity $\tilde{G}^2 = \partial_z\tilde{G}$ to transform the estimate for $\tr \tilde{Q}_t\tilde{G}^2$ to the estimate for $\tr \tilde{Q}_t\tilde{G}$. In addition, Lemma \ref{070420} can be also applied to the estimation of $\mathsf{I}_{t2}$ (c.f. (\ref{061351})) , where $\tr Q_tG(z_1)G(z_2)$ becomes $\tr \tilde{Q}_t\tilde{G}(z_1)A\tilde{G}(z_2)$ in the current case.

%
%

For the estimation of the expectation, we let
\begin{align*}
	\tilde{E}_t(z) : = \E^{\chi} [ \Tr \tilde{P}_t\tilde{G} ] - (N-1)(1+\tilde{\omega}_t(z)\tilde{m}_{\boxplus}(z)), \quad \tilde{E}(z) := \E^{\chi} [\Tr \tilde{G}A ] - (N-1)\tilde{m}_{\boxplus}(z).
\end{align*}
Then following the same procedure as the calculation fo the expectation in Theorem \ref{FreeCLT}, we have the approximation for $\tilde{E}(z)$. By the trivial identity $N + z\Tr \tilde{G}=\sum_{t=1}^k \Tr \tilde{P}_t\tilde{G} = \sum_{t=1}^k \Tr \tilde{P}_t\tilde{G}A = N-1+z\Tr \tilde{G}A $, we have
\begin{align*}
	\E^{\chi} \big[\Tr\tilde{G} \big] = \E^{\chi}  \big[\Tr \tilde{G}A \big] -\frac{1}{z} + O_{\prec}(N^{-D}) =  (N-1)\tilde{m}_{\boxplus}(z) + \tilde{E}(z)-\frac{1}{z} .
\end{align*}
This gives the expectation in Theorem \ref{RemoveSampleMean}.

\section{Relaxing the moment condition: Proof of Theorem \ref{mainth22}}\label{Proofofmainth22}
In this section, we use a Green function comparison argument to relax the moment condition. We only consider the relaxation of Theorem \ref{FreeCLT}, and the others can be done similarly. We start with the following lemma \cite{he2020mesoscopic}.
\begin{lemma}\label{truncation}
	Fix $m > 2$ and let $x$ be a real random variable, with absolutely continuous law, satisfying 
	\begin{align*}
		\E [x] = 0, \quad \E [x^2] = \sigma^2, \quad \E [|x|^m] \le C_m,
	\end{align*}
	for some constant $C_m > 0$. Let $\lambda > 2\sigma$. Then there exists a real random variable $y$ that satisfies
	\begin{align*}
		\E [y] = 0, \quad \E [y^2] = \sigma^2, \quad |y| \le \lambda,\quad \mathbb{P}\left( x \neq y \right) \le 2C_m\lambda^{-m}.
	\end{align*}
	In particular, $\E [| y|^m] \le 3C_m$. Moreover, if $m>4$ and $\sigma = 1$, then there exists a real random variable $z$ matching the first four moments of $y$, and satisfies $|z| \le 6C_m$.
\end{lemma}

Now let $H$ be the matrix defined in (\ref{def of H}) under Assumption \ref{assum2} and \ref{assum3}. 
By using the foregoing lemma, we can construct the following two versions of $X_t, t \in [\![ k]\!]$. Let $x := \sqrt{N}X_{t,ij}, \lambda := N^{1/2 - \epsilon}, C_m := C$, and $\epsilon = \delta / (4(4+\delta)) > 0$, we can construct a random variable $X_{t,ij}^{(1)} := N^{-1/2}y$, such that 
\begin{align*}
	\E \big[X_{t,ij}^{(1)}\big] = 0, \quad \E \big[\big(X_{t,ij}^{(1)}\big)^2\big] = \E \big[\left(X_{t,ij}\right)^2\big], \quad |X_{t,ij}^{(1)}| \le N^{-\epsilon}, \quad
		\mathbb{P}\left( X_{t,ij} \neq X_{t,ij}^{(1)} \right) \le 2CN^{-2-\delta/4},
\end{align*}
With $X_{t,ij}^{(1)}$ we define $H^{(1)} := \sum_{t=1}^k \big(X_{t,ij}^{(1)}\big)'\big( X_{t,ij}^{(1)} \big( X_{t,ij}^{(1)} \big)'\big)^{-1}X_{t,ij}^{(1)}$. Further using the second part of Lemma \ref{truncation} on $y = \sqrt{N}X_{t,ij}^{(1)}$, we can construct a random variable $X_{t,ij}^{(2)} : = N^{-1/2}z$, and $X_{t,ij}^{(2)}$ satisfies Assumption \ref{assum1}. Then we can define $H^{(2)}$ analogously. The main idea is due to $\mathbb{P}(X_{t,ij} \neq X_{t,ij}^{(1)}) \le 2CN^{-2-\delta/4}$, we have $\mathbb{P}(H \neq H^{(1)}) = O(N^{-\frac{\delta}{4}})$. Hence, we can work with $H^{(1)}$ instead of $H$. Then we compare the statistics of $H^{(1)}$ with $H^{(2)}$, where the latter satisfies Assumption \ref{assum1} and thus all previous result in this work hold for $H^{(2)}$.

Let 
$$H^{(t,(i)(j))} := \sum_{s=1}^{k} \big(\mathcal{X}_s^{(t,(i)(j))}\big)'\big(\mathcal{X}_s^{(t,(i)(j))}\big(\mathcal{X}_s^{(t,(i)(j))}\big)'\big)^{-1}\mathcal{X}_s^{(t,(i)(j))},$$ 
where $\mathcal{X}_s^{(t,(i)(j))}$ is defined entrywise by $\mathcal{X}_{s,uv}^{(t,(i)(j))} := X_{t,uv}^{(2)}$ if $u \le j, v \le i$, and $s \le t$; otherwise, $\mathcal{X}_{s,uv}^{(t,(i)(j))} := X_{t,uv}^{(1)}$. In particular, $H^{(1,(1)(0))} := H^{(1)}$, $H^{(k,(N)(p_k))} := H^{(2)} $ and $H^{(s,(1)(0))} :=  H^{(s-1,(N)(p_{s-1}))}$. The Green function of $H^{(t,(i)(j))}$ is denoted by $G^{(t,(i)(j))} \equiv G^{(t,(i)(j))}(z) := (H^{(t,(i)(j))} -z)^{-1}$.

By the rigidity of eigenvalues of sample covariance matrices with bounded support condition \cite{xi2020convergence}, i.e., $|X_{t,ij}^{(1)}|, |X_{t,ij}^{(2)}| \le N^{-\epsilon}$, we have 
\begin{align}
	c<\lambda_{\min}\big(\mathcal{X}_s^{(t,(i)(j))}\big(\mathcal{X}_s^{(t,(i)(j))}\big)' \big) \le \lambda_{\max}\big(\mathcal{X}_s^{(t,(i)(j))}\big(\mathcal{X}_s^{(t,(i)(j))}\big)' \big) < C, \label{070601}
\end{align}
for some strictly positive constants $c, C > 0$ with high probability. Note that the result presented in \cite{xi2020convergence} only consider the case of $y_s \sim 1$. Based on (\ref{070601}), we can easily show that (\ref{070601}) holds for the case of $0 <y_s < 1$ by Cauchy Interlacing Theorem. More specifically, the results of lower rank covariance matrices follow from the one with large rank. Therefore, we still have the following high probability bounds,
\begin{align}
	\|G^{(t,(i)(j))} \|,\; \|\mathcal{X}_s^{(t,(i)(j))} \|,\; \| \big(\mathcal{X}_s^{(t,(i)(j))}\big(\mathcal{X}_s^{(t,(i)(j))}\big)'\big)^{-1} \| \prec 1. \label{070610}
\end{align}
Then, similar to (\ref{062810}), we have (with high probability) for suitable contour $\gamma$ ,
\begin{align*}
	\Tr f(H^{(t,(i)(j))})=&\frac{-1}{2\pi{\rm i}}\oint_{\gamma} \Tr G^{(t,(i)(j))}f(z){\rm d} z=\frac{-1}{2\pi{\rm i}}\oint_{\gamma \setminus \{|\Im z| < N^{-K}\} } \Tr G^{(t,(i)(j))}f(z){\rm d} z \cdot \Xi^{(t,(i)(j))}  +O_\prec(N^{-K+1}) \\
	&=:L_N^{(t,(i)(j))}(f)+O_\prec(N^{-K+1}),
\end{align*}
here the choice of $\gamma$ depends on the test function $f$ (c.f. (\ref{contour12})). The truncation function $\Xi^{(t,(i)(j))}$ is defined as 
\begin{align}
	\Xi^{(t,(i)(j))}:=\prod_{s=1}^k\chi \big(\tr \big(\mathcal{X}_s^{(t,(i)(j))}\big(\mathcal{X}_s^{(t,(i)(j))}\big)'\big)^{-1} \big )\chi\big(\tr \mathcal{X}_s^{(t,(i)(j))}\big(\mathcal{X}_s^{(t,(i)(j))}\big)'\big), \label{Xtij}
\end{align}
which is used to control $\big\|\big(\mathcal{X}_s^{(t,(i)(j))}\big(\mathcal{X}_s^{(t,(i)(j))}\big)'\big)^{-1}\big \|$ and $\big\|\mathcal{X}_s^{(t,(i)(j))}\big(\mathcal{X}_s^{(t,(i)(j))}\big)'\big\|$ crudely but deterministically.

Next we do the comparison on the $L^{(1,(1)(0))}_{N}(f)$ (corresponds to $H^{(1)}$) and $L^{(k,(N)(p_k))}_{N}(f)$ (corresponds to $H^{(2)}$). Let $F = F(x + \mathrm{i}y)$ be a complex-valued, smooth, bounded function, with bounded derivatives. Then
\begin{align*}
	&\E \Big[ F\left( L^{(k,(N)(p_k))}_{N}(f)   \right) \Big] -\E \Big[ F\left( L^{(1,(1)(0))}_{N}(f)\right) \Big] = \sum_{t=1}^{k}\sum_{ij}^{(t)} \E \Big[ F\left(L_N^{(t,(i)(j))}(f)  \right) \Big] - \E \Big[ F\left( L_N^{(t,(i)(j-1))}(f) \right) \Big]. 
\end{align*}
Let us focus on the one step difference
\begin{align}\label{RelaxError1}
	&\E \Big[ F\left(L_N^{(t,(i)(j))}(f)  \right) \Big] - \E \Big[ F\left( L_N^{(t,(i)(j-1))}(f) \right) \Big].
\end{align}
Hereafter, we use the three tuple $(s,u,v)$ to locate the entries of $H^{(t,(i)(j))}$, i.e., the $(s,u,v)$ entry of $H^{(t,(i)(j))}$ is $\mathcal{X}_{s,uv}^{(t,(i)(j))}$. Notice that  $H^{(t,(i)(j))}$ and $H^{(t,(i)(j-1))}$ have only one different entry at the position $(t,i,j)$. With a slight abuse of notation, we use $\partial_{t,ij}$ to denote the partial derivative w.r.t the $(t,i,j)$ entry. Therefore, we can perform Taylor expansion around $0$ at the $(t,i,j)$ entry of the two terms in (\ref{RelaxError1}). We view $F$ is a function of $(t,i,j)$ entry, and use the notation $F|_{(t,i,j) = x}$ to denote the value of $F$ with $(t,i,j)$ entry being $x$. 
\begin{align*}
	F\left( L_N^{(t,(i)(j))}(f)  \right) &= F\left( L_N^{(t,(i)(j))}(f)  \right)\Big|_{(t,i,j) = 0} + X_{t,ij}^{(2)} \cdot \partial_{t,ij}F\left( L_N^{(t,(i)(j))}(f)  \right)\Big|_{(t,i,j) = 0}\\
	 &+ \frac{\big(X_{t,ij}^{(2)} \big)^2}{2} \partial^2_{t,ij}F\left( L_N^{(t,(i)(j))}(f)   \right)\Big|_{(t,i,j) = 0}+\frac{\big(X_{t,ij}^{(2)} \big)^3}{6} \partial^3_{t,ij}F\left(L_N^{(t,(i)(j))}(f)   \right)  \Big|_{(t,i,j)= 0} \\
	 &+ \frac{\big(X_{t,ij}^{(2)} \big)^4}{24} \partial^4_{t,ij}F\left( L_N^{(t,(i)(j))}(f)  \right) \Big|_{(t,i,j) = 0} + \frac{\big(X_{t,ij}^{(2)} \big)^5}{120} \partial^5_{t,ij}F\left( L_N^{(t,(i)(j))}(f) \right) \Big|_{(t,i,j) = h_2},
\end{align*}
where $h_2$ is a random variable satisfying $|h_2| \le |X_{t,ij}^{(2)}|$.
Similarly,
\begin{align*}
	F\left( L_N^{(t,(i)(j-1))}(f)  \right) &= F\left( L_N^{(t,(i)(j-1))}(f)  \right)\Big|_{(t,i,j) = 0} + X_{t,ij}^{(1)} \cdot \partial_{t,ij}F\left( L_N^{(t,(i)(j-1))}(f)  \right)\Big|_{(t,i,j) = 0}\\
	 &+ \frac{\big(X_{t,ij}^{(1)} \big)^2}{2} \partial^2_{t,ij}F\left( L_N^{(t,(i)(j-1))}(f)   \right)\Big|_{(t,i,j) = 0}+\frac{\big(X_{t,ij}^{(1)} \big)^3}{6} \partial^3_{t,ij}F\left(L_N^{(t,(i)(j-1))}(f)   \right)  \Big|_{(t,i,j)= 0} \\
	 &+ \frac{\big(X_{t,ij}^{(1)} \big)^4}{24} \partial^4_{t,ij}F\left( L_N^{(t,(i)(j-1))}(f)   \right) \Big|_{(t,i,j) = 0} + \frac{\big(X_{t,ij}^{(1)} \big)^5}{120} \partial^5_{t,ij}F\left( L_N^{(t,(i)(j-1))}(f)  \right) \Big|_{(t,i,j) = h_1},
\end{align*}
where $h_1$ is a random variable satisfying $|h_1| \le |X_{t,ij}^{(1)}|$. Since $H^{(t,(i)(j))}$ and $H^{(t,(i)(j-1))}$ have only one different entry at the position $(t,i,j)$, we have
$$\partial^{a}_{t,ij}F\left( L_N^{(t,(i)(j-1))}(f)   \right)\Big|_{(t,i,j) = 0} = \partial^{a}_{t,ij}F\left( L_N^{(t,(i)(j))}(f)   \right)\Big|_{(t,i,j) = 0}, \quad a = 1,2,3,4.$$
Together with the fact that the first four moments of $X_{t,ij}^{(1)}$ and $X_{t,ij}^{(2)}$ are identical (c.f. Lemma \ref{truncation}), we have 
\begin{align*}
	&\E \Big[ F\left(L_N^{(t,(i)(j))}(f)  \right) \Big] - \E \Big[ F\left( L_N^{(t,(i)(j-1))}(f) \right) \Big] =\frac{1}{120}\bigg( \E\left[\big(X_{t,ij}^{(2)} \big)^5 \partial^5_{t,ij}F\left( L_N^{(t,(i)(j))}(f) \right) \Big|_{(t,i,j) = h_2} \right] \\
	&- \E \left[ \big(X_{t,ij}^{(1)} \big)^5\partial^5_{t,ij}F\left( L_N^{(t,(i)(j-1))}(f)  \right) \Big|_{(t,i,j) = h_1} \right] \bigg) =: I^{(2)} + I^{(1)}.
\end{align*}
For $I^{(2)}$, we have
\begin{align*}
	|I^{(2)}| \lesssim& \E\left[\big|X_{t,ij}^{(2)} \big|^5 \cdot \sup_{|h_2| \le N^{-\epsilon}}\bigg| \partial^5_{t,ij}F\left( L_N^{(t,(i)(j))}(f) \right) \Big|_{(t,i,j) = h_2} \bigg|  \right]\\
	 =& \E \left[\big|X_{t,ij}^{(2)} \big|^5\right]\E\left[\sup_{|h_2| \le N^{-\epsilon}}\bigg| \partial^5_{t,ij}F\left( L_N^{(t,(i)(j))}(f) \right) \Big|_{(t,i,j) = h_2} \bigg|  \right]
\end{align*}
Let $\tilde{\mathcal{X}}_t^{(t,(i)(j))}: = \mathcal{X}_t^{(t,(i)(j))}(h_2)$ be the matrix with $\mathcal{X}_{t,ij}^{(t,(i)(j))}$ replaced by $h_2$. By a similar argument as Lemma \ref{remainderLemma}, together with the bounds in (\ref{070610}), we have the following estimates hold uniformly in $|h_2| \le N^{-\epsilon}$,
\begin{align*}
	\|\tilde{\mathcal{X}}_t^{(t,(i)(j))}\|, \; \big\| \big(\tilde{\mathcal{X}}_t^{(t,(i)(j))} \big(\tilde{\mathcal{X}}_t^{(t,(i)(j))} \big)' \big) ^{-1}\big\| \prec 1.
\end{align*}
Therefore, using chain rule with the fact that $F$ has bounded derivatives, we get
\begin{align*}
	\sup_{|h_2| \le N^{-\epsilon}}\bigg| \partial^5_{t,ij}F\left( L_N^{(t,(i)(j))}(f) \right) \Big|_{(t,i,j) = h_2} \bigg| \prec 1.
\end{align*}
Then with the help of the truncation function $ \Xi^{(t,(i)(j))}$ defined in (\ref{Xtij}), we can conclude that
\begin{align*}
	\E\left[\sup_{|h_2| \le N^{-\epsilon}}\bigg| \partial^5_{t,ij}F\left( L_N^{(t,(i)(j))}(f) \right) \Big|_{(t,i,j) = h_2} \bigg|  \right] = O_{\prec}(1).
\end{align*}
Finally, by the construction of $X_{t,ij}^{(2)}$, we get
\begin{align*}
	|I^{(2)}| \prec N^{-\frac{5}{2}}.
\end{align*}
Similarly, for $I^{(1)}$, we have
\begin{align*}
	|I^{(1)}| \prec  \E \left[ \big|X_{t,ij}^{(1)} \big|^5\right] \le N^{-\epsilon}\E \left[ \big|X_{t,ij}^{(1)} \big|^4\right] = O(N^{-2-\epsilon}). 
\end{align*}
Therefore,
\begin{align*}
	\E \Big[ F\left( L^{(k,(N)(p_k))}_{N}(f)   \right) \Big] -\E \Big[ F\left( L^{(1,(1)(0))}_{N}(f)\right) \Big] = \sum_{t=1}^{k}\sum_{ij}^{(t)} O(N^{-2-\epsilon/2}) = O(N^{-\epsilon/2}).
\end{align*}
By approximation argument, for $F$ in the above class, together with the definition of $L^{(k,(N)(p_k))}_{N}(f)$ and $L^{(1,(1)(0))}_{N}(f)$, we have
\begin{align}
\Big|\E \Big[ F\left(  \Tr f(H^{(1)}) \right) \Big]- \E \big[ F\left( \Tr f(H^{(2)})  \right) \big]\Big|= O(N^{-\epsilon/2}) \label{070701}
\end{align}
The transition from $H^{(1)}$ to $H$ is immediate, we have
\begin{align*}
		&\Big|\E \Big[ F\left(  \Tr f(H^{(1)}) \right) \Big]- \E \big[ F\left( \Tr f(H)  \right) \big]\Big|\le   O\big( \mathbb{P}\big( H^{(1)} \neq H \big) \big) = \sum_{t=1}^{k}\sum_{ij}^{(t)}O\big(\mathbb{P}\big( X_{t,ij}^{(1)} \neq X_{t,ij} \big) \big)= O\left( N^{-c/4} \right).
\end{align*}
Combining with (\ref{070701}), we conclude the proof of Theorem \ref{mainth22}.

\section{Discussion on the contours}\label{Discussion on the contours}
In this section, we argue that by choosing $\epsilon_{1i}$, $\epsilon_{2i}$, $M_{1i}$ and $M_{2i}$, $i=1,2$ appropriately, we can always have that $\{m_{\boxplus}(z): z\in \gamma_1^0\}$ and $ \{m_{\boxplus}(z): z\in \gamma_2^0\}$ are well separated and  the same holds if $(\gamma_1^0,\gamma_2^0)$ is replaced by $(\gamma_1,\gamma_2)$. We consider the following two cases, contours $(\gamma_1,\gamma_2)$ with $\hat{y} \in (0,\infty)$ and contours $(\gamma_1^0,\gamma_2^0)$ with $\hat{y} \in (0,1)$.

The former case is trivial since $|m_{\boxplus}(z)| \sim |z|^{-1}$ as $|z|$ goes to infinity. Therefore, we can set $M_{11}$ and $M_{21}$ much greater than $M_{12}$ and $M_{22}$ such that $|m_{\boxplus}(z_1)|$ and $|m_{\boxplus}(z_2)|$ are of different orders (in $|z|$) if $z_1 \in \gamma_1$ and $z_2 \in \gamma_2$.

Next, we consider the case of contours $(\gamma_1^0,\gamma_2^0)$ with $\hat{y} \in (0,1)$. To show $\{m_{\boxplus}(z): z\in \gamma_1^0\}$ and $ \{m_{\boxplus}(z): z\in \gamma_2^0\}$ are well separated, it suffices to show that for any $z_1 \in \gamma_1^0$ and $z_2 \in \gamma_2^0$, we have $m_{\boxplus}(z_1) \neq m_{\boxplus}(z_2)$. We only consider the following cases while the others are trivial.

Case 1: If $z_1 \in \mathcal{C}_1(\epsilon_{11},\epsilon_{21})$ and $z_2 \in \mathcal{C}_1(\epsilon_{12},\epsilon_{22})$, by setting $\epsilon_{1a}, a = 1,2$ sufficiently small, we have $|m_{\boxplus}(z_a)| \sim \epsilon_{1a}^{-1}$ (since $\mu_{\boxplus}$ has point mass at $0$ by Lemma \ref{supportofmu}). Therefore, we can choose $\epsilon_{11}$ much smaller than $\epsilon_{12}$, i.e. $\epsilon_{11} = \epsilon_{12}^2$, to ensure that $|m_{\boxplus}(z_1)| > |m_{\boxplus}(z_2)|$.

Case 2: If $z_1 \in \mathcal{C}_1(\epsilon_{11},\epsilon_{21})$ and $z_2 \in \mathcal{C}_2(\epsilon_{12},\epsilon_{22}, M_{12})$, using Lemma \ref{supportofmu} together with the fact that $\Re z_2 < 0$, we have
$$
|m_{\boxplus}(z_2)| \le \int \frac{1}{|\lambda - z_2|} {\rm d}\mu_{\boxplus}(\lambda) \le \frac{1}{|z_2|} \le \frac{1}{\epsilon_{12}}.
$$
Therefore, together with Case 1, we can still obtain $|m_{\boxplus}(z_1)| > |m_{\boxplus}(z_2)|$.

Case 3: If $z_1 \in \mathcal{C}_2(\epsilon_{11},\epsilon_{21}, M_{11})$ and $z_2 \in \mathcal{C}_2(\epsilon_{12},\epsilon_{22}, M_{12})$, we compare the imaginary parts of $m_{\boxplus}(z_a), a=1,2$. By the definition of the Stieltjes transform, we have 
$$
\Im m_{\boxplus}(z_a) = \int \frac{\Im z_a}{|\lambda - \Re z_a|^2 + |\Im z_a|^2} {\rm d}\mu_{\boxplus}(\lambda), \qquad a = 1,2.
$$
Let $\Im z_a$ be sufficiently small ($\epsilon_{2a}$ sufficiently close to $\epsilon_{1a}$) with $\Im z_2 \ge \frac{2M_{12}^2}{\epsilon_{11}^2}\Im z_1$, we have
\begin{align*}
	\Im m_{\boxplus}(z_2) \ge \frac{2M_{12}^2}{\epsilon_{11}^2} \int \frac{\Im z_1}{|\lambda - \Re z_2|^2 + |\Im z_2|^2} {\rm d}\mu_{\boxplus}(\lambda) \ge \int \frac{\Im z_1}{|z_1|^2} {\rm d}\mu_{\boxplus}(\lambda)\ge \Im m_{\boxplus}(z_1),
\end{align*}
where in the last step, we used Lemma \ref{supportofmu} and the fact that $\Re z_1 < 0$. Therefore, we can conclude that $m_{\boxplus}(z_1) \neq m_{\boxplus}(z_2)$ in this case.

Case 4: If $z_1 \in \mathcal{C}_2(\epsilon_{11},\epsilon_{21}, M_{11})$ and $z_2 \in \mathcal{C}_1(\epsilon_{12},\epsilon_{22}) \cap \{\Im z \ge \sqrt{\epsilon_{12}^2 - \epsilon_{22}^2}\}$, by Lemma \ref{supportofmu}, we have there exists a constant $c > 0$, such that 
\begin{align*}
	\Im m_{\boxplus}(z_2) \ge c \frac{\Im z_2}{|z_2|^2} = c\frac{\Im z_2}{\epsilon_{12}^2} = c\frac{\Im z_2}{\epsilon_{11}},
\end{align*}
where in the last step we used $\epsilon_{11} = \epsilon_{12}^2$ (Case 1). Therefore, choosing $\Im z_2 \ge \max \{\frac{2M_{12}^2}{\epsilon_{11}^2}, \frac{1}{c\epsilon_{11}} \}\Im z_1$, we arrives at
\begin{align*}
	\Im m_{\boxplus}(z_2) \ge \frac{\Im z_1}{\epsilon_{11}^2} = \int \frac{\Im z_1}{|z_1|^2} {\rm d}\mu_{\boxplus}(\lambda)\ge \Im m_{\boxplus}(z_1).
\end{align*}
This gives $m_{\boxplus}(z_1) \neq m_{\boxplus}(z_2)$ in this case.

If $z_1 \in \mathcal{C}_2(\epsilon_{11},\epsilon_{21}, M_{11})$ and $z_2 \in \mathcal{C}_1(\epsilon_{12},\epsilon_{22}) \cap \{\Im z < \sqrt{\epsilon_{12}^2 - \epsilon_{22}^2}\}$, it is easy to check that $\Re m_{\boxplus}(z_1)$ and $\Re m_{\boxplus}(z_2)$ have different sign, when $\epsilon_{12}$ is sufficiently small. 

For the other cases, the discussion is similar and simpler, namely, we can always compare either the total sizes of $m_{\boxplus}$'s,  or the real parts or imaginary parts of them.  Hence, we omit the details. 

\section{Discussion on Semicircle Law} \label{s.case 3 semicircle}
In this section, we discuss the case when $y\to \infty$ as $N \to \infty$. We first define
\begin{align*}
	\mathcal{H} := \frac{\sum_{t=1}^k(P_t-y_tI)}{\tilde{y}}, \qquad \tilde{y}:=\sqrt{y - \sum_t^k y_t^2}.
\end{align*}
The corresponding Green function can be written as
\begin{align*}
	\mathcal{G}(z) := (\mathcal{H} - z)^{-1}.
\end{align*}
Further, with $\tilde{z} := y+z\tilde{y}$, we have
\begin{align}
	\mathcal{G}(z) = \left( \frac{\sum_{t=1}^k(P_t-y_tI)}{\tilde{y}} - z\right)^{-1} = \tilde{y}\left(\sum_{t=1}^kP_t -\tilde{z}\right)^{-1} = \tilde{y}G(\tilde{z}) \label{070820}
\end{align}
By the method of recursive moment estimation, one can show the following estimates
\begin{align}
	&\tr\bbP_tG(\tilde{z})\bbP_t-(1-y_t)\tr\bbQ_tG(\tilde{z})=O_{\prec}(\frac{1}{\tilde{y}}\sqrt{\frac{p_t}{N^3}})\label{TrPGPy}\\
&\tr\bbP_tG(\tilde{z})-(\tr(\bbX_t\bbX_t')^{-1}-\tr\bbQ_tG(\tilde{z}))(\tr G(\tilde{z})-\tr\bbP_tG(\tilde{z}))=O_{\prec}\left(\frac{1}{\tilde{y}}\sqrt{\frac{p_t}{N^3}}\right)\label{TrPGy}
\end{align}
Notice that for any fixed $z \in \mathbb{C}^{+}$, we have $\| G(\tilde{z})\| = \| \mathcal{G}(z) \| / \tilde{y} \sim \tilde{y}^{-1}$. Plugging (\ref{trXX}) and (\ref{TrPGPy}) into (\ref{TrPGPy}), we have
\begin{align}
	\tr\bbP_tG(\tilde{z})-\frac{y_t -\tr P_t G(\tilde{z})}{1-y_t}(\tr G(\tilde{z})-\tr\bbP_tG(\tilde{z}))=O_{\prec}\left(\frac{1}{\tilde{y}}\sqrt{\frac{p_t}{N^3}}\right). \label{062626}
\end{align}
Multiplying both sides by $(1-y_t)$, and then summing over $t$, we have
\begin{align}
	\tr HG(\tilde{z}) - y\tr G(\tilde{z})-\sum_{t=1}^k (\tr P_tG(\tilde{z}))^2 + \tr HG(\tilde{z}) \tr G(\tilde{z}) = O_{\prec}\left(\max_i \frac{\tilde{y}}{\sqrt{Np_i}} \right). \label{070801}
\end{align}

Next, we present an approximation for $\sum_{t=1}^k(\tr P_tG(\tilde{z}))^2$. By (\ref{062626}), we have
\begin{align*}
	\tr P_tG(\tilde{z}) -y_t\tr G(\tilde{z})  = (\tr P_tG(\tilde{z}))^2 - \tr P_tG(\tilde{z}) \tr G(\tilde{z})  + O_{\prec}\left(\frac{1-y_t}{\tilde{y}}\sqrt{\frac{p_t}{N^3}} \right) =  O_{\prec}\left(\frac{1-y_t}{\tilde{y}}\sqrt{\frac{p_t}{N^3}} \vee \frac{y_t}{\tilde{y}^2} \right).
\end{align*}
Therefore, we can have the following estimates,
\begin{align*}
	\left|\sum_{t=1}^k (\tr P_tG(\tilde{z}))^2  - \sum_{t=1}^ky_t^2\tr G(\tilde{z})\right| \leq&  \sum_{t=1}^k\left|\tr P_t G(\tilde{z}) - y_t\tr G(\tilde{z})\right|\left|\tr P_t G(\tilde{z}) + y_t\tr G(\tilde{z})\right| \\
	\prec & \sum_{t=1}^k\frac{y_t}{\tilde{y}} \left|\tr P_t G(\tilde{z}) - y_t\tr G(\tilde{z})\right| = O_{\prec}\left( \tilde{y}^2\sqrt{\frac{p_{\max}}{N^3}}\vee \frac{1}{\tilde{y}}  \right).
\end{align*}
Plugging the above estimates back into (\ref{070801}), we have
\begin{align*}
	\tr HG(\tilde{z}) - y\tr G(\tilde{z})- \sum_{t=1}^ky_t^2\tr G(\tilde{z})+ \tr HG(\tilde{z}) \tr G(\tilde{z}) = O_{\prec}\left(\max_i \frac{\tilde{y}}{\sqrt{Np_i}} \vee \tilde{y}^2\sqrt{\frac{p_{\max}}{N^3}}\vee \frac{1}{\tilde{y}} \right).
\end{align*}
Further simplifying the LHS of the above equation using the trivial identity $\tr HG(\tilde{z}) = 1+ \tilde{z}\tr G(\tilde{z})$ together with (\ref{070820}), we have
\begin{align*}
	1 + z\tr \mathcal{G}(z) + (\tr \mathcal{G}(z))^2 =&  -\tr G(\tilde{z})(1 + z\tr \mathcal{G}(z)) + O_{\prec}\left(\max_i \frac{\tilde{y}}{\sqrt{Np_i}} \vee \tilde{y}^2\sqrt{\frac{p_{\max}}{N^3}}\vee \frac{1}{\tilde{y}} \right)\\
	=&O_{\prec}\left(\max_i \frac{\tilde{y}}{\sqrt{Np_i}} \vee \tilde{y}^2\sqrt{\frac{p_{\max}}{N^3}}\vee \frac{1}{\tilde{y}} \right).
\end{align*}
If $k \sim N^{1-\epsilon}$ for some $\epsilon > 0$, we have
\begin{align*}
	1 + z\tr \mathcal{G}(z) + (\tr \mathcal{G}(z))^2 = O_{\prec}\left( \frac{1}{N^{\epsilon/2}}\vee \frac{1}{\tilde{y}} \right).
\end{align*}
Which implies that 
\begin{align*}
	|m_{sc}(z) - \tr \mathcal{G}(z) | = O_{\prec}\left( \frac{1}{N^{\epsilon/2}}\vee \frac{1}{\tilde{y}} \right),
\end{align*}
where $m_{sc}(z)$ is the Stieltjes transform of the semicircle law. Let $\hat{\mu}_N$ be the ESD of $\mathcal{H}$. By the continuity theorem of Stieltjes transform, we conclude that $\hat{\mu}_{N}$ converges weakly in probability to $\mu_{sc}$, where
\begin{align*}
	\mu_{sc} := \frac{1}{2\pi}\sqrt{(4-x^2)_+}{\rm d}x,
\end{align*}
under the assumption that $k \sim N^{1-\epsilon}$ for some small $\epsilon > 0$, $y\to \infty$ as $N \to \infty$, and Assumption \ref{assum1}.
\end{appendix}

\bibliographystyle{plain}

\end{document}